\documentclass[a4paper,10pt]{article}
\usepackage[utf8]{inputenc}
\usepackage[T1]{fontenc}
\usepackage{geometry}
\usepackage{tabularx}
\usepackage{booktabs}
\usepackage{array}
\usepackage{microtype}
\usepackage{setspace}
\usepackage{xcolor}
\definecolor{lightgreen}{rgb}{.20,.60,.22}
\usepackage{pdfpages}
\usepackage[normalem]{ulem}

\usepackage{amssymb,amsmath,amsopn,amsxtra,amsthm,amsfonts}
\usepackage{mathtools}
\usepackage[mathcal]{euscript}
\let\mathscr\mathcal
\usepackage[bb=px]{mathalfa}

\usepackage{tikz}
\usetikzlibrary{matrix}
\usepackage{tikz-cd}
\usepackage{enumitem}
\usepackage{xspace}
\usepackage{xifthen}
\usepackage{xparse}

\usepackage[
backend=biber,
style=alphabetic,
maxnames=99,        
maxalphanames=99,   
minalphanames=99    
]{biblatex}
\setlist[enumerate,1]{label={(\arabic*)},itemsep=\parskip} 
\setlist[itemize,1]{itemsep=\parskip} 

\newlist{thmlist}{enumerate}{2}
\setlist[thmlist,1]{label={\em(\roman*)},ref={(\roman*)},%
	itemsep=\parskip,leftmargin=*,align=left}
\setlist[thmlist,2]{label={\em(\alph*)},ref={(\alph*)},%
	itemsep=\parskip,leftmargin=*,align=left,topsep=0.1cm}

\newlist{defnlist}{enumerate}{2}
\setlist[defnlist,1]{label={(\roman*)},ref={(\roman*)},itemsep=\parskip,%
	leftmargin=*,align=left}
\setlist[defnlist,2]{label={(\alph*)},ref={(\alph*)},itemsep=\parskip,%
	leftmargin=*,align=left,topsep=0.1cm}

\DeclareFontFamily{U}{min}{}
\DeclareFontShape{U}{min}{m}{n}{<-> udmj30}{}

\newcommand{\nc}{\newcommand}
\nc{\renc}{\renewcommand}
\nc{\ssec}{\subsection}
\nc{\sssec}{\subsubsection}
\nc{\on}{\operatorname}
\nc{\term}[1]{#1\xspace}

\nc{\sA}{\ensuremath{\mathcal{A}}\xspace}
\nc{\sB}{\ensuremath{\mathcal{B}}\xspace}
\nc{\sC}{\ensuremath{\mathcal{C}}\xspace}
\nc{\sD}{\ensuremath{\mathcal{D}}\xspace}
\nc{\sE}{\ensuremath{\mathcal{E}}\xspace}
\nc{\sF}{\ensuremath{\mathcal{F}}\xspace}
\nc{\sG}{\ensuremath{\mathcal{G}}\xspace}
\nc{\sH}{\ensuremath{\mathcal{H}}\xspace}
\nc{\sI}{\ensuremath{\mathcal{I}}\xspace}
\nc{\sJ}{\ensuremath{\mathcal{J}}\xspace}
\nc{\sK}{\ensuremath{\mathcal{K}}\xspace}
\nc{\sL}{\ensuremath{\mathcal{L}}\xspace}
\nc{\sM}{\ensuremath{\mathcal{M}}\xspace}
\nc{\sN}{\ensuremath{\mathcal{N}}\xspace}
\nc{\sO}{\ensuremath{\mathcal{O}}\xspace}
\nc{\sP}{\ensuremath{\mathcal{P}}\xspace}
\nc{\sQ}{\ensuremath{\mathcal{Q}}\xspace}
\nc{\sR}{\ensuremath{\mathcal{R}}\xspace}
\nc{\sS}{\ensuremath{\mathcal{S}}\xspace}
\nc{\sT}{\ensuremath{\mathcal{T}}\xspace}
\nc{\sU}{\ensuremath{\mathcal{U}}\xspace}
\nc{\sV}{\ensuremath{\mathcal{V}}\xspace}
\nc{\sW}{\ensuremath{\mathcal{W}}\xspace}
\nc{\sX}{\ensuremath{\mathcal{X}}\xspace}
\nc{\sY}{\ensuremath{\mathcal{Y}}\xspace}
\nc{\sZ}{\ensuremath{\mathcal{Z}}\xspace}

\nc{\bA}{\ensuremath{\mathbf{A}}\xspace}
\nc{\bB}{\ensuremath{\mathbf{B}}\xspace}
\nc{\bC}{\ensuremath{\mathbf{C}}\xspace}
\nc{\bD}{\ensuremath{\mathbf{D}}\xspace}
\nc{\bE}{\ensuremath{\mathbf{E}}\xspace}
\nc{\bF}{\ensuremath{\mathbf{F}}\xspace}
\nc{\bG}{\ensuremath{\mathbf{G}}\xspace}
\nc{\bH}{\ensuremath{\mathbf{H}}\xspace}
\nc{\bI}{\ensuremath{\mathbf{I}}\xspace}
\nc{\bJ}{\ensuremath{\mathbf{J}}\xspace}
\nc{\bK}{\ensuremath{\mathbf{K}}\xspace}
\nc{\bL}{\ensuremath{\mathbf{L}}\xspace}
\nc{\bM}{\ensuremath{\mathbf{M}}\xspace}
\nc{\bN}{\ensuremath{\mathbf{N}}\xspace}
\nc{\bO}{\ensuremath{\mathbf{O}}\xspace}
\nc{\bP}{\ensuremath{\mathbf{P}}\xspace}
\nc{\bQ}{\ensuremath{\mathbf{Q}}\xspace}
\nc{\bR}{\ensuremath{\mathbf{R}}\xspace}
\nc{\bS}{\ensuremath{\mathbf{S}}\xspace}
\nc{\bT}{\ensuremath{\mathbf{T}}\xspace}
\nc{\bU}{\ensuremath{\mathbf{U}}\xspace}
\nc{\bV}{\ensuremath{\mathbf{V}}\xspace}
\nc{\bW}{\ensuremath{\mathbf{W}}\xspace}
\nc{\bX}{\ensuremath{\mathbf{X}}\xspace}
\nc{\bY}{\ensuremath{\mathbf{Y}}\xspace}
\nc{\bZ}{\ensuremath{\mathbf{Z}}\xspace}

\nc{\dA}{\ensuremath{\mathds{A}}\xspace}
\nc{\dB}{\ensuremath{\mathds{B}}\xspace}
\nc{\dC}{\ensuremath{\mathds{C}}\xspace}
\nc{\dD}{\ensuremath{\mathds{D}}\xspace}
\nc{\dE}{\ensuremath{\mathds{E}}\xspace}
\nc{\dF}{\ensuremath{\mathds{F}}\xspace}
\nc{\dG}{\ensuremath{\mathds{G}}\xspace}
\nc{\dH}{\ensuremath{\mathds{H}}\xspace}
\nc{\dI}{\ensuremath{\mathds{I}}\xspace}
\nc{\dJ}{\ensuremath{\mathds{J}}\xspace}
\nc{\dK}{\ensuremath{\mathds{K}}\xspace}
\nc{\dL}{\ensuremath{\mathds{L}}\xspace}
\nc{\dM}{\ensuremath{\mathds{M}}\xspace}
\nc{\dN}{\ensuremath{\mathds{N}}\xspace}
\nc{\dO}{\ensuremath{\mathds{O}}\xspace}
\nc{\dP}{\ensuremath{\mathds{P}}\xspace}
\nc{\dQ}{\ensuremath{\mathds{Q}}\xspace}
\nc{\dR}{\ensuremath{\mathds{R}}\xspace}
\nc{\dS}{\ensuremath{\mathds{S}}\xspace}
\nc{\dT}{\ensuremath{\mathds{T}}\xspace}
\nc{\dU}{\ensuremath{\mathds{U}}\xspace}
\nc{\dV}{\ensuremath{\mathds{V}}\xspace}
\nc{\dW}{\ensuremath{\mathds{W}}\xspace}
\nc{\dX}{\ensuremath{\mathds{X}}\xspace}
\nc{\dY}{\ensuremath{\mathds{Y}}\xspace}
\nc{\dZ}{\ensuremath{\mathds{Z}}\xspace}

\nc{\bbA}{\ensuremath{\mathbb{A}}\xspace}
\nc{\bbB}{\ensuremath{\mathbb{B}}\xspace}
\nc{\bbC}{\ensuremath{\mathbb{C}}\xspace}
\nc{\bbD}{\ensuremath{\mathbb{D}}\xspace}
\nc{\bbE}{\ensuremath{\mathbb{E}}\xspace}
\nc{\bbF}{\ensuremath{\mathbb{F}}\xspace}
\nc{\bbG}{\ensuremath{\mathbb{G}}\xspace}
\nc{\bbH}{\ensuremath{\mathbb{H}}\xspace}
\nc{\bbI}{\ensuremath{\mathbb{I}}\xspace}
\nc{\bbJ}{\ensuremath{\mathbb{J}}\xspace}
\nc{\bbK}{\ensuremath{\mathbb{K}}\xspace}
\nc{\bbL}{\ensuremath{\mathbb{L}}\xspace}
\nc{\bbM}{\ensuremath{\mathbb{M}}\xspace}
\nc{\bbN}{\ensuremath{\mathbb{N}}\xspace}
\nc{\bbO}{\ensuremath{\mathbb{O}}\xspace}
\nc{\bbP}{\ensuremath{\mathbb{P}}\xspace}
\nc{\bbQ}{\ensuremath{\mathbb{Q}}\xspace}
\nc{\bbR}{\ensuremath{\mathbb{R}}\xspace}
\nc{\bbS}{\ensuremath{\mathbb{S}}\xspace}
\nc{\bbT}{\ensuremath{\mathbb{T}}\xspace}
\nc{\bbU}{\ensuremath{\mathbb{U}}\xspace}
\nc{\bbV}{\ensuremath{\mathbb{V}}\xspace}
\nc{\bbW}{\ensuremath{\mathbb{W}}\xspace}
\nc{\bbX}{\ensuremath{\mathbb{X}}\xspace}
\nc{\bbY}{\ensuremath{\mathbb{Y}}\xspace}
\nc{\bbZ}{\ensuremath{\mathbb{Z}}\xspace}

\nc{\mrm}[1]{\ensuremath{\mathrm{#1}}\xspace}
\nc{\mbf}[1]{\ensuremath{\mathbf{#1}}\xspace}
\nc{\mcal}[1]{\ensuremath{\mathcal{#1}}\xspace}
\nc{\msc}[1]{\ensuremath{\mathscr{#1}}\xspace}
\nc{\mfr}[1]{\ensuremath{\mathfrak{#1}}\xspace}

\renc{\bar}[1]{\overline{#1}}

\let\S\relax

\nc{\sub}{\subset}
\nc{\too}{\longrightarrow}
\nc{\hook}{\hookrightarrow}
\nc*{\hooklongrightarrow}{\ensuremath{\lhook\joinrel\relbar\joinrel\rightarrow}}
\nc{\hooklong}{\hooklongrightarrow}
\nc{\twoheadlongrightarrow}{\relbar\joinrel\twoheadrightarrow}
\nc{\shiso}{\approx}
\nc{\isoto}{\xrightarrow{\sim}}
\nc{\isofrom}{\xleftarrow{\sim}}
\renc{\ge}{\geqslant}
\renc{\le}{\leqslant}

\nc{\id}{\mathrm{id}}

\DeclareMathOperator{\Hom}{\on{Hom}}
\nc{\uHom}{\underline{\smash{\Hom}}}

\DeclareMathOperator{\End}{\on{End}}

\nc{\uEnd}{\underline{\smash{\End}}}

\renc{\lim}{\varprojlim}

\makeatletter
\newcommand{\colim@}[2]{%
	\vtop{\m@th\ialign{##\cr
			\hfil$#1\operator@font colim$\hfil\cr
			\noalign{\nointerlineskip\kern1.5\ex@}#2\cr
			\noalign{\nointerlineskip\kern-\ex@}\cr}}%
}
\newcommand{\colim}{%
	\mathop{\mathpalette\colim@{\rightarrowfill@\textstyle}}\nmlimits@
}
\makeatother

\nc{\Cofib}{\on{Cofib}}
\nc{\Fib}{\on{Fib}}
\nc{\initial}{\varnothing}
\newcommand{\opp}{\mathrm{op}}

\usepackage{url}
\usepackage{hyperref}
\hypersetup{
	colorlinks=true,
	linkcolor=purple,
	citecolor=lightgreen,
	urlcolor=cyan,
}

\usepackage[capitalise,nameinlink]{cleveref}

\newtheorem{ThmAlpha}{Theorem}
\newtheorem{thm}{Theorem}[section]
\newtheorem{cor}[thm]{Corollary}
\newtheorem{lem}[thm]{Lemma}
\newtheorem{prop}[thm]{Proposition}

\newtheorem{qu}[thm]{Question}

\theoremstyle{definition}

\newtheorem{de}[thm]{Definition}

\newtheorem{rem}[thm]{Remark}
\newtheorem{warning}[thm]{Warning}

\newtheorem{ex}[thm]{Example}
\newtheorem{exam}[thm]{Example}

\newtheorem{nota}[thm]{Notation}

\newtheorem{warn}[thm]{Warning}
\newtheorem{cov}[thm]{Convention}

\renewcommand{\eqref}[1]{(\ref{#1})}

\numberwithin{equation}{subsection}

\nc{\Spc}{\mrm{Spc}}
\nc{\Spt}{\mrm{Spt}}
\nc{\Spec}{\on{Spec}}
\nc{\Stk}{\mrm{Stk}}
\nc{\Sch}{\mrm{Sch}}
\nc{\aff}{\mrm{aff}}
\nc{\A}{\mbf{A}}
\renc{\P}{\mbf{P}}
\nc{\cl}{{\mrm{cl}}}
\nc{\bDelta}{\mathbf{\Delta}}
\nc{\un}{\mathbf{1}}
\nc{\Tot}{\on{Tot}}
\nc{\Cech}{\textnormal{\v{C}}}
\nc{\Mod}{\mrm{Mod}}
\nc{\Qcoh}{\on{Qcoh}}
\nc{\free}{\mrm{free}}

\nc{\perf}{\mrm{perf}}
\nc{\aperf}{\mrm{aperf}}
\nc{\coh}{\mrm{coh}}
\newcommand{\Cat}{\mrm{Cat}}
\nc{\unitm}{\mbf{1}}
\nc{\sphere}{\mbf{S}}
\nc{\Z}{\mbf{Z}}
\nc{\Map}{\mrm{Map}}
\nc{\map}{\mrm{map}}
\nc{\PrL}{\mathcal{P}r^\mrm{L}}
\nc{\PrLst}{\mathcal{P}r^\mrm{L}_\mrm{St}}
\nc{\Motnc}{\mathcal{M}_{\mrm{loc}}}
\nc{\Motadd}{\mathcal{M}_{\mrm{add}}}
\nc{\Einfty}{{\sE_\infty}}
\nc{\E}[1]{{\sE_{#1}}}
\nc{\modmod}{/\!\!/}
\nc{\heart}{\heartsuit}
\nc{\proj}{\mrm{proj}}
\nc{\LL}{\on{L}}
\nc{\K}{\on{K}}
\nc{\G}{\on{G}}
\nc{\GL}{\on{GL}}
\nc{\BGL}{\on{BGL}}
\nc{\M}{\on{M}}
\nc{\KH}{\on{KH}}
\nc{\Alg}{\on{Alg}}
\nc{\CAlg}{\on{CAlg}}
\nc{\cn}{\mrm{cn}}
\nc{\hw}{\mrm{Hw}}
\nc{\htt}{\mrm{Ht}}
\nc{\Fun}{\on{Fun}}
\nc{\Funadd}{\on{Fun}_{\mrm{add}}}
\nc{\Funex}{\on{Fun}_{\mrm{ex}}}
\nc{\Ind}{\on{Ind}}
\nc{\Pro}{\on{Pro}}
\nc{\Kar}{\on{Kar}}
\nc{\Obj}{\on{Obj}}

\nc{\scr}{\term{simplicial commutative ring}}
\nc{\scrs}{\term{simplicial commutative rings}}

\nc{\Einfring}{\term{$\Einfty$-ring}}
\nc{\Einfrings}{\term{$\Einfty$-rings}}

\nc{\Ering}{\term{$\sE_1$-ring}}
\nc{\Erings}{\term{$\sE_1$-rings}}

\nc{\inftyCat}{\term{$\infty$-category}}
\nc{\inftyCats}{\term{$\infty$-categories}}

\nc{\inftyTop}{\term{$\infty$-topos}}
\nc{\inftyTops}{\term{$\infty$-toposes}}

\nc{\inftyGrpd}{\term{$\infty$-groupoid}}
\nc{\inftyGrpds}{\term{$\infty$-groupoids}}

\def\mc{\mathcal}
\def\mb{\mathbf}

\def\op{\mathrm}
\def\ein{\term{$\mathbb{E}_{\infty}$}}
\long\def\enu#1{%
	\begin{enumerate}[label=(\arabic*),font=\normalfont]
		#1
	\end{enumerate}
}

\newcommand{\alg}{\operatorname{Alg}}
\newcommand{\calg}{\operatorname{CAlg}}
\newcommand{\ccalg}{\operatorname{cCAlg}}

\newcommand{\cmon}{\operatorname{CMon}}
\newcommand{\modu}{\operatorname{Mod}}
\newcommand{\lmodu}{\operatorname{LMod}}
\newcommand{\rmodu}{\operatorname{RMod}}
\newcommand{\amodu}{\operatorname{aMod}}
\newcommand{\moducn}{\ensuremath{\operatorname{Mod}^{\operatorname{cn}}}\xspace}
\newcommand{\modurcn}{\ensuremath{\operatorname{Mod}_{R}^{\operatorname{cn}}}\xspace}
\newcommand{\mapp}{\operatorname{Map}}
\newcommand{\unmap}{\underline{\operatorname{Map}}}
\newcommand{\udmap}{\underline{\operatorname{Map}}}

\newcommand{\colimit}{\mathrm{colim}}
\newcommand{\funct}{\operatorname{Fun}}

\newcommand{\catadidem}{\ensuremath{\mathrm{Cat}_{\mathrm{ad}}^{\mathrm{idem}}}\xspace}
\newcommand{\prl}{\ensuremath{\mc{P}\mathrm{r}^L}\xspace}
\newcommand{\prlv}{\ensuremath{\mc{P}\mathrm{r}^L_{\mathcal{V}}}\xspace}
\newcommand{\prlw}{\ensuremath{\mc{P}\mathrm{r}^L_{\mathcal{W}}}\xspace}

\newcommand{\prv}{\ensuremath{\mc{P}\mathrm{r}_{\mathcal{V}}}\xspace}

\newcommand{\pr}{\ensuremath{\mc{P}\mathrm{r}}\xspace}
\newcommand{\prad}{\ensuremath{\mc{P}\mathrm{r}_{\op{ad}}}\xspace}
\newcommand{\prlad}{\ensuremath{\mc{P}\mathrm{r}_{\op{ad}}^L}\xspace}
\newcommand{\praddbl}{\ensuremath{\mc{P}\mathrm{r}_{\op{ad}}^{\op{dbl}}}\xspace}

\newcommand{\pradone}{\ensuremath{\mc{P}\mathrm{r}_{\op{ad},1}}\xspace}
\newcommand{\prladone}{\ensuremath{\mc{P}\mathrm{r}^L_{\op{ad},1}}\xspace}
\newcommand{\prlst}{\ensuremath{\mc{P}\mathrm{r}_{\op{st}}^L}\xspace}

\newcommand{\prvdbl}{\ensuremath{\mc{P}\mathrm{r}_{\mc{V}}^{\op{dbl}}}\xspace}
\newcommand{\prvat}{\ensuremath{\mc{P}\mathrm{r}_{\mc{V}}^{\op{at}}}\xspace}
\newcommand{\prwdbl}{\ensuremath{\mc{P}\mathrm{r}_{\mc{W}}^{\op{dbl}}}\xspace}

\newcommand{\prcdbl}{\ensuremath{\mc{P}\mathrm{r}_{\mc{C}}^{\op{dbl}}}\xspace}

\newcommand{\prdbl}{\ensuremath{\mc{P}\mathrm{r}^{\op{dbl}}}\xspace}
\newcommand{\prat}{\ensuremath{\mc{P}\mathrm{r}^{\op{at}}}\xspace}

\newcommand{\prvil}{\ensuremath{\mc{P}\mathrm{r}^{iL}_{\mcv}}\xspace}
\newcommand{\cofib}{\operatorname{cofib}}
\newcommand{\fib}{\operatorname{fib}}
\newcommand{\spgeq}{\ensuremath{\operatorname{Sp}_{\geq0}}\xspace}

\newcommand{\eone}{\term{\mathbb{E}_1}}

\newcommand{\einfring}{\term{$\mathbb E_\infty$-ring}}
\newcommand{\einfrings}{\term{$\mathbb E_\infty$-rings}}

\newcommand{\infcat}{\term{$\infty$-category}}
\newcommand{\infcats}{\term{$\infty$-categories}}
\newcommand{\syminfcat}{\term{symmetric monoidal $\infty$-category}}
\newcommand{\syminfcats}{\term{symmetric monoidal $\infty$-categories}}
\newcommand{\dualaddinfcat}{\term{dualizable additive $\infty$-category}}
\newcommand{\dualaddinfcats}{\term{dualizable additive $\infty$-categories}}

\newcommand{\shv}{\operatorname{Shv}}

\newcommand{\shvhyp}{\operatorname{Shv}^{\wedge}}

\newcommand{\idl}{\op{Idl}}

\newcommand{\idlc}{\op{Idl}(\mc{C})}

\newcommand{\im}{\op{Im}}

\newcommand{\mca}{\ensuremath{\mathcal{A}}\xspace}
\newcommand{\mcb}{\ensuremath{\mathcal{B}}\xspace}
\newcommand{\mcc}{\ensuremath{\mathcal{C}}\xspace}
\newcommand{\mcd}{\ensuremath{\mathcal{D}}\xspace}

\newcommand{\mcf}{\ensuremath{\mathcal{F}}\xspace}

\newcommand{\mch}{\ensuremath{\mathcal{H}}\xspace}
\newcommand{\mci}{\ensuremath{\mathcal{I}}\xspace}

\newcommand{\mck}{\ensuremath{\mathcal{K}}\xspace}
\newcommand{\mcl}{\ensuremath{\mathcal{L}}\xspace}
\newcommand{\mcm}{\ensuremath{\mathcal{M}}\xspace}
\newcommand{\mcn}{\ensuremath{\mathcal{N}}\xspace}
\newcommand{\mco}{\ensuremath{\mathcal{O}}\xspace}
\newcommand{\mcp}{\ensuremath{\mathcal{P}}\xspace}
\newcommand{\mcq}{\ensuremath{\mathcal{Q}}\xspace}

\newcommand{\mcs}{\ensuremath{\mathcal{S}}\xspace}

\newcommand{\mcu}{\ensuremath{\mathcal{U}}\xspace}
\newcommand{\mcv}{\ensuremath{\mathcal{V}}\xspace}
\newcommand{\mcw}{\ensuremath{\mathcal{W}}\xspace}
\newcommand{\mcx}{\ensuremath{\mathcal{X}}\xspace}

\newcommand{\mcz}{\ensuremath{\mathcal{Z}}\xspace}

\newcommand{\spec}{\op{Spec}}

\newcommand{\frmcoh}{\mathsf{Frm}_{\mathsf{coh}}}
\newcommand{\frm}{\mathsf{Frm}}
\newcommand{\pfrm}{\mathsf{PFrm}}

\newcommand{\coker}{\op{coker}}
\newcommand{\catper}{\op{Cat}^{\op{perf}}}

\newcommand{\opsp}{\ensuremath{\op{Sp}}\xspace}
\newcommand{\sppcpl}{\ensuremath{\op{Sp}_{p\text{-}\op{cpl}}}\xspace}
\newcommand{\idem}{\ensuremath{\op{idem}}\xspace}
\newcommand{\Idem}{\ensuremath{\op{Idem}}\xspace}
\newcommand{\cidem}{\ensuremath{\op{cIdem}}\xspace}
\newcommand{\cidemc}{\ensuremath{\op{cIdem}(\mathcal{C})}\xspace}

\newcommand{\nuc}{\ensuremath{\op{Nuc}}\xspace}

\begin{document}

\title{Generalized Telescope Conjecture}
\author{Jiacheng Liang}

\date{}
\let\mathbb=\mathbf

\maketitle

\setstretch{1.1}
\setcounter{tocdepth}{2}

{\renewcommand{\thefootnote}{}\footnotetext{Date: October 4, 2026}}

\vspace{-1em}
\begin{abstract}
	We extend Balmer's tensor-triangular geometry to arbitrary presentably symmetric monoidal \(\infty\)-categories \(\mathcal V\) by introducing the atomic smashing frame. This construction is functorial in \(\mathcal V\), recovers the classical Balmer spectrum in the stable compactly-rigidly generated case, and yields a formulation of the telescope conjecture for arbitrary \(\mathcal V\).

	Using dualizable and locally rigid \(\infty\)-categories, we identify
	smashing ideals with locally rigid localizations and establish recollement results for smashing frames and atomic smashing frames. As a consequence, we obtain a recollement principle for the
	telescope conjecture. We also establish descent for the telescope conjecture along several natural classes of coverings.
	
	 As a
	principal application to chromatic homotopy theory, we analyze the structure of the smashing
	frame of $\mathrm{Sp}$. We prove a prime decomposition in terms of the smashing
	frames of the \(p\)-local $\infty$-categories \(\mathrm{Sp}_{(p)}\) and, at each
	prime, an embedding into the product of the smashing frames of the
	monochromatic layers \(\mathrm{Sp}_{T(n)}\). In particular, the spatiality of the smashing frame of $\mathrm{Sp}$ reduces entirely to that of $\mathrm{Sp}_{T(n)}$.
	
	Beyond the stable setting, we characterize the telescope conjecture for $\infty$-topoi in terms of smashing fields, and for connective module categories and hypercomplete connective sheaves in terms of Pierce-type conditions. Finally, we introduce the Serre smashing frame. Over a connective $\mathbb{E}_\infty$-ring $R$, this frame sits between the atomic and usual smashing frames, providing an intermediate structural layer in the study of the telescope conjecture for the connective $R$-module $\infty$-category.
\end{abstract}

\tableofcontents

\section{Introduction}
\subsection{Background and motivation}
The telescope conjecture is one of the organizing problems of chromatic
homotopy theory. Fix a prime \(p\). For every height \(n\), a finite
\(p\)-local spectrum of type \(n\) admits a \(v_n\)-self map, whose telescope
defines the corresponding telescopic theory \(T(n)\). The telescope
conjecture asserts that localization with respect to \(T(n)\) agrees with
localization with respect to Morava \(K\)-theory \(K(n)\).
The conjecture holds at heights \(n=0\) and \(n=1\), by
\cite{mahowald1981bo,miller1981relations}. After remaining open for several
decades, it was recently disproved at every height \(n\geq 2\) by Burklund,
Hahn, Levy, and Schlank \cite{burklund2023k}. 

The failure of the classical telescope conjecture for $\opsp$ motivates the development of a broader categorical framework to precisely quantify this telescopic gap. Beyond its chromatic origins, this gap has a natural interpretation in tensor-triangular geometry. Let
\(\mcc\) be a compactly-rigidly generated big \(tt\)-category. The Balmer
frame \cite{Balmer05a}, given by the thick ideals of
\(\mcc^\omega\), admits a canonical embedding into the smashing frame
\cite{BalmerKrauseStevenson20}:
\[
\Idem_f(\mcc)\hookrightarrow\Idem(\mcc).
\]
Under this identification, the telescope conjecture asks precisely whether the Balmer frame already accounts for every smashing
localization.
While the telescope conjecture fails for $\opsp$, it has been shown to hold in several algebraic contexts. Most notably, Neeman \cite{neeman1992chromatic} established it for derived categories of commutative Noetherian rings, a result later generalized by Antieau and Stevenson \cite{AS16} to representations of quivers over commutative Noetherian rings, fully proving the conjecture for Dynkin quivers.

This frame-theoretic viewpoint suggests that, in order to understand the
telescope gap in a broader picture, one should first identify the appropriate
analogue of Balmer geometry beyond the compactly-rigidly generated setting.
This leads to the following basic question.
\begin{qu}\label{qu2}
	Does there exist a natural extension of the Balmer frame to an arbitrary $\mcv \in \calg(\prl)$ such that:
	\enu{
		\item it depends functorially on $\mcv$, inducing a functor $\calg(\prl) \to \frm$;
		\item it admits a natural embedding into the smashing frame of $\mcv$;
		\item it recovers the classical Balmer frame when $\mcv$ is stable and compactly-rigidly generated?
	}
\end{qu}
The central observation is that ordinary compactness is not the correct
notion outside the compactly-rigidly generated setting. For a general
\(\mcv\)-module one must distinguish absolute compactness from compactness
relative to the base \(\mcv\). The latter is captured by the notion of a
\(\mcv\)-atomic object. In the special case of \(\mcv\) itself, the
\(\mcv\)-atomic objects are precisely its dualizable objects. Consequently,
the natural analogue of a compactly generated smashing ideal is a smashing
ideal generated by dualizable objects.

The first purpose of this paper is to show that this observation produces
the desired extension of Balmer geometry.

\subsection{Main results}
\subsubsection*{Generalized Balmer geometry}

Outside the stable setting, idempotent objects themselves do not form the
appropriate smashing frame. Aoki observed that the correct replacement is
the frame of coidempotent objects \cite{aoki2023sheaves}. For
$
\mcv\in\calg(\prl),
$
we write
$
\cidem(\mcv)
$
for this frame and call it the \textbf{smashing frame} of \(\mcv\).
In the stable setting coidempotent and idempotent objects agree, so this
recovers the usual smashing frame.

In joint work of the author \cite{Zariski_Bal}, smashing ideals are related
to dualizable \infcats by a natural equivalence
\[
\cidem(\mcv)
\xrightarrow[\sim]{x\mapsto\langle x\rangle}
\cidem(\prvdbl).
\]
Here \(\prvdbl\) denotes the symmetric monoidal \(\infty\)-category of
dualizable \(\mcv\)-modules and internal left adjoints.

Guided by this correspondence, the atomic smashing frame naturally emerges, as follows.
\begin{prop}[{\ref{atsmashingideals}}]
	Let $\mcv\in\calg(\prl)$. The image of the canonical embedding 
	$$\cidem(\prvat)\hookrightarrow\cidem(\prvdbl)$$ 
	can be identified with the $\mcv$-atomically generated smashing ideals of \mcv, where 
	$
	\prvat\subset\prvdbl
	$
	denote the full symmetric monoidal subcategory spanned by the atomically
	generated \(\mcv\)-modules. In particular, the subposet of atomically generated smashing ideals forms a subframe.
\end{prop}
\begin{de}
Let \(\mcv \in \calg(\prl)\). A smashing ideal
\(\mci \hookrightarrow \mcv\) is called an \textbf{atomic smashing ideal} if
it is \(\mcv\)-atomically generated, or equivalently, if it is generated as a
\(\mcv\)-submodule by dualizable objects of \(\mcv\).
	We define the \textbf{atomic smashing frame} of \(\mcv\) to be the
	subframe
	\[
	\cidem_f(\mcv) \subset \cidem(\mcv)
	\]
	consisting of those coidempotent objects \(x \to \mathbf{1}\) whose
	associated smashing ideal \(\langle x\rangle \hookrightarrow \mcv\) is
	atomic. 
\end{de}
	
This definition can be illustrated by the following diagram:
\[
\begin{tikzcd}
	\cidem_f(\mcv) \arrow[r, hook] \arrow[d, "\sim", "\langle - \rangle"']
	& \cidem(\mcv) \arrow[d, "\sim", "\langle - \rangle"'] \\
	\cidem(\prvat) \arrow[r, hook]
	& \cidem(\prvdbl).
\end{tikzcd}
\] 

The following results show that this construction provides the desired extension
asked for in \cref{qu2}.

\begin{ThmAlpha}[{\ref{atsmfrmfunctorial}, \ref{localrigatgen}, \ref{atsmfrmcoh}, \ref{atsmfrmloccoh}}]
	The assignment
	$
	\mcv\longmapsto \cidem_f(\mcv)
	$
	defines a functor
	\[
	\cidem_f(-)\colon \calg(\prl)\longrightarrow \frm.
	\]
	Moreover, if \(\mcc\in\calg(\prlst)\) is stable and
	compactly-rigidly generated, then there is a natural equivalence
	\[
	\cidem_f(\mcc)
	\simeq
	\op{Thick}(\mcc^\omega).
	\]
	Thus the atomic smashing frame recovers the classical Balmer frame.
	
	Finally, let \(\mcv\in\calg(\prl)\). Then:
	\begin{enumerate}
		\item If the unit of \(\mcv\) is compact, then
		\(\cidem_f(\mcv)\) is a coherent frame.
		
		\item If the compact dualizable objects of \(\mcv\) generate
		\(\mcv\) as a localizing ideal, then \(\cidem_f(\mcv)\) is a
		locally coherent frame.
	\end{enumerate}
	In particular, in either case, \(\cidem_f(\mcv)\) is a spatial frame\footnote{A frame is called spatial if it is isomorphic to the frame of open subsets of a topological space; equivalently, if it has enough points to distinguish its elements.}.
\end{ThmAlpha}

Thus \(\cidem_f(\mcv)\) behaves as a generalized Balmer frame even when
\(\mcv\) is neither stable nor compactly generated.
This leads to a formulation of the telescope conjecture for an arbitrary
presentably symmetric monoidal \(\infty\)-category.

\begin{de}
	Let \(\mcv\in\calg(\prl)\). We say that the
	\textbf{telescope conjecture} holds for \(\mcv\) if the canonical
	inclusion of frames
	\[
	\cidem_f(\mcv)\hookrightarrow\cidem(\mcv)
	\]
	is an equivalence, or equivalently, if every smashing ideal of \(\mcv\) is
	generated by dualizable objects.
\end{de}

In the stable compactly-rigidly generated case this is the usual telescope
conjecture. On the other hand, for a projectively rigid
\(\spgeq\)-algebra it asks whether every smashing ideal is compact
projectively generated. Thus the same geometric construction interpolates
between the classical stable telescope conjecture and its connective
analogue.

The following examples illustrate the range of the atomic smashing frame
and its relation to the full smashing frame.
\begin{center}
	\renewcommand{\arraystretch}{1.35}
	\small
	\begin{tabularx}{\textwidth}{>{\raggedright\arraybackslash}p{0.22\textwidth}
			>{\raggedright\arraybackslash}p{0.34\textwidth}
			>{\raggedright\arraybackslash}X}
		\hline
		\textbf{Category} 
		& \textbf{Atomic smashing frame} 
		& \textbf{Smashing frame} \\
		\hline
		
		$\opsp^{X}$, 
		particularly when $X=BG$ (\cref{ex:borel-spectra})
		&
		\(\displaystyle
		\cidem_f(\opsp^X)\simeq
		\prod_{\pi_0X}\cidem_f(\opsp)
		\)
		&
		\(\displaystyle
		\cidem(\opsp^X)\simeq
		\prod_{\pi_0X}\cidem(\opsp)
		\).
		\\
		\hline
		
		$\opsp^{\mathbb{N}}$ graded spectra
		(\cref{ex:graded-spectra})
		&
		\(\displaystyle
		\cidem_f(\opsp^{\mathbb{N}})
		\simeq
		\cidem_f(\opsp)
		\)
		&
		\(\displaystyle
		\cidem(\opsp^{\mathbb{N}})
		\simeq
		\cidem(\opsp)
		\).
		\\
		\hline
		$\shv(\mcx;\mcd(k))$
		(\cref{atsmfrmsheaf})
		&
		$\op{Z}(\mcx_{\leq-1})$
		&
		$\mcx_{\leq-1}$
		\\
		\hline
		$\sppcpl$ (\cref{thmC})
		&
		\(
		\cidem_f(\opsp_{(p)})_{\leq \mathbb{Q}}
		\)\footnote{This frame can be identified with $\{\cdots< L^f_n\mathbb{S}< L^f_{n-1}\mathbb{S} <\dots<L^f_{0}\mathbb{S}=\mathbb{Q}\}$.} 
		&
		$
		\cidem(\opsp_{(p)})_{\leq \mathbb{Q}}$ 
		\\
		\hline
		$\opsp_{K(n)}$
		(\cref{smfieldex})
		&
		\(\displaystyle
		\{0,1\}
		\)
		&
		\(\displaystyle
		\{0,1\}
		\)
		\\
		\hline
		$\opsp_{T(n)}$
		(\cref{ex:atomic-Tn})
		&
		\(
		\{0,1\}
		\)
		&
		Not completely known yet; by \cite{burklund2023k}, it has at least 3
		elements for \(n\geq 2\).
		\\
		\hline	
	\end{tabularx}
\end{center}

We also establish local-to-global principles for the telescope conjecture. We
introduce three natural classes of coverings: atomic Hopf coverings, atomic
smashing-open coverings, and atomic smashing-closed coverings.
\begin{ThmAlpha}[{\ref{tcdescent}, \ref{atopensmcovtcdescent}, \ref{finiteatsmcloseddescent}}]
	Let \(\mcv\in\calg(\prl)\), and let
	$
	\{f_i:\mcv\to\mcv_i\}
	$
	be one of the following:
	\begin{enumerate}
		\item an atomic Hopf covering;
		\item an atomic smashing-open covering;
		\item an atomic smashing-closed covering satisfying the coinitiality
		condition, provided that \(\mcv\) is stable and
		compactly-rigidly generated.
	\end{enumerate}
	If every \(\mcv_i\) satisfies the telescope conjecture, then so does
	\(\mcv\). In cases \((2)\) and \((3)\), the converse also holds.
\end{ThmAlpha}

\subsubsection*{Recollement for (atomic) smashing frames}
In fact, smashing ideals admit an intrinsic description in terms of locally rigid
\infcats.
Let
$
i:\mci\hookrightarrow\mcv
$
be a smashing ideal. Its right adjoint
\[
i^R:\mcv\longrightarrow\mci
\]
is a localization compatible with the symmetric monoidal structure and
equips \(\mci\) with the structure of a locally rigid algebra over \(\mcv\).
Conversely, locally rigid localizations arise in this way. This leads to a
natural identification (see \cref{smidlvslocrig}):
\[
\cidem(\mcv)
\xrightarrow{\sim}
\bigl\{
\text{locally rigid localizations out of }\mcv
\bigr\}.
\]

In the (pre)stable setting, a smashing ideal
admits a complementary smashing quotient. In fact, 
the smashing frames of these two pieces reconstruct that of the ambient
category.
\begin{ThmAlpha}[{\ref{recollementfrms}, \ref{recollementatsmfrms}}]\label{thmC}
		Let $\mcc\in\calg(\prlad)$ be prestable. Let
	\[
	\mci\xhookrightarrow{i} \mcc \xrightarrow{L} \mcc/\mci
	\]
	be a smashing localization sequence. Then: 
	\enu{
		\item There are natural equivalences
		\[
		\cidem(\mci)
		\simeq
		\cidem(\mcc)_{\leq\mci},
		\qquad
		\cidem(\mcc/\mci)
		\simeq
		\cidem(\mcc)_{\geq\mci}.
		\]
		Consequently, the canonical morphism
		$
		\cidem(\mcc)
		\xrightarrow{(i^R,L)}
		\cidem(\mci)\times\cidem(\mcc/\mci)
		$
		is an embedding of frames.
		Moreover, under the preceding identifications,
		\(\cidem(\mcc)\) is obtained by Artin gluing:
		\[
		\cidem(\mcc)
		\simeq
		\cidem(\mci)
		\overleftarrow{\times}_{\!\phi_\mci}
		\cidem(\mcc/\mci),
		\]
		where the gluing functor is explicitly given by
		\[
		\phi_\mci(\mca)
		=
		\mci\vee(\mci\backslash \mca),
		\qquad
		\mca\in\cidem(\mcc)_{\leq\mci}.
		\]
		Here \(\mci\backslash \mca\) denotes the Heyting implication in the frame
		\(\cidem(\mcc)\).
		\item 
		If, in addition, \(\mcc\) is stable and compactly-rigidly generated and \(\mci\) is atomic, then the assertions in \((1)\) hold with all smashing frames replaced by atomic smashing frames.
	}
\end{ThmAlpha}
The gluing functor admits a categorical description in terms of internal
Homs of dualizable categories. More precisely, let
$$
\mca\subset\mci\subset\mcc
$$
be smashing ideals of \mcc, and let
$
q:\mcc/\mca\longrightarrow\mcc/\mci
$
be the induced smashing localization. Then
$
\underline{\operatorname{Hom}}^{\op{dbl}}_{\mcc}
(\mcc/\mci,\mcc/\mca)
$
carries a canonical structure of a smashing ideal of \(\mcc/\mci\), whose
inclusion is induced by postcomposition $
\underline{\operatorname{Hom}}^{\op{dbl}}_{\mcc}
(\mcc/\mci,\mcc/\mca)
\xrightarrow{\,q\circ-\,}
\underline{\operatorname{Hom}}^{\op{dbl}}_{\mcc}
(\mcc/\mci,\mcc/\mci)
\simeq
\mcc/\mci.
$
Under the canonical identification
$
\cidem(\mcc/\mci)
\simeq
\cidem(\mcc)_{\geq\mci},
$
we have (see \cref{artin-gluing-internal-hom})
$$
\mci\vee(\mci\backslash\mca)\simeq\underline{\operatorname{Hom}}^{\op{dbl}}_{\mcc}
(\mcc/\mci,\mcc/\mca) \quad\text{in}\quad \cidem(\mcc/\mci).
$$

The recollement theorem has several immediate frame-theoretic consequences. For a
frame \(F\), let \(\op{Z}(F)\) denote the subframe generated by the
complemented elements of \(F\), often called the \textbf{zero-dimensional part} of
\(F\).
\begin{cor}[\ref{recollementspatial}, \ref{splitclopen}]
	Let \(\mcc\in\calg(\prlad)\) be prestable. Then:
	\begin{enumerate}
		\item For any smashing localization sequence
		$
		\mci\hookrightarrow\mcc\rightarrow\mcc/\mci,
		$
	 the smashing frame \(\cidem(\mcc)\) is spatial if and only if
		both \(\cidem(\mci)\) and \(\cidem(\mcc/\mci)\) are spatial.
		
		\item The zero-dimensional part of the smashing frame is contained
		in the atomic smashing frame:
		\[
		\op{Z}\bigl(\cidem(\mcc)\bigr)
		\subset
		\cidem_f(\mcc).
		\]
	\end{enumerate}
\end{cor}

In the stable compactly-rigidly generated setting, the same recollement
principle also applies directly to the telescope conjecture.

\begin{ThmAlpha}[{\ref{recollement-telescope}}]
	Let
	$
	\mcc\in\calg(\prlst)
	$
	be compactly-rigidly generated, and let
	\[
	\mci\xhookrightarrow{i}\mcc\xrightarrow{L}\mcc/\mci
	\]
	be a smashing localization sequence with \(\mci\) an atomic smashing ideal. Then
	\[
	\mcc\text{ satisfies the telescope conjecture}
	\quad\Longleftrightarrow\quad
	\mci\text{ and }\mcc/\mci
	\text{ satisfy the telescope conjecture}.
	\]
\end{ThmAlpha}
Thus atomic smashing localizations give genuine local decompositions for the
telescope conjecture in the compactly-rigidly generated stable setting.

 \subsubsection*{Structure of the smashing frame of \opsp}
 We next apply the preceding framework to the original motivating example: the
 \(\infty\)-category \(\opsp\) of spectra. 
 We analyze the arithmetic and chromatic layers of \(\Idem(\opsp)\). The main
 result is the following.
 \begin{ThmAlpha}[{\ref{primedecomp}, \ref{smspstructure}}]\label{thmF}
 	The smashing frame and the atomic smashing frame of \(\opsp\) admit
 	natural prime decompositions:
 	$$
 	\begin{tikzcd}
 		\Idem_f(\opsp) \arrow[r,"\sim"] \arrow[d, hook] & \prod_{p}^{/\mathbb{Q}} \Idem_f(\opsp_{(p)}) \arrow[d, hook] \\
 		\Idem(\opsp) \arrow[r,"\sim"]             & \prod_{p}^{/\mathbb{Q}} \Idem(\opsp_{(p)})
 	\end{tikzcd}
 	$$
 	For each prime \(p\), the natural projections along the finite
 	chromatic tower induce an embedding of frames
 	\[
 	\Idem(\opsp_{(p)})
 	\hookrightarrow
 	\prod_{n=0}^{\infty}\Idem(\opsp_{T(n)}).
 	\]
 	Consequently there are natural embeddings
 	\[
 	\Idem(\opsp)
 	\hookrightarrow
 	\prod_p^{/\mathbb Q}
 	\prod_{n=0}^{\infty}\Idem(\opsp_{T(n)})
 	\hookrightarrow
 	\prod_p
 	\prod_{n=0}^{\infty}\Idem(\opsp_{T(n)}).
 	\]
 \end{ThmAlpha}
  Thus the smashing geometry of spectra decomposes first primewise and then
 into monochromatic layers.
 
 Recent work of Barthel--Heard 
 \cite{barthel2026nonspatialframesmashingideals} shows that smashing frames need not be spatial
 in general, even for derived categories of commutative rings. However, whether the smashing frame
 \[
 \Idem(\opsp)
 \]
 of \opsp is spatial remains a fundamental open problem; see
 \cite{balchin2021big,aoki2024smashing}. If the 
 telescope conjecture held, this would follow immediately from the spatiality of
 the classical Balmer frame. The recent disproof of the telescope conjecture
 \cite{burklund2023k} shows, however, that genuinely non-atomic smashing
 localizations of spectra exist, while the global structure of
 \(\Idem(\opsp)\) remains poorly understood.
 
 Spatiality of $\Idem(\opsp)$ would guarantee the existence of enough prime elements (``points'') to completely recover the frame. This would allow one to classify all smashing localizations via point-set topological methods; see
 \cite{balchin2021big,verasdanis2023stratification} for detailed studies of
 smashing frames in the spatial case.
 
 While the spatiality of \(\Idem(\opsp)\) remains open, the
 decomposition of \cref{thmF} reduces this global problem entirely to the monochromatic layers:
 
 \begin{ThmAlpha}[{\ref{sptnspatial}}]
 	The smashing frame \(\Idem(\opsp)\) is spatial if and only if, for
 	every prime \(p\) and every height \(n\geq 0\), the smashing frame
 	\[
 	\Idem(\opsp_{T(n)})
 	\]
 	is spatial.
 \end{ThmAlpha}
 
 We also obtain a retraction property for the telescope inclusion.
 
 \begin{cor}[{\ref{retracttcinclu}}]
 	For every prime \(p\), the diagram
 	\[
 	\begin{tikzcd}
 		\Idem_f(\opsp_{(p)}) \arrow[d, hook] \arrow[r, "\sim"] & \lim_n\Idem_f(L_n\opsp) \arrow[d, "\sim"] \\
 		\Idem(\opsp_{(p)}) \arrow[r]                           & \lim_n\Idem(L_n\opsp)
 	\end{tikzcd}
 	\]
 	induces a natural frame retraction of
 	$
 	\Idem_f(\opsp_{(p)})
 	\hookrightarrow
 	\Idem(\opsp_{(p)})
 	$ over $\Idem(\opsp_{\mathbb{Q}})$.
 	In particular, the global telescope inclusion
 	\[
 	\Idem_f(\opsp)\hookrightarrow\Idem(\opsp)
 	\]
 	admits a natural frame retraction by \cref{thmF}.
 \end{cor}

 \subsubsection*{Smashing fields}
 
  We introduce smashing fields, which
 can be regarded as an analogue of tensor-triangular fields
 introduced in \cite{MR3911737}, formulated from the perspective of smashing ideals.
 
 \begin{de}
 	Let $\mathcal{C}$ be a symmetric monoidal $\infty$-category that is compatible with the initial object. We say that $\mathcal{C}$ is
 	a \textbf{smashing field} if its frame of coidempotent objects is the
 	two-element frame:
 	\[
 	\cidem(\mathcal{C})=\{\emptyset,\mathbf{1}\}.
 	\]
 \end{de}
 
 Equivalently, a smashing field has no nontrivial smashing ideals. Since the zero smashing ideal is atomically generated vacuously, while the unit ideal is generated by the dualizable tensor unit, \emph{every smashing field satisfies the telescope conjecture}. The following examples, discussed in more detail in
 \cref{smfieldex}, illustrate the scope of this notion:
 
 \begin{center}
 	\renewcommand{\arraystretch}{1.25}
 	\small
 	\begin{tabularx}{\textwidth}{
 			>{\raggedright\arraybackslash}p{0.36\textwidth}
 			>{\raggedright\arraybackslash}X}
 		\hline
 		\textbf{Smashing field} & \textbf{Input} \\
 		\hline
 		
 		\(\mathcal{D}(k)\), where \(k\) is a field
 		&
 		Neeman's telescope theorem \cite{neeman1992chromatic}
 		\\
 		\hline
 		
 		\(\mathcal{D}(\mathbb{Z})_{p\text{-}\op{cpl}}\)
 		&
 		detected by the conservative functor
 		\(-\otimes^{L}\mathbb{F}_p\)
 		\\
 		\hline
 		
 		\(\mcs\) and \(\mcs_{\leq n}\) for \(n\geq -1\)
 		&
 		\cref{tctopos}
 		\\
 		\hline
 		
 		\(\mcs_*\)
 		&
 		\cref{sptsmfield}
 		\\
 		\hline
 		
 		\(\cmon(\mcs)\)
 		&
 		\cref{cmonssmfield}
 		\\
 		\hline
 		
 		\(\opsp_{\geq 0}\)
 		&
 		\cref{noethertelescope}
 		\\
 		\hline
 		
 		\(\opsp_{K(n)}\)
 		&
 		\cite[Theorem 7.5]{hovey1999morava}
 		\\
 		\hline
 		
 		\(\op{Sp}_{\mathbb{F}_p}\)
 		&
 		\cite[Proposition 6.2]{MR3374070}
 		\\
 		\hline
 		
 		\(\op{Sp}_{I_p}\)
 		&
 		\cite[Proposition 6.6]{MR3374070}
 		\\
 		\hline
 		
 	\end{tabularx}
 \end{center}
 
 For Cartesian symmetric monoidal \infcats, the smashing frame is especially
 simple: 
 $$\cidem(\mcx)\simeq\mcx_{\leq -1}.$$
  This yields the following characterization of the telescope conjecture in this
  setting, and in particular for \(\infty\)-topoi.
 
 \begin{ThmAlpha}[{\ref{tctopos}}]
 	Let \(\mcx\in\calg(\prl)\) be presentably Cartesian symmetric monoidal
 	and nontrivial. Then
 	\[
 	\cidem_f(\mcx)
 	=
 	\{\emptyset_\mcx,*\}.
 	\]
 	Consequently, the following conditions are equivalent:
 	\enu{\item $\mcx$ satisfies the telescope conjecture.
 		\item $\mcx$ is a smashing field.
 		\item $\mcx_{\leq-1}$ consists of only two elements.
 	}
 \end{ThmAlpha}
 \subsubsection*{Telescope conjecture in the additive setting}
 The telescope conjecture becomes particularly algebraic in the prestable and
 abelian settings. Using generalized higher almost algebra
 \cite{hattt}, we show that the smashing frame of a left-complete prestable
 \(\infty\)-category is determined by its heart. More precisely,
 \cref{smashidlidemidliden} gives natural equivalences
 \[
 \cidem(\mcc)
 \simeq
 \cidem(\mcc^\heartsuit)
 \simeq
 \cidem\bigl(\idl(\mcc^\heartsuit)\bigr).
 \]
 
 Under projective rigidity, this equivalence also detects the atomic part.
 
 \begin{ThmAlpha}[{\ref{comparisontcinc}}]
 	Let \(\mcc\in\calg(\prlad)\) be a projectively rigid
 	\(\spgeq\)-algebra. Then there is a commutative diagram
 	\[
 	\begin{tikzcd}
 		\cidem_f(\mcc)
 		\arrow[d,"\sim"]
 		\arrow[r,hook]
 		&
 		\cidem(\mcc)
 		\arrow[d,"\sim"]
 		\\
 		\cidem_f(\mcc^\heartsuit)
 		\arrow[d,"\sim"]
 		\arrow[r,hook]
 		&
 		\cidem(\mcc^\heartsuit)
 		\arrow[d,"\sim"]
 		\\
 		\idl_p(\mcc^\heartsuit)
 		\arrow[r,hook]
 		&
 		\cidem(\idl(\mcc^\heartsuit)).
 	\end{tikzcd}
 	\]
 	 Here \(\idl_p(\mcc^\heartsuit)\) denotes the frame of Pierce ideals. In particular,
 	\[
 	\mcc\text{ satisfies the telescope conjecture}
 	\quad\Longleftrightarrow\quad
 	\mcc^\heartsuit\text{ does}.
 	\]
 \end{ThmAlpha}

 We further introduce the \textbf{Serre smashing frame}
 \[
 \cidem_s(\mcc)\subset\cidem(\mcc)
 \]
 for a presentably symmetric monoidal separated Grothendieck prestable
 \(\infty\)-category. A smashing ideal is Serre if its quotient localization
 is left exact. Serre smashing ideals form a subframe, and in the
 projectively rigid setting this frame has its own geometric structure.
 
 \begin{ThmAlpha}[{\ref{serresmfrmspatial}}]
 	Let \(\mcc\in\calg(\prlad)\) be a projectively rigid
 	\(\spgeq\)-algebra. Then the Serre smashing frame
 	\[
 	\cidem_s(\mcc)
 	\]
 	is spatial and its spectrum is quasi-compact.
 \end{ThmAlpha}
 
 The corresponding \textbf{flat telescope conjecture} asks whether
 \[
 \cidem_s(\mcc)\hookrightarrow\cidem(\mcc)
 \]
 is an equivalence.
 
 For connective module \infcats, all three frames admit a classical
 commutative-algebraic interpretation.
 
 \begin{ThmAlpha}[{\ref{main2}, \ref{atsmidlisserre}}]
 	Let $R$ be a connective \ein-ring. Then any atomic smashing ideal of $\modu_R^{\op{cn}}=\modu_R(\spgeq)$ is a Serre smashing ideal. Consequently, the Serre smashing frame resides as an intermediate structural layer of the telescope inclusion:
 $$\begin{tikzcd}[column sep=0em]
 	& \cidem_s(\modu_R^{\op{cn}}) \arrow[rd, dashed, hook] &                           \\
 	\cidem_f(\modu_R^{\op{cn}}) \arrow[rr, hook] \arrow[ru, dashed, hook] &                                                      & \cidem(\modu_R^{\op{cn}})
 \end{tikzcd}$$
 Moreover, this hierarchy of frames corresponds precisely to the following inclusions among idempotent ideals of $\pi_0 R$:
 $$\{\text{Pierce ideals}\} \subset \{\text{pure ideals}\} \subset \{\text{idempotent ideals}\}.$$
 \end{ThmAlpha}
 
 More explicitly, if \(I\subset\pi_0R\) is idempotent, then the almost
 module \infcat
 \[
 \amodu_{(R,I)}(\spgeq)
 \]
 is compact projectively generated if and only if \(I\) is Pierce, and it
 is a Serre smashing ideal if and only if \(I\) is pure.
 Consequently,
 $
 \modu_R(\spgeq)
 $
 satisfies the telescope conjecture if and only if every idempotent ideal of
 \(\pi_0R\) is Pierce, while it satisfies the flat telescope conjecture if
 and only if every idempotent ideal is pure.
 
 Geometrically, the assignment $I\mapsto V(I)^c$ relates Pierce and pure ideals to the following classes of open subsets of the Zariski spectrum $\spec(\pi_0R)$:
 \[
 \begin{tikzcd}[column sep=-0.5em] 
 	\{\text{Pierce ideals}\} \arrow[d, "V(-)^c"', "\sim"] & \subset & \{\text{pure ideals}\} \arrow[d, "V(-)^c"', "\sim"] \\
 	\{\text{Pierce open subsets}\}                        & \subset & \left\{\begin{tabular}{c} open subsets closed \\ under specializations \end{tabular}\right\}
 \end{tikzcd}.
 \] 

 Finally, we obtain an analogous description for hypercomplete connective sheaves.
 
 \begin{ThmAlpha}[{\ref{pierceAbshv}, \ref{pierceSpgeq0shv}}]
 	Let \(X\) be a topological space. Then there are natural equivalences
 	\[
 	\mco(X)
 	\xrightarrow{\sim}
 	\cidem(\shvhyp(X;\spgeq))
 	\xrightarrow{\sim}
 	\cidem(\shv(X;\op{Ab})).
 	\]
 	Every smashing ideal of
 	\(\shvhyp(X;\spgeq)\) is Serre, and an open subset
 	\(U\subset X\) determines an atomic smashing ideal if and only if
 	\(U\) is a Pierce open subset, i.e.\ a union of clopen subsets.
 	In particular,
 	\[
 	\mco_p(X)
 	\xrightarrow{\sim}
 	\cidem_f(\shvhyp(X;\spgeq))
 	\xrightarrow{\sim}
 	\cidem_f(\shv(X;\op{Ab})).
 	\]
 \end{ThmAlpha}
 
 \begin{cor}[{\ref{shvabtc}}]
 	The \infcat \(\shvhyp(X;\spgeq)\) always satisfies the flat telescope
 	conjecture. Moreover, the following are equivalent:
 	\begin{enumerate}
 		\item \(\shvhyp(X;\spgeq)\) satisfies the telescope conjecture;
 		\item \(\shv(X;\op{Ab})\) satisfies the telescope conjecture;
 		\item every open subset of \(X\) is a Pierce open subset.
 	\end{enumerate}
 	In particular, these conditions hold whenever \(X\) is profinite.
 \end{cor}
 
 \subsection{Outline}
 
 The paper is divided into three parts.
 
 \emph{Part I develops the general theory of smashing and atomic smashing
 	frames.}
 In \cref{sec2} we review internal adjoints, dualizable  \infcats, and atomic
 generation in \(\Pr^L_\mcv\).
 In \cref{sec3} we study smashing ideals through dualizable  \infcats
 and identify them with locally rigid localizations.
 In \cref{sec4} we introduce the atomic smashing frame, formulate the
 generalized telescope conjecture, establish coherence properties of the atomic
 smashing frame, and prove local-to-global principles for the telescope
 conjecture.
 In \cref{secrecollement} we develop the recollement and Artin gluing theory
 for (atomic) smashing frames. We give a formula of the Artin gluing functor in terms of internal Homs of
 dualizable categories and prove a recollement principle for
 the telescope conjecture.
 
 \emph{Part II is devoted to stable and chromatic applications.}
 In \cref{sec5} we analyze the smashing frame of \opsp. We prove its
 prime decomposition, study its relation with finite chromatic localization,
 and construct its embedding into the product of the monochromatic smashing
 frames
 \[
 \Idem(\opsp_{T(n)}).
 \]
 This yields the reduction of the spatiality problem for
 \(\Idem(\opsp)\) to the corresponding monochromatic problems.
 
 \emph{Part III develops the prestable and additive theory.}
 In \cref{tcadd} we use higher almost algebra to compare the smashing and
 atomic smashing frames of a projectively rigid \spgeq-algebra with those of its
 heart. For connective \(\mathbb E_\infty\)-rings, atomic smashing ideals
 are identified with Pierce ideals.
 In \cref{sec7} we introduce dualizable flatness and the Serre smashing
 frame, identify the latter with the pure frame in the \spgeq-rigid
 setting, and study the resulting flat telescope conjecture. We conclude
 with applications to \infcats of connective modules and sheaves.
\subsection{Conventions and notations}
For the reader's convenience, we record the recurring notions and notation used throughout.

\begin{itemize}
	\item We denote by $\mathcal{S}$ the \infcat of animas (i.e., $\infty$-groupoids), and by $\mathrm{Cat}_{\infty}$ the \infcat of small \infcats. Similarly, $\Cat_{\infty}^*$ denotes the \infcat of small pointed \infcats and point-preserving functors, while $\mathrm{Sp}$ denotes the \infcat of spectra.
	
	\item For \(n\in\mathbb{N}\), \(\mathcal{S}_{\leq n}\) is the full subcategory of \(\mathcal{S}\) spanned by \(n\)-truncated animas. For example, \(\mathcal{S}_{\leq -1}\) is the full subcategory spanned by \(\varnothing\) and \(*\), and \(\mathcal{S}_{\leq -2}\) is the full subcategory spanned by \(*\).
	
	\item We write \(\mcs_*\) for the \(\infty\)-category of pointed animas.
	We write \(\cmon(\mcs)\) for the \(\infty\)-category of
	\(\mathbb{E}_\infty\)-animas, equivalently commutative monoid objects in
	\(\mcs\). More generally, \(\cmon_m(\mcs)\) denotes the \(m\)-semiadditive
	mode of \cite{Harpaz_2020}.
	\item For functor categories \(\mathrm{Fun}(\mathcal{C},\mathcal{D})\) we use superscripts \text{lex}, \text{rex}, \text{ex}, $R$ and $L$ to indicate the full subcategories of functors preserving finite limits, finite colimits, finite limits and finite colimits (exact functors), small limits, and small colimits, respectively.
	
	\item Given a symmetric monoidal $\infty$-category $\mcc$. We denote the full subcategory of dualizable objects by $\mcc^d$. The $\infty$-category of commutative monoids in $\mcc$ will be denoted by $\mathrm{CAlg}(\mathcal{C})$ and we refer to its objects as commutative algebras or \ein-algebras in $\mathcal{C}$. For $\mathcal{C}=\mathrm{Sp}$ equipped with its natural symmetric monoidal structure, we usually say $\mathbb{E}_{\infty}$-ring instead of commutative algebra.
	\item We write \(\spgeq\) for the \(\infty\)-category of connective
	spectra. For a connective \(\mathbb{E}_\infty\)-ring \(R\), we write
	\(\modu_R^{\op{cn}}\) or \(\modu_R(\spgeq)\) for the \(\infty\)-category
	of connective \(R\)-modules.
	\item We use standard chromatic notation at a fixed prime \(p\), which is
	implicit unless otherwise specified. Thus \(K(n)\) denotes Morava
	\(K\)-theory at \(p\), \(T(n)\) denotes the telescope of a \(v_n\)-self map
	at \(p\), \(\opsp_{K(n)}\) and \(\opsp_{T(n)}\) denote the corresponding
	local categories of spectra, and \(L_n\) and \(L_n^f\) denote the usual and
	finite chromatic localizations at \(p\). 
	
	\item For a prestable or \(t\)-structured stable \(\infty\)-category
	\(\mathcal{C}\), we write \(\mathcal{C}^{\heartsuit}\) for its heart.
	\item We use $\mapp(-,-)$ to denote mapping animas, $\unmap(-,-)$ for enriched mapping objects (when the enriched context is clear), and $\underline{\Hom}(-,-)$ for internal homs.
	\item For an \infcat \mcc, we denote by $\mcc^\kappa$ the full subcategory of $\kappa$-compact objects; we denote by $\mcc^{\op{cproj}}$ the full subcategory of compact projective objects.  
	\item \(\prl\) denotes the \(\infty\)-category of presentable \(\infty\)-categories with left adjoint functors. Its default symmetric monoidal structure is the Lurie tensor product, which corepresents functors out of the Cartesian product that are colimit-preserving in each variable. 
	
\item 	For a regular cardinal \(\kappa\), \(\prl_{\kappa}\subset\prl\) denotes the (non-full) subcategory of \(\kappa\)-accessible presentable \(\infty\)-categories with functors that preserve \(\kappa\)-compact objects.
	
\item 	For a natural number $n\geq 0$, \(\prl_{n}\subset\prl\) denotes the full subcategory of presentable $n$-categories, and $\prl_{\kappa,n} := \prl_{n}\cap \prl_{\kappa}$ for a regular cardinal \(\kappa\). 
	\item  A (symmetric) monoidal presentable $\infty$-category $\mathcal{C}$ is called presentably (symmetric) monoidal if the tensor product $-\otimes-$ preserves colimits separately in each variable.
\item We write \(\prlst\subset\prl\) for the full subcategory of stable
presentable \(\infty\)-categories and colimit-preserving functors, and
\(\prlad\subset\prl\) for the full subcategory of additive presentable
\(\infty\)-categories. If \(\mathcal{V}\in\calg(\prl)\), we write
\(\prlv\) for the \(\infty\)-category of presentable \(\mathcal{V}\)-modules
and \(\mathcal{V}\)-linear left adjoint functors.  
\item We refer an object in $\calg(\prlst)$ to a big $tt$-category.
\item We denote by $\catper$ the \infcat of (small) idempotent-complete stable \infcats. We denote by \catadidem the \infcat of (small) idempotent-complete additive \infcats.

\item By a space we mean a topological space, not an anima. For a
space \(X\) and an \(\infty\)-category \(\mathcal{C}\), we write
\(\shv(X;\mathcal{C})\) for the \(\infty\)-category of
\(\mathcal{C}\)-valued sheaves on \(X\), and \(\shvhyp(X;\mathcal{C})\) for
its full subcategory of hypercomplete sheaves.

\item For a space \(X\), we write \(\mco(X)\) for the poset of open
subsets of \(X\), and \(\mco_p(X)\subset\mco(X)\) for the subposet of
Pierce open subsets, namely open subsets which are unions of clopen
subsets.
\end{itemize}

\subsection{Acknowledgments}
The author is especially grateful to Yifan Jin and Maxime Ramzi for many
helpful conversations that inspired several of the results in this paper.

The author would also like to thank Anish Chedalavada, Daniel Dugger,
Sasha Efimov, Thorger Geiß, David Gepner, Mike Hill, Ishan Levy,
Zhenpeng Li, Tomer Schlank, Vova Sosnilo, Leonard Tokic, Yuchen Wu,
Albert Jinghui Yang, and Changhan Zou for useful discussions and comments
at various stages of this project.

AI tools were used only to search for examples and to
assist with editing. The author takes full responsibility for the content of
this paper.
\part{General Theory}
\section{Preliminaries}\label{sec2}
This section establishes the necessary categorical background. By reviewing the mechanics of (pre)frames,  categorical ideals and dualizable categories, we lay the groundwork for the subsequent analysis of non-stable smashing ideals.
\subsection{Preframes and frames}
This subsection provides a brief review of the lattice-theoretic foundations used throughout the paper. Specifically, we recall the essential definitions and properties of preframes, coherent frames, and Artin recollements. These geometric and order-theoretic concepts serve as the background for the gluing of smashing frames in the main text.

We recall the theory of 
preframes and frames. We refer the reader to \cite[Appendix A]{anel2023left}\footnote{Preframes are called $\otimes$-frames in \cite{anel2023left}.} or \cite[Section~2]{Zariski_Bal} for more details about preframes.
\begin{de}\label{de:preframe}
	A \textbf{preframe} is a presentably symmetric monoidal $0$-category $(Q,\star,1)$ whose unit $1$ is terminal. 
	A morphism of preframes is a symmetric monoidal functor that preserves small joins. 
	We denote by $\pfrm$ the category of preframes.
	
	Recall that a \textbf{frame} is a presentably symmetric monoidal $0$-category whose monoidal structure is Cartesian. 
	A morphism of frames is simply a morphism of the underlying preframes.
	We denote by $\frm$ the category of frames.
\end{de}

\begin{rem}

	Every frame is canonically a preframe, yielding a fully faithful embedding
	\[
	\frm \hookrightarrow \pfrm.
	\]
	This embedding fits into the commutative diagram
	\[
	\begin{tikzcd}
		\frm \arrow[r, hook] \arrow[d, "\sim"]
		& \pfrm \arrow[d, "\sim"] \\
		\calg(\prl_0)_{\op{Car}} \arrow[r, hook]
		& \calg(\prl_0)_{\op{ter}}
	\end{tikzcd}
	\]
	where $\calg(\prl_0)_{\op{ter}}$ denotes the full subcategory of $\calg(\prl_0)$ spanned by presentably symmetric monoidal $0$-categories whose unit is terminal, and $\calg(\prl_0)_{\op{Car}}$ denotes the full subcategory spanned by those whose tensor product is given by the meet.
\end{rem}
\begin{de}\label{defkcoh}
	Let $\kappa\geq \omega$ be a regular cardinal. A preframe $Q$ is called $\kappa$-coherent\footnote{When $\kappa=\omega$ we will simply say it is coherent.}, if $Q$ lies in $\calg(\prl_\kappa)$, i.e. the following conditions hold:
	\enu{
		\item Every element is the join of $\kappa$-compact elements;
		\item The unit is $\kappa$-compact;
		\item The $\kappa$-compact elements are closed under tensor products.
	}
	We say a frame $F$ is $\kappa$-coherent if $F$ regarded as a preframe is $\kappa$-coherent, which coincides with the $\kappa$-coherence appeared in  \cite[\textsection3]{krause2023spectrum}.
\end{de}
\begin{de}\label{defradprime}
	Let $Q$ be a preframe. We call an element $r \in Q$ \textbf{radical} if $x^2 \leq r$ implies $x \leq r$ for every $x \in Q$. We write $Q_{\op{rad}}$ for the poset of radical elements of $Q$. An element $p \in Q$ is called a \textbf{prime} if $x_1 \cdot x_2 \leq p$ implies either $x_1 \leq p$ or $x_2 \leq p$, for all $x_1, x_2 \in Q$. 
\end{de}

\begin{rem}\label{operationsqrt}
	We define the nil-radical of an element $x \in Q$ to be
	\[
	\sqrt x = \bigwedge \big\{y\in Q|y\text{ is radical and }y\geq x\big\}.
	\]
	Since the radical objects are the local objects for the relations $y^2 \leq y$, we can also construct $\sqrt{x}$ by a small object argument. Let us put
	$$
	\sqrt[1]{x}=\sup \left\{y \mid \exists k\colon y^k \leq x\right\}
	$$
	and $\sqrt[n+1]{x}=\sqrt[1]{\sqrt[n]{x}}$, then the transfinite colimit of $x \leq \sqrt[1]{x} \leq \sqrt[2]{x} \leq \ldots$ is $\sqrt{x}$. When $Q$ is coherent in the sense of \cref{defkcoh}, we have in fact $\sqrt[1]{x}=\sqrt{x}$.
\end{rem}

\begin{prop}[{\cite[Proposition A.4.3 \& A.4.5]{anel2023left}}]
	Let $Q$ be a preframe. Then the full subcategory $Q_{\op{rad}}$ is a reflective subcategory, and the correspondent localization is given by $a\mapsto \sqrt{a}$. Furthermore, the localization is compatible with the symmetric monoidal structure. Thereby it induces a symmetric monoidal adjunction 
	$$\begin{tikzcd}
		Q \arrow[r, "\sqrt{(-)}", shift left=1ex]   & \arrow[l,"\perp"', hook', shift left=1ex] Q_{\op{rad}}
	\end{tikzcd} .$$
	Moreover, $Q_{\op{rad}}$ is a frame.
\end{prop}

\begin{de}\label{internalhomfrm}
	Let $Q\in\calg(\prl_0)$, i.e. a $\otimes$-suplattice. By definition, the map $a \cdot(-): Q \rightarrow Q$ preserves small joins for every $a \in Q$. Thus, we have
	$$
	a \cdot\left(\bigvee_{x \in S} x\right)=\bigvee_{x \in S} (a \cdot x)
	$$
	for every subset $S \subset Q$. It follows that the map $a \cdot(-): Q \rightarrow Q$ has a right adjoint, which is denoted by $a \backslash(-): Q \rightarrow Q$, and called the division by $a \in Q$. For every $a, b, c \in Q$ we have
	$$
	a \cdot b \leq c \Leftrightarrow b \leq a \backslash c .
	$$
\end{de}
\begin{rem}
	By definition, the tensor product of two elements $x$ and $y$ of $Q$ is $x \star y$, and the internal hom $\underline{\operatorname{Hom}}(x, y)$ is the element $x \backslash y$ for every $x, y \in Q$. When \(Q\) is a frame, the operation \(\backslash\) also refers to the \textbf{Heyting
	implication}.
\end{rem}
\begin{prop}\label{implication-principal-filter}
	Let $Q$ be a preframe and let $x\in Q$. Endow the principal filter
	$
	Q_{\geq x}
	$
	with its canonical preframe structure, whose tensor product is given by
	$
	u\star_x v:=x\vee(u\star v).
	$
	For any $u,v\in Q$ satisfying
	$
	u,v\geq x,
	$
	one has
	\[
	(u\backslash v)_{Q_{\geq x}}
	=
	(u\backslash v)_Q.
	\]
	In other words, division on the principal filter $Q_{\geq x}$ is the
	restriction of division on $Q$.
\end{prop}

\begin{proof}
	First observe that, since the unit of $Q$ is terminal, we have
	$u\leq 1$, and hence
	\[
	x\star u\leq x\star 1=x\leq v.
	\]
	By the defining property of division, this gives
	$
	x\leq (u\backslash v)_Q.
	$
	Thus $(u\backslash v)_Q$ itself belongs to the principal filter
	$Q_{\geq x}$.
	
	Now let $z\geq x$. By the defining property of division in
	$Q_{\geq x}$,
	\[
	z\leq (u\backslash v)_{Q_{\geq x}}
	\quad\Longleftrightarrow\quad
	u\star_x z\leq v.
	\]
	Since
	$
	u\star_x z=x\vee(u\star z)
	$
	and $x\leq v$, the latter condition is equivalent to
	$
	u\star z\leq v.
	$
	By the defining property of division in $Q$, this is equivalent to
	$
	z\leq (u\backslash v)_Q.
	$
	Hence $(u\backslash v)_{Q_{\geq x}}$ and $(u\backslash v)_Q$
	have the same universal property inside $Q_{\geq x}$, and therefore
	$
	(u\backslash v)_{Q_{\geq x}}
	=
	(u\backslash v)_Q.
	$
\end{proof}

\begin{rem}
	If $Q$ is a frame, then $\star=\wedge$, and for $u,v\geq x$ one has
	\[
	x\vee(u\wedge v)=u\wedge v.
	\]
	Thus the preframe structure on $Q_{\geq x}$ is simply the usual
	principal-filter frame structure.
\end{rem}

\begin{de}
	Let $Q$ be a preframe. We define $\cidem(Q) \subset Q$ to be the subposet spanned by coidempotent elements, i.e. those $x$ such that $x^2=x$.
\end{de}
\begin{prop}
	$\cidem(Q)$ is a sub-preframe of $Q$. Moreover, $\cidem(Q)$ is a frame.
\end{prop}
\begin{proof}
	The poset $\cidem(Q)$ is closed under tensor product since, for every $a, b \in Q$, we have $(a \cdot b)^2=a^2 \cdot b^2=a \cdot b$. Let us see that it is closed under small joins: for any family of elements $a_i \in Q$, we have
	$$
	\bigvee a_i \geq\left(\bigvee a_i\right)^2=\bigvee a_i \cdot a_j \geq \bigvee a_i^2=\bigvee a_i .
	$$
	This shows that $\bigvee a_i=\left(\bigvee a_i\right)^2$ and that $\cidem(Q)$ is a sub-preframe of $Q$. Then it is a frame since every element is idempotent.
\end{proof}
One can deduce the following statement from \cref{cidemkappaprl} directly, but here we provide a straightforward argument in the preframe setting.
\begin{prop}
	Let $Q$ be a $\kappa$-coherent preframe, where $\kappa>\omega$ is an uncountable regular cardinal. Then $\cidem(Q)$ is a $\kappa$-coherent frame. Furthermore we have $\cidem(Q)^\kappa\simeq \cidem(Q)\cap Q^\kappa$.
\end{prop}
\begin{proof}
	Note that we always have $\cidem(Q)\cap Q^\kappa\subset \cidem(Q)^\kappa$. So it suffices to show $$\op{Ind}_\kappa(\cidem(Q)\cap Q^\kappa)\simeq \cidem(Q).$$
	Now given an idempotent element $x\in \cidem(Q)$, by assumption $x\simeq \bigvee_{\alpha\in I} x_\alpha$ can be written as a $\kappa$-filtered join of $\kappa$-compact elements. For each $\alpha\in I$, we will construct an increasing sequence of $\kappa$-compact elements by induction $$x_\alpha=x_{\alpha,0}\leq x_{\alpha,1}\leq ... $$
	such that each $x_{\alpha,n}\leq x$ and $x_{\alpha,n}\leq x_{\alpha,n+1}^2$. We let $x_{\alpha,0}=x_{\alpha}$. Assume that $x_{\alpha,n}$ has been constructed, then we have $$x_{\alpha,n}\leq x=x^2\simeq \bigvee_{\beta\in I} x_\beta^2.$$  The $\kappa$-compactness of $x_{\alpha,n}$ implies there exists a $\beta\geq \alpha$ such that $x_{\alpha,n}\leq x_\beta^2$. We let $x_{\alpha,n+1}:=x_\beta$. Then we get the desired sequence of $\kappa$-compact elements. 
	
	We define $\bar{x}_\alpha:=\bigvee_n x_{\alpha,n}$. Then $\bar{x}_\alpha$ is idempotent because $\bar{x}_\alpha= \bigvee_n x_{\alpha,n}\leq \bigvee_n x_{\alpha,n+1}^2\leq \bar{x}_\alpha^2$. Furthermore, $\bar{x}_\alpha$ is $\kappa$-compact because the sequential colimit is a $\kappa$-small diagram. Consequently, $x\simeq \bigvee_\alpha x_\alpha\simeq\bigvee_\alpha \bar{x}_\alpha$ proves that any element in $\cidem(Q)$ is a join of elements in $\cidem(Q)\cap Q^\kappa.$
\end{proof}
\begin{prop}[{\cite[Corollary A.4.9]{anel2023left}}]\label{prefrmadjoints}
	The operations $(-)_{\op{rad}}$ and  $\cidem(-)$ induce the following adjunctions:
	$$\begin{tikzcd}
		\pfrm \arrow[r, "(-)_{\op{rad}}", shift left=2ex] \arrow[r, "\cidem(-)"', shift right=2ex]  & \arrow[l,"\perp","\perp"', hook'] \frm
	\end{tikzcd}  \text{\quad or equivalently\quad} \begin{tikzcd}
		\calg(\prl_0)_{\op{ter}} \arrow[r, "(-)_{\op{rad}}", shift left=2ex] \arrow[r, "\cidem(-)"', shift right=2ex]  & \arrow[l,"\perp","\perp"', hook'] \calg(\prl_0)_{\op{Car}}
	\end{tikzcd} .$$
\end{prop}

\subsection{Artin gluing of frames}\label{glueframes}

Artin gluing is a fundamental construction in locale theory and topos theory, firstly introduced in \cite[Expos\'e IV, \S9]{artin1972theorie}.
\begin{thm}\label{artingluethm}
	Let $O$ and $K$ be frames, and let $\varphi \colon O \to K$ be an order-preserving map that preserves finite meets.
	Then the subposet
	\[
	O\overleftarrow{\times}_{\!\varphi}K
	:=
	\{(a,b)\mid \varphi(a) \geq b\} \subset  O\times K
	\]
	equipped with the pointwise partial order
	forms a subframe of $O\times K$. Furthermore,
	\[
	(O\overleftarrow{\times}_{\!\varphi}K)_{\leq (1,0)}\simeq O
	\quad\text{and}\quad
	(O\overleftarrow{\times}_{\!\varphi}K)_{\geq (1,0)}\simeq K,
	\]
	so that $O$ identifies with an open sublocale\footnote{Note that this is not  a subframe.} of $O\overleftarrow{\times}_{\!\varphi}K$, and $K$ identifies with its closed complement.
\end{thm}
\begin{proof}
	Let $F$ denote $O\overleftarrow{\times}_{\!\varphi}K$. It suffices to show that the inclusion $F\subset O\times K$ is closed under arbitrary joins and finite meets.
	It is obviously closed under finite meets by assumption.	
	We now check that it is closed under arbitrary joins.
	Let $\{(a_i,b_i)\}_{i\in I}\subset F$. Their pointwise join is $(\bigvee_I a_i,\, \bigvee_I b_i)$. We check the gluing condition:
	Given $b_i \leq \varphi(a_i)$ for all $i$, and using $a_i \leq \bigvee_I a_k$ together with order preservation of $\varphi$, we get
	\[
	\varphi(a_i) \leq \varphi\Bigl(\bigvee_I a_k\Bigr) \quad \text{for all } i.
	\]
	Hence $b_i \leq \varphi(\bigvee_I a_k)$ for all $i$, and by the definition of joins,
	\[
	\bigvee_I b_i \;\leq\; \varphi\Bigl(\bigvee_I a_k\Bigr).
	\]
	Thus arbitrary pointwise joins remain in $F$. The bottom element $(0,0)$ is also in $F$ since $0 \leq \varphi(0)$ holds.

	Now 
	the principal lower set is defined by
	\[
	F_{\leq (1,0)} = \{(a,b)\in F \mid (a,b) \leq (1, 0)\}.
	\]
	By the pointwise order, $(a,b) \leq (1, 0)$ means:
	\[
	a \leq 1 \quad\text{(always true),}\qquad
	b \leq 0 \quad\text{(hence } b = 0\text{)}.
	\]
	Plugging the forced condition $b = 0$ back into the definition of $F$, the gluing condition becomes
	$
	0 \leq \varphi(a).
	$
	This holds for all $a \in O$ without any constraint. Therefore, the elements of $F_{\leq (1,0)}$ are precisely
	\[
	F_{\leq (1,0)} = \{(a,\, 0) \mid a \in O\}.
	\]
	Hence
	$
	O \simeq F_{\leq (1,0)},
	$
	realizing $O$ as the open sublocale given by the principal lower set of $(1,0)$.
	
	The argument for the closed part $F_{\geq (1,0)}\simeq K$ is similar, we leave it to the reader.
\end{proof}
\begin{rem}
	Note that the inclusion $O\overleftarrow{\times}_{\!\varphi}K\hookrightarrow O\times K$ is an equivalence if and only if the gluing functor $\varphi(O)=\{1\}$ is trivial.
\end{rem}
\begin{prop}\label{frmopencloseddecomp}
	Let $F$ be a frame and $u\in F$ be an element. Then
	\[
	F \;\simeq\; \Bigl\{(a,b)\in F_{\leq u} \times F_{\geq u} \;\Bigm|\;  u \vee (u \backslash a)\geq b \Bigr\},
	\]
	where
	$u \backslash a \;:=\; \bigvee \{\, y \in F \mid y \wedge u \leq a \,\}.$
	That means $F$ is obtained by gluing along the map $$\varphi:F_{\leq u}\xrightarrow{a\mapsto u\vee (u\backslash a)}F_{\geq u}.$$
\end{prop}
\begin{proof}
	Let
	$P \;=\; \Bigl\{(a,b)\in F_{\leq u} \times F_{\geq u} \;\Bigm|\; u \vee (u \backslash a)\geq b \Bigr\}.$
	We wish to show the canonical map
	$$j_u \colon F \longrightarrow F_{\leq u} \times F_{\geq u}, 
	\qquad
	j_u(x) \;=\; (x \wedge u,\; x \vee u)$$
	restricts an equivalence $j_u:F\xrightarrow{\sim}P$.
	
	We first check for any $x \in F$, its image $(a,b)=(x \wedge u,\, x \vee u)$ lies in $P$, i.e., satisfies the gluing condition. By the implication properties,
	$$u \backslash a \;=\; u \backslash (x \wedge u) \;=\; (u \backslash x) \wedge (u \backslash u) \;=\; u \backslash x,$$ because $u \backslash u=1$.
	One also has $x \leq u \backslash x$, hence
	$$b \;=\; x \vee u \;\leq\; (u \backslash x) \vee u \;=\; u \vee (u \backslash a),$$
	so $j_u(x) \in P$.
	
	We now show $j_u$ is an embedding. By \cref{conserinjectfrms}, it suffices to show that if $x \wedge u = y \wedge u$ and $x \vee u = y \vee u$, then $x=y$.  Since frames are infinitely distributive (hence distributive), expand using distributivity:
	$$x \;=\; x \wedge (x \vee u) \;=\; x \wedge (y \vee u) \;=\; (x \wedge y) \vee (x \wedge u).$$
	Substitute $x \wedge u = y \wedge u$:
	$$x \;=\; (x \wedge y) \vee (y \wedge u) \;=\; y \wedge (x \vee u).$$
	Substitute $x \vee u = y \vee u$:
	$$x \;=\; y \wedge (y \vee u) \;=\; y.$$
	Hence $j_u$ is an embedding.
	
	We now show $j_u$ is surjective.
	Given any $(a,b) \in P$ satisfying the gluing condition, we reconstruct a unique preimage $x \in F$ by
	$x \;=\; b \wedge (u \backslash a).$
	We first verify $x \wedge u = a$. We have
	$$x \wedge u \;=\; (b \wedge (u \backslash a)) \wedge u \;=\; b \wedge ((u \backslash a) \wedge u).$$
	By implication properties, $(u \backslash a) \wedge u = a \wedge u = a$, hence $x \wedge u = b \wedge a$. Since $(a,b) \in F_{\leq u} \times F_{\geq u}$ implies $a \leq u \leq b$, we get $b \wedge a = a$, so $x \wedge u = a$.
	We then verify $x \vee u = b$. We have
	$$x \vee u \;=\; (b \wedge (u \backslash a)) \vee u \;=\; (b \vee u) \wedge ((u \backslash a) \vee u),$$
	using distributivity in frames. Since $b \geq u$ (as $b \in F_{\geq u}$), we have $b \vee u = b$, hence
	$$x \vee u \;=\; b \wedge \bigl(u \vee (u \backslash a)\bigr).$$
	By the gluing condition for $(a,b) \in P$, $b \leq u \vee (u \backslash a)$, so the meet equals the smaller element:
	$$x \vee u \;=\; b.$$
	Thus $j_u$ is surjective.
\end{proof}

\begin{prop}\label{lem:interval-equivalence-implies-injective}
	Let $f:F\to L$ be a morphism of frames, and let
	$x\in F$. Suppose that the induced maps
	\[
	F_{\leq x}\longrightarrow L_{\leq f(x)}
	\qquad\text{and}\qquad
	F_{\geq x}\longrightarrow L_{\geq f(x)}
	\]
	are equivalences of posets. Then $f$ is an embedding. 
	
	Moreover, if $f$ preserves internal homs, i.e., it is a Heyting map, 
	then $f$ is an equivalence.
\end{prop}

\begin{proof}
	By \cref{frmopencloseddecomp}, for every frame $F$ and every
	$x\in F$, the canonical map
	\[
	j_x:F\longrightarrow
	F_{\leq x}\times F_{\geq x},
	\qquad
	y\longmapsto (y\wedge x,y\vee x),
	\]
	is an embedding. Moreover, this construction is natural in frame maps.
	Hence we have a commutative diagram
	\[
	\begin{tikzcd}
		F \arrow[r,"f"] \arrow[d,"j_x"',hook]
		&
		L \arrow[d,"j_{f(x)}",hook]
		\\
		F_{\leq x}\times F_{\geq x}
		\arrow[r,"\sim"']
		&
		L_{\leq f(x)}\times L_{\geq f(x)} .
	\end{tikzcd}
	\]
	Since the bottom horizontal map is an equivalence and both vertical maps
	are embeddings, $f$ is an embedding.

	When $f$ is a Heyting map, the second statement follows from \cref{frmopencloseddecomp}.
\end{proof}
\begin{ex}
	Note that for a general map of frames $F\xrightarrow{f}L\in \frm$ inducing the equivalences $$F_{\leq x}\xrightarrow{\sim}L_{\leq f(x)}\,\,\text{and}\,\, F_{\geq x}\xrightarrow{\sim}L_{\geq f(x)},$$ it does not imply $f$ is an equivalence.
	
	In topology, a continuous bijection that is a homeomorphism on an open set and on its closed complement need not be a global homeomorphism, because the topology on the domain may be finer than that of the target space, so that the closed and open pieces do not “glue together” in the same way as in the target space.
	
	For example, consider the following finite frames (finite distributive lattices are frames). Let
	\[
	F=\{0, x, 1\}
	\]
	be the three-element total order (the topology of the Sierpiński space). Let
	\[
	L=\{0, A, B, 1\}
	\]
	be the four-element Boolean algebra with \(A \wedge B = 0\) and \(A \vee B = 1\). Define a frame map \(f \colon F \to L\) by
	\[
	f(0)=0,\qquad f(x)=A,\qquad f(1)=1.
	\]
	Clearly \(f\) preserves finite meets and arbitrary joins, hence is a legitimate frame map.
	\(F_{\leq x}=\{0,x\}\) maps to \(L_{\leq A}=\{0,A\}\), yielding the induced map \(\{0,x\}\xrightarrow{\ \sim\ }\{0,A\}\), which is an isomorphism.
	\(F_{\geq x}=\{x,1\}\) maps to \(L_{\geq A}=\{A,1\}\), yielding the induced map \(\{x,1\}\xrightarrow{\ \sim\ }\{A,1\}\), which is also an isomorphism.
	However, \(F\) has only three elements while \(L\) has four; the map \(f\) misses the element \(B\). Hence \(f\) is not an isomorphism.
\end{ex}
We end with a criterion when a continuous map of spaces induces a Heyting  map of frames.
\begin{nota}
	Let $X$ be a topological space and $A,B$ be open subsets. We denote $A\Rightarrow B$ (instead of $A\backslash B$) to be the internal hom $\underline{\Hom}_{\mco(X)}(A,B)$ to avoid confusion.
\end{nota}
\begin{prop}\label{prop:inverse-image-preserves-heyting}
	Let $f:X\to Y$ be a continuous map of topological spaces. Write
	$
	f^{-1}:\mco(Y)\rightarrow \mco(X)
	$
	for the inverse-image morphism of frames. Then the following conditions are equivalent:
	\begin{enumerate}
		\item The map $f^{-1}$ preserves Heyting implication, i.e. for all open subsets
		$A,B\subset Y$ one has
		\[
		f^{-1}(B\Rightarrow A)
		=
		f^{-1}(B)\Rightarrow f^{-1}(A).
		\]
		\item For every locally closed subset $L\subset Y$, one has
		\[
		f^{-1}(\overline L)
		=
		\overline{f^{-1}(L)}.
		\]
	\end{enumerate}
	In particular, every open continuous map $f:X\to Y$ satisfies these equivalent
	conditions.
\end{prop}

\begin{proof}
	Recall that the Heyting implication in the frame of open subsets of a topological
	space is given by
	\[
	B\Rightarrow A
	=
	\operatorname{int}\bigl((Y\setminus B)\cup A\bigr)
	=
	Y\setminus \overline{B\setminus A}.
	\]
	Hence, for open subsets $A,B\subset Y$, we have
	\[
	f^{-1}(B\Rightarrow A)
	=
	X\setminus f^{-1}\bigl(\overline{B\setminus A}\bigr).
	\]
	On the other hand,
	\[
	f^{-1}(B)\Rightarrow f^{-1}(A)
	=
	X\setminus
	\overline{f^{-1}(B)\setminus f^{-1}(A)}
	=
	X\setminus
	\overline{f^{-1}(B\setminus A)}.
	\]
	Therefore
	$
	f^{-1}(B\Rightarrow A)
	=
	f^{-1}(B)\Rightarrow f^{-1}(A)
	$
	if and only if
	$
	f^{-1}\bigl(\overline{B\setminus A}\bigr)
	=
	\overline{f^{-1}(B\setminus A)}.
	$
	
	It remains only to observe that the subsets of the form $B\setminus A$, with
	$A,B\subset Y$ open, are precisely the locally closed subsets of $Y$. Indeed, such
	a subset is open in the closed subset $Y\setminus A$, hence locally closed.
	Conversely, if $L$ is locally closed, then $L=O\cap C$ for some open subset
	$O\subset Y$ and some closed subset $C\subset Y$, so
	\[
	L=O\setminus (Y\setminus C),
	\]
	with both $O$ and $Y\setminus C$ open. This proves the equivalence.
	
	Finally, suppose that $f$ is open. We claim that for every subset $S\subset Y$,
	one has
	\[
	f^{-1}(\overline S)
	=
	\overline{f^{-1}(S)}.
	\]
	The inclusion
	$
	\overline{f^{-1}(S)}
	\subset
	f^{-1}(\overline S)
	$
	follows from continuity. Conversely, let $x\in f^{-1}(\overline S)$, and let
	$U\subset X$ be an open neighbourhood of $x$. Since $f$ is open, $f(U)$ is an
	open neighbourhood of $f(x)$. Since $f(x)\in \overline S$, we have
	\[
	f(U)\cap S\neq \varnothing.
	\]
	Thus $U\cap f^{-1}(S)\neq \varnothing$. Hence
	$x\in \overline{f^{-1}(S)}$. Applying this to every locally closed
	$S=L\subset Y$ gives the desired condition.
\end{proof}

\subsection{Categorical ideals}\label{sec2.1}
In this subsection, we recall the notion of categorical ideals (often called sub-unit objects) within symmetric monoidal $\infty$-categories.
 This framework allows for well-behaved higher-categorical operations, including the intersection, union, and radicalization of ideals. For a more comprehensive account of categorical ideals, we refer the reader to \cite[Section~2]{Zariski_Bal}.
	\begin{de}\label{imagedef}
	Let \mcc be a \syminfcat. We say that a map $i: I\to\mb{1}$ is an \textbf{ideal} of $\mc{C}$ if $i$ is a monomorphism, i.e. a $(-1)$-truncated morphism in $\mc{C}$. This  is often called a sub-unit object. For notational convenience, we often denote an ideal $i: I\to\mb{1}$ simply by $I$.
	
	We define the \infcat of ideals $\op{Idl}(\mc{C})$ to be $(\mc{C}_{/\mb{1}})_{\leq-1}$, which inherits a symmetric monoidal structure, because by \cite[\textsection2.2.2 and 4.8.2.15]{ha} the inclusion $\idlc\hookrightarrow \mc{C}_{/\mb{1}}$ admits a symmetric monoidal left adjoint as follows
		$$\begin{tikzcd}
		\mc{C}_{/\mb{1}} \arrow[r, "\im(-)", shift left=1ex]   & \arrow[l,"\perp"', hook', shift left=1ex] \op{Idl}(\mc{C})
	\end{tikzcd} .$$
	 Therefore the symmetric monoidal structure  $\op{Idl}(\mc{C})$ is given by $I\cdot J:= \im(I\otimes J)$.
\end{de}
 The next result records the basic finiteness and radicality properties of this poset of ideals.
\begin{thm}[{\cite[Proposition 2.2.5,  2.2.15 and 2.2.16]{Zariski_Bal}}]\label{idlprefrm}
	Let $\mcc\in\calg(\prl_\kappa)$, where $\kappa\geq\omega$ is a regular cardinal. Suppose that $S\subset \mcc^\kappa$ be a generating set. Then:
	\enu{\item $\op{Idl}(\mc{C})$ is a $\kappa$-coherent preframe in the sense of \cref{defkcoh}.
	Moreover, an ideal $I$ is $\kappa$-compact in $\idlc$ if and only if there exists a $\kappa$-small collection of objects $\{X_\alpha\}\subset S$ and  maps $X_\alpha\to \mb{1}$ such that $\bigvee_\alpha\op{Im}(X_\alpha)=I$.
	\item An ideal $I\subset\mb{1}$ is radical (see \cref{defradprime}) if and only if, for any 
	ideal $J$ such that $J=\im(X)$ for some map $X\to\mb{1}$ with $X\in S$ and that $J^n\subset I$ for some $n\geq 1$, we have $J\subset I$.
	\item Let $\idlc_{\op{rad}}\subset\idlc$ denote the full subcategory spanned by radical ideals. It is a $\kappa$-coherent frame. Moreover, a radical ideal $I$ is $\kappa$-compact in $\idlc_{\op{rad}}$ if and only if there exists a $\kappa$-compact ideal $J\in \idlc^\kappa$ such that $\sqrt{J}=I$. 
	
	}
	Moreover, suppose that $\kappa=\omega$. Let $I\subset\mb{1}$ be an ideal. By \cref{operationsqrt}, we have $\sqrt{I}\simeq\sqrt[1]{I}= \bigvee_{\alpha\in K}I_\alpha$ where $K\subset \idlc$ consists of those compact ideals $I_\alpha$ such that $I^n_\alpha\subset I$ for some $n\geq 1$.
\end{thm}
Thus, under compact generation hypotheses, categorical ideals behave much like ordinary ideals in commutative algebra: they form a coherent frame upon passing to radical ideals, and radicalization can be detected on compact generators. This provides the order-theoretic input required for our later study of smashing ideals. For a more detailed treatment in this direction, we refer the reader to \cite{Zariski_Bal}.

\subsection{Dualizability and atomic generation}
To extend our framework beyond absolute compact generation, this subsection discusses the notions of dualizable $\mathcal{V}$-modules and atomic objects. The use of $\mathcal{V}$-atomic objects provides the appropriate relative version of compactness needed to generalize the classical Balmer spectrum framework. We briefly recall these definitions below; for a comprehensive exposition, we refer the reader to \cite{efimov2024k,ramzi2024dualizable}.

We begin with the relevant notion of morphism between $\mathcal{V}$-modules. 
\begin{de}
	Let $\mc{V}\in \calg(\prl)$ and $\mc{M}\in\prlv:=\modu_{\mc{V}}(\prl)$. 
 We say a \mcv-module map
		$f:\mcm\to\mc{N} \in \prl_\mcv$ is an internal left adjoint if it is a left adjoint in the $(\infty, 2)$-category $\prl_\mcv$. 

\end{de}

\begin{de}\label{dualdef}
	Let $\mc{V}\in \calg(\prl)$. We denote by $\pr^{i L}_\mcv$ the (non-full) subcategory of $\operatorname{Mod}_{\mathcal{V}}(\prl)$ whose morphisms are the internal left adjoints. We denote $\infty$-category $\pr_{\mathcal{V}}^{\mathrm{dbl}}$ to be the full subcategory of $\pr^{i L}_\mcv$ spanned by dualizable $\mathcal{V}$-modules, in other words the (non-full) subcategory of $\operatorname{Mod}_{\mathcal{V}}(\prl)$ spanned by dualizable $\mathcal{V}$-modules, and with morphisms the internal left adjoints. 
\end{de}

\begin{rem}\label{monoidalprvil}
	Note that $\pr^{i L}_\mcv$ and $\pr_{\mathcal{V}}^{\mathrm{dbl}}$ naturally inherit a symmetric monoidal structure from $\pr_{\mathcal{V}}^{L}$ because  for any $\mcc\in\prl_\mcv$, tensor product functor $\mcc\otimes_\mcv-: \prl_\mcv\to\prl_\mcv$ is a 2-functor and thus preserves internal left adjoints.
\end{rem}
We next recall their basic cocompleteness properties. These facts will allow us to create colimits of dualizable categories on underlying categories:
\begin{prop}[{\cite[Corollary 1.32 and Proposition 1.62]{ramzi2024dualizable}}]\label{colimforget}
	The following functors create small colimits
	$$
	\pr^{i L}_\mcv \rightarrow \pr^{L}_\mcv \text{ \,and \,} \prvdbl\to \pr^{iL}_\mcv .
	$$
		Therefore the sources of the functors
	have all small colimits.
\end{prop}
We will also need a way to formalize the idea that a morphism has sufficiently large image. This is captured by the following notion.
\begin{de}
	We say that a functor $F:\mcc\to\mcd$ between \infcats is \textbf{generating} if $F$ admits a conservative right adjoint. 
\end{de}
For module categories, this condition admits several equivalent formulations in terms of generation under colimits and under the $\mathcal{V}$-action. We spell this out since these equivalent criteria will be used repeatedly.
\begin{prop}\label{generatingequiconds}
	Let $\mcv\in\calg(\prl)$ and $F:\mcm\to\mcn\in\prlv$. Then the following are equivalent:
	\enu{\item $F$ is generating.
		\item  $F$ generates $\mcn$ under small colimits.
		\item  $F$ generates $\mcn$ under small colimits and \mcv-actions.
	} 
\end{prop}
\begin{proof}
	The direction $(2)\implies (3)$ is obvious. 
	
	For $(1)\implies (2)$: Let $\mcn' \subset \mcn$ be the full subcategory generated by the essential image of $F$ under small colimits. Since $\mcn'$ is presentable, the inclusion $i: \mcn' \hookrightarrow \mcn$ admits a right adjoint $i^R$.  By the factorization of $F$ $$\mcm\xrightarrow{F'}\mcn'\xrightarrow{i}\mcn,$$
	we see that $i^R$ is conservative too.
	Thus $\mcn' = \mcn$, meaning $F$ generates $\mcn$ under small colimits.
	
	For $(3)\implies (1)$: Assume $F$ generates $\mcn$ under small colimits and $\mcv$-actions. Let $R$ be the right adjoint of $F$. Suppose $Y\to Z \in \mcn$ is a morphism such that $R(Y)\to R(Z)$ is an equivalence. Because $F$ is a $\mcv$-linear functor, we have $F(V\otimes X) \simeq V\otimes F(X) $ for any $V \in \mcv$ and $X \in \mcm$. Therefore, by adjunction, we have:$$\mathrm{Map}_\mcn(V\otimes F(X) , Y) \simeq \mathrm{Map}_\mcn(F(V\otimes X ), Y) \simeq \mathrm{Map}_\mcm(V\otimes X , R(Y))$$ Since the collection of such objects generates $\mcn$ under small colimits by assumption, we must have $Y\to Z$ is an equivalence. Therefore, $R$ is conservative, and $F$ is generating.
\end{proof}

\begin{prop}\label{tensorgenerating}
	Let $\mcv \in \calg(\prl)$. 
	If $F_1:\mcm_1\to \mcn_1$ and $F_2:\mcm_2 \to \mcn_2$ are two generating $\mcv$-module maps, then
	\[
	F_1\otimes_\mcv F_2:\mcm_1\otimes_\mcv \mcm_2 \to \mcn_1\otimes_\mcv \mcn_2
	\]
	is generating too.
\end{prop}
\begin{proof}
	By \cite[Theorem 4.4.2.8]{ha}, $\mcm_1\otimes_\mcv \mcm_2$ can be identified with the geometric realization of the following diagram in \prl
	$$\begin{tikzcd}[row sep=4em] 
		\vdots 
		\arrow[d, shift right=10pt] 
		\arrow[d] 
		\arrow[d, shift left=10pt] 
		\\
		\mcm_1 \otimes \mcv \otimes\mcm_2 
		\arrow[d, "d_0" description, shift right=10pt] 
		\arrow[d, "d_1" description, shift left=10pt] 
		\arrow[u, shift right=5pt] 
		\arrow[u, shift left=5pt] 
	\\
		\mcm_1 \otimes\mcm_2 
		\arrow[u, "s_0" description]                                                                                                                            
	\end{tikzcd}$$
 Combining the construction of Lurie tensor product and the fact that a limit of conservative functors in $\pr^R$ is still conservative, we are done.
\end{proof}
\begin{cor}\label{tensorgenerator}
	Let $\mcv \in \calg(\prl)$ and let $\mcm,\mcn \in \prlv$. 
	If $\{x_\alpha\} \subset \mcm$ and $\{y_\beta\} \subset \mcn$ generate $\mcm$ and $\mcn$, respectively, under small colimits and $\mcv$-actions, then the collection
	$
	\{x_\alpha \otimes y_\beta\}
	$
	generates $\mcm \otimes_\mcv \mcn$ under small colimits and $\mcv$-actions.
\end{cor}
We now turn from generating morphisms to generating objects. The relative analogue of compactness in a $\mathcal{V}$-module is the notion of a $\mathcal{V}$-atomic object.
\begin{de}
	Let $\mc{V}\in \calg(\prl)$ and $\mc{M}\in\prlv:=\modu_{\mc{V}}(\prl)$. 
	
	\enu{\item An object $x \in \mathcal{M}$ is called $\mathcal{V}$-atomic, or simply atomic if the base $\mathcal{V}$ is
		understood, if the $\mathcal{V}$-linear functor $\mathcal{V} \xrightarrow{-\otimes x} \mathcal{M}$ classifies an internal left adjoint, i.e. if $$\udmap_{\mathcal{M}}(x,-): \mathcal{M} \rightarrow \mathcal{V}$$ preserves small colimits and the canonical map $$v \otimes \udmap_{\mathcal{M}}(x, y) \rightarrow \udmap_{\mathcal{M}}(x, v \otimes y)$$ is an equivalence for all $v \in \mathcal{V}, y \in \mathcal{M}$.
		\item The $\mathcal{M}$ is said to be $\mathcal{V}$-atomically generated if the smallest full $\mathcal{V}$-submodule of $\mathcal{M}$ closed under colimits and containing the atomics of $\mathcal{M}$ is $\mathcal{M}$ itself.
	} 
\end{de}
\begin{rem}\label{atomictensor}
	If $x\in\mcm$ and $y\in \mcn$ are \mcv-atomic for two \mcv-modules \mcm and \mcn, then $x\otimes y\in \mcm\otimes_\mcv \mcn$ is \mcv-atomic too by \cref{monoidalprvil}.
\end{rem}
\begin{rem}
	Let $\mcm$ be a \mcv-module. If $x:\mcv\to \mcm$ is \mcv-atomic, then $1\otimes x:\mcw \to \mcw\otimes_{\mcv} \mcm$ is \mcw-atomic. Therefore, by \cref{tensorgenerating}, if \mcm is a \mcv-atomically generated \mcv-module, then $\mcw\otimes_{\mcv} \mcm$ is a \mcw-atomically generated \mcw-module.
\end{rem}
\begin{ex}
	For \mcv-itself,  \mcv-atomic objects in \mcv are exactly those dualizable objects, namely $$\mcv^{\op{at}}=\mcv^d.$$
\end{ex}
\begin{ex}
	For a presentable stable \infcat $\mcc$, an $\op{Sp}$-atomic object in $\mcc$ is precisely a compact object. For a presentable (semi-)additive \infcat $\mcd$, a ($\cmon(\mcs)$-atomic) $\spgeq$-atomic object in $\mcd$ is precisely a compact projective object.
	More generally, if $\mathcal{V}$ is a mode, an object $m$ in a $\mcv$-module $\mathcal{M}$ is $\mathcal{V}$-atomic if and only if $\udmap_{\mathcal{M}}(m,-): \mathcal{M} \rightarrow \mathcal{V}$ preserves colimits. The case over a mode was studied in \cite{ben2024higher}.
\end{ex}
We next recall the fundamental link between atomic generation and dualizability.
\begin{prop}[{\cite[Lemma 1.30, Proposition 1.40 and 1.41]{ramzi2024dualizable}}]\label{conserdual}
	Let $\mcv\in\calg(\prl)$.
	\enu{
		\item  Let $F: \mathcal{M} \rightarrow \mathcal{N}$ be a map in $\prlv$; and let $\iota: \mathcal{K}\rightarrow \mathcal{M}\in\prvil$ be an internal left adjoint such that  $\iota$ is generating. In this case, $F$ is an internal left adjoint if and only if  $F\circ\iota$ is.
		\item 
Let $\iota: \mathcal{K}\rightarrow \mathcal{M}\in\prvil$ be an internal left adjoint such that $\iota$ is generating. If $\mck$ is a dualizable \mcv-module (resp. atomically generated \mcv-module), then so is \mcm.
	\item Any atomically generated \mcv-module is dualizable.
	}
\end{prop}

\begin{de}
Let $\mcv\in\calg(\prl)$.	We denote $\infty$-category $\pr_{\mathcal{V}}^{\mathrm{at}}$ as the full subcategory of $\prvdbl$ spanned by those atomically generated \mcv-modules.
\end{de}

\begin{ex}
	When $\mcv=\opsp$, we have a natural equivalence $$\prat_{\opsp}\xrightarrow[\sim]{(-)^\omega}\catper.$$  When $\mcv=\spgeq$, we have a natural equivalence $$\prat_{\spgeq}\xrightarrow[\sim]{(-)^{\op{cproj}}}\catadidem.$$
\end{ex}
\begin{prop}\label{tensorat}
	Let $\mcv \in \calg(\prl)$. The inclusion $\prvat\hookrightarrow\prvdbl$ is closed under small colimits and tensor products.
\end{prop}
\begin{proof}
	The closure statement for colimits appeared in the proof of \cite[Proposition 1.62]{ramzi2024dualizable}. The closure statement for tensor product follows by combining  \cref{tensorgenerator} and \cref{atomictensor}.
\end{proof}

We will also need a relative version of the fact that the subcategory generated by the image of a functor is again well behaved. In the present setting, dualizability is preserved by this construction.
\begin{prop}[{\cite[Lemma 2.61]{ramzi2024dualizable}}]\label{generatedual}
	Let $\mcv\in\calg(\prl)$ and $f: \mathcal{M} \rightarrow \mathcal{N}$ be a map in $\prvdbl$. The full subcategory of $\mathcal{N}$ generated under colimits and \mcv-tensors by the image of $\mathcal{M}$ is dualizable, and both its inclusion into $\mathcal{N}$ as well as the corestriction of $f$ are internal left adjoints.
\end{prop}
As a consequence, any internal left adjoint between dualizable modules can be factored through the submodule generated by its image. This gives the following factorization system:
\begin{cor}\label{genincfactorization}
	The pair $(\op{Gen},\op{Inc})$ forms a factorization system on \prvdbl in the sense of \cite[Definition 5.2.8.8]{htt}, where $\op{Gen}\subset \funct(\Delta^1,\prvdbl)$ denotes the class of generating morphisms, and $\op{Inc}\subset \funct(\Delta^1,\prvdbl)$ denotes the class of fully faithful morphisms.
	
	Moreover, this restricts to a
	factorization system on \(\prvat\) too.
\end{cor}
We recall a presentability result about categories of dualizable and atomically generated modules themselves.

\begin{thm}[{\cite[Theorem 3.1, Corollary 3.14 and 3.15]{ramzi2024dualizable}}]\label{dualprl}
	Let $\mcv\in\calg(\prl_\kappa)$ where $\kappa>\omega$ is a uncountable regular cardinal. Then: \enu{
		\item We have $\pr^{\op{dbl}}_\mcv\in \calg(\prl_\kappa)$.
		The natural (non-full) inclusion $\prvdbl\to \prl_{\kappa,\mcv}$ preserves and reflects $\lambda$-compact objects for any regular cardinal $\lambda\geq\kappa$.
		\item We have $\pr^{\op{at}}_\mcv$ is presentably symmetric monoidal.}
\end{thm}

By the preceding presentability result, the categories \(\prvdbl\) and
\(\prvat\) admit well-defined dualizable internal Hom objects. We denote these
by
\[
\underline{\operatorname{Hom}}^{\op{dbl}}_{\mcv}(-,-)
\qquad\text{and}\qquad
\underline{\operatorname{Hom}}^{\op{at}}_{\mcv}(-,-),
\]
respectively. Moreover, the right class of the above factorization system,
namely the class of fully faithful morphisms, is preserved by these internal
Hom functors, as follows.

\begin{prop}\label{hom-preserves-ff}
	Let \(\mcv\in\calg(\prl)\) and \(\mcm\in\prvdbl\). Then the internal Hom
	functor
	\[
	\underline{\operatorname{Hom}}^{\op{dbl}}_{\mcv}(\mcm,-)
	\colon
	\prvdbl\longrightarrow\prvdbl
	\]
	preserves fully faithful morphisms.
	
	If \(\mcm\) is atomically generated, then the internal Hom
	functor
	\[
	\underline{\operatorname{Hom}}^{\op{at}}_{\mcv}(\mcm,-)
	\colon
	\prvat\longrightarrow\prvat
	\]
	also preserves fully faithful morphisms.
\end{prop}

\begin{proof}
		We prove the statement for \(\prvdbl\); the proof in \(\prvat\) is
	identical, using the restricted factorization system of
	\cref{genincfactorization}.
	
	Let \(f\colon\mcm\to\mcn\) be a fully faithful morphism in \(\prvdbl\).
	By \cref{genincfactorization}, it suffices to show that
	$
	\underline{\operatorname{Hom}}^{\op{dbl}}_{\mcv}(\mcm,f)
	$
	is right orthogonal to every generating morphism.
	
	Let
	\(g\colon\mcp\to\mcq\) be a generating morphism. By the tensor--Hom
	adjunction, the lifting problem from \(g\) to
	\(\underline{\operatorname{Hom}}^{\op{dbl}}_{\mcv}(\mcm,f)\) identifies
	naturally with the lifting problem from
	\[
	g\otimes_{\mcv}\mcm
	\colon
	\mcp\otimes_{\mcv}\mcm
	\longrightarrow
	\mcq\otimes_{\mcv}\mcm
	\]
	to \(f\). Since tensoring with \(\mcm\) preserves generating morphisms (see \cref{tensorgenerating}),
	the morphism \(g\otimes_{\mcv}\mcm\) belongs to \(\op{Gen}\). Since
	\(f\in\op{Inc}\), the latter lifting space is contractible. Hence
	$
	\underline{\operatorname{Hom}}^{\op{dbl}}_{\mcv}(\mcm,f)
	\in\op{Inc},
	$
	as required.
\end{proof}
\begin{prop}\label{hom-localization-ff}
	Let \(\mcv\in\calg(\prl)\) and let
	$
	f:\mcm\rightarrow\mcn\in\prvdbl
	$
	be a localization. Then, for every \(\mcz\in\prvdbl\),
	precomposition with \(f\) induces a fully faithful morphism
	\[
	f^*:
	\underline{\operatorname{Hom}}^{\op{dbl}}_{\mcv}(\mcn,\mcz)
	\longrightarrow
	\underline{\operatorname{Hom}}^{\op{dbl}}_{\mcv}(\mcm,\mcz).
	\]
	
	The analogous statement holds in \(\prvat\) too: if
	\(f:\mcm\to\mcn \in \prvat\) is a localization, then for every
	\(\mcz\in\prvat\),
	precomposition with \(f\) induces a fully faithful morphism
	\[
	f^*:
	\underline{\operatorname{Hom}}^{\op{at}}_{\mcv}(\mcn,\mcz)
	\longrightarrow
	\underline{\operatorname{Hom}}^{\op{at}}_{\mcv}(\mcm,\mcz).
	\]
\end{prop}

\begin{proof}
	We prove the statement for \(\prvdbl\); the proof in \(\prvat\) is
	identical, using the restricted factorization system of
	\cref{genincfactorization}.
	
	By \cref{genincfactorization}, it suffices to show that \(f^*\) is right
	orthogonal to every generating morphism. Let
	$
	g:\mcp\longrightarrow\mcq
	$
	be generating. By the tensor--Hom adjunction, the lifting space from
	\(g\) to \(f^*\) identifies with the lifting space induced by the
	pushout-product
	\[
	\begin{tikzcd}
		(\mcq\otimes_{\mcv}\mcm) 	\amalg_{\mcp\otimes_{\mcv}\mcm} 	(\mcp\otimes_{\mcv}\mcn) 	 \arrow[d, "g\square f"] \arrow[r] & \mcz \\
		\mcq\otimes_{\mcv}\mcn \arrow[ru, dashed]                                                                               &     
	\end{tikzcd}.
	\]
	Thus it is enough to prove that \(g\square f\) is an equivalence.
	
	Consider the commutative square
	\[
	\begin{tikzcd}
		\mcp\otimes_{\mcv}\mcm
		\arrow[r,"\op{Id}_{\mcp}\otimes f"]
		\arrow[d,"g\otimes\op{Id}_{\mcm}"']
		&
		\mcp\otimes_{\mcv}\mcn
		\arrow[d,"g\otimes\op{Id}_{\mcn}"]
		\\
		\mcq\otimes_{\mcv}\mcm
		\arrow[r,"\op{Id}_{\mcq}\otimes f"']
		&
		\mcq\otimes_{\mcv}\mcn .
	\end{tikzcd}
	\]
	Since tensor product is a \(2\)-functor, it preserves adjunctions and
	localizations. Hence
	$
	\op{Id}_{\mcp}\otimes f
	$ and $
	\op{Id}_{\mcq}\otimes f
	$
	are localizations, and the above square is both horizontally and vertically right adjointable.
	Since \(g\) is generating, \cref{tensorgenerating} implies that
	\(g\otimes\op{Id}_{\mcm}\) is generating. Hence
	$
	(g\otimes\op{Id}_{\mcm})^R
	$
	is conservative.
	Let
	\[
	\eta_{\mcp}:
	\op{Id}
	\longrightarrow
	(\op{Id}_{\mcp}\otimes f)^R(\op{Id}_{\mcp}\otimes f),
	\qquad
	\eta_{\mcq}:
	\op{Id}
	\longrightarrow
	(\op{Id}_{\mcq}\otimes f)^R(\op{Id}_{\mcq}\otimes f)
	\]
	be the localization units. Right adjointability gives a natural
	equivalence
	\[
	(g\otimes\op{Id}_{\mcm})^R
	(\op{Id}_{\mcq}\otimes f)^R
	(\op{Id}_{\mcq}\otimes f)
	\simeq
	(\op{Id}_{\mcp}\otimes f)^R
	(\op{Id}_{\mcp}\otimes f)
	(g\otimes\op{Id}_{\mcm})^R,
	\]
	under which
	$
	(g\otimes\op{Id}_{\mcm})^R(\eta_{\mcq}(y))
	$
	identifies with
	$
	\eta_{\mcp}\bigl((g\otimes\op{Id}_{\mcm})^R(y)\bigr).
	$
	Since \((g\otimes\op{Id}_{\mcm})^R\) is conservative, we have
	\[
	\begin{aligned}
		y\text{ is }(\op{Id}_{\mcq}\otimes f)\text{-local}
		&\Longleftrightarrow
		\eta_{\mcq}(y)\text{ is an equivalence}  \\
		&\Longleftrightarrow
		\eta_{\mcp}\bigl((g\otimes\op{Id}_{\mcm})^R(y)\bigr)
		\text{ is an equivalence} \\
		&\Longleftrightarrow
		(g\otimes\op{Id}_{\mcm})^R(y)
		\text{ is }(\op{Id}_{\mcp}\otimes f)\text{-local}.
	\end{aligned}
	\]
	Therefore, passing to right adjoints, the defining
	square becomes a pullback square in \(\pr^R\).
	Consequently,  by \cref{colimforget} the canonical comparison becomes an equivalence:
	\[
	g\square f:
	(\mcq\otimes_{\mcv}\mcm)
	\amalg_{\mcp\otimes_{\mcv}\mcm}
	(\mcp\otimes_{\mcv}\mcn)
	\xlongrightarrow{\sim}
	\mcq\otimes_{\mcv}\mcn.
	\]
\end{proof}
\begin{rem}
Note that \cref{hom-localization-ff} does not hold if $f: \mcm\to\mcn$ is merely a generating map.
\end{rem}

\section{Smashing Ideals and the Smashing Frames}\label{sec3}
The generalization from classical tensor-triangulated categories to arbitrary presentably symmetric monoidal $\infty$-categories necessitates a reformalization of smashing ideals and smashing frames. In this section, we study these objects through the lens of dualizable and rigid $\infty$-categories.
\subsection{Smashing frames}
Outside the stable setting, the poset of idempotent objects does not in general form a frame. Therefore the foundational definition of the smashing frame must be shifted toward the poset of \emph{coidempotent} objects, as indicated by Aoki \cite{aoki2023sheaves}.

We begin by recalling the relevant notions of idempotent and coidempotent objects, emphasizing the latter as it is the structure that retains good frame-theoretic properties in the non-stable setting.
	\begin{de}
	Let $\mcc$ be a symmetric monoidal $\infty$-category. A \textbf{coidempotent object} is an object $c: C \rightarrow \mb{1}$ in the overcategory $\mathcal{C}_{/ \mathbf{1}}$ such that $C \otimes c: C \otimes C \rightarrow C \otimes \mb{1} \simeq C$ is an equivalence. 
	A cocommutative coalgebra in $C\in\ccalg(\mcc)$ is called coidempotent  if its underlying $\mathbb{E}_0$-coalgebra is coidempotent. 
	
	We denote by $\op{Idem}(\mcc)\subset\mcc_{\mb{1}/}$ the full subcategory  spanned by idempotent objects. We denote by $\cidemc\subset\mcc_{/\mb{1}}$
	the full subcategory  spanned by coidempotent objects. 
\end{de}

\begin{rem}\label{stablecidemidem}
	Let $\mcc\in \calg(\op{Cat}^{\op{st}}_{\infty})$ be stably symmetric monoidal. 
	Then $$\cidemc \simeq \op{Idem}(\mcc)\simeq\calg(\mcc)^{\op{idem}}  ,$$ where the first equivalence follows from \cite[Proposition 2.14]{aoki2023sheaves} and the second one follows from  \cite[Proposition 4.8.2.9]{ha}. 
	
	Note that if $\mcc$ is not stable, then in general $\cidem(\mcc) \neq \Idem(\mcc)$. For instance taking $\mcc=\spgeq$, the inclusion
	\[
	\cidem(\spgeq) \simeq \calg(\spgeq)^{\op{cn}, \idem} \subsetneq \calg(\spgeq)^{\idem} \simeq \Idem(\spgeq)
	\]
	is strict (where the first equivalence follows from \cref{connloc}), because the map $\mathbb{S} \to \mathbb{S}_{(p)}$ is an idempotent algebra that is not $\pi_0$-epimorphic.
\end{rem}
Motivated by this stable comparison, and following Aoki, we use coidempotent objects to define the smashing frame in arbitrary presentably symmetric monoidal $\infty$-categories.
\begin{de}
	Let $\mcv\in \calg(\prl)$. We call $\cidem(\mcv)$ the \textbf{smashing frame} of \mcv. 
\end{de}
This terminology makes sense due to the following theorem by Aoki, which shows that the poset of coidempotent objects is not merely an auxiliary replacement for idempotents, but genuinely forms a frame and participates in a sheaf-theoretic adjunction.
\begin{thm}[{\cite[\textsection3.2]{aoki2023sheaves}}]\label{shvsmadj}
	If $\mcv\in \calg(\prl)$, then $\cidem(\mcv)$ is a frame. Furthermore, it induces  the following sheaves-spectrum adjunction:
	$$\begin{tikzcd}
		\frm  \arrow[r, "\op{Shv}(-)",hook, shift right=-1ex]  & \arrow[l,"\perp"', "\cidem(-)", shift right=-1ex] \calg(\prl) \end{tikzcd}$$
		In particular, $\cidem(-)$ preserves small limits.
\end{thm}

\begin{ex}\label{cartesiancidem}
Assume that $\mcx\in\calg(\prl)$ is Cartesian monoidal. Then	$$\cidem(\mcx)\simeq \mcx_{\leq -1}$$ due to the following diagram:
$$\begin{tikzcd}
	X \arrow[r, "\Delta"] \arrow[rd, Rightarrow, no head] & X\times X \arrow[d] \\
	& *\times X          
\end{tikzcd}$$
because $\Delta$ is an equivalence if and only if the vertical arrow is an equivalence.
\end{ex}
We will later compare smashing frames along symmetric monoidal functors. For this purpose, it is useful to know when such a comparison map is injective and when spatiality can be inherited. We first isolate two elementary frame-theoretic observations.
\begin{lem}\label{conserinjectfrms}
	Let $f:F\to L$ be a conservative finite-meet-preserving functor between posets admitting finite meets. Then $f$ is an embedding\footnote{An embedding between posets can be viewed as a fully faithful functor between $0$-categories.}.
\end{lem}
\begin{proof}
	Indeed, given a such map $f:F\to L$ and $x,y\in F$; in this case $x\leq y$ if and only if the natural map $x\wedge y \to x$ is an equivalence. Therefore $f(x)\leq f(y)\Leftrightarrow x\leq y$.
	
\end{proof}
\begin{de}
Let $L_0\subset L$ be an embedding of posets such that $L$ is a frame. We say $L_0$ is a \textbf{subframe} of $L$ if it is closed under small joins and finite meets.
\end{de}
\begin{lem}\label{subfrmspatial}
	Let $L$ be a spatial frame. Then any subframe of $L$ is spatial too.
\end{lem}
\begin{proof}
	Let $i:L_0\subset L$ be a subframe. Since $L$ is spatial, the collection of all frame maps $\op{pt}(L):=\{F_\alpha:L\to \mcs_{\leq-1}\}$ is jointly conservative. Therefore the collection $\{F_\alpha\circ i:L_0\to \mcs_{\leq-1}|F_\alpha\in \op{pt}(L)\}$ is jointly conservative and $L_0$ is spatial.
\end{proof}

Combining these two elementary observations gives the following useful criterion for comparing smashing frames through conservative monoidal functors.
\begin{prop}\label{conservinjfrms}
Let $F:\mcc\to\mcd$ be a morphism in $\calg(\Cat_{\infty})$ whose underlying functor is conservative. Then the induced map of posets
	$$\cidem(\mcc)\to\cidem(\mcd)$$ is an embedding (cf. \cite[Lemma 3.4]{aoki2024smashing}). 
	
	Furthermore, suppose that both $\mcc,\mcd$ are presentably symmetric monoidal and the morphism $\mcc\to\mcd$ lies in $\calg(\prl)$. In this case if $\cidem(\mcd)$ is spatial, then so is  $\cidem(\mcc)$.
\end{prop}
\begin{proof}
The embedding follows from \cref{conserinjectfrms}, because both $\cidem(\mcc)$ and $\cidem(\mcd)$ admit finite meets, which are given by the tensor product. The statement about the spatiality follows from \cref{subfrmspatial}.
\end{proof}

\subsection{Smashing ideals}
In this subsection, we first recall the identification between coidempotent objects and smashing ideals within the framework of dualizable $\mathcal{V}$-modules introduced in \cite{Zariski_Bal}. Based on this identification, we prove the $\kappa$-coherence of the smashing frame when $\mcv$ is $\kappa$-presentably symmetric monoidal for an uncountable $\kappa$.

	\begin{de}
	Let $\mcv\in \calg(\prl)$. We call a fully faithful map $i:\mci\hookrightarrow \mcv \in \prlv$ of \mcv-modules
	as  a \textbf{localizing ideal}.
	We call a fully faithful internal left adjoint $i:\mci\hookrightarrow \mcv \in \prvil$ of \mcv-modules as a \textbf{smashing ideal} of \mcv.
	
\end{de}
\begin{rem}
	A  localizing ideal can be identified with a presentable full subcategory  closed under colimits and \mcv-actions (see \cite[Proposition B.16]{Zariski_Bal}).
	A smashing ideal of \mcv is always a dualizable \mcv-module, since it is a retract of \mcv in \prlv.
\end{rem}
\begin{nota}	 	
	If $S\subset \mcv$ is a small set of objects, we define $\left<S\right>\subset\mcv$ to be the smallest \mcv-submodule (or equivalently smallest localizing ideal) containing $S$. 
\end{nota}

The following theorem from \cite{Zariski_Bal} identifies coidempotent objects of $\mcv$, coidempotent objects of $\prvdbl$, and smashing ideals of $\mcv$. In particular, it relates the study of smashing ideals to the theory
of dualizable categories.
\begin{thm}[{\cite[Theorem 3.2.3 and 3.2.8]{Zariski_Bal}}]\label{smideal}
	Let $\mcv\in \calg(\prl)$. 
	\enu{ \item 
		Let $i:\mci\to \mcv\in\prvdbl$ be a \mcv-internal left adjoint between dualizable \mcv-modules. Then $i:\mci\to \mcv$ is coidempotent in \prvdbl  if and only if $i$ is fully faithful. Therefore we obtain the following one-to-one correspondence
		$$\cidem(\prvdbl)
		\simeq\{\text{smashing ideals of }\mcv\}.$$
		\item 	The assignment $x\mapsto \left<x \right>$ induces the following equivalence between coidempotent objects and smashing ideals $$\left<- \right>: \cidem(\mcv)\xrightarrow{\sim}\cidem(\prvdbl).$$
	}
\end{thm}

\begin{rem}
	Note that the equivalence $\left<- \right>: \cidem(\mcv)\xrightarrow{\sim}\cidem(\prvdbl)$ is functorial in \mcv.
	Indeed, given $F:\mcv\to \mcw \in\calg(\prl) $ and a coidempotent object $x\in\cidem(\mcv)$. Then we have a natural equivalence $$\left<F(x)\right>\simeq \left<x\right>\otimes_\mcv \mcw,$$ because $\prvdbl\xrightarrow{(-)\otimes_\mcv\mcw}\prwdbl$ preserves both generating functors and fully faithful functors.
\end{rem}
\begin{rem}
	For two smashing ideals $\mci, \mck$ of $\mcv$, their meet in $\cidem(\prvdbl)$ is
	given by $\mci\otimes_{\mcv}\mck$, which can be identified with the
	intersection $\mci\cap\mck\subset\mcv$.
\end{rem}
Although the terminology may suggest that smashing ideals are kernels of
smashing localizations, in the non-stable setting they are, however, not in
general in one-to-one correspondence with smashing localizations. We therefore
pause to clarify the relation between the two notions.

\begin{de}\label{smloc}
	Let $\mcv\in\calg(\prl)$. We say that a morphism of \mcv-modules $(\mcv\to\mcw)\in \prlv$ is a \textbf{smashing localization} if it is both a \mcv-internal left adjoint and a localization.
\end{de}
\begin{rem}
	Since a smashing localization $\mcv\to \mcw$ preserves $\mcv$-actions by definition, it is compatible with the monoidal structure. Consequently, $\mcw$ inherits a natural symmetric monoidal structure.
\end{rem}
\begin{warn}\label{smaidllocnotmatch}
	We emphasize that, in the general non-stable setting, smashing ideals are not in one-to-one correspondence with smashing localizations. More explicitly, we have
	\[
	\{\text{smashing localizations of }\mcv\}\simeq \Idem(\mcv)
	\neq
	\cidem(\mcv)
	\simeq
	\{\text{smashing ideals of }\mcv\};
	\]
cf. \cref{stablecidemidem} or \cite[Remark 4.5.4]{Zariski_Bal}.
\end{warn}

Having clarified the relation between coidempotents, smashing ideals, and smashing localizations, we now turn to a finiteness property of the smashing frame. We end this subsection by proving its $\kappa$-coherence for uncountable $\kappa$.
\begin{de}
Let $\mcv \in \calg(\prl_\kappa)$, where $\kappa > \omega$ is an uncountable regular cardinal.	We let 
	\[
	\underline{\idl}(\prvdbl)\hookrightarrow(\prvdbl)_{/\mcv}
	\quad\text{and}\quad
	\underline{\idl}(\prl_{\kappa,\mcv})\hookrightarrow(\prl_{\kappa,\mcv})_{/\mcv}
	\]
	denote the full subcategories spanned by the fully faithful morphisms with target $\mcv$.
\end{de}
\begin{rem}
	By \cref{smideal}, the category $\underline{\idl}(\prvdbl)$ identifies with the poset $\cidem(\prvdbl)$ of smashing ideals. Note that $\underline{\idl}(-)$ here refers to a notion of $(\infty,2)$-ideal, rather than the $(\infty,1)$-ideals appearing in \cref{imagedef}. Since we will not use this notion again, we do not go into further details about $(\infty,2)$-ideals.
\end{rem}
The next proposition is the main technical input for the coherence statement.
\begin{prop}\label{unidl}
	Let $\mcv \in \calg(\prl_\kappa)$, where $\kappa > \omega$ is an uncountable regular cardinal. The full subcategory $\underline{\idl}(\prvdbl)\hookrightarrow(\prvdbl)_{/\mcv}$ is a $\kappa$-accessible reflective subcategory and is compatible with the monoidal structure on $(\prvdbl)_{/\mcv}$. Therefore the induced \syminfcat $\underline{\idl}(\prvdbl)$ lies in $\calg(\prl_{\kappa})$.

\end{prop}
\begin{proof}
	By \cref{generatedual}, we see that the following diagram is horizontally left adjointable in the sense of
	\cite[Definition 4.7.4.13]{ha}:
		$$\begin{tikzcd}
		(\prvdbl)_{/\mcv} \arrow[d] \arrow[r, "\underline{\op{Im}}", dashed, shift left]    & \underline{\idl}(\prvdbl) \arrow[d] \arrow[l, hook', shift left]    \\
		{(\prl_{\kappa,\mcv})_{/\mcv}} \arrow[r, "\underline{\op{Im}}", dashed, shift left] & {\underline{\idl}(\prl_{\kappa,\mcv})} \arrow[l, hook', shift left]
	\end{tikzcd}$$
	Since $\prl_{\kappa,\mcv}\simeq \modu_{\mcv^\kappa}(\Cat_\infty^{\kappa\text{-}\op{rex}})$, the inclusion $\underline{\idl}(\prl_{\kappa,\mcv})\hookrightarrow(\prl_{\kappa,\mcv})_{/\mcv}$ is closed under $\kappa$-filtered colimits. Thus the inclusion $$\underline{\idl}(\prvdbl)\hookrightarrow (\prvdbl)_{/\mcv}$$ is closed under $\kappa$-filtered colimits. By \cite[Lemma 2.2.4]{Zariski_Bal}, we know the over category $(\prvdbl)_{/\mcv}$ is $\kappa$-presentable too, hence so is $\underline{\idl}(\prvdbl)$.
	
It remains to prove that the localization $(\prvdbl)_{/\mcv}  \xrightarrow{\underline{\im}} \underline{\idl}(\prvdbl)$ is compatible with tensor product. That follows from the fact that both generating morphisms (see \cref{tensorgenerating}) and fully faithful morphisms in \prvdbl are closed under tensor product.
\end{proof}

Via the identification of smashing ideals with coidempotent objects, this $\kappa$-presentability assertion immediately implies the coherence of the smashing frame, generalizing \cite[Proposition 3.4.1]{Zariski_Bal} to the non-stable setting:
\begin{cor}\label{cidemkappaprl}
	Let $\mcv \in \calg(\prl_\kappa)$, where $\kappa > \omega$ is an uncountable regular cardinal.
	Then its smashing frame $\cidem(\mcv)$ is a $\kappa$-coherent frame.
\end{cor}
\begin{proof}
Since by \cref{smideal} we have the equivalences $\cidem(\mcv)\simeq\cidem(\prvdbl)\simeq\underline{\idl}(\prvdbl)$, the claim follows directly from \cref{unidl}.
\end{proof}

\subsection{Smashing ideals versus locally rigid localizations}
A smashing ideal can often be realized as a completion, a phenomenon that naturally motivates the notion of local rigidity. For instance, the smashing ideal $\mcd(\mathbb{Z})_{p\text{-}\op{nil}} \hookrightarrow \mcd(\mathbb{Z})$ of $p$-nilpotent complexes can be recovered via the $p$-completion functor $\mcd(\mathbb{Z}) \to \mcd(\mathbb{Z})_{p\text{-}\op{cpl}} \simeq \mcd(\mathbb{Z})_{p\text{-}\op{nil}}$, which exhibits local rigidity. We will show that this identification holds in full generality.

For a comprehensive treatment of local rigidity, we refer the reader to \cite{arinkin2020stack,krause2024sheaves, ramzi2024locally}. Here, we recall only the specialized formulation needed for our purposes, namely, for commutative algebras over a fixed presentably \syminfcat $\mcv$.
\begin{de}
	Let $\mc{V}\in \calg(\prl)$ and $\mc{W}\in\calg_\mc{V}:=\calg(\prlv)\simeq \calg(\prl)_{\mcv/}$. We say  $\mc{W}$ is a \textbf{locally rigid} $\mc{V}$-algebra, if the following (1) and (2) hold, a \textbf{rigid} $\mc{V}$-algebra if it further satisfies the (3).\enu{
		\item The $\mc{W}$ is a dualizable $\mc{V}$-module.
		\item the multiplication map $\mathcal{W} \otimes_{\mc{V}} \mathcal{W} \rightarrow \mathcal{W}$ is an internal left adjoint in $\operatorname{Mod}_{\mathcal{W} \otimes_{\mc{V}} \mathcal{W}}(\prl)$.
		
		\item The unit $\mb{1}_{\mc{W}}\in\mc{W}$ is $\mathcal{V}$-atomic.
	}
	
\end{de}
The first connection with smashing ideals is that the right adjoint to a smashing ideal inclusion is automatically compatible with the monoidal structure, and in fact produces a locally rigid algebra:
\begin{prop}\label{symlocal}
	Let $\mcv\in\calg(\prl)$ and $i:\mci\hookrightarrow \mcv$ be a smashing ideal. Then the localization $i^R$ is compatible with the symmetric monoidal structure of $\mcv$. Furthermore, the symmetric monoidal localization $i^R:\mcv\to \mci$ is a locally rigid map.
\end{prop}
\begin{proof}
	Given a morphism $f:x\to y \in \mcv$ such that $i^R(f)$ is an equivalence; we wish to show that $i^R(v\otimes f)$ is also an equivalence for any $v\in\mcv$. However by the definition of a smashing ideal, $i^R$ is \mcv-linear, so $i^R(v\otimes f)\simeq v\otimes i^R(f)$ is an equivalence indeed.
	
	For the local rigidity, see \cite[Proposition 4.62]{ramzi2024locally}.
\end{proof}
\begin{rem}
	In particular, a smashing ideal can be viewed as a ``cosmashing'' localization (cf. \cref{smloc}).
\end{rem}
To formulate the resulting correspondence precisely, we introduce the following terminology for localizations that arise in this way.
\begin{de}\label{localrigloc}
	Let $F:\mcv\to \mcw\in\calg(\prl)$. We say $F$ is a \textbf{locally rigid localization} if $F$ admits a fully faithful \mcv-linear left adjoint $F^L$.
\end{de}
\begin{rem}
	A locally rigid localization  is indeed a locally rigid map by \cite[Proposition 4.62]{ramzi2024locally}.
\end{rem}
The following theorem makes precise the slogan that smashing ideals encode the same thing as locally rigid localizations: 
\begin{thm}\label{smidlvslocrig}
	Let $\mcv\in\calg(\prl)$. The construction of \cref{symlocal} provides equivalences of posets $$\calg^{iR\text{-}\mathrm{loc}}(\prlv)^\opp\xrightarrow{\sim}\big((\prlv)_{\mcv/}^{iR\text{-}\mathrm{loc}}\big)^{\opp}\xrightarrow{\sim}\cidem(\mcv),$$
	where the first category denotes the full subcategory of $\calg(\prlv)^\opp$ spanned by locally rigid localizations out of \mcv, and the middle one denotes the full subcategory of $(\prlv)_{\mcv/}^\opp$ spanned by those \mcv-maps $\mcv \to\mcm$ that admits a fully faithful \mcv-left adjoint.
\end{thm}
\begin{proof}
 Applying \cite[Theorem D.3.17]{heyer20246} to the $(\infty,2)$-category $\prlv$, we obtain an equivalence of $(\infty,1)$-categories $$(\prlv)^{iR,\opp}\simeq (\prlv)^{iL}.$$ This induces the second equivalence in the statement $$\big((\prlv)_{\mcv/}^{iR\text{-}\mathrm{loc}}\big)^{\opp}\simeq(\prlv)_{/\mcv}^{iL\text{-}\op{ff}}=\cidem(\prvdbl).$$
 
 For the first equivalence, by \cite[Proposition 4.8.2.9]{ha} it suffices to show that for any smashing ideal $\mci\xhookrightarrow{i}\mcv$, the right adjoint $i^R:\mcv\to\mci$ makes it idempotent in the sense that $\mcv\otimes_\mcv\mci\xrightarrow[i^R\otimes 1]{\sim}\mci\otimes_\mcv\mci$. That holds because by \cref{symlocal} it is a symmetric monoidal localization.
\end{proof}
One important advantage of this localization viewpoint is that it identifies the smashing frame of a smashing ideal with an open part of the original smashing frame.
\begin{cor}\label{cidemideal}
	Let $\mcv\in\calg(\prl)$ and $i:\mci\hookrightarrow \mcv$ be a smashing ideal. Then $i^R:\mcv\to \mci$ induces a localization of frames $\cidem(\mcv)\xrightarrow{} \cidem(\mci),$ whose reflective subcategory
	$$\cidem(\mcv)_{\leq \mci}\simeq\cidem(\mci)$$
can be identified with the open sublocale\footnote{We do not distinguish $\mci$ from its associated coidempotent object $x_\mci$, so that $\cidem(\mcv)_{\leq \mci}$ means $\cidem(\mcv)_{\leq x_\mci}$.} associated to $ \mci\in\cidem(\mcv)$.
\end{cor}
\begin{proof}
	By \cref{smidlvslocrig}, it suffices to show that for a given morphism $G:\mci\to \mck\in \calg(\prl)$, $G$ is a locally rigid localization if and only if the composite $G\circ i^R:\mcv \to \mck$ is. That follows from \cite[Proposition 4.18]{ramzi2024locally}, because $\mcv\xrightarrow{i^R}\mci$ is a locally rigid map.
\end{proof}
This description is compatible with base change. In frame-theoretic terms, passing to a smashing ideal and then extending scalars produces a pushout square.
\begin{cor}\label{smidlpushoutfrms}
	Let $\mcv\to\mcw\in\calg(\prl)$ and $i:\mci\hookrightarrow \mcv$ be a smashing ideal. Then $$\mci\otimes_\mcv\mcw \hookrightarrow \mcv\otimes_\mcv\mcw\simeq\mcw$$ is again a smashing ideal, and the induced square of smashing frames is a pushout in $\frm$:
	$$\begin{tikzcd}[row sep=20pt, column sep=50pt]
		\cidem(\mcv) \arrow[d] \arrow[r, "\cidem(i^R)"]        & \cidem(\mci) \arrow[d]        \\
		\cidem(\mcw) \arrow[r, "\cidem(i^R\otimes_\mcv \mcw)"] & \cidem(\mci\otimes_\mcv \mcw)
	\end{tikzcd}$$
\end{cor}
\begin{proof}
By \cite[Section~6.4.5]{htt}, there exists a natural adjunction
\[
\begin{tikzcd}[column sep=large]
	\frm \simeq \mathcal{LT}\op{op}_0 \arrow[r, "\shv(-)", hook, shift left=1ex] & \arrow[l, "(-)_{\leq-1}", "\perp"', shift left=1ex] \mathcal{LT}\op{op}_\infty
\end{tikzcd}
\]
where $\mathcal{LT}\op{op}_n$ (for $n=0,\infty$) denotes the \infcat of $n$-topoi and left geometric morphisms.
	We denote $\mck:=\mci\otimes_\mcv \mcw$;	by \cref{cidemideal}, we have $$\cidem(\mci)\simeq \cidem(\mcv)_{\leq\mci} \,\,\text{ and }\,\,\cidem(\mck)\simeq \cidem(\mcw)_{\leq\mck}.$$ Therefore applying $\shv(-)$ to the diagram in question, we get 
	$$\begin{tikzcd}
		\shv(\cidem(\mcv)) \arrow[d] \arrow[r] & \shv(\cidem(\mcv))_{/h_{\mci}} \arrow[d] \\
		\shv(\cidem(\mcw)) \arrow[r]           & \shv(\cidem(\mcw))_{/h_{\mck}}          
	\end{tikzcd},$$
	which is a pushout diagram in $\mathcal{LT}\op{op}_\infty$ by \cite[Remark 6.3.5.8]{htt}. The result then follows because $\shv(-)$ creates small colimits.
\end{proof}
\begin{rem}\label{slicelocalepushout}
	In fact, the argument of \cref{smidlpushoutfrms} shows that for any frame map $\mcf\xrightarrow{f} \mcl \in \frm$ and any object $x\in \mcf$, the following diagram $$\begin{tikzcd}
		\mcf \arrow[d] \arrow[r] & \mcf_{/x} \arrow[d] \\
		\mcl \arrow[r]           & \mcl_{/f(x)}         
	\end{tikzcd}$$ is a pushout in $\frm$. Topologically, this corresponds to a pullback of spaces along an open subspace.
\end{rem}
\subsection{Split smashing ideals}

This subsection aims to address a natural question\footnote{We thank Yifan Jin for asking the author this question.}: when is a smashing ideal \mci\ rigid over \mcv?  The answer is governed by whether the right adjoint to the inclusion has one more layer of $\mathcal{V}$-linear adjointness. This leads to the following definition.
\begin{de}
	Let $\mcv\in\calg(\prl)$. A smashing ideal
	$\mci\xhookrightarrow{i}\mcv$ is said to be \textbf{split} if
	$i^R$ is a $\mcv$-internal left adjoint; equivalently, if
	$i^{RR}$ preserves colimits and is $\mcv$-linear.
\end{de}

\begin{rem}
	This condition is equivalent to saying that $\mci$ is a smooth
	dualizable $\mcv$-module, in the sense that its coevaluation functor
	is a $\mcv$-internal left adjoint
	(cf. \cite[Proposition~4.68]{ramzi2024locally}).
\end{rem}
To connect this adjointness condition with rigidity, we need a small observation about dualizable coidempotent objects by Ramzi.
\begin{lem}[{\cite{ramzi-small-idempotent}}]\label{idemdualretract}
	Let $\mcc\in\calg(\Cat_{\infty})$ be a \syminfcat and let $(x\xrightarrow{f} \mb{1})\in \cidem(\mcc)$. If $x$ is dualizable in \mcc, then the composite $$x\xrightarrow{f}\mb{1}\xrightarrow{f^\vee}x^\vee$$ is an equivalence. Similarly, if $(\mb{1}\xrightarrow{g}y)\in \Idem(\mcc)$ is an idempotent object such that $y$ is dualizable in \mcc, then the composite $$y^\vee \xrightarrow{g^\vee }\mb{1}\xrightarrow{g}y$$ is an equivalence.
\end{lem}
\begin{proof}
	We only prove the first statement, because the second one is totally parallel. Since $x\xrightarrow{f}\mb{1}$ is coidempotent, $\mb{1}\xrightarrow{f^\vee}x^\vee$ is idempotent. Note that by dualizability, $x$ is a retract of $x\otimes x^\vee \otimes x$, so $x$ is $x^\vee$-local; similarly we can show that $x^\vee$ is $x$-colocal.
	Therefore we obtain the following diagram:
	$$\begin{tikzcd}
		x\otimes\mb{1} \arrow[d, "f\otimes 1"] \arrow[r, "1\otimes f^\vee", "\sim"'] & x\otimes x^\vee \arrow[d, "f\otimes 1", "\sim"'] \\
		\mb{1}\otimes\mb{1} \arrow[r, "1\otimes f^\vee"]                    & \mb{1}\otimes x^\vee      
	\end{tikzcd}$$
	as desired.
\end{proof}
This lemma shows that dualizability of the associated coidempotent object is equivalent to the existence of a retraction. Combining this with the description of the adjoints associated to a smashing ideal gives the following equivalent characterizations:
\begin{thm}\label{splitsmidl}
Let $\mcv\in\calg(\prl)$.	Let $ \mci\xhookrightarrow{i}\mcv$ be a smashing ideal and $x\xrightarrow{f} \mb{1}$ be the coidempotent object associated to \mci. Then the following are equivalent:
	\enu{
	\item \mci is a split smashing ideal.
	\item The symmetric monoidal localization $i^R:\mcv\to\mci$ exhibits $\mci$ as a rigid algebra over \mcv.
	\item  $x$  is dualizable.
	\item $x\xrightarrow{f} \mb{1}$ admits a retract.
	}
\end{thm}
\begin{proof}
	Since $\mci$ is locally rigid over \mcv, by definition, (2) is equivalent to (3). By \cref{idemdualretract}, (3) is equivalent to (4).
	It suffices to show $(1)\Leftrightarrow  (2)$.
	
	 Now consider the following adjunctions.
	$$\begin{tikzcd}
		\mcv \arrow[r, "i^R", shift left] \arrow[rr, "x\otimes-", bend left=49] & \mci \arrow[l, "i^{RR}", hook', shift left] \arrow[r, "i", hook, shift left] & \mcv \arrow[l, "i^R", shift left] \arrow[ll, "{\unmap_\mcv(x,-)}", bend left=49]
	\end{tikzcd}$$
	Assume that (1) holds; then $\unmap_\mcv(x,-)$ preserves colimits and \mcv-actions. Thus $x$ is atomic in \mcv and hence dualizable.
	
	Assume that (2) holds. Then $x$ is atomic in \mcv and hence $\unmap_\mcv(x,-)$ preserves colimits and \mcv-actions. Since $i^R$ is a localization that preserves colimits and \mcv-actions, we conclude that so is $i^{RR}$.
\end{proof}
Consequently, the previous correspondence between smashing ideals and locally rigid localizations restricts to a correspondence between split smashing ideals and rigid localizations.
\begin{cor}
	The equivalences in \cref{smidlvslocrig} restrict to the following equivalences of posets
$$\calg_{\op{rig}}^{iLR\text{-}\mathrm{loc}}(\prlv)^\opp\xrightarrow{\sim}\big((\prlv)_{\mcv/}^{iLR\text{-}\mathrm{loc}}\big)^{\opp}\xrightarrow{\sim}\cidem_{\op{spl}}(\mcv),$$ where the first one denotes the full subcategory of $\calg(\prlv)^\opp$ spanned by those rigid \mcv-algebras $\mcv\to\mcw$ that admits a fully faithful \mcv-left adjoint, and the middle one denotes the full subcategory of $(\prlv)_{\mcv/}^\opp$ spanned by those \mcv-module maps $\mcv \to\mcm$ that admits both a fully faithful \mcv-left adjoint and a \mcv-right adjoint.
\end{cor}
\begin{rem}
	Split smashing ideals are closed under finite meets in the smashing frame $\cidem(\mcv)$, 
	but not under small joins, and hence do not form a subframe of  $\cidem(\mcv)$ in general.
\end{rem}

\section{Generalized Telescope Conjecture}\label{sec4}
In this section, we introduce the main object of this paper: the \emph{atomic smashing frame}. The (generalized) telescope conjecture asks whether the canonical inclusion from the atomic smashing frame into the full smashing frame is an equivalence.
\subsection{Atomic smashing frames}
In this subsection, we define atomic smashing ideals as those localizing ideals generated by  $\mathcal{V}$-atomic objects. This approach modifies the classical framework by replacing (absolutely) compact generators with a relative counterpart, providing the necessary framework for our study. 

Recall from \cref{smideal} that smashing ideals of $\mcv$ can be identified with coidempotent objects of $\prvdbl$. Since $\prvat$ is the full subcategory of $\prvdbl$ spanned by atomically generated modules, it is natural to isolate those coidempotent objects which already lie in $\prvat$. Our starting point is the following identification.
\begin{prop}\label{atsmashingideals}
	Let $\mcv\in\calg(\prl)$.	The objects in $$\cidem(\prvat)\hookrightarrow\cidem(\prvdbl)$$ can be identified with those $\mcv$-atomically generated smashing ideals of \mcv. In particular, the subposet of atomically generated smashing ideals forms a subframe.
\end{prop}
\begin{proof}
	It follows by combining \cref{smideal} and the definition of \prvat.
\end{proof}
This motivates the following terminology.
\begin{de}
Let $\mcv\in\calg(\prl)$. We call $\mcv$-atomically generated smashing ideals \textbf{atomic smashing ideals} of $\mcv$.

\end{de}
Equivalently, an atomic smashing ideal is a smashing ideal whose underlying $\mcv$-module is generated by dualizable objects of \mcv:
\begin{rem}
		Since $\mcv^{\op{at}}\simeq\mcv^d$, a smashing ideal of \mcv is atomic if and only if it is a localizing ideal generated by a set of dualizable objects in \mcv.
\end{rem}
In the stable case this recovers the usual compact-generation condition, since atomic objects over spectra are precisely compact objects:
\begin{ex}
	Since \opsp-atomic objects coincide with compact objects, we have that $\prat_{\opsp}\simeq\Cat^{\op{perf}}$ and that $\cidem(\prat_{\opsp})$ can be identified with the Balmer frame $\op{Thick}(\opsp^\omega)$, i.e. the poset of thick ideals of $\opsp^\omega$.
\end{ex}
We now translate this subcollection of smashing ideals back into the language of coidempotent objects in $\mcv$.
	\begin{de}
	Using the equivalence in \cref{smideal},
	we define $\cidem_f(\mcv)\subset \cidem(\mcv)$ to be the subframe determined by the following commutative diagram:
	\[
	\begin{tikzcd}
		\cidem_f(\mcv) \arrow[r, hook] \arrow[d, "\sim", "\langle - \rangle"']
		& \cidem(\mcv) \arrow[d, "\sim" , "\langle - \rangle"'] \\
		\cidem(\prvat) \arrow[r, hook]
		& \cidem(\prvdbl).
	\end{tikzcd}
	\]
	Explicitly, the frame $\cidem_f(\mcv)$ consists precisely of those coidempotent objects
	$
	x \to \mathbf{1}
	$
	in $\mcv$ for which the associated smashing ideal
	$
	\langle x \rangle \hookrightarrow \mcv
	$
	is atomic.

	\end{de}

\begin{cov}
		Accordingly, we will freely identify $\cidem(\mcv)$ with $\cidem(\prvdbl)$, and $\cidem_f(\mcv)$ with $\cidem(\prvat)$.
\end{cov}
Although the terminology is suggestive, one should distinguish carefully between atomicity of the associated smashing ideal and atomicity of the coidempotent object itself:
\begin{warning}
	Beware that for a coidempotent object $x\in\cidem(\mcv)$, the smashing ideal $\langle x \rangle \hookrightarrow \mcv$ being atomic does not imply $x$ is \mcv-atomic. In other words, although every split smashing ideal is atomic, the converse need not hold.
	
	For example, let $\Gamma_p \mathbb{S}:=\ker(\mathbb{S}\to\mathbb{S}[p^{-1}])\in\opsp$. Then $\Gamma_p \mathbb{S}\to \mathbb{S}$ is coidempotent and $$\left<\Gamma_p \mathbb{S}\right>=\opsp_{p\text{-}\mathrm{nil}}\hookrightarrow\opsp$$ is an atomic smashing ideal of \opsp because $\opsp_{p\text{-}\mathrm{nil}}$ is compactly generated (\opsp-atomically generated), but $\Gamma_p \mathbb{S}$ itself is not compact.
\end{warning}
 We can now formulate the telescope conjecture for \mcv.
\begin{de}
	Let $\mcv\in\calg(\prl)$. We define $\cidem_f(\mcv)$ to be the \textbf{atomic smashing frame} of \mcv. We say that $\mcv$ satisfies the \textbf{telescope conjecture} if the canonical telescope inclusion is an equivalence of frames: $$\cidem_f(\mcv)\hookrightarrow\cidem(\mcv).$$
\end{de}

\begin{rem}
	By \cref{atsmashingideals}, $\mcv$ satisfies the telescope conjecture if and only if every smashing ideal of $ \mcv$ is an atomic smashing ideal.
\end{rem}
\begin{rem}
	In the literature, the poset \(\op{Thick}(\mcc^d)\) of thick ideals in
	\(\mcc^d\) is sometimes used as a kind of ``generalized Balmer frame'' for
	a big \(tt\)-category \(\mcc\in\calg(\prlst)\). In fact, there is a natural
	identification
	\[
	\op{Thick}(\mcc^d)
	\simeq
	\cidem_f\bigl(\op{Ind}(\mcc^d)\bigr).
	\]
	However, this frame is generally different from the atomic smashing frame
	\(\cidem_f(\mcc)\), and the natural frame map 
	\[
	\op{Thick}(\mcc^d)\simeq
	\cidem_f\bigl(\op{Ind}(\mcc^d)\bigr)
	\xrightarrow{}
	\cidem_f(\mcc)
	\]
	is not an embedding. Consequently, it is not a suitable
	candidate for \cref{qu2}.
	
	For example, let \(\mcc=\sppcpl\). Then
	$
	\op{Ind}(\sppcpl^d)
	\simeq
	\modu_{\mathbb{S}_{p}^{\wedge}}(\opsp)
	$
	by \cite[Theorem 11.3]{naumann2024separable}, and the natural map
	\[
	\op{Thick}(\sppcpl^d)
	\simeq
	\cidem_f\bigl(\modu_{\mathbb{S}_{p}^{\wedge}}(\opsp)\bigr)
	\longrightarrow
	\cidem_f(\sppcpl)
	\]
	is not an embedding.\footnote{In fact, it is a nontrivial localization; see
		\cref{cidematideal}.} In other words, \(\op{Thick}(\sppcpl^d)\) does not
	remember the genuinely \(p\)-complete information.
\end{rem}
One advantage of the atomic smashing frame is its functoriality in \(\mcv\).
\begin{rem}\label{atsmfrmfunctorial}
	Note that any morphism $\mcv\to\mcw$ in $\calg(\prl)$ induces a natural functor $\prvat\to\pr_{\mcw}^{\operatorname{at}}$, since base-change preserves atomic generation. Hence it induces a functor $\cidem_f(\mcv)\to\cidem_f(\mcw)$. Therefore, taking the atomic smashing frame defines a functor
	\[
	\cidem_f(-)\colon\calg(\prl)\to\frm.
	\]
	
\end{rem}
The following product-preservation property gives a first indication that the
atomic smashing frame behaves more naturally than \(\op{Thick}(\mcc^d)\); cf. \cite[Remark~1.6]{MR4814789}.
\begin{prop}\label{atsmfrmprod}
	Let \(\{\mcc_i\}_{i\in I}\subset \calg(\prl_*)\) be a small collection of pointed
	presentably symmetric monoidal \(\infty\)-categories. Then the natural
	map
	\[
	\cidem_f\Bigl(\prod_{i\in I}\mcc_i\Bigr)
	\longrightarrow
	\prod_{i\in I}\cidem_f(\mcc_i)
	\]
	is an equivalence of frames. 
	
	In particular, $\prod_{i\in I}\mcc_i$ satisfies the telescope conjecture if and only if each $\mcc_i$ does.
\end{prop}
\begin{proof}
 Consider the following diagram: 
	$$\begin{tikzcd}
		\cidem_f(\prod_i\mcc_i) \arrow[d, hook] \arrow[r] & \prod_i\cidem_f(\mcc_i) \arrow[d, hook] \\
		\cidem(\prod_i\mcc_i) \arrow[r, "\sim"]           & \prod_i\cidem(\mcc_i) .                 
	\end{tikzcd}$$
	Since \(\cidem(-)\) preserves small products, the top map is an
	embedding. It therefore suffices to prove surjectivity.
	
	Let
	$
	(\mci_i)_{i\in I}
	\in
	\prod_{i\in I}\cidem_f(\mcc_i).
	$
	For every \(i\), choose a small collection
	$
	S_i\subset \mci_i^{\mcc_i\text{-}\op{at}}
	$
	generating \(\mci_i\) as an \(\mcc_i\)-module. If
	\(\mci_i=0\), we may take
	$
	S_i=\{0\},
	$
	because \(\mcc_i\) is pointed and hence the unique  zero
	object is \(\mcc_i\)-atomic\footnote{This fails in the unpointed setting. For example, $\emptyset \in \mcs$ is not dualizable.}.
	
	 Given an \(i\in I\) and
	\(s\in S_i\), let
	\[
	\widetilde s^{\,i}
	:=
	(s_j)_{j\in I}
	\in
	\prod_j\mci_j,
	\qquad
	s_i=s,\quad
	s_j=0\;(j\neq i).
	\]
	Since atomicity in a product module is detected componentwise,
	every \(\widetilde s^{\,i}\) is
	\(\prod_j\mcc_j\)-atomic.
	Let \(e_i\in\prod_j\mcc_j\) be the object whose \(i\)-th component is
	the tensor unit and whose other components are zero. Then
	\[
	e_i\otimes\widetilde s^{\,i}
	=
	(0,\ldots,0,s,0,\ldots),
	\]
	so the \(\prod_j\mcc_j\)-submodule generated by the
	\(\widetilde s^{\,i}\)'s contains all coordinatewise generators.
	Since colimits in the product are computed componentwise, these objects
	generate
	$
	\prod_i\mci_i.
	$
	Hence \(\prod_i\mci_i\) is an atomic smashing ideal.
\end{proof}
\begin{rem}
Note that \cref{atsmfrmprod} fails in the unpointed setting. For example (see \cref{tctopos}), 
	$$\{0,1\}=\cidem_f(\mcs\times \mcs)\neq \cidem_f(\mcs)\times \cidem_f(\mcs)=\{0,1\}^2.$$
\end{rem}
	One may naturally ask whether 
$\cidem_f(-)$
preserves small limits (in the pointed setting), as $\cidem(-)$ does by \cref{shvsmadj}. An affirmative answer would imply that the validity of telescope conjecture is closed under small limits in $\calg(\prl_*)$. 
However, $\cidem_f(-)$ does not preserve small limits in general, and the telescope conjecture is not closed under small limits, even in $\calg(\prlst)$, as the following example shows:
\begin{ex}

	For example, let the indexing category be $I = \mathbb{Z}$, viewed as a poset. Let $k$ be a field. For each $n \in \mathbb{Z}$, define the $\infty$-category
	\[
	\mathcal{C}_n := \mcd(k[x]),
	\]
	the unbounded derived category of $k[x]$. They satisfy the telescope conjecture by \cite{neeman1992chromatic}, because $k[x]$ is Noetherian. Define all transition morphisms $
	F_n \colon \mathcal{C}_n \to \mathcal{C}_{n+1}
	$ in $\calg(\prl)$
	to be induced by base change along the ring homomorphism $k[x] \xrightarrow{ x^2} k[x]$. 	Now consider the limit $\infty$-category
	\[
	\mathcal{C}_\infty = \lim_{n \in \mathbb{Z}} \mathcal{C}_n.
	\]
	We claim that $\mcc_\infty$ does \emph{not} satisfy the telescope conjecture.
	
	Note that $k[x]$ is a free module of rank $2$ over $k[x]$ via $x \mapsto x^2$ (with basis $1, x$), so the map $k[x]\xrightarrow{x^2}k[x]$ is flat.
	In each $\mathcal{C}_n$, define $\mathcal{I}_n$ to be the localizing ideal of $x$-nilpotent complexes. The ideal $\mathcal{I}_n$ is generated by dualizable (compact) objects, i.e, it is an atomic smashing ideal; one perfect generator is the Koszul complex
	$
	K(x) = \cofib\bigl(k[x] \xrightarrow{x} k[x]\bigr).
	$
	We first verify that
	$
	\mathcal{I}_n \otimes_{\mathcal{C}_n} \mathcal{C}_{n+1}=\mathcal{I}_{n+1}.
	$
	The functor $F_n$ sends the generator $K(x)$ to
	\[
	K(x) \otimes_{x \mapsto x^2} k[x] \simeq K(x^2).
	\]
	Since a complex is $x$-nilpotent if and only if it is $x^2$-nilpotent, the localizing ideal generated by $K(x^2)$ coincides with that generated by $K(x)$, namely $\mathcal{I}_{n+1}$. Thus, the image of $\mathcal{I}_n \otimes_{\mathcal{C}_n} \mathcal{C}_{n+1}$ generates $\mathcal{I}_{n+1}$, and the condition holds.
	
	We will show that $\mathcal{C}_\infty^d \cap \mathcal{I}_\infty = \{0\}$ and $\mathcal{I}_\infty \neq 0$; therefore, $\mathcal{I}_\infty\hookrightarrow\mcc_\infty$ is no longer an atomic smashing ideal and $\mcc_\infty$ does not satisfy the telescope conjecture.
	
	Suppose $P \in \mathcal{C}_\infty^d \cap \mathcal{I}_\infty$. Then $P$ is a sequence $(P_n)_{n \in \mathbb{Z}}$ with $P_n \in \mathcal{C}_n^d \cap \mathcal{I}_n$, i.e., each $P_n$ is an $x$-nilpotent perfect complex. We have transition equivalences
	\[
	P_{n+1} \simeq P_n \otimes_{x \mapsto x^2} k[x].
	\]
Since \(P_n\) is perfect and \(x\)-nilpotent, each \(H_i(P_n)\) is a
finitely generated \(x\)-power-torsion \(k[x]\)-module. By the structure
theorem over the PID \(k[x]\), each \(H_i(P_n)\) is a finite direct sum
of modules of the form \(k[x]/(x^m)\). In particular,
\(H_i(P_n)\) is finite-dimensional over \(k\). Its length coincides with its dimension as a $k$-vector space:
	\[
	\operatorname{length}_{k[x]}\bigl(H_i(P_n)\bigr)
	=
	\dim_k\bigl(H_i(P_n)\bigr)
	<
	\infty.
	\]
	Define the “size”
	\[
	l(P_n) = \sum_{i \in \mathbb{Z}} \dim_k\bigl(H_i(P_n)\bigr).
	\]
	Since $P_n$ has only finitely many nonvanishing degrees, $l(P_n)$ is a nonnegative integer (and $P_n \simeq 0 \iff l(P_n)=0$). Because $F_n$ is t-exact and doubles the rank, we have
	\[
	H_i(P_{n+1}) \simeq H_i(P_n) \otimes_{x \mapsto x^2} k[x],
	\]
	which doubles the length of the homology groups:
	$
	l(P_{n+1}) = 2\,l(P_n).
$
	This holds for all $n \in \mathbb{Z}$, including negative $n$. Hence, for any $k>0$,
	$
	l(P_0) = 2^k \, l(P_{-k}).
	$
	Since lengths are integers, this forces $l(P_0)$ to be divisible by arbitrarily large powers of $2$, and therefore $l(P_0)=0$. Thus $P_0 \simeq 0$, and consequently all $P_n \simeq 0$. We conclude that $\mathcal{C}_\infty^d \cap \mathcal{I}_\infty = \{0\}$.
	
	Finally, we prove $\mathcal{I}_\infty\neq0$. An object $(M_n)_{n \in \mathbb{Z}}$ in $\mathcal{I}_\infty$ must satisfy the base-change equivalences, and each $M_n$ must be $x$-nilpotent. We construct an infinite-dimensional torsion module that is strictly preserved under base change: the Prüfer module (the injective hull of $k$)
	\[
	M_n = k[x,x^{-1}]/k[x]
	\]
	(concentrated in degree $0$). This module is torsion: every element is annihilated by some power of $x$, so it is supported at $x=0$, hence $M_n \in \mathcal{I}_n$. In fact, the underlying complex of the idempotent \ein-algebra associated to $\mci_n$ has homology identified with a suspension of this Prüfer module.
	Taking the base change of the short exact sequence
	\[
	0 \to k[x] \xrightarrow{} k[x,x^{-1}] \to M_n \to 0
	\]
	along the flat map $k[x] \xrightarrow{x^2} k[x]$, we again obtain
	\[
	0 \to k[x] \xrightarrow{} k[x,x^{-1}] \to M_n \otimes_{x \mapsto x^2}k[x] \to 0,
	\]
	because
	$
	k[x,x^{-1}] \otimes_{x \mapsto x^2} k[x]
	\xrightarrow[\sim]{x\otimes 1\mapsto x^2}
	k[x,x^{-1}].
	$
	Thus the quotient satisfies
	\[
	M_n \otimes_{ x \mapsto x^2} k[x]
	\simeq
	k[x,x^{-1}]/k[x]
	=
	M_{n+1}.
	\]
	Therefore, the sequence $\bigl(k[x,x^{-1}]/k[x]\bigr)_{n \in \mathbb{Z}}$ defines a strictly nonzero object in $\mathcal{I}_\infty$. 
\end{ex}
\subsection{Recovering the classical case}
In this subsection, we show that our generalized telescope conjecture recovers the classical formulation. In particular, we prove that when \mcv is stable and compactly-rigidly generated, its atomic smashing frame coincides with the classical Balmer frame.

The key point is that, under a rigid or locally rigid change of base, atomic generation can be detected after restricting scalars. We begin with the rigid case.

\begin{prop}\label{atomicmodu}
	Let $\mcv\in\calg(\prl)$. Let $\mcw\in \calg(\prlv)$ be a $\mcv$-atomically generated rigid algebra over \mcv. Let $\mcm\in\prlw$ be a \mcw-module. Then $\mcm$ is $\mcw$-atomically generated if and only if \mcm, regarded as a \mcv-module, is \mcv-atomically generated. Consequently, the equivalence $\modu_{\mcw}(\prvdbl)\simeq \modu_{\mcw}(\prlv)^{\op{dbl}}$ in \cite[Corollary 4.49]{ramzi2024locally} restricts to an equivalence $$\modu_{\mcw}(\prvat)\simeq \modu_{\mcw}(\prlv)^{\op{at}}.$$
\end{prop}
\begin{proof}
	Let $\mcm$ be a \mcw-module. By \cite[Corollary 4.55]{ramzi2024locally},    \mcv-atomic objects coincide with \mcw-atomic objects in $\mcm$. Thus the ``if'' direction holds.
	
	For the  ``only if'' direction, assume that \mcm is \mcw-atomically generated as a \mcw-module. Then \mcm is generated by $\{w\otimes x\mid w\in \mcw \text{ is \mcv-atomic and } x \in \mcm \text{ is \mcv-atomic} \}$ as a \mcv-module. So it suffices to show that such $w\otimes x$ is \mcv-atomic in \mcm, but that follows from \cite[Proposition 4.11]{ramzi2024locally}.
\end{proof}
Applying this comparison to smashing ideals gives the following criterion for the telescope conjecture over a rigid algebra.
\begin{cor}\label{compactrig}
	Let $\mcw\in \calg(\prlv)$ be a $\mcv$-atomically generated rigid algebra over \mcv. Then $\mcw$ satisfies the telescope conjecture if and only if every smashing ideal $\mci\hookrightarrow\mcw$, regarded as a \mcv-module, is $\mcv$-atomically generated.
\end{cor}
Specializing to the spectral base recovers the familiar compactly generated formulation:
\begin{ex}
	In particular when $\mcv=\opsp$,  \cref{compactrig} indicates that for a compactly-rigidly generated $tt$-category \mcw, a smashing ideal $\mci\hookrightarrow\mcw$ is  atomic if and only if \mci is compactly generated. In this case, the atomic smashing frame $\cidem_f(\mcw)$ can be identified with the Balmer frame $\op{Thick}(\mcw^\omega)$.
	Consequently, \mcw satisfies the telescope conjecture if and only if every smashing ideal of \mcw is compactly generated, reducing to the classical sense.
\end{ex}
The same statement has a connective analogue over the additive base $\spgeq$.
\begin{ex}\label{projrigidtc}
	When $\mcv=\spgeq$, \cref{compactrig} indicates that  for a projectively rigid \spgeq-algebra \mcw, a smashing ideal $\mci\hookrightarrow\mcw$ is  atomic if and only if \mci is compact projectively generated.  We will discuss this more in \cref{tcadd}.
\end{ex}
The rigidity assumption in the preceding comparison can in fact be weakened. The same conclusion holds for locally rigid algebras\footnote{We thank Anish Chedalavada and Yifan Jin for pointing this out.}. 

\begin{prop}\label{localrigatgen1}
		Let $\mcw\in \calg(\prlv)$ be a locally rigid algebra over \mcv such that \mcw is \mcv-atomically generated. Let $\mcm$ be a \mcw-module. Then $\mcm$ is  \mcw-atomically generated if and only if $\mcm$, as a \mcv-module, is \mcv-atomically generated.
\end{prop}
\begin{proof}
	Assume that \mcm is \mcw-atomically generated. Since $\mcm$, as a \mcv-module, can be generated by $$\{w\otimes m\mid w\in \mcw^{\mcv\text{-}\op{at}}, m \in \mcm^{\mcw\text{-}\op{at}}\} \subset \mcm$$ by assumption, it suffices to show that any such $w\otimes m\in \mcm$ is \mcv-atomic. That follows from the natural equivalence
	$$\unmap_{\mcm}^{\mcv}(w\otimes m,-)\simeq \unmap_{\mcm}^{\mcv}\big(w, \unmap_{\mcm}^{\mcw}(m,-)\big).$$
	
	Conversely, assume that \mcm, as a \mcv-module, is \mcv-atomically generated. By \cite[Proposition 4.11]{ramzi2024locally}, the action map $\mcw\otimes_{\mcv}\mcm\to\mcm$ is a \mcw-internal left adjoint (which is clearly a generating functor). By \cref{conserdual}, it suffices to observe that $\mcw\otimes_\mcv \mcm$ is \mcw-atomically generated, but that holds because basechange preserves atomic generation.
\end{proof}
\begin{cor}\label{localrigatgen}
	Let $\mcw\in \calg(\prlv)$ be a locally rigid algebra over \mcv such that \mcw is \mcv-atomically generated. Given a smashing ideal $\mci\hookrightarrow\mcw$, then \mci is an atomic smashing ideal of \mcw if and only if, regarded as a \mcv-module, \mci is \mcv-atomically generated.
\end{cor}
\begin{ex}
	Suppose that $\mcc\in\calg(\prlst)$ is locally rigid stable and compactly generated (e.g. $\opsp_{T(n)}, \opsp_{K(n)}$ and $\opsp_{p\text{-}\op{cpl}}$). Then a smashing ideal $\mci\hookrightarrow \mcc$ is atomic if and only if it is compactly generated.
\end{ex}

\subsection{Basic properties of atomic smashing frames}
In this subsection, we study some basic properties of atomic smashing ideals and atomic smashing frames. In particular, we prove that if $\mcv$ has compact unit, then its atomic smashing frame is coherent, and hence spatial. 

\begin{prop}\label{cidematideal}
Let $\mcv\in\calg(\prl)$. Let $i:\mci\hookrightarrow\mcv$ be an atomic smashing ideal. Then the following diagram induced by the symmetric monoidal localization $i^R:\mcv\to \mci$ (see \cref{symlocal}) is horizontally right adjointable.
	$$\begin{tikzcd}
		\cidem_f(\mcv) \arrow[d, hook] \arrow[r] & \cidem_f(\mci) \arrow[d, hook] \arrow[l, dashed, hook', shift left=2] \\
		\cidem(\mcv) \arrow[r]                   & \cidem(\mci) \arrow[l, dashed, hook', shift left=2]                  
	\end{tikzcd}$$
	In particular, we have $$\cidem_f(\mci)\simeq \cidem_f(\mcv)_{\leq \mci}.$$
\end{prop}
\begin{proof}
	By \cref{cidemideal}, we have $\cidem(\mci)\simeq \cidem(\mcv)_{\leq \mci}$. Applying \cref{localrigatgen} to $(\mcv\to\mcw) = (i^R:\mcv\to \mci)$, we see that the equivalence $\cidem(\mci)\simeq \cidem(\mcv)_{\leq \mci}$ restricts to $$\cidem_f(\mci)\simeq \cidem_f(\mcv)_{\leq \mci},$$
	as desired.
\end{proof}
This compatibility implies that the atomicity of a smashing ideal inside an atomic smashing ideal $\mci$ can be detected in \mcv:
\begin{rem}\label{atsmidltc}
	The induced diagram from adjointability in \cref{cidematideal}
	$$
	\begin{tikzcd}
		\cidem_f(\mcv) \arrow[d, hook]  & \cidem_f(\mci) \arrow[d, hook] \arrow[l, hook'] \\
		\cidem(\mcv)                   & \cidem(\mci) \arrow[l, hook']
	\end{tikzcd}
	$$
	is a pullback diagram\footnote{Note that the horizontal inclusions are not frame morphisms, so this is not a pullback diagram in $\frm$.} of posets by \cref{localrigatgen}.
	In particular, if $\mci\hookrightarrow\mcv$ is an atomic smashing ideal and $\mcv$ satisfies the telescope conjecture, then  $\mci$ satisfies the telescope conjecture too.
\end{rem}
\begin{cor}\label{atsmidlpushoutfrms}
	Let
	$
	F:\mcv\to\mcw
	$
	be a morphism in \(\calg(\prl)\), and let
	$
	i:\mci\hookrightarrow\mcv
	$
	be an atomic smashing ideal. Then
	$
	\mci\otimes_\mcv\mcw\hookrightarrow\mcw
	$
	is again an atomic smashing ideal, and the induced square
	\[
	\begin{tikzcd}[row sep=20pt, column sep=50pt]
		\cidem_f(\mcv) \arrow[r] \arrow[d]
		&
		\cidem_f(\mci) \arrow[d]
		\\
		\cidem_f(\mcw) \arrow[r]
		&
		\cidem_f(\mci\otimes_\mcv\mcw)
	\end{tikzcd}
	\]
	is a pushout square in \(\frm\).
\end{cor}

\begin{proof}
	Atomic smashing ideals are preserved under base change. By
	\cref{cidematideal}, the horizontal maps identify with the principal
	open localizations
	\[
	\cidem_f(\mcv)
	\longrightarrow
	\cidem_f(\mcv)_{\leq\mci}
	\]
	and
	\[
	\cidem_f(\mcw)
	\longrightarrow
	\cidem_f(\mcw)_{\leq\mci\otimes_\mcv\mcw}.
	\]
	The claim therefore follows from \cref{slicelocalepushout}.
\end{proof}

\begin{prop}\label{atsmfrminjsurj}
	Let \(f\colon \mcc\to \mcd\) be a morphism in \(\calg(\prl)\). Then:
	\enu{
		\item If \(f\) is conservative, then the induced map
		$
		\cidem_f(\mcc)\to \cidem_f(\mcd)
		$
		is an embedding.
		
		\item If \(f^d\colon \mcc^d\to \mcd^d\) is essentially surjective on
		dualizable objects, then the induced map
		$
		\cidem_f(\mcc)\to\cidem_f(\mcd)
		$
		is surjective.
	}
\end{prop}

\begin{proof}
	For (1), consider the commutative diagram
	\[
	\begin{tikzcd}
		\cidem_f(\mcc) \arrow[d, hook] \arrow[r]
		& \cidem_f(\mcd) \arrow[d, hook] \\
		\cidem(\mcc) \arrow[r, hook]
		& \cidem(\mcd).
	\end{tikzcd}
	\]
	The bottom horizontal map is an embedding by \cref{conservinjfrms};
	hence the top horizontal map is also an embedding.
	
	For (2), it suffices to observe that, for every full subcategory
	\(S\subset \mcc^d\), there is a canonical identification\footnote{This identification need not hold if \(S\) is not contained in
		\(\mcc^d\).}
	\[
	\mcd\otimes_{\mcc}\langle S\rangle_{\mcc}
	\simeq
	\langle f(S)\rangle_{\mcd}
	\subset \mcd .
	\]
	Since \(f^d\colon \mcc^d\to \mcd^d\) is essentially surjective, every
	atomic smashing ideal of \(\mcd\) is generated by the image of one from
	\(\mcc\). Thus \(\cidem_f(\mcc)\to\cidem_f(\mcd)\) is surjective.
\end{proof}

We now state the main theorem of this section.
\begin{thm}\label{atsmfrmcoh}
Let $\mcv\in\calg(\prl)$ and let $\kappa\geq\omega$ be a regular cardinal. If the unit $\mb{1}_{\mcv}$ is
	$\kappa$-compact, then the atomic smashing frame
	$\cidem_f(\mcv)$ is $\kappa$-coherent. In particular, when $\kappa=\omega$, $\cidem_f(\mcv)$ is a coherent frame (and hence spatial).
	
	Moreover, its \(\kappa\)-compact elements are exactly the
	\(\kappa\)-small joins of principal atomic smashing ideals:
	\[
	\cidem_f(\mcv)^\kappa
	=
	\left\{
	\bigvee_{\alpha\in A}\langle x_\alpha\rangle
	\;\middle|\;
	A \text{ is }\kappa\text{-small,}\quad
	x_\alpha\in\mcv^d
	\right\}.
	\]
\end{thm}

\begin{proof}
	Set
	$$
	L:=\cidem_f(\mcv),
	\qquad
	L_0:=
	\{\langle x\rangle\mid x\in\mcv^d\}\subset L.
	$$
	It suffices to prove that \(L_0\) is a join-dense collection of $\kappa$-compact
	elements of \(L\), closed under finite meets. This implies that \(L\) is $\kappa$-coherent.
	
	We first claim that
	$
	L_0\subset L^\kappa.
	$ Since $\mb{1}_{\mcv}$ is $\kappa$-compact, every dualizable object
	$x\in\mcv^d$ is $\kappa$-compact.
	Let $x\in\mcv^d$ and suppose that
	\[
	\langle x\rangle
	\leq
	\bigvee_{\alpha\in A}\mathcal I_\alpha.
	\]
	Let $\mathcal P_{<\kappa}$ denote the $\kappa$-filtered poset of subsets
	$F\subset A$ of cardinality $<\kappa$. Put
	$
	\mathcal I_F:=\bigvee_{\alpha\in F}\mathcal I_\alpha,
	$
	and let
	$
	c_F\to\mb{1}_{\mcv}$
	 and
$	c\to\mb{1}_{\mcv}
	$
	be the corresponding coidempotent objects. Then
	$
	c\simeq
	\colim_{F\in\mathcal P_{<\kappa}}c_F.
	$
	Since $\langle x\rangle\leq\bigvee_\alpha\mathcal I_\alpha$,
	we have
	\[
	x\simeq c\otimes x
	\simeq
	\colim_F(c_F\otimes x).
	\]
	As $x$ is $\kappa$-compact, the inverse of this equivalence factors
	through $c_F\otimes x$ for some $\kappa$-small $F$. Hence $x$ is
	a retract of $c_F\otimes x$, so
	\[
	\langle x\rangle
	\leq
	\mathcal I_F
	=
	\bigvee_{\alpha\in F}\mathcal I_\alpha.
	\]
	Thus $\langle x\rangle$ is $\kappa$-compact in $L$.
	
	By definition of the atomic smashing frame, every
	$\mathcal I\in L$ is of the form
	\[
	\mathcal I
	=
	\langle S\rangle
	=
	\bigvee_{x\in S}\langle x\rangle
	\]
	for some $S\subset\mcv^d$. Hence $L_0$ is join-dense in $L$. Therefore we obtain
	$
	L\simeq\op{Ind}_\kappa(L^\kappa).
	$
	
	It remains to show that $L^\kappa$ is closed under finite meets. The
	top element is
	$
	\langle\mb{1}_{\mcv}\rangle\in L,
	$
	and is $\kappa$-compact by the first part. It remains to show that \(L_0\) is closed under binary meets. Let
	\(x,y\in\mcv^d\). Since \(\langle x\rangle\) and
	\(\langle y\rangle\) are smashing ideals,
	\[
	\langle x\rangle\wedge \langle y\rangle
	\simeq
	\langle x\otimes y\rangle .
	\]
	But
	$
	x\otimes y\in\mcv^d .
	$
	Hence
	$
	\langle x\rangle\wedge\langle y\rangle\in L_0 .
	$
	
\end{proof}

We next turn to the locally rigid case. Here one should not expect global
coherence in general; instead, under a natural generation hypothesis, the
atomic smashing frame is locally coherent. The key step is to recognize when
a locally rigid \(\mcv\)-algebra can be realized as an atomic smashing ideal
inside a rigid \(\mcv\)-algebra.

\begin{prop}\label{atsmidlrealize}
	Let \(\mcv\to\mcw\) be a morphism in \(\calg(\prl)\). Then the following
	are equivalent:
	\enu{
		\item \(\mcw\) can be realized as an atomic smashing ideal of a rigid
		\(\mcv\)-algebra.
		
		\item \(\mcw\) is locally rigid over \(\mcv\), and
		$
		\left\langle \mcw^{\mcv\text{-}\op{at}}\right\rangle
		=
		\mcw .$\footnote{This condition is weaker than being \(\mcv\)-atomically
			generated as a \(\mcv\)-module.}
	}
\end{prop}

\begin{proof}
	We first record an observation which will be used in both directions.
	Suppose that
	$
	p\colon \mcd\longrightarrow \mcw
	$
	is a locally rigid localization in \(\calg(\prlv)\), with \(\mcd\) rigid
	over \(\mcv\). Denote its fully faithful \(\mcd\)-linear left adjoint by
	$
	i\colon \mcw\hookrightarrow \mcd .
	$
	Then, for every \(w\in\mcw\), we have
	\[
	w\in\mcw^{\mcv\text{-}\op{at}}
	\quad\Longleftrightarrow\quad
	i(w)\in\mcd^{\mcv\text{-}\op{at}}.
	\]
	Equivalently,
	$
	i\bigl(\mcw^{\mcv\text{-}\op{at}}\bigr)
	=
	i(\mcw)\cap \mcd^{\mcv\text{-}\op{at}}.
$
	Since \(\mcd\) is rigid over \(\mcv\), we moreover have
	$
	\mcd^{\mcv\text{-}\op{at}}=\mcd^{d}.
	$
	
	We now prove the equivalence. Assume first that \(\mcw\) can be realized
	as an atomic smashing ideal of a rigid \(\mcv\)-algebra \(\mcd\). Thus
	there is a locally rigid localization
	$
	p\colon \mcd\longrightarrow \mcw
	$
	and a collection
	\[
	K\subset \mcw\cap\mcd^{d}
	\simeq \mcw^{\mcv\text{-}\op{at}}
	\]
	such that
	$
	\mcw\simeq \langle K\rangle_{\mcd}.
	$
	In particular, \(\mcw\) is locally rigid over \(\mcv\). By the observation
	above, we obtain
	\[
	\mcw
	\subset
	\langle K\rangle_{\mcw}
	\subset
	\left\langle\mcw^{\mcv\text{-}\op{at}}\right\rangle_{\mcw}
	\subset
	\mcw.
	\]
	Hence
	$
	\mcw=
	\left\langle\mcw^{\mcv\text{-}\op{at}}\right\rangle_{\mcw},
	$
	which proves \((1)\Rightarrow(2)\).
	
	Conversely, suppose that \(\mcw\) is locally rigid over \(\mcv\) and that
	$
	\left\langle\mcw^{\mcv\text{-}\op{at}}\right\rangle_{\mcw}
	=
	\mcw.
	$
	Consider the canonical rigidification
	\[
	p\colon
	\mcd:=\op{Rig}_{\mcv}(\mcw)
	\longrightarrow
	\mcw
	\]
	of \cite[Theorem 4.67]{ramzi2024locally}. By the rigidification theorem,
	\(\mcd\) is rigid over \(\mcv\), and \(p\) admits a fully faithful
	\(\mcv\)-linear left adjoint
	$
	i\colon\mcw\hookrightarrow\mcd.
	$
	By \cite[Lemma 4.63]{ramzi2024locally}, this left adjoint is
	automatically \(\mcd\)-linear. Thus \(i:\mcw\hookrightarrow\mcd\) is a smashing
	ideal. By the observation above, we have
	\[
	\mcw\cap\mcd^{d}
	\simeq
	\mcw^{\mcv\text{-}\op{at}}.
	\]
	The assumed generation condition therefore shows that \(\mcw\) is an
	atomic smashing ideal of \(\mcd\), generated by
	\(\mcw\cap\mcd^{d}\). This proves \((2)\Rightarrow(1)\).
\end{proof}

We recall the frame-theoretic notion which will describe the output in the
locally rigid case.

\begin{de}
	A frame \(L\) is called \textbf{locally \(\kappa\)-coherent} if
	\(L\) is \(\kappa\)-compactly generated and the \(\kappa\)-compact elements of \(L\)
	are closed under binary meets. When \(\kappa=\omega\), we simply say that \(L\) is
	\textbf{locally coherent}.
\end{de}

\begin{rem}
	A locally coherent frame is spatial. Moreover, by Stone duality
	\cite[Proposition A.1.3.3]{sag}, locally coherent frames correspond to
	quasi-separated sober spaces admitting a basis of compact open subsets.
\end{rem}

\begin{ex}
	If \(F\) is a $\kappa$-coherent frame and \(x\in F\), then the open sublocale
	\(F_{\leq x}\) is locally  $\kappa$-coherent.
\end{ex}

We can now combine the realization criterion with the coherent case proved
above. The idea is that, under the generation hypothesis, a locally rigid
\(\mcv\)-algebra can be embedded as an atomic smashing ideal in a rigid one.

\begin{cor}\label{locrigloccoh}
	Let \(\mcv\in\calg(\prl)\) have $\kappa$-compact unit, and let
	\(\mcw\in\calg(\prl_{\mcv})\) be locally rigid over \(\mcv\). Suppose that
	\[
	\left\langle \mcw^{\mcv\text{-}\op{at}}\right\rangle
	=
	\mcw .
	\]
	Then \(\cidem_f(\mcw)\) is a locally $\kappa$-coherent frame.
\end{cor}

\begin{proof}
	By \cref{atsmidlrealize}, the hypotheses imply that \(\mcw\) can be
	realized as an atomic smashing ideal of a rigid \(\mcv\)-algebra
	\(\mcd\). By \cref{cidematideal}, this identifies
	\(\cidem_f(\mcw)\) with an open sublocale of \(\cidem_f(\mcd)\).
	Since \(\mcd\) is rigid over \(\mcv\), its unit is \(\mcv\)-atomic and hence $\kappa$-compact.
	Hence \(\cidem_f(\mcd)\) is $\kappa$-coherent by \cref{atsmfrmcoh}. An open
	sublocale of a $\kappa$-coherent frame is locally $\kappa$-coherent, so
	\(\cidem_f(\mcw)\) is locally $\kappa$-coherent.
\end{proof}
In fact, local rigidity is not essential for local coherence of the atomic
smashing frame.\footnote{We thank Ishan Levy for pointing this out.} It would
be enough to realize \(\mcw\) as an atomic smashing ideal in a
\(\mcv\)-algebra whose unit is \(\mcv\)-atomic. One way to do this is through
a categorical one-point compactification. Since setting up that construction
would be unnecessarily elaborate for the present purpose, we give instead a
direct argument of the following stronger result than \cref{locrigloccoh}.
\begin{thm}\label{atsmfrmloccoh}
	Let $\mcv\in\calg(\prl)$. Suppose that
	$
	\left\langle
	\mcv^\kappa\cap\mcv^d
	\right\rangle
	=
	\mcv .$\footnote{Note that this is a weaker condition than having $\kappa$-compact unit.}
	Then the atomic smashing frame $\cidem_f(\mcv)$ is locally $\kappa$-coherent.
	
	Moreover, its \(\kappa\)-compact elements are exactly the
	\(\kappa\)-small joins of principal atomic smashing ideals generated by
	\(\kappa\)-compact dualizable objects:
	\[
	\cidem_f(\mcv)^\kappa
	=
	\left\{
	\bigvee_{\alpha\in A}\langle x_\alpha\rangle
	\;\middle|\;
	A \text{ is }\kappa\text{-small,}\quad
	x_\alpha\in\mcv^\kappa\cap\mcv^d
	\right\}.
	\]
\end{thm}
\begin{proof}
	Set
	\[
	L:=\cidem_f(\mcv),\qquad
	L_0:=\{\langle x\rangle\mid x\in \mcv^\kappa\cap \mcv^d\}\subset L .
	\]
	It suffices to prove that \(L_0\) is a join-dense collection of \(\kappa\)-compact
	elements of \(L\), closed under binary meets. This implies that \(L\) is
	locally \(\kappa\)-coherent.
	
	We first note that $\langle x\rangle$ lies in $ L^\kappa$ for every
	\(x\in\mcv^\kappa\cap\mcv^d\) by the same argument as \cref{atsmfrmcoh}.
	
	We now prove that \(L_0\) is join-dense in \(L\). Let \(\mathcal I\in L\).
	By definition of the atomic smashing frame, there exists
	\(S\subset \mcv^d\) such that
	$
	\mathcal I=\langle S\rangle .
	$
	For \(s\in S\), we claim
	\[
	\langle s\rangle
	=
	\bigl\langle s\otimes a\mid a\in \mcv^\kappa\cap\mcv^d\bigr\rangle .
	\]
	The inclusion \(\supset\) is clear. Conversely, since
	$
	\bigl\langle \mcv^\kappa\cap\mcv^d\bigr\rangle=\mcv ,
	$
	we have
	$
	\mb 1_{\mcv}\in \bigl\langle \mcv^\kappa\cap\mcv^d\bigr\rangle .
	$
	Applying the colimit-preserving functor \(s\otimes-\), we get
	\[
	s=s\otimes \mb 1_{\mcv}
	\in
	\bigl\langle s\otimes a\mid a\in \mcv^\kappa\cap\mcv^d\bigr\rangle .
	\]
	This proves the claim.
	
	Moreover, if \(s\in\mcv^d\) and
	\(a\in\mcv^\kappa\cap\mcv^d\), then
	$
	s\otimes a\in \mcv^\kappa\cap\mcv^d .
	$
	Indeed, \(s\otimes a\) is dualizable, and it is \(\kappa\)-compact since \(a\) is
	\(\kappa\)-compact and \(s^\vee\otimes-\) preserves $\kappa$-filtered colimits:
	\[
	\Map_{\mcv}(s\otimes a,-)
	\simeq
	\Map_{\mcv}(a,s^\vee\otimes -).
	\]
	Consequently,
	\[
	\mathcal I
	=
	\bigvee_{s\in S}\langle s\rangle
	=
	\bigvee_{\substack{s\in S\\ a\in\mcv^\kappa\cap\mcv^d}}
	\langle s\otimes a\rangle ,
	\]
	and every term on the right belongs to \(L_0\). Hence \(L_0\) is
	join-dense in \(L\).
	
	It remains to show that \(L_0\) is closed under binary meets. Let
	\(x,y\in\mcv^\kappa\cap\mcv^d\). Since \(\langle x\rangle\) and
	\(\langle y\rangle\) are smashing ideals,
	\[
	\langle x\rangle\wedge \langle y\rangle
	\simeq
	\langle x\otimes y\rangle .
	\]
	But
	$
	x\otimes y\in\mcv^\kappa\cap\mcv^d .
	$
	Hence
	$
	\langle x\rangle\wedge\langle y\rangle\in L_0 .
	$
	Thus \(L_0\) is closed under binary meets.
\end{proof}
We end this discussion with a computation of the atomic smashing frame of
\(k\)-valued sheaves, which shows that local rigidity alone does not imply
local coherence of the atomic smashing frame.
\begin{de}[{\cite{MR4074766}}]\label{defzdimcorefl}
	
	For a frame $L$, let
	$
	\op{Comp}(L)\subset L
$
	denote the Boolean algebra of complemented elements, and let
	\[
	\op{Z}(L)
	:=
	\left\{
	\bigvee_{\alpha}x_{\alpha}
	\;\middle|\;
	x_{\alpha}\in\op{Comp}(L)
	\right\}
	\subset L
	\]
	be the subframe generated by them. We call $\op{Z}(L)$ as the
	\textbf{zero-dimensional coreflection}\footnote{When $L=\mco(X)$ for a space $X$, one can see that $\op{Z}(L)=\mco_p(X).$} of $L$.
\end{de}

\begin{prop}\label{atsmfrmsheaf}
Let $\mathcal X$ be an $\infty$-topos and let $k$ be a field. 
By Aoki's computation of the smashing frame
\cite[Theorem A]{aoki2024smashing}, there is a canonical isomorphism
$$
\cidem\big(\op{Shv}(\mathcal X;\mcd(k))\big)
\simeq
\mcx_{\leq -1}.
$$
Under this identification, the atomic smashing frame is precisely the
zero-dimensional part:
\[
\cidem_f\big(\op{Shv}(\mathcal X;\mcd(k))\big)
\simeq
\op{Z}\bigl(\mcx_{\leq -1}\bigr).
\]
\end{prop}
\begin{proof}
	Set
	$
	\mcc:=\op{Shv}(\mathcal X;\mcd(k)).
	$ By \cite[Corollary 2.5.4.12]{martini2025presentabilitytopoiinternalhigher}, the dualizable objects of $\mcc$ are
	precisely the locally constant sheaves with perfect values. Let
	$F\in\mcc^d$, and let
	$
	U_F\in\mcx_{\leq -1}
	$
	be the element corresponding to the principal smashing ideal
	$\langle F\rangle_\mcc$. Since $F$ is locally constant, there is an
	effective epimorphism
	$
	q\colon V\longrightarrow \mb{1}_{\mathcal X}
	$
	such that, after decomposing $V$ into a coproduct if necessary,
	$
	q^*F\simeq\underline P
	$
	on each component, for some $P\in\op{Perf}(k)$. Since $k$ is a field,
	a nonzero perfect $k$-complex generates $\mcd(k)$ as a localizing
	tensor ideal. Consequently, after pullback to $V$, the smashing ideal
	generated by $F$ is either zero or the whole category according as
	$P=0$ or $P\neq0$. Equivalently, $q^*U_F$ is locally either $0$ or $1$.
	Hence $q^*U_F$ is complemented. Since complemented subterminal objects
	satisfy effective descent, $U_F$ is complemented as well.
	
	Conversely, let
	$
	U\in\op{Comp}\bigl(\mcx_{\leq -1}\bigr)
	$
	with complement $U^c$. Then
	$$
	\mb{1}_{\mathcal X}\simeq U\amalg U^c,
	$$
	so $\mathcal X$ decomposes into the corresponding two open pieces.
	Let $k_U\in\mcc$ be equal to the constant sheaf $k$ over $U$ and to
	zero over $U^c$. This is locally constant with perfect values, hence
	dualizable, and under Aoki's identification the principal smashing ideal
	$\langle k_U\rangle_\mcc$ corresponds exactly to $U$. Thus principal
	smashing ideals generated by dualizable objects correspond precisely to
	complemented subterminal objects. Taking arbitrary joins gives
	$
	\cidem_f(\mcc)
	\simeq
	\op{Z}\bigl(\mcx_{\leq -1}\bigr).
	$
\end{proof}
\begin{rem}
	In particular, let \(X\) be a topological space. If \(zX\) denotes the
	zero-dimensional reflection of \(X\), namely the topology on the underlying
	set of \(X\) generated by its clopen subsets, then
	\[
	\cidem_f\bigl(\op{Shv}(X;\mcd(k))\bigr)
	\simeq
	\op{Z}\bigl(\mathcal O(X)\bigr)
	\simeq
	\mco_p(X)
	\simeq
	\mathcal O(zX).
	\]
\end{rem}

\begin{ex}
	There exists a locally rigid \opsp-algebra whose atomic smashing frame is not
	locally coherent. Indeed, there exists a locally compact Hausdorff space
	\(Y\) whose zero-dimensional reflection is not locally coherent; see
	\cite[Example 2]{MR1832154}. Therefore
	\[
	\cidem_f\bigl(\op{Shv}(Y;\mcd(k))\bigr)
	\]
	is not locally coherent, although \(\op{Shv}(Y;\mcd(k))\) is locally rigid over \opsp
	(see \cite[Example 4.10]{ramzi2024locally}).
\end{ex}

\subsection{Descent for the telescope conjecture}
In this subsection, we study the descent properties of the telescope conjecture with respect to various Grothendieck topologies. The resulting descent theorems show that the validity of the telescope conjecture is preserved under atomic Hopf and smashing-open coverings.

We begin with a class of morphisms for which dualizability and atomic generation can be detected after base change.
\begin{de}
	Let $f \colon \mcc \to \mcd$ be a symmetric monoidal functor between \syminfcats. We say $f$ is a \textbf{Hopf functor} if it admits a $\mathcal{C}$-linear left adjoint $f^L$.
\end{de}
\begin{rem}
	This is equivalent to the condition that $f$ admits a left adjoint $f^L$ whose natural oplax $\mathcal{C}$-linear structure is $\mathcal{C}$-linear. The terminology stems from the notion of a Hopf adjunction (see \cite{Bruguieres2011}).
\end{rem}
\begin{ex}
	A locally rigid localization (see \cref{localrigloc}) is always a Hopf functor.
\end{ex}
\begin{de}
	Let $\mathcal{V} \in \operatorname{CAlg}(\prl)$, and let $\{f_i: \mathcal{V} \to \mathcal{V}_i\}_{i \in I}$ be a set of morphisms in $\operatorname{CAlg}(\prl)$. We say that this collection is a \textbf{Hopf covering} of $\mathcal{V}$ if the following conditions hold:
	\begin{enumerate}[label=(\arabic*), font=\normalfont]
		\item For each $i \in I$, the functor $f_i: \mathcal{V} \to \mathcal{V}_i$ admits a $\mathcal{V}$-linear left adjoint $f_{i}^L$.
		\item The collection of functors $\{f_i: \mathcal{V} \to \mathcal{V}_i\}_{i \in I}$ is jointly conservative.
	\end{enumerate}
	
	Furthermore, we say that this collection is an \textbf{atomic Hopf covering} of $\mathcal{V}$ if, in addition, each $\mathcal{V}_i$ is $\mathcal{V}$-atomically generated.
\end{de}
\begin{rem}\label{preservinternalhom}
	For any Hopf morphism $f:\mcv\to \mcw \in\calg(\prl)$,  $f$ preserves internal homs.
\end{rem}
\begin{proof}
	Let $f^L: \mcw \to \mcv$ denote the left adjoint to $f$. 
	 We can verify this by checking the universal property via the Yoneda lemma. For any test object $w \in \mcw$, we compute the mapping anima:
	\begin{align*}
		\mathrm{Map}_{\mcw}\big(w, f(\underline{\mathrm{Hom}}_{\mcv}(v_1, v_2))\big) 
		&\simeq \mathrm{Map}_{\mcv}\big(f^L(w), \underline{\mathrm{Hom}}_{\mcv}(v_1, v_2)\big) \quad &&(\text{adjunction } f^L \dashv f) \\
		&\simeq \mathrm{Map}_{\mcv}\big(f^L(w) \otimes v_1, v_2\big) \quad &&(\text{adjunction } \otimes \dashv \underline{\mathrm{Hom}} \text{ in } \mcv) \\
		&\simeq \mathrm{Map}_{\mcv}\big(f^L(w \otimes f(v_1)), v_2\big) \quad &&(\text{by } \mcv\text{-linearity of } f^L) \\
		&\simeq \mathrm{Map}_{\mcw}\big(w \otimes f(v_1), f(v_2)\big) \quad &&(\text{adjunction } f^L \dashv f) \\
		&\simeq \mathrm{Map}_{\mcw}\big(w, \underline{\mathrm{Hom}}_{\mcw}(f(v_1), f(v_2))\big) \quad &&(\text{adjunction } \otimes \dashv \underline{\mathrm{Hom}} \text{ in } \mcw)
	\end{align*}
	Since this equivalence holds naturally for all test objects $w \in \mcw$, the Yoneda lemma implies the desired equivalence:
	$$ f(\underline{\mathrm{Hom}}_{\mcv}(v_1, v_2)) \simeq \underline{\mathrm{Hom}}_{\mcw}(f(v_1), f(v_2)) $$
\end{proof}

\begin{prop}\label{hopfcovproperties}
	Let $\mcv\in\calg(\prl)$.
	\enu{
	\item For a Hopf covering $\{\mcv\to \mcv_i\}$ of \mcv, it jointly reflects dualizable objects. 
	\item If $\{\mcv\to \mcv_i\}$ is an atomic Hopf covering, then for any \mcv-module $\mcm\in\prlv$, it is \mcv-atomically generated if and only if each $\mcv_i\otimes_{\mcv}\mcm$ is $\mcv_i$-atomically generated.
	}

\end{prop}
\begin{proof}
	(1) Note that for an object $x\in\mcv$, it is dualizable if and only if for any $y\in \mcv$, the natural map $\underline{\op{Hom}}_{\mcv}(x,\mb{1})\otimes y\to \underline{\op{Hom}}_{\mcv}(x,y)$ is an equivalence. Therefore it suffices to show that each $\mcv\to\mcv_i$ preserves internal homs, but that follows from \cref{preservinternalhom}.\\
	(2) It only suffices to show the ``if'' direction.  Let $\mcm$ be a \mcv-module such that each $\mcv_i\otimes_{\mcv}\mcm$ is $\mcv_i$-atomically generated.  By assumption, the (\mcv-module) map $\bigoplus_i \mcv_i\xrightarrow{\sum_i f_i^L}\mcv$ is generating. Therefore, the \mcv-module map $$\bigoplus_i (\mcv_i\otimes_{\mcv} \mcm) \xrightarrow{\sum_i f_i^L\otimes \op{id}}  \mcv\otimes_{\mcv}\mcm=\mcm$$ is generating by \cref{tensorgenerating}. Since both $\mcv_i$ and $\mcv_i\otimes_{\mcv}\mcm$ are $\mcv$-atomically generated, applying \cref{conserdual} to a $\mcv_i$-atomic generating set $$\bigoplus_{J_i}\mcv_i \to \mcv_i\otimes_{\mcv}\mcm$$ of $\mcv_i\otimes_{\mcv}\mcm$,  we deduce that $\mcv_i\otimes_{\mcv}\mcm$ is \mcv-atomically generated for each $i\in I$. Applying \cref{conserdual} again to the \mcv-module map $\bigoplus_i (\mcv_i\otimes_{\mcv} \mcm) \xrightarrow{\sum_i f_i^L\otimes \op{id}}  \mcv\otimes_{\mcv}\mcm=\mcm$, we conclude that $\mcm$ is \mcv-atomically generated, as desired.
\end{proof}
This detection statement immediately yields descent for the telescope conjecture along atomic Hopf coverings.
\begin{thm}\label{tcdescent}
	Let $\mcv\in\calg(\prl)$, and let $\{\mcv\to \mcv_i\}$ be an atomic Hopf covering of \mcv. If each $\mcv_i$ satisfies the telescope conjecture, then so does \mcv.
\end{thm}
\begin{proof}
	It directly follows from \cref{hopfcovproperties}(2).
\end{proof}
\begin{rem}
	In fact, \cref{hopfcovproperties}(2) indicates that for an atomic Hopf covering, the following square is
	a pullback in \(\calg(\prl)\):
	\[
	\begin{tikzcd}
		\prvat \arrow[d] \arrow[r, hook]
		&
		\prvdbl \arrow[d]
		\\
		\prod_i \prat_{\mcv_i} \arrow[r, hook]
		&
		\prod_i \prdbl_{\mcv_i}.
	\end{tikzcd}
	\]
	Applying \(\cidem(-)\), one obtains the following pullback square of
	frames:
	\[
\begin{tikzcd}
	\cidem_f(\mcv) \arrow[d, hook] \arrow[r, hook] & \cidem(\mcv) \arrow[d, hook] \\
	\prod_i \cidem_f(\mcv_i) \arrow[r, hook]       & \prod_i \cidem(\mcv_i)      
\end{tikzcd}
	\]
\end{rem}
\begin{ex}
	Let $\mcv\in\calg(\prl)$. Then for any connected anima $X$ with a basepoint $x$, the evaluation functor
	$$x^*:\funct(X,\mcv)\to \mcv$$ forms a Hopf covering, but in general it is not an atomic Hopf covering.
\end{ex}
There is another natural class of coverings adapted more directly to the smashing frame itself, namely coverings by smashing opens.
\begin{de}\label{defsmopencov}
		Let $\mathcal{V} \in \operatorname{CAlg}(\prl)$, and let $\{f_i: \mathcal{V} \to \mathcal{V}_i\}_{i \in I}$ be a (small) collection of morphisms in $\operatorname{CAlg}(\prl)$. We say that this collection is a \textbf{smashing-open covering} of $\mathcal{V}$ if each $\mathcal{V}\to\mcv_i$ is a locally rigid localization and $\{\cidem(\mcv)\to\cidem(\mcv_i)\}$ is jointly conservative\footnote{or equivalently, $\cidem(\mcv)\to \prod_{i}\cidem(\mcv_i)$ is an embedding (see \cref{conserinjectfrms})}.
		
			Furthermore, we say that this collection is an \textbf{atomic smashing-open covering} of $\mathcal{V}$ if, in addition, each $\mathcal{V}_i$ is $\mathcal{V}$-atomically generated.
\end{de}
\begin{thm}\label{atopensmcovtcdescent}
	Let $\mcv\in\calg(\prl)$, and let $\{f_i:\mcv\to \mcv_i\}$ be an atomic smashing-open covering of \mcv. Then $\mcv$ satisfies the telescope conjecture if and only if each $\mcv_i$ does.
\end{thm}
\begin{proof}
	For each \(i\), let
	\(
	\mcu_i\hookrightarrow\mcv
	\)
	be the atomic smashing ideal corresponding to the locally rigid
	localization
	\(
	f_i:\mcv\longrightarrow\mcv_i.
	\)
By \cref{cidemideal} the induced
	map on smashing frames identifies with the principal open localization
	$$
	\cidem(\mcv)
	\longrightarrow
	\cidem(\mcv)_{\leq\mcu_i},
	\qquad
	\mci\longmapsto \mci\wedge\mcu_i.
	$$
	Since the covering is atomic, each \(\mcu_i\) is an atomic smashing
	ideal, and by \cref{cidematideal} we similarly have
	\(
	\cidem_f(\mcv_i)
	\simeq
	\cidem_f(\mcv)_{\leq\mcu_i}.
	\)
	
The ``only if'' direction follows from \cref{atsmidltc}.
	
	Conversely, suppose that every \(\mcv_i\) satisfies the telescope
	conjecture. We claim first that
	\[
	\bigvee_{i\in I}\mcu_i=\mb{1}
	\]
	in \(\cidem(\mcv)\). Indeed, for every \(j\in I\), the images of
	\(\mb{1}\) and \(\bigvee_i\mcu_i\) in
	\(\cidem(\mcv_j)\simeq\cidem(\mcv)_{\leq\mcu_j}\) are both equal to
	\(\mcu_j\):
	$
	\mb{1}\wedge\mcu_j=\mcu_j
	$
	and
	\[
	\left(\bigvee_i\mcu_i\right)\wedge\mcu_j
	=
	\bigvee_i(\mcu_i\wedge\mcu_j)
	=
	\mcu_j,
	\]
	since the join contains the term
	\(\mcu_j\wedge\mcu_j=\mcu_j\).
	As the family
	$
	\left\{
	\cidem(\mcv)\longrightarrow\cidem(\mcv_i)
	\right\}_{i\in I}
	$
	is jointly conservative, it follows that
	$
	\bigvee_i\mcu_i=\mb{1}.
	$
	
	Now let
	$
	\mci\in\cidem(\mcv)
	$
	be an arbitrary smashing ideal. Its image in \(\cidem(\mcv_i)\) is
	$
	\mci\wedge\mcu_i.
	$
	Since \(\mcv_i\) satisfies the telescope conjecture, this image is
	atomic. Using
$
	\cidem_f(\mcv_i)
	\simeq
	\cidem_f(\mcv)_{\leq\mcu_i},
	$
	we conclude that
	$
	\mci\wedge\mcu_i
	\in
	\cidem_f(\mcv)
	$
	for every \(i\).
	
	Finally, using infinite distributivity in the frame
	\(\cidem(\mcv)\), we obtain
	\[
	\mci
	=
	\mci\wedge\mb{1}
	=
	\mci\wedge\bigvee_i\mcu_i
	=
	\bigvee_i(\mci\wedge\mcu_i).
	\]
	Each term on the right belongs to \(\cidem_f(\mcv)\), and
	\(\cidem_f(\mcv)\subset\cidem(\mcv)\) is a subframe, hence is closed
	under arbitrary joins. Therefore
	$
	\mci\in\cidem_f(\mcv).
	$
	Thus every smashing ideal of \(\mcv\) is atomic, and \(\mcv\)
	satisfies the telescope conjecture.	
\end{proof}

\begin{rem}
	It is natural to ask whether a descent result holds for a \emph{smashing-closed} counterpart. We will show that in the compactly-rigidly generated stable case, a similar descent property does hold for atomic smashing-closed coverings (see \cref{finiteatsmcloseddescent}). However, this fails in the unstable setting; the primary reason is that without stability, the canonical correspondence between smashing localizations and smashing ideals fails.
\end{rem}

\subsection{Smashing fields}
In the unstable context, particularly within the framework of $\infty$-topoi, the telescope conjecture admits a simple characterization. In this subsection, we introduce an important class of symmetric monoidal $\infty$-categories satisfying the telescope conjecture, which we call \emph{smashing fields}. We then show that an $\infty$-topos satisfies the telescope conjecture if and only if it is a smashing field.
The terminology is meant to emphasize the analogy with tensor-triangular fields.
\begin{de}
	Let $\Cat_{\infty}^\emptyset$ denote the \infcat of small \infcats admitting an initial object, with morphisms given by functors that preserve the initial object. This category inherits a symmetric monoidal structure from \cite[Corollary 4.8.1.4]{ha}.
	
	 An object $\mcc \in \calg(\Cat_{\infty}^\emptyset)$ corresponds to a symmetric monoidal \infcat whose underlying category has an initial object and whose tensor product preserves the initial object in each variable.
We say that such a symmetric monoidal $\infty$-category $\mcc$ is a \textbf{smashing field} if its frame of coidempotent objects is the two-element frame, that is,
\[
\cidem(\mcc)=\{\emptyset,\mb{1}\}.
\]
\end{de}
\begin{rem}
	A smashing field is always nontrivial because $\emptyset\neq \mb{1}$ in it. 
\end{rem}

The property of being a smashing field can be detected along conservative symmetric monoidal functors.
\begin{prop}\label{consercatfield}
	Let $\mcv\to\mcw\in\calg(\Cat_{\infty}^\emptyset)$ be a conservative symmetric monoidal functor preserving the initial object. If $\mcw$ is a smashing field, then so is \mcv.
\end{prop}
\begin{proof}
	It follows from  \cref{conservinjfrms}.
\end{proof}

\begin{prop}\label{catfieldsTC}
	Let $\mcv \in \calg(\prl)$ be a smashing field. Then it satisfies the telescope conjecture.
\end{prop}
\begin{proof}
	It suffices to observe that both ${\emptyset,\mb{1}}$ lie in $\cidem_f(\mcv)$.
\end{proof}
\begin{rem}
	If $\mcv\in\calg(\prl)$ is a smashing field, then for any connected anima $X$, the \infcat $\funct(X,\mcv)$ is a smashing field too, by the same argument as \cref{ex:borel-spectra}.
\end{rem}
\begin{ex}\label{smfieldex}
	\begin{enumerate}
		\item For any field $k$, the derived category $\mcd(k)$ is a smashing field by \cite{neeman1992chromatic}.
	\item The \infcat $\opsp_{K(n)}$ of $K(n)$-local spectra is a smashing field (see \cite[Theorem 7.5]{hovey1999morava}).
	\item The \infcat $\mcd(\mathbb{Z})_{p\text{-}\op{cpl}}$ of $p$-complete complexes is a smashing field, because the functor $$\mcd(\mathbb{Z})_{p\text{-}\op{cpl}}\xrightarrow{-\otimes^L \mathbb{F}_p} \mcd(\mathbb{F}_p)$$ satisfies the condition in \cref{consercatfield}.
	\item (Counterexample) Note that, however, the \infcat $\sppcpl$ of $p$-complete spectra is \emph{not} a smashing field, because $$\sppcpl\to \opsp_{K(1)}$$ is a non-trivial smashing localization, due to the following diagram
	$$\begin{tikzcd}
		\sppcpl \arrow[r] \arrow[d] & \opsp_{(p)} \arrow[r] \arrow[d] & \opsp_{\mathbb{Q}} \arrow[d, "\sim"] \\
		\opsp_{K(1)} \arrow[r]      & L_1 \opsp \arrow[r]           & L_0 \opsp  .                       
	\end{tikzcd}$$
	\item The \infcat 	\(\op{Sp}_{\mathbb F_p}
	\)
	of $\mathbb F_p$-local spectra is a smashing field (see \cite[Proposition 6.2]{MR3374070}).
	\item 	The \infcat \(
	\op{Sp}_{I_p}
	\) 	of $I_p$-local spectra is a smashing field (see \cite[Proposition 6.6]{MR3374070}). Here $I_p$ denotes the $p$-local Brown--Comenetz dual characterized by the natural equivalences
	\[
	\pi_n\unmap_{\op{Sp}}(X,I_p)
	\simeq
	\op{Hom}_{\op{Ab}}
	\bigl(\pi_{-n}X,\mathbb Q/\mathbb Z_{(p)}\bigr).
	\]
	\item  We will see that the \infcat of connective spectra \spgeq is a smashing field in \cref{noethertelescope}.
	\end{enumerate}
	
\end{ex}

For Cartesian monoidal categories, the smashing frame is especially simple: by \cref{cartesiancidem}, it is identified with the frame of subterminal objects. This leads to the following characterization of the telescope conjecture for an $\infty$-topos.
\begin{thm}\label{tctopos}
Let \(\mcx\in\calg(\prl)\) be presentably Cartesian symmetric monoidal
and nontrivial. Then
\[
\cidem_f(\mcx)
=
\{\emptyset_\mcx,*\}.
\]
Consequently, the following conditions are equivalent:
	\enu{\item $\mcx$ satisfies the telescope conjecture.
		\item $\mcx$ is a smashing field.
		\item $\mcx_{\leq-1}$ consists of only two elements.
	}
\end{thm}
\begin{proof}
	By \cref{cartesiancidem},
	$
	\cidem(\mcx)\simeq\mcx_{\leq-1}.
	$
	We first compute the atomic smashing frame. We claim that the only dualizable
	object of \(\mcx\) is the unit;
	then the atomic smashing frame $$\cidem_f(\mcx)\simeq\{\emptyset_\mcx,*\}$$ consists of only two objects. 
	
	Suppose \(A \in \mcx\) is a dualizable object with dual \(A^\vee\). By definition, there exist maps
	\[
	\operatorname{ev}\colon A^\vee \times A \to * \qquad \text{and} \qquad \operatorname{coev}\colon * \to A \times A^\vee
	\]
	satisfying the triangle identity, i.e., the composite
	\[
	A \simeq * \times A \xrightarrow{\ \operatorname{coev} \times \operatorname{id}_A\ } A \times A^\vee \times A \xrightarrow{\ \operatorname{id}_A \times \operatorname{ev}\ } A \times * \simeq A
	\]
	is the identity \(\operatorname{id}_A\).
	We now analyze these maps in the Cartesian setting. Since \(*\) is terminal, the evaluation map \(\operatorname{ev}\colon A^\vee \times A \to *\) is uniquely determined. The coevaluation map \(\operatorname{coev}\colon * \to A \times A^\vee\) is equivalently a global point of \(A \times A^\vee\). By the universal property of the product, this is the same as specifying a pair of global points
	\[
	a\colon * \to A \qquad \text{and} \qquad a'\colon * \to A^\vee.
	\]
	Substituting these observations into the triangle identity, we interpret the composite geometrically as follows: for any \(x \in A\),
	\[
	x \longmapsto (*,x) \longmapsto (a,a',x) \longmapsto (a,*) \simeq a.
	\]
	Hence, the composite required to be \(\operatorname{id}_A\) is in fact a constant map with value \(a\). Therefore, the identity \(\operatorname{id}_A\) factors through the terminal object:
	\[
	\operatorname{id}_A \colon A \to * \xrightarrow{\ a\ } A.
	\]
	So \(A\) is a retract of \(*\).  Therefore,
	$
	A \simeq *.
	$
	
	Consequently
	\[
	\cidem_f(\mcx)=\{\emptyset_\mcx,*\}.
	\]
	The three stated conditions are now immediately equivalent.
\end{proof}
\begin{cor}
The Cartesian \syminfcats $\mcs$ and $\mcs_{\leq n}$, corresponding to animas and $n$-truncated animas respectively,  are smashing fields and hence satisfy the telescope conjecture for all $n \geq -1$. 
\end{cor}
\begin{rem}
	Note that when $n=-2$, $\mcs_{\leq n}=*$  satisfies the telescope conjecture trivially, though it is not a smashing field.
\end{rem}
We now show that the symmetric monoidal $\infty$-categories of pointed animas and of $\mathbb{E}_\infty$-animas are also smashing fields and hence satisfy the telescope conjecture. The proofs use a simple criterion: a coidempotent object admitting a section of its counit must already be trivial:
\begin{lem}\label{retractcidem}
	Let $\mathcal{D}$ be an $\mathbb{E}_1$-monoidal $\infty$-category with unit object $\mathbf{1}$. Let $X \in \mathcal{D}$ be a coidempotent object, meaning it is equipped with a morphism $f: X \to \mathbf{1}$ such that the natural morphisms induced by the unit constraints,
	\[ f \otimes \id_X: X \otimes X \xrightarrow{\sim} \mathbf{1} \otimes X \simeq X \text{\quad and\quad} \id_X \otimes f: X \otimes X \xrightarrow{\sim} X \otimes \mathbf{1} \simeq X, \]
	are equivalences in $\mathcal{D}$. 
	If there exists a section $g: \mathbf{1} \to X$ such that $f \circ g \simeq \id_{\mathbf{1}}$, then $f$ is an equivalence.
\end{lem}

\begin{proof}
	Since $f \circ g \simeq \id_{\mathbf{1}}$, applying the functor $(-) \otimes X$ yields:
	\[ (f \circ g) \otimes \id_X \simeq \id_{\mathbf{1}} \otimes \id_X \simeq \id_X. \]
	By the bifunctoriality (interchange law) of the tensor product, we can rewrite the left-hand side as:
	\[ (f \otimes \id_X) \circ (g \otimes \id_X) \simeq \id_X. \]
	By the coidempotent hypothesis, $f \otimes \id_X$ is an equivalence. The equation above implies that $g \otimes \id_X$ is its right inverse, and therefore $g \otimes \id_X$ is also an equivalence.
	
	Now consider the endomorphism $h = g \circ f : X \to X$ and the following diagram:
	$$\begin{tikzcd}
		\mathbf{1} \otimes X \arrow[d, "\id_{\mathbf{1}} \otimes f"] \arrow[r, "g \otimes \id_X"] &  X \otimes X \arrow[d, "\id_X \otimes f"] \\
		\mathbf{1} \otimes \mathbf{1} \arrow[r, "g\otimes \id_{\mathbf{1}} "]                     &  X \otimes \mathbf{1}                    
	\end{tikzcd}$$
	We get:
	\[  g \otimes f= (\id_X \otimes f) \circ (g \otimes \id_X)  \simeq g \circ f = h. \]
	Since both $g \otimes \id_X$ and $\id_X \otimes f$ are equivalences, their composition $h$ must be an equivalence.
	Finally, we compute $h^2$:
	\[ h^2 = (g \circ f) \circ (g \circ f) \simeq g \circ (f \circ g) \circ f \simeq g \circ \id_{\mathbf{1}} \circ f \simeq g \circ f = h. \]
	We have an equivalence $h$ satisfying $h^2 \simeq h$. Multiplying both sides by $h^{-1}$ yields $h \simeq \id_X$. 
	Since $g \circ f \simeq \id_X$ and $f \circ g \simeq \id_{\mathbf{1}}$, $f$ and $g$ are mutually inverse equivalences.
\end{proof}
We first apply this criterion to pointed animas.
\begin{ex}\label{sptsmfield}
	The pointed mode $\mcs_*$, i.e. the  \syminfcat of pointed animas, is a smashing field.
\end{ex}

\begin{proof}
	Let $X$ be a coidempotent object in $\mathcal{S}_*$ with structure map $f: X \to S^0$. The anima $S^0$ consists of two contractible connected components: the basepoint component $\{0\}$ and the non-basepoint component $\{1\}$. We proceed by analyzing the image of $X$ under $f$.
	
	Assume that the map $f$ does not hit the component $\{1\}$.
	In this case, the essential image of $f$ is contained entirely in the basepoint component $\{0\} \simeq *$. Consequently, $f$ factors through the point and is therefore nullhomotopic. 
	By the definition of a coidempotent object, the following composition is an equivalence:
	\[  X \wedge X \xrightarrow{\id_X \wedge f} X \wedge S^0 \simeq X. \]
	Since $f$ factors through $*$, the morphism $\id_X \wedge f$ must factor through $X \wedge * \simeq *$. This implies that the equivalence $X \xrightarrow{\sim} X$ factors through the zero object $*$, which immediately forces $X \simeq *$.
	
	Assume that the map $f$ hits the component $\{1\}$.
	Suppose there exists a point $x \in X$ such that $f(x)$ lies in the component $\{1\}$. We can construct a pointed map $g: S^0 \to X$ by mapping the basepoint $0 \in S^0$ to the basepoint $*_X \in X$, and mapping the non-basepoint $1 \in S^0$ to $x \in X$.
 Therefore, $f \circ g \simeq \id_{S^0}$ is a section.
	Applying \cref{retractcidem} directly yields $X \simeq S^0$.
	
	Combining both cases, we conclude that any coidempotent object in $\mathcal{S}_*$ must be equivalent to either $*$ or $S^0$.
\end{proof}
The next example gives a discrete algebraic model which will be used in the $\mathbb{E}_\infty$-anima case.
\begin{ex}\label{cmonsetsmfield}
	Let $\cmon(\mathrm{Set})$ be the (1-)category of commutative monoids with the symmetric monoidal structure given by tensor product. Then $\cmon(\mathrm{Set})$ is a smashing field.
\end{ex}

\begin{proof}
	The unit in $\cmon(\mathrm{Set})$ is the additive monoid of natural numbers $\mathbb{N}$.  Let  $\epsilon: C \to \mathbb{N}$ be a coidempotent object.  We wish to show that $C$ is either $0$ or the unit $\mathbb{N}$. 
	
	The counit $\epsilon: C \to \mathbb{N}$ admits a unique cocommutative coalgebra structure.
	The inverse of $\Delta$ is given by $f = (\mathrm{id} \otimes \epsilon): C \otimes C \to C$, which acts on pure tensors as $f(x \otimes y) = \epsilon(y)x$ (meaning $x$ added to itself $\epsilon(y)$ times). 
	By the symmetry of the tensor product, we must have $$\epsilon(y)x = \epsilon(x)y$$ for all $x, y \in C$. Let $I = \epsilon(C) \subset \mathbb{N}$ be the image of the counit. Since $\epsilon$ is a homomorphism, $I$ is an additive submonoid of $\mathbb{N}$. Since $f$ is an isomorphism, it is surjective, so any $z \in C$ can be written as $z = \sum \epsilon(y_i)x_i$. Applying $\epsilon$ to both sides yields $\epsilon(z) = \sum \epsilon(y_i)\epsilon(x_i)$. 
	This implies that $I = I \cdot I$.
	However, the only additive submonoids of $\mathbb{N}$ satisfying $I = I \cdot I$ are $0$ and $\mathbb{N}$ itself. 
	
	If $I = 0$, then $\epsilon(y) = 0$ for all $y \in C$. The isomorphism $f$ then yields $x = f(x \otimes e) = 0 \cdot x = 0$, which forces $C \simeq 0$.
	
	If $I = \mathbb{N}$, there exists an element $e \in C$ such that $\epsilon(e) = 1$. Using the symmetry equation, for any $y \in C$, $y = 1 \cdot y = \epsilon(e)y = \epsilon(y)e$. The map $\mathbb{N} \to C$ given by $n \mapsto ne$ is therefore an isomorphism with inverse $\epsilon$. Hence $C \simeq \mathbb{N}$.
\end{proof}

Passing from discrete commutative monoids to $\mathbb{E}_\infty$-animas, we combine the previous calculation on $\pi_0$ with group completion.

\begin{ex}\label{cmonssmfield}
	The (0-)semiadditive mode $\cmon(\mcs)$, i.e. the \syminfcat of \ein-animas, is a smashing field.
\end{ex}

\begin{proof}
	Let $C \in \cmon(\mathcal{S})$ be a coidempotent object with counit $C \to \op{Fin}^\simeq$. We wish to show that $C$ is either $0$ or the unit $\op{Fin}^\simeq$.
	
	Consider the group completion functor $(-)^{\op{gp}}: \cmon(\mathcal{S}) \to \opsp_{\geq 0}$, which is strong symmetric monoidal. The object $E := C^{\op{gp}}$ is therefore a coidempotent object in connective spectra over the sphere spectrum $\mathbb{S}$. Since the functor $\pi_0: \cmon(\mathcal{S}) \to \cmon(\mathrm{Set})$ is also strong symmetric monoidal, $\pi_0 C$ is a coidempotent commutative monoid over $\pi_0(\op{Fin}^\simeq) \simeq \mathbb{N}$. By \cref{cmonsetsmfield}, $\pi_0 C$ is either $0$ or $\mathbb{N}$.
	
	Assume that $\pi_0 C \simeq 0$.
	In this case, $C$ is a connected $\mathbb{E}_\infty$-anima and hence has been group-like, i.e. $$C \xrightarrow{\sim} C^{\op{gp}}=E.$$ Furthermore, $\pi_0(E) = (\pi_0 C)^{\op{gp}} = 0$. 
	By \cref{noethertelescope}, $E$ is either $0$ or $\mathbb{S}$, so we must have $C\simeq E=0$.
	
	Assume that $\pi_0 C \simeq \mathbb{N}$. 
Recall that the unit object $\op{Fin}^\simeq \simeq \coprod_{n \geq 0} B\Sigma_n$ is the free $\mathbb{E}_\infty$-anima generated by a single point $* \simeq B\Sigma_1$. By the universal property, specifying an $\mathbb{E}_\infty$-morphism from $\op{Fin}^\simeq$ to $C$ is equivalent to choosing a point in the underlying anima of $C$.
Since $\pi_0(f)$ is an isomorphism, there exists a unique connected component $C_1 \subset C$ that maps to the component $(\op{Fin}^\simeq)_1 \simeq B\Sigma_1 \simeq *$ under $f$. We choose an arbitrary point $x: * \to C_1$. 
By the free universal property, the point $x$ extends to an $\mathbb{E}_\infty$-morphism $g: \op{Fin}^\simeq \to C$. 
 Since $f(x)$ lies in the contractible anima $(\op{Fin}^\simeq)_1$, $f(x)$ is homotopic to the standard generator of $\op{Fin}^\simeq$ and the universal property indicates that $f \circ g \simeq \id_{\op{Fin}^\simeq}$. Thus, $g$ is a section of $f$.
Since $C$ is a coidempotent object in the monoidal $\infty$-category $\cmon(\mathcal{S})$ and $f$ admits a section $g$, \cref{retractcidem} immediately implies that $f: C \xrightarrow{} \op{Fin}^\simeq$ is an equivalence.
\end{proof}
These examples suggest that higher semiadditive modes may also behave like smashing fields. We leave the following question open.
\begin{qu}
	Let $m\geq-2$ be an integer. Is the $m$-semiadditive mode $\cmon_m(\mcs)$ appearing in \cite{Harpaz_2020} a smashing field? (We have shown that it holds when $m\in[-2,0]$.)
\end{qu}

\section{Recollement and Gluing of (Atomic) Smashing Frames}
\label{secrecollement}

In both the stable and prestable settings, a smashing ideal admits a
complementary quotient, and the smashing frame of the ambient category can be
reconstructed from the smashing frames of these two pieces.
The resulting structure is naturally described by Artin gluing.

\subsection{Smashing quotients in the prestable setting}
The stable and prestable cases are formally very similar, but there is one
important distinction. In the stable setting, smashing ideals and smashing
localizations are related by the usual fiber--cofiber correspondence. In the
prestable setting, the corresponding quotients are instead the
\emph{connective smashing localizations}. We begin by recalling this
prestable analogue.

\begin{de}
	Let \(\mcc\in\calg(\prlad)\) be prestable. We define
	\[
	\calg(\mcc)^{\op{cn\text{-}idem}}
	\subset
	\calg(\mcc)^{\op{idem}}
	\]
	to be the full subcategory spanned by those idempotent algebras
	\(\mb{1}\to R\) for which
	$
	\pi_0\mb{1}\longrightarrow\pi_0R
	$
	is an epimorphism in \(\mcc^\heartsuit\).
\end{de}

The following lemma identifies these connective idempotents with
coidempotent objects.

\begin{lem}\label{connloc}
	Let \(\mcc\in\calg(\prlad)\) be prestable. Then there are natural
	equivalences
	\[
	\cidem(\mcc)
	\simeq
	\Idem(\mcc)^{\op{cn}}
	\simeq
	\calg(\mcc)^{\op{cn\text{-}idem}},
	\]
	where \(\Idem(\mcc)^{\op{cn}}\) denotes the full subposet of
	idempotent objects
	$
	\mb{1}\longrightarrow R
	$
	whose induced morphism on \(\pi_0\) is an epimorphism.
\end{lem}

\begin{proof}
	Passing to the stabilization gives
	\[
	\cidem(\opsp(\mcc))
	\simeq
	\Idem(\opsp(\mcc))
	\simeq
	\calg(\opsp(\mcc))^{\op{idem}}.
	\]
	Let
	$
	x\longrightarrow\mb{1}
	$
	be a coidempotent object of \(\mcc\). Its cofiber in
	\(\opsp(\mcc)\),
	$
	\mb{1}\longrightarrow R,
	$
	is idempotent. Conversely, for an idempotent object
	\(\mb{1}\to R\) lying in \(\mcc\), its fiber lies in \(\mcc\) precisely
	when
	$
	\pi_0\mb{1}\longrightarrow\pi_0R
	$
	is an epimorphism. Thus the stable equivalence between coidempotents and
	idempotents restricts to
	\[
	\cidem(\mcc)
	\simeq
	\Idem(\mcc)^{\op{cn}}.
	\]
	The second equivalence follows from the usual identification of
	idempotent objects with idempotent commutative algebras.
\end{proof}
This motivates the following terminology for smashing localizations of prestable $\infty$-categories (cf. \cite[Definition 3.16]{levy-liang-sosnilo:dualizable}).
\begin{de}
	Let \(\mcc\in\calg(\prlad)\) be prestable. A smashing localization
	$
	L:\mcc\longrightarrow\mcd
	$
	is called a \textbf{connective smashing localization} if, for every
	\(x\in\mcc\), the unit morphism
	\[
	x\longrightarrow L^RL(x)
	\]
	induces an epimorphism on \(\pi_0\).
\end{de}

Thus connective smashing localizations are exactly the prestable
localizations classified by the idempotents of \cref{connloc}.

\begin{cor}\label{connlocandclosub}
	Let \(\mcc\in\calg(\prlad)\) be prestable. Taking kernel and quotient
	induces mutually inverse equivalences of posets
	\[
	\begin{tikzcd}[column sep=large]
		\{\text{smashing ideals of }\mcc\}
		\arrow[r, shift left, "\mci\mapsto \mcc/\mci"]
		&
		\{\text{connective smashing localizations out of }\mcc\}
		\arrow[l, shift left, "\ker"]
	\end{tikzcd}.
	\]
	Equivalently, there is a commutative diagram of equivalences
	\[
	\begin{tikzcd}
		\{\text{smashing ideals of }\mcc\}
		\arrow[r, "\sim"]
		\arrow[d, "\sim"']
		&
		\{\text{connective smashing localizations of }\mcc\}
		\arrow[d, "\sim"]
		\\
		\cidem(\mcc)
		\arrow[r, "\cofib"', "\sim"]
		&
		\Idem(\mcc)^{\op{cn}}.
	\end{tikzcd}
	\]
\end{cor}

\begin{proof}
	Let
	$
	i:\mci\hookrightarrow\mcc
	$
	be a smashing ideal, and let
	$
	x\longrightarrow\mb{1}
	$
	be its corresponding coidempotent object. By \cref{connloc}, the
	cofiber
	\[
	\mb{1}\longrightarrow c:=\cofib(x\to\mb{1})
	\]
	is a connective idempotent object. The quotient
	$
	\mcc\longrightarrow\mcc/\mci
	$
	is therefore the associated connective smashing localization, and its
	kernel is \(\mci\).
	
	Conversely, let
	$
	L:\mcc\longrightarrow\mcd
	$
	be a connective smashing localization. Its associated idempotent object
	\(\mb{1}\to R\) belongs to
	\(\Idem(\mcc)^{\op{cn}}\), so by \cref{connloc} its fiber
	\[
	x:=\fib(\mb{1}\to R)\longrightarrow\mb{1}
	\]
	is coidempotent. The corresponding smashing ideal is precisely
	\(\ker(L)\), and quotienting by it recovers \(L\).
\end{proof}

\begin{rem}\label{remshortex1add}
	The preceding statement is the prestable analogue of the familiar
	stable correspondence between smashing ideals and smashing
	localizations. However, \cref{connlocandclosub} fails in the setting of additive $1$-categories. This suggests that smashing ideals exhibit better structural properties in the prestable context.
\end{rem}
\begin{rem}
	Let $\mcc \in \calg(\prlad)$ be a prestable $\infty$-category. Connective smashing localizations of $\mcc$ correspond precisely to normal epimorphisms in $\prcdbl$ (in the sense of \cite{Zariski_Bal}). Furthermore, the kernel of any connective smashing localization forms a bifiber sequence in $\prcdbl$ (see \cite[Section~3]{levy-liang-sosnilo:dualizable}).
\end{rem}
\subsection{Recollement and Artin gluing}

We now treat the stable and prestable cases simultaneously.

\begin{thm}\label{recollementfrms}
	Let \(\mcc\in\calg(\prlad)\) be prestable.
	Let
	$
	\mci\xhookrightarrow{i}\mcc\xrightarrow{L}\mcc/\mci
	$
	be the quotient associated to a smashing ideal
	\(\mci\hookrightarrow\mcc\). Then there are natural equivalences
	\[
	\cidem(\mci)
	\simeq
	\cidem(\mcc)_{\leq\mci},
	\qquad
	\cidem(\mcc/\mci)
	\simeq
	\cidem(\mcc)_{\geq\mci}.
	\]
	Consequently, the canonical morphism
	$
	\cidem(\mcc)
	\xrightarrow{(i^R,L)}
	\cidem(\mci)\times\cidem(\mcc/\mci)
	$
	is an embedding of frames.
	
	Moreover, under the preceding identifications,
	\(\cidem(\mcc)\) is obtained by Artin gluing:
	\[
	\cidem(\mcc)
	\simeq
	\cidem(\mci)
	\overleftarrow{\times}_{\!\phi_\mci}
	\cidem(\mcc/\mci),
	\]
	where the gluing functor is explicitly given by
	\[
	\phi_\mci(\mca)
	=
	\mci\vee(\mci\backslash \mca),
	\qquad
	\mca\in\cidem(\mcc)_{\leq\mci}.
	\]
	Here \(\mci\backslash \mca\) denotes the Heyting implication in the frame
	\(\cidem(\mcc)\).
\end{thm}

\begin{proof}
	The first equivalence
	$
	\cidem(\mci)
	\simeq
	\cidem(\mcc)_{\leq\mci}
	$
	holds without any stability assumption by \cref{cidemideal}.
	
	It remains to identify the closed part.	
 By \cref{connlocandclosub}, smashing ideals of \(\mcc/\mci\) correspond
	to connective smashing localizations out of \(\mcc/\mci\).
	Composition with
	\[
	L:\mcc\longrightarrow\mcc/\mci
	\]
	identifies these with connective smashing localizations out of \(\mcc\)
	whose kernel contains \(\mci\). Indeed, any such localization factors
	uniquely through \(L\), and the induced localization of
	\(\mcc/\mci\) is again connective smashing. Hence
	\[
	\cidem(\mcc/\mci)
	\simeq
	\cidem(\mcc)_{\geq\mci}.
	\]
	
	The Artin gluing description now follows formally from
	\cref{frmopencloseddecomp}. Under the identifications
	\[
	\cidem(\mci)\simeq\cidem(\mcc)_{\leq\mci},
	\qquad
	\cidem(\mcc/\mci)\simeq\cidem(\mcc)_{\geq\mci},
	\]
	the gluing functor is precisely
	$
	a\longmapsto
	\mci\vee(\mci\backslash a).
	$
\end{proof}

\begin{rem}
	We call a sequence
	\[
	\mci\hookrightarrow\mcc\rightarrow\mcc/\mci,
	\]
	consisting of a smashing ideal \(\mci\) and its quotient, a \textbf{smashing
		localization sequence}.\footnote{More precisely, these could be called
		\emph{connective smashing localization sequences}; for convenience, we use the
		shorter terminology.}
	
	The theorem gives a precise sense in which a smashing ideal is an
	``open part'' of the smashing spectrum, while its quotient is the
	complementary ``closed part''. In \cref{artin-gluing-internal-hom}, we give
	a further description of the Artin gluing functor in terms of dualizable
	internal Homs.
\end{rem}
\begin{rem}\label{smlocpullback}
	The closed-part equivalence in \cref{recollementfrms} admits a concrete
	description by pullback.
	Let
	$
	L:\mcc\longrightarrow\mcc/\mci
	$
	be the connective smashing localization associated to a smashing ideal
	\(\mci\subset\mcc\). For a smashing ideal
	$
	\mck\hookrightarrow\mcc/\mci,
	$
	let
	$
	q:\mcc/\mci\longrightarrow\mcd
	$
	be the corresponding connective smashing localization, so that
	\(\ker(q)=\mck\). Then the composite
	\[
	qL:\mcc\longrightarrow\mcd
	\]
	is again a connective smashing localization, and hence its kernel is a
	smashing ideal of \(\mcc\). Moreover,
	\[
	\ker(qL)
	=
	\{x\in\mcc\mid L(x)\in\mck\}
	\simeq
	\mck\times_{\mcc/\mci}\mcc.
	\]
	Thus the equivalence
	$
	\cidem(\mcc/\mci)
	\simeq
	\cidem(\mcc)_{\geq\mci}
	$
	is explicitly given by
	$$\begin{tikzcd}[column sep=5em] \cidem(\mcc/\mci) \arrow[r, "\mck\mapsto \mck\times_{\mcc/\mci}\mcc", shift right=-1ex] & \arrow[l,"\sim"', "\mcm/\mci \reflectbox{$\mapsto$}\mcm", shift right=-1ex] \cidem(\mcc)_{\geq \mci} \end{tikzcd}.$$
\end{rem}

The recollement is compatible with arbitrary
symmetric monoidal base change.

\begin{prop}\label{smlocpushoutfrms}
	Let
	$
	F:\mcc\longrightarrow\mcd
	$
	be a morphism in \(\calg(\prlad)\) such that both are prestable, let
	\(\mci\hookrightarrow\mcc\) be a smashing ideal, and set
$
	\mck:=\mci\otimes_\mcc\mcd.
	$
	Then the diagram
	\[
	\begin{tikzcd}
		\mci
		\arrow[r, hook, shift left, "i"]
		\arrow[d]
		&
		\mcc
		\arrow[l, shift left, "i^R"]
		\arrow[r, shift left]
		\arrow[d]
		&
		\mcc/\mci
		\arrow[l, hook', shift left]
		\arrow[d]
		\\
		\mck
		\arrow[r, hook, shift left]
		&
		\mcd
		\arrow[l, shift left]
		\arrow[r, shift left]
		&
		\mcd/\mck
		\arrow[l, hook', shift left]
	\end{tikzcd}
	\]
	is horizontally right adjointable.
	Therefore, after identifying the smashing frames with their Artin gluing
	descriptions via \cref{recollementfrms}, we obtain a natural commutative
	diagram
	\[
	\begin{tikzcd}
		\cidem(\mcc)
		\arrow[r, "\sim"]
		\arrow[d]
		&
		\cidem(\mci)
		\overleftarrow{\times}_{\!\phi_\mci}
		\cidem(\mcc/\mci)
		\arrow[r, hook]
		\arrow[d]
		&
		\cidem(\mci)\times\cidem(\mcc/\mci)
		\arrow[d]
		\\
		\cidem(\mcd)
		\arrow[r, "\sim"]
		&
		\cidem(\mck)
		\overleftarrow{\times}_{\!\phi_\mck}
		\cidem(\mcd/\mck)
		\arrow[r, hook]
		&
		\cidem(\mck)\times\cidem(\mcd/\mck).
	\end{tikzcd}
	\]
	
	Furthermore, the square
	\[
	\begin{tikzcd}[row sep=20pt, column sep=50pt]
		\cidem(\mcc)
		\arrow[r]
		\arrow[d]
		&
		\cidem(\mcc/\mci)
		\arrow[d]
		\\
		\cidem(\mcd)
		\arrow[r]
		&
		\cidem(\mcd/\mck)
	\end{tikzcd}
	\]
	is a pushout in \(\frm\).
\end{prop}

\begin{proof}
	The left square of categories is formally right adjointable. The right
	square is right adjointable by
	\cite[Proposition A.16]{ramzi2024dualizable}\footnote{This also works in the prestable setting}. The induced compatibility
	of the Artin gluing descriptions follows from
	\cref{recollementfrms}.
	
	The pushout assertion follows by the same argument as
	\cref{smidlpushoutfrms}, together with
	\cite[Proposition 7.3.2.12]{htt}.
\end{proof}

\begin{rem}
	Topologically, the pushout in \cref{smlocpushoutfrms} corresponds to
	pullback along the closed part of the associated recollement. This is
	the closed analogue of the open base-change square of
	\cref{smidlpushoutfrms}.
\end{rem}

Artin gluing also gives an immediate local criterion for spatiality.

\begin{cor}\label{recollementspatial}
	Let \(\mcc\in\calg(\prlad)\) be prestable, and let
	$
	\mci\hookrightarrow\mcc\rightarrow\mcc/\mci
	$
	be a smashing localization sequence. Then
	$
	\cidem(\mcc)
	$
	is spatial if and only if both
	$
	\cidem(\mci)$ and $
	\cidem(\mcc/\mci)
	$
	are spatial.
\end{cor}

\begin{proof}
	If both pieces are spatial, then
	\[
	\cidem(\mcc)
	\hookrightarrow
	\cidem(\mci)\times\cidem(\mcc/\mci)
	\]
	exhibits \(\cidem(\mcc)\) as a subframe of a spatial frame, so it is
	spatial by \cref{subfrmspatial}.
	
	Conversely, by \cref{recollementfrms},
	\(\cidem(\mci)\) is an open sublocale and
	\(\cidem(\mcc/\mci)\) is its closed complement. Open and closed
	sublocales of a spatial locale are spatial.
\end{proof}

\begin{rem}
	The prestable hypothesis is essential for the closed part of
	the recollement. For a general
	\(\mcv\in\calg(\prl)\), smashing ideals are classified by
	coidempotent objects, whereas smashing localizations are classified by
	idempotent objects, and these two notions need not agree; see
	\cref{smaidllocnotmatch}.
\end{rem}

\subsection{Split smashing ideals and clopen decompositions}

The recollement becomes a product decomposition precisely when the
corresponding smashing ideal is split.

\begin{prop}\label{splitstablesmid}
	Let \(\mcc\in\calg(\prlad)\) be prestable. Let
	$
	\mci\xhookrightarrow{i}\mcc\xrightarrow{L}\mcc/\mci
	$
	be a smashing localization sequence, and let
	$
	x\xrightarrow{f}\mb{1}
	$
	be the coidempotent object associated to \(\mci\). Write
	$
	c:=\cofib(f).
	$
	Then the following conditions are equivalent:
	\enu{
		\item \(\mci\) is a split smashing ideal;
		\item \(c\) is dualizable;
		\item the canonical symmetric monoidal functor
		$
		\mcc
		\xrightarrow{(i^R,L)}
		\mci\times\mcc/\mci
		$
		is an equivalence;
		\item the canonical morphism of frames
		$
		\cidem(\mcc)
		\longrightarrow
		\cidem(\mci)\times\cidem(\mcc/\mci)
		$
		is an equivalence.
	}
\end{prop}

\begin{proof}
	By \cref{splitsmidl}, condition (1) is equivalent to the dualizability
	of \(x\).
	
	Assume first that \(x\) is dualizable. Again by
	\cref{splitsmidl}, the coidempotent morphism
	$
	x\longrightarrow\mb{1}
	$
	splits. Hence the cofiber sequence
	\[
	x\longrightarrow\mb{1}\longrightarrow c
	\]
	is split, so \(c\) is a retract of the tensor unit and therefore
	dualizable. Thus (1) implies (2).
	
	Conversely, suppose that \(c\) is dualizable. Since
	$
	\mb{1}\longrightarrow c
	$
	is idempotent, \cref{idemdualretract} implies that it admits a section.
	Hence the preceding cofiber sequence splits, and \(x\) is a retract of
	the tensor unit. In particular \(x\) is dualizable, and
	\cref{splitsmidl} implies that \(\mci\) is split. Thus (2) implies (1).
	
	We next prove that (2) implies (3). By
	\cref{idemdualretract}, the splitting above identifies the two
	complementary summands of the unit as
	\[
	\mb{1}
	\simeq
	x^\vee\oplus c.
	\]
	Consequently, every \(y\in\mcc\) admits a natural decomposition
	\[
	y
	\simeq
	(y\otimes x^\vee)
	\oplus
	(y\otimes c).
	\]
	The first summand lies in \(\mci\), while the second is local for the
	quotient \(\mcc/\mci\). This identifies
	$
	\mcc
	\simeq
	\mci\times\mcc/\mci
	$
	compatibly with the symmetric monoidal structures.
	
	The implication (3) \(\Rightarrow\) (4) is immediate.
	
	Finally, assume (4). Consider
	\[
	(0,1)
	\in
	\cidem(\mci)\times\cidem(\mcc/\mci).
	\]
	By surjectivity, it is represented by a coidempotent object
	$
	a\xrightarrow{g}\mb{1}
	$
	of \(\mcc\) satisfying
	\[
	a\otimes x\simeq0,
	\qquad
	a\otimes c\simeq c.
	\]
	Hence the commutative square
	\[
	\begin{tikzcd}
		a\otimes\mb{1}
		\arrow[r, "1\otimes f_c", "\sim"']
		\arrow[d, "g\otimes1"]
		&
		a\otimes c
		\arrow[d, "g\otimes1", "\sim"']
		\\
		\mb{1}\otimes\mb{1}
		\arrow[r, "1\otimes f_c"]
		&
		\mb{1}\otimes c
	\end{tikzcd}
	\]
	shows that \(c\) is a retract of the tensor unit. Thus \(c\) is
	dualizable, proving (4) \(\Rightarrow\) (2).
\end{proof}

The preceding proposition admits a frame-theoretic interpretation. In particular, we show that in the prestable setting, the atomic smashing frame contains the zero-dimensional part of the smashing frame.

\begin{cor}\label{splitclopen}
	Let \(\mcc\in\calg(\prlad)\) be prestable, and let
	\(\mci\in\cidem(\mcc)\) be a smashing ideal. Then the following are
	equivalent:
	\enu{
		\item \(\mci\) is a split smashing ideal;
		\item \(\mci\) is a complemented element of the frame
		\(\cidem(\mcc)\).
	}
	In particular, we obtain the following inclusion from the zero-dimensional coreflection
	\[
	\op{Z}\bigl(\cidem(\mcc)\bigr)
	\subset
	\cidem_f(\mcc).
	\]
\end{cor}

\begin{proof}
	By \cref{splitstablesmid}, splitness is equivalent to the canonical
	morphism
	\[
	\cidem(\mcc)
	\longrightarrow
	\cidem(\mci)\times\cidem(\mcc/\mci)
	\]
	being an equivalence. By \cref{recollementfrms}, this identifies with
	the canonical morphism
	\[
	F
	\longrightarrow
	F_{\leq u}\times F_{\geq u},
	\qquad
	F=\cidem(\mcc),\quad u=\mci.
	\]
	For any frame \(F\), this morphism is an equivalence if and only if
	\(u\) is complemented. This is in turn equivalent to the associated
	open sublocale being clopen.
	
	Finally, every split smashing ideal is atomic by
	\cref{splitsmidl}. Hence every complemented element of
	\(\cidem(\mcc)\) belongs to \(\cidem_f(\mcc)\). Since
	\(\cidem_f(\mcc)\) is a subframe, it contains the subframe generated by
	all complemented elements:
	\[
	\op{Z}\bigl(\cidem(\mcc)\bigr)
	\subset
	\cidem_f(\mcc).
	\]
\end{proof}

\begin{rem}
	The prestable hypothesis in \cref{splitclopen} cannot be
	removed.	
	For example, let
	\[
	\mcv=\mathcal S\times\mathcal S
	\]
	with its Cartesian symmetric monoidal structure. Then
	$
	\cidem(\mcv)
	\simeq
	\mcv_{\leq-1}
	$
	is the four-element Boolean frame, so every coidempotent object is
	complemented.
	
	On the other hand, the only dualizable object of \(\mcv\) is its tensor
	unit \((*,*)\). The two nontrivial proper smashing ideals corresponding
	to
	\[
	(*,\emptyset),
	\qquad
	(\emptyset,*)
	\]
	contain no dualizable objects of \(\mcv\), and hence are not atomic.
	In particular, they are not split by \cref{splitsmidl}.
\end{rem}

\subsection{Atomic recollement in the stable setting}

We now turn to atomic smashing ideals. The open part of a recollement behaves
well for atomic smashing ideals in complete generality, by
\cref{cidematideal}. The closed part is subtler: atomic generators of a
quotient need not lift to atomic generators of the ambient category.

Nevertheless, for compactly-rigidly generated stable categories, the required lifting is
provided by the Thomason--Neeman theorem.

\begin{thm}\label{recollementatsmfrms}
	Let \(\mcc\in\calg(\prlst)\) be compactly-rigidly generated, and let
	$
	\mci\xhookrightarrow{i}\mcc\xrightarrow{L}\mcc/\mci
	$
	be a smashing localization such that \(\mci\) is atomic. Then
	\[
	\cidem_f(\mci)
	\simeq
	\cidem_f(\mcc)_{\leq\mci},
	\qquad
	\cidem_f(\mcc/\mci)
	\simeq
	\cidem_f(\mcc)_{\geq\mci}.
	\]
	Consequently,
	\[
	\cidem_f(\mcc)
	\simeq
	\cidem_f(\mci)
	\overleftarrow{\times}_{\!\phi^f_\mci}
	\cidem_f(\mcc/\mci),
	\]
	where
	$
	\phi^f_\mci(\mca)
	=
	\mci\vee
	(\mci\backslash_f \mca),
	$
	and \(\backslash_f\) denotes Heyting implication in the frame
	\(\cidem_f(\mcc)\).\footnote{The Heyting implication in
		\(\cidem_f(\mcc)\) generally differs from that in \(\cidem(\mcc)\).}
\end{thm}

\begin{proof}
	By \cref{cidematideal},
	$
	\cidem_f(\mci)
	\simeq
	\cidem_f(\mcc)_{\leq\mci}.
	$
	It remains to identify the closed part.
	
	By \cref{recollementfrms},
	$
	\cidem(\mcc/\mci)
	\simeq
	\cidem(\mcc)_{\geq\mci}.
	$
	Thus it suffices to prove the following: if
	\[
	\mci\subset\mck\subset\mcc
	\]
	are two smashing ideals of \mcc, then
	\[
	\mck\text{ is atomic in }\mcc
	\quad\Longleftrightarrow\quad
	\mck/\mci\text{ is atomic in }\mcc/\mci.
	\]
	
	The forward implication follows from preservation of atomic smashing
	ideals under base change.
	
	Conversely, assume that
	\[
	\mck/\mci
	\hookrightarrow
	\mcc/\mci
	\]
	is atomic. Since \(\mcc\) is compactly-rigidly generated, atomic
	smashing ideals are precisely compactly generated smashing ideals.
	Moreover, since \(\mci\) is atomic, the localization sequence induces a
	Karoubi sequence
	\[
	\mci^\omega
	\longrightarrow
	\mcc^\omega
	\xrightarrow{L}
	(\mcc/\mci)^\omega.
	\]
	
	Choose compact generators
	$
	\{y_\alpha\}_{\alpha\in A}
	\subset
	(\mck/\mci)^\omega
	$
	for \(\mck/\mci\). By the Thomason--Neeman theorem for Verdier
	quotients, an individual \(y_\alpha\) need not lift to an object of
	\(\mcc^\omega\), but
	\[
	y_\alpha\oplus\Sigma y_\alpha
	\]
	does. Thus there exists
	$
	c_\alpha\in\mcc^\omega
	$
	with
	$
	L(c_\alpha)
	\simeq
	y_\alpha\oplus\Sigma y_\alpha.
	$
	Since \(L(c_\alpha)\in\mck/\mci\) and
	\(\mci\subset\mck\), \cref{smlocpullback} implies that \(c_\alpha\in\mck\).
	
	On the other hand, choose compact generators
	$
	\{x_\beta\}_{\beta\in B}
	\subset
	\mci^\omega
	$
	for \(\mci\). The collection
	\[
	\{x_\beta\}_{\beta\in B}
	\cup
	\{c_\alpha\}_{\alpha\in A}
	\]
	consists of compact objects of \(\mck\). Their images generate both the
	open part \(\mci\) and the closed part \(\mck/\mci\). Since
	\[
	\cidem(\mcc)
	\hookrightarrow
	\cidem(\mci)\times\cidem(\mcc/\mci)
	\]
	is conservative by \cref{recollementfrms}, they generate
	\(\mck\) itself. Hence \(\mck\) is compactly generated, and therefore
	atomic.
	
	The Artin gluing description and the formula for \(\phi^f_\mci\)
	follow from \cref{frmopencloseddecomp}.
\end{proof}

The compact-rigid generation hypothesis in
\cref{recollementatsmfrms} is not merely technical.

\begin{rem}\label{qu510}
	In a general stable presentably symmetric monoidal
	\(\infty\)-category, atomicity need not glue across a recollement.
	More precisely, one may have smashing ideals
	\[
	\mci\subset\mck\subset\mcc
	\]
	such that \(\mci\) is atomic and
	$
	\mck/\mci
	\hookrightarrow
	\mcc/\mci
	$
	is atomic, while \(\mck\) itself is not atomic. Thus the closed-part
	statement of \cref{recollementatsmfrms} fails without a hypothesis
	ensuring the lifting of atomic generators from the quotient.
	
	An explicit example is given in
	\cref{exnonliftableatomicquotient}.
\end{rem}
As an immediate consequence, the telescope conjecture passes to atomic
smashing quotients.

\begin{cor}\label{atsmloctc}
	Let \(\mcc\in\calg(\prlst)\) be compactly-rigidly generated, and let
	$
	\mci\hookrightarrow\mcc\rightarrow\mcc/\mci
	$
	be a smashing localization sequence with \(\mci\) atomic. If
	\(\mcc\) satisfies the telescope conjecture, then so does
	\(\mcc/\mci\).
\end{cor}

\begin{proof}
	Consider the commutative diagram
	\[
	\begin{tikzcd}
		\cidem_f(\mcc/\mci)
		\arrow[r, hook]
		\arrow[d, "\sim"']
		&
		\cidem(\mcc/\mci)
		\arrow[d, "\sim"]
		\\
		\cidem_f(\mcc)_{\geq\mci}
		\arrow[r, hook]
		&
		\cidem(\mcc)_{\geq\mci}.
	\end{tikzcd}
	\]
	The vertical equivalences follow from
	\cref{recollementfrms} and \cref{recollementatsmfrms}, while the bottom map is an
	equivalence when \(\mcc\) satisfies the telescope conjecture.
\end{proof}

The preceding atomic recollement gives a closed version of the descent
theorem for compactly-rigidly generated stable categories.

\begin{de}
	Let \(\mcc\in\calg(\prlst)\), and let
	$
	\{f_i:\mcc\to\mcc_i\}_{i\in I}
	$
	be a small family of smashing localizations.
	We call this family a \textbf{smashing-closed covering} if the induced
	family
	\[
	\{\cidem(\mcc)\to\cidem(\mcc_i)\}_{i\in I}
	\]
	is jointly conservative.
	It is called an \textbf{atomic smashing-closed covering} if, in
	addition, the kernel of every \(f_i\) is an atomic smashing ideal.
\end{de}
\begin{thm}
	\label{finiteatsmcloseddescent}
	Let
	\(
	\mcc\in\calg(\prlst)
	\)
	be compactly-rigidly generated, and let
	\(
	\{f_i:\mcc\to\mcc_i\}_{i\in I}
	\)
	be an atomic smashing-closed covering. For each \(i\in I\), write
	\(
	\mck_i:=\ker(f_i)\in\cidem_f(\mcc).
	\)
	Assume that the finite-meet closure of the family
	\(\{\mck_i\}_{i\in I}\) is coinitial among the non-atomic smashing
	ideals, in the following sense: for every
	\(
	\mci\in\cidem(\mcc)\setminus\cidem_f(\mcc),
	\)
	there exists a nonempty finite subset \(J\subset I\) such that
	\(
	\bigwedge_{j\in J}\mck_j\leq\mci.
	\)

	Then atomicity is detected by the covering: a smashing ideal
	$
	\mci\in\cidem(\mcc)
	$
	is atomic if and only if its image in
	\(\cidem(\mcc_i)\) is atomic for every \(i\in I\).
	In particular,
	$
	\mcc
	$
	satisfies the telescope conjecture if and only if every
	$
	\mcc_i
	$
	does.

\end{thm}

\begin{proof}
	By \cref{recollementfrms}, for every \(i\in I\) the induced map on
	smashing frames identifies with the principal closed localization
	\[
	\cidem(\mcc)
	\longrightarrow
	\cidem(\mcc)_{\geq\mck_i},
	\qquad
	\mci\longmapsto\mci\vee\mck_i,
	\]
	and there is an equivalence
	\(
	\cidem(\mcc_i)
	\simeq
	\cidem(\mcc)_{\geq\mck_i}.
	\)
	Since \(\mcc\) is compactly-rigidly generated and each \(\mck_i\) is
	atomic, atomic recollement \cref{recollementatsmfrms} moreover gives
	\(
	\cidem_f(\mcc_i)
	\simeq
	\cidem_f(\mcc)_{\geq\mck_i}.
	\)
	Consequently, the image of a smashing ideal \(\mci\) in
	\(\cidem(\mcc_i)\) is atomic if and only if
	\(
	\mci\vee\mck_i\in\cidem_f(\mcc).
	\)

	If \(\mci\) is atomic, then all its images are atomic by base change.
	Conversely, suppose that the image of \(\mci\) in every
	\(\cidem(\mcc_i)\) is atomic. Thus
	\[
	\mci\vee\mck_i\in\cidem_f(\mcc)
	\qquad
	\text{for every }i\in I.
	\]
	By the coinitiality hypothesis, there exists a nonempty finite subset
	\(J\subset I\) such that
	$
	\mck_J:=
	\bigwedge_{j\in J}\mck_j
	\leq
	\mci.
	$
	Using finite distributivity in the frame \(\cidem(\mcc)\), we obtain
	\[
		\mci
		=
		\mci\vee\mck_J
		=
		\mci\vee
		\bigwedge_{j\in J}\mck_j
		=
		\bigwedge_{j\in J}
		(\mci\vee\mck_j).
	\]
	Each factor on the right belongs to \(\cidem_f(\mcc)\), and
	\(\cidem_f(\mcc)\subset\cidem(\mcc)\) is a subframe, hence is closed
	under finite meets. Therefore
	$
	\mci\in\cidem_f(\mcc)
	$
 is atomic.
	
\end{proof}

\begin{rem}\label{rematclsmcov1}
	The coinitiality hypothesis is automatic for a finite atomic
	smashing-closed covering. Indeed, if \(I\) is  finite, then
	\(
	\bigwedge_{i\in I}\mck_i=0.
	\)
	To see this, let
	\(
	\mck:=\bigwedge_{i\in I}\mck_i.
	\)
	For every \(i\),
	\(
	\mck\vee\mck_i
	=
	\mck_i
	=
	0\vee\mck_i.
	\)
	Since the family is jointly conservative on smashing frames, it follows
	that \(\mck=0\). 
\end{rem}

\begin{rem}\label{rematclsmcov2}
	If the family of kernels is downward directed, for example if
	\(
	\mck_0\geq\mck_1\geq\mck_2\geq\cdots,
	\)
	then the hypothesis simplifies to the following condition:
	\[
	\mci\notin\cidem_f(\mcc)
	\quad\Longrightarrow\quad
	\mck_i\leq\mci
	\quad
	\text{for some }i.
	\]
	Indeed, for every finite subset \(J\subset I\), downward directedness
	provides an index \(i\) such that
	\(
	\mck_i\leq\bigwedge_{j\in J}\mck_j.
	\)
\end{rem}

\begin{ex}\label{exchromaticcloseddescent}
	Consider the family
	\[
	\left\{
	\opsp
	\longrightarrow
	L_n^f\opsp
	\right\}_{p\in\mathrm{primes},\, n\geq 0},
	\]
	where \(L_n^f\) is taken at the prime \(p\). By \cref{imageinLnf}, the
	induced map
	$
	\Idem(\opsp)
	\longrightarrow
	\prod_{p}\prod_{n\geq 0}\Idem(L_n^f\opsp)
	$
	is an embedding. Thus this family is an atomic smashing-closed covering.
	Together with \cref{finiteatsmcloseddescent,rematclsmcov2}, this shows that
	\(\opsp\) satisfies the telescope conjecture if and only if \(L_n^f\opsp\)
	does for every prime \(p\) and every \(n\geq 0\), recovering the classical
	chromatic reduction.
\end{ex}
\begin{qu}
	Is the coinitiality hypothesis in \cref{finiteatsmcloseddescent} removable?
\end{qu}
\subsection{Recollement for the telescope conjecture}

In this subsection, we  give an explicit
formula for the Artin gluing functor in terms of dualizable internal Homs. We then
use this description to compare the gluing functors for the smashing frame and
the atomic smashing frame, leading to a recollement principle for the telescope
conjecture.

We begin with an internal-Hom description of Heyting implication in the smashing frame.
\begin{prop}\label{heyting-internal-hom}
Let \(\mcv\in\calg(\prl)\), and let \(\mca,\mcb\) be smashing ideals of
\(\mcv\). Write
$
\mca\backslash\mcb
$
for the Heyting implication in \(\cidem(\mcv)\). Then there is a canonical
equivalence, where the pullback is taken in \(\prvdbl\),
\[
\mca\backslash\mcb
\simeq
\mcv
\times^{\op{dbl}}_{
	\underline{\operatorname{Hom}}^{\op{dbl}}_{\mcv}(\mca,\mcv)
}
\underline{\operatorname{Hom}}^{\op{dbl}}_{\mcv}(\mca,\mcb).
\]

Similarly, if \(\mca\) and \(\mcb\) are atomic smashing ideals, write
$
\mca\backslash_f\mcb
$
for the Heyting implication in \(\cidem_f(\mcv)\). Then there is a canonical
equivalence, where the pullback is taken in \(\prvat\),
\[
\mca\backslash_f\mcb
\simeq
\mcv
\times^{\op{at}}_{
	\underline{\operatorname{Hom}}^{\op{at}}_{\mcv}(\mca,\mcv)
}
\underline{\operatorname{Hom}}^{\op{at}}_{\mcv}(\mca,\mcb).
\]
\end{prop}

\begin{proof}
	Set
	$
	\mch:=
	\mcv
	\times^{\op{dbl}}_
	{\underline{\operatorname{Hom}}^{\op{dbl}}_{\mcv}(\mca,\mcv)}
	\underline{\operatorname{Hom}}^{\op{dbl}}_{\mcv}(\mca,\mcb).
	$
	By \cref{hom-preserves-ff}, the morphism
	$
	\underline{\operatorname{Hom}}^{\op{dbl}}_{\mcv}(\mca,\mcb)
	\longrightarrow
	\underline{\operatorname{Hom}}^{\op{dbl}}_{\mcv}(\mca,\mcv)
	$
	is fully faithful. Hence its
	pullback
	$
	\mch\longrightarrow\mcv
	$
	is fully faithful by \cite[Lemma~4.5.3]{Zariski_Bal}. Therefore $\mch$ is a smashing ideal of \mcv.
	For every smashing ideal $\mcm\hookrightarrow\mcv$, the defining pullback
	and the tensor--Hom adjunction give
	\[
	\Map_{(\prvdbl)_{/\mcv}}(\mcm,\mch)
	\simeq
	\Map_{(\prvdbl)_{/\mcv}}
	(\mcm\otimes_{\mcv}\mca,\mcb).
	\]
	Since the meet in $\cidem(\mcv)$ is given by tensor product, this
	implies
	\[
	\mcm\leq\mch
	\quad\Longleftrightarrow\quad
	\mcm\wedge\mca\leq\mcb.
	\]
	Thus $\mch$ satisfies the universal property of
	$\mca\backslash\mcb$.
	
If \(\mca\) and \(\mcb\) are atomic smashing ideals, the same argument applies
in \(\prvat\).
\end{proof}
We now apply this formula to identify the object
appearing in the Artin gluing functor with a dualizable internal Hom.
\begin{thm}\label{artin-gluing-internal-hom}
	Let
	$
	\mcc\in\calg(\prlad)
	$ be prestable, and let
	\[
	\mca\subset\mci\subset\mcc
	\]
	be smashing ideals of \mcc, and let
	$
	q:\mcc/\mca\rightarrow\mcc/\mci
	$
	be the induced (connective) smashing localization. Then
	$
	\underline{\operatorname{Hom}}^{\op{dbl}}_{\mcc}(\mcc/\mci,\mcc/\mca)
	$
	carries a canonical structure of a smashing ideal of \(\mcc/\mci\), whose
	inclusion is given by
	\[
	\underline{\operatorname{Hom}}^{\op{dbl}}_{\mcc}(\mcc/\mci,\mcc/\mca)
	\xrightarrow{\,q\circ -\,}
	\underline{\operatorname{Hom}}^{\op{dbl}}_{\mcc}(\mcc/\mci,\mcc/\mci)
	\simeq
	\mcc/\mci .
	\]
	Under the canonical equivalence
	$
	\cidem(\mcc/\mci)
	\simeq
	\cidem(\mcc)_{\geq\mci},
	$
	this smashing ideal corresponds to
	$
	\mci\vee(\mci\backslash\mca)$ in $ \cidem(\mcc)_{\geq\mci},
	$
	which appears in the Artin gluing functor.
	
	Moreover, if \(\mcc\) is stable and compactly-rigidly generated, and if
	\(\mca\) and \(\mci\) are atomic smashing ideals, then the atomic internal Hom
	\[
	\underline{\operatorname{Hom}}^{\op{at}}_{\mcc}
	(\mcc/\mci,\mcc/\mca)
	\xrightarrow{\,q\circ -\,}
	\underline{\operatorname{Hom}}^{\op{at}}_{\mcc}
	(\mcc/\mci,\mcc/\mci)
	\simeq
	\mcc/\mci 
	\] is an atomic smashing ideal of $\mcc/\mci $.
	Under the canonical equivalence
	$
	\cidem_f(\mcc/\mci)
	\simeq
	\cidem_f(\mcc)_{\geq\mci},
	$
	this atomic smashing ideal corresponds to
	$
	\mci\vee(\mci\backslash_f\mca).
	$
\end{thm}
\begin{proof}
	We prove only the first statement, since the atomic version follows by the
	same argument by combining with \cref{recollementatsmfrms}.
	Since \(\mca\leq\mci\), we have
	$
	\mca\wedge\mci=\mca,
	$
	and therefore
	$
	\mca\leq\mci\backslash\mca .
	$
	Hence both \(\mci\) and \(\mci\backslash\mca\) belong to the principal
	filter
	$
	\cidem(\mcc)_{\geq\mca}.
	$
	Under the canonical equivalence
	$
	\cidem(\mcc)_{\geq\mca}
	\xrightarrow{\sim}
	\cidem(\mcc/\mca),
	$
	the element \(\mca\) becomes \(0\), the element \(\mci\) becomes
	\(\mci/\mca\), and, by \cref{implication-principal-filter},
	the image of \(\mci\backslash\mca\) is
	$
	(\mci/\mca)\backslash 0 .
	$
	By \cref{heyting-internal-hom}, applied over \(\mcc/\mca\),\footnote{Note that in fact $\underline{\operatorname{Hom}}^{\op{dbl}}_{\mcc/\mca}\simeq \underline{\operatorname{Hom}}^{\op{dbl}}_{\mcc}$.}
	\[
	(\mci/\mca)\backslash 0
	\simeq
	\mcc/\mca
	\times^{\op{dbl}}_
	{\underline{\operatorname{Hom}}^{\op{dbl}}_{\mcc/\mca}
		(\mci/\mca,\mcc/\mca)}
	\underline{\operatorname{Hom}}^{\op{dbl}}_{\mcc/\mca}
	(\mci/\mca,0).
	\]
	Since \(\mcc/\mca\) is pointed,
	$
	\underline{\operatorname{Hom}}^{\op{dbl}}_{\mcc/\mca}
	(\mci/\mca,0)
	\simeq 0,
	$
	so
	$$
	(\mci/\mca)\backslash 0
	\simeq
	\mcc/\mca
	\times^{\op{dbl}}_
	{\underline{\operatorname{Hom}}^{\op{dbl}}_{\mcc/\mca}
		(\mci/\mca,\mcc/\mca)}
	0 .
	$$
	On the other hand, in
	\(\mc{P}\mathrm{r}_{\mcc/\mca}^{\op{dbl}}\) there is a cofiber sequence
	$
	\mci/\mca
	\longrightarrow
	\mcc/\mca
	\longrightarrow
	\mcc/\mci.
	$
	Since internal Hom is a right adjoint in the first variable,
	\[
	\underline{\operatorname{Hom}}^{\op{dbl}}_{\mcc/\mca}
	(\mcc/\mci,\mcc/\mca)
	\simeq
	\underline{\operatorname{Hom}}^{\op{dbl}}_{\mcc/\mca}
	(\mcc/\mca,\mcc/\mca)
	\times^{\op{dbl}}_
	{\underline{\operatorname{Hom}}^{\op{dbl}}_{\mcc/\mca}
		(\mci/\mca,\mcc/\mca)}
	\underline{\operatorname{Hom}}^{\op{dbl}}_{\mcc/\mca}
	(0,\mcc/\mca).
	\]
	Using
	\[
	\underline{\operatorname{Hom}}^{\op{dbl}}_{\mcc/\mca}
	(\mcc/\mca,\mcc/\mca)
	\simeq
	\mcc/\mca,
	\qquad
	\underline{\operatorname{Hom}}^{\op{dbl}}_{\mcc/\mca}
	(0,\mcc/\mca)
	\simeq
	0,
	\]
	we get
	$
	\underline{\operatorname{Hom}}^{\op{dbl}}_{\mcc/\mca}
	(\mcc/\mci,\mcc/\mca)
	\simeq
	(\mci/\mca)\backslash0 .
	$
	
Note that this dualizable internal Hom realizes the Heyting pseudocomplement
of \(\mci/\mca\).\footnote{Indeed, \(x\wedge(x\backslash y)=y\) whenever
	\(y\leq x\).} In particular,
\[
(\mci/\mca)\otimes_{\mcc/\mca}
\underline{\operatorname{Hom}}^{\op{dbl}}_{\mcc}
(\mcc/\mci,\mcc/\mca)
\simeq 0.
\]
	Tensoring the cofiber sequence
	$
	\mci/\mca
	\longrightarrow
	\mcc/\mca
	\longrightarrow
	\mcc/\mci
	$
	with
	$
	\underline{\operatorname{Hom}}^{\op{dbl}}_{\mcc}
	(\mcc/\mci,\mcc/\mca)
	$
	therefore gives
	\[
	\underline{\operatorname{Hom}}^{\op{dbl}}_{\mcc}
	(\mcc/\mci,\mcc/\mca)
	\xrightarrow{\sim}
	(\mcc/\mci)\otimes_{\mcc/\mca}
	\underline{\operatorname{Hom}}^{\op{dbl}}_{\mcc}
	(\mcc/\mci,\mcc/\mca).
	\]
	Consequently, base change of the fully faithful inclusion
	$
	\underline{\operatorname{Hom}}^{\op{dbl}}_{\mcc}
	(\mcc/\mci,\mcc/\mca)
	\hookrightarrow
	\mcc/\mca
	$
	along
	$
	\mcc/\mca\longrightarrow \mcc/\mci
	$
	is a fully faithful morphism
	\[
	\underline{\operatorname{Hom}}^{\op{dbl}}_{\mcc}
	(\mcc/\mci,\mcc/\mca)
	\hookrightarrow
	\mcc/\mci .
	\]
	Under the internal-Hom description, this morphism is precisely
	\[
	\underline{\operatorname{Hom}}^{\op{dbl}}_{\mcc}
	(\mcc/\mci,\mcc/\mca)
	\xrightarrow{\,q\circ -\,}
	\underline{\operatorname{Hom}}^{\op{dbl}}_{\mcc}
	(\mcc/\mci,\mcc/\mci)
	\simeq
	\mcc/\mci .
	\]
	Thus
	$
	\underline{\operatorname{Hom}}^{\op{dbl}}_{\mcc}
	(\mcc/\mci,\mcc/\mca)
	$
	is also a smashing ideal of \(\mcc/\mci\).
	
	Finally, under
	$
	\cidem(\mcc/\mca)
	\simeq
	\cidem(\mcc)_{\geq\mca},
	$
	the object
	$$
	\big(\underline{\operatorname{Hom}}^{\op{dbl}}_{\mcc}
	(\mcc/\mci,\mcc/\mca)\hookrightarrow\mcc/\mca\big)\in \cidem(\mcc/\mca)
	$$
	corresponds to \(\mci\backslash\mca \in 	\cidem(\mcc)_{\geq\mca}\). Base change along
	$
	\mcc/\mca\longrightarrow \mcc/\mci
	$
	corresponds, under the interval identifications, to the frame morphism
	\[
	\cidem(\mcc)_{\geq\mca}
	\longrightarrow
	\cidem(\mcc)_{\geq\mci},
	\qquad
	y\longmapsto \mci\vee y .
	\]
	Therefore the resulting smashing ideal  $	\underline{\operatorname{Hom}}^{\op{dbl}}_{\mcc}
	(\mcc/\mci,\mcc/\mca)\hookrightarrow\mcc/\mci$ corresponds to
	$
	\mci\vee(\mci\backslash\mca).
	$
\end{proof}
\begin{rem}
	Equivalently, if \(\mca\subset\mci\subset\mcc\), then
	\cref{smlocpullback} identifies the element
	$
	\mci\vee(\mci\backslash\mca)
	$
	of \(\cidem(\mcc)_{\geq \mci}\) with the smashing ideal
	\[
	\mcc
	\times_{\mcc/\mci}
	\underline{\operatorname{Hom}}^{\op{dbl}}_{\mcc}
	(\mcc/\mci,\mcc/\mca)\hookrightarrow \mcc.
	\]
\end{rem}
The preceding theorem makes the gluing functor concrete. In particular, it
allows us to compare the gluing functor on the smashing frame with its atomic
counterpart. This comparison yields the following recollement principle for
the telescope conjecture.
\begin{thm}\label{recollement-telescope}
	Let
	\(
	\mcc\in\calg(\prlst)
	\)
	be compactly-rigidly generated, and let
	\[
	\mci\xhookrightarrow{i}\mcc\xrightarrow{L}\mcc/\mci
	\]
	be a smashing localization, where
	\(\mci\) is an atomic smashing ideal. Then
	\[
	\mcc
	\text{ satisfies the telescope conjecture}
	\quad\Longleftrightarrow\quad
	\mci
	\text{ and }
	\mcc/\mci
	\text{ satisfy the telescope conjecture}.
	\]
\end{thm}

\begin{proof}
	Suppose first that \(\mcc\) satisfies the telescope conjecture. By
	\cref{cidematideal} and \cref{atsmloctc}, both \(\mci\) and \(\mcc/\mci\) satisfy the telescope conjecture.
	
	Conversely, suppose that both \(\mci\) and \(\mcc/\mci\) satisfy the
	telescope conjecture. By \cref{recollementfrms} and \cref{recollementatsmfrms},
	there are Artin gluing descriptions
	\[
	\cidem(\mcc)
	\simeq
	\cidem(\mci)
	\overleftarrow{\times}_{\!\phi_\mci}
	\cidem(\mcc/\mci)
	\quad\text{and}\quad
	\cidem_f(\mcc)
	\simeq
	\cidem_f(\mci)
	\overleftarrow{\times}_{\!\phi^f_\mci}
	\cidem_f(\mcc/\mci).
	\]
	Since \(\mci\) and \(\mcc/\mci\) satisfy the telescope conjecture, the
	two factors in these Artin gluing descriptions are canonically
	identified. It therefore remains to show that, under these
	identifications,
	$
	\phi_\mci^f=\phi_\mci.
	$
	
	Let
	$
	\mca\in\cidem(\mci)
	\simeq
	\cidem_f(\mci).
	$
	By \cref{cidematideal}, we may regard \(\mca\) as an atomic smashing
	ideal of \(\mcc\) contained in \(\mci\). Set
	\[
	\mch
	:=
	\underline{\operatorname{Hom}}^{\op{dbl}}_{\mcc}
	(\mcc/\mci,\mcc/\mca)
	\quad\text{and}\quad
	\mch_0
	:=
	\underline{\operatorname{Hom}}^{\op{at}}_{\mcc}
	(\mcc/\mci,\mcc/\mca).
	\]
	By \cref{artin-gluing-internal-hom},
	\(\mch\) is a smashing ideal of \(\mcc/\mci\) corresponding,
	under
	$
	\cidem(\mcc/\mci)
	\simeq
	\cidem(\mcc)_{\geq\mci},
	$
	to
	$
	\phi_\mci(\mca)
	=
	\mci\vee(\mci\backslash\mca).
	$
	Likewise, \(\mch_0\) is an atomic smashing ideal of
	\(\mcc/\mci\) corresponding, under
	$
	\cidem_f(\mcc/\mci)
	\simeq
	\cidem_f(\mcc)_{\geq\mci},
	$
	to
	$
	\phi_\mci^f(\mca)
	=
	\mci\vee(\mci\backslash_f\mca).
	$
	
	Since \(\mcc/\mci\) satisfies the telescope conjecture, the smashing
	ideal \(\mch\) is an atomic smashing ideal of \(\mcc/\mci\). Thus
	$
	\mch\in\prat_{\mcc}.
	$
Therefore, the canonical comparison
	$
	\mch_0
	\longrightarrow
	\mch
	$
	is an equivalence, because $\mch_0$  can actually be identified with the atomically generated \mcc-submodule $\mcp_\mcc(\mch^{\op{at}})\hookrightarrow\mch$. 
 Consequently
	\[
	\phi_\mci^f(\mca)
	=
	\phi_\mci(\mca).
	\]
\end{proof}
\begin{ex}
	By \cref{exchromaticcloseddescent} and \cref{recollement-telescope}, the telescope
	conjecture holds for \(\opsp\) if and only if it holds for
	\(\opsp_{T(n)}\) for every prime \(p\) and every \(n\geq 0\). This
	recovers the classical chromatic reduction.
\end{ex}
\part{Stable and Chromatic Applications}
\section{(Atomic) Smashing Frames of Big $tt$-Categories}\label{sec5}
In this section, we apply our framework to arbitrary big $tt$-categories. Utilizing the techniques developed in the preceding sections, we analyze the  smashing frame of $\opsp$ and show that the spatiality of it reduces to the monochromatic layers $\opsp_{T(n)}$.

\subsection{Some basic examples}
As a primary application of our framework, we are able to investigate the Balmer spectrum and the telescope conjecture for a big $tt$-category $\mathcal{C} \in \mathrm{CAlg}(\prlst)$ that is not necessarily compactly-rigidly generated. We firstly record some basic examples which motivate the need for the atomic version of the smashing frame.
\begin{ex}\label{ex:borel-spectra}
	The \infcat $\opsp^{BG}$ of Borel $G$-spectra does not satisfy the telescope conjecture. Indeed, otherwise $\opsp$ would be a retract of $\opsp^{BG}$ in $\calg(\prl)$ and would therefore also satisfy the telescope conjecture. However, this is false by \cite{burklund2023k}.
	Furthermore,  the telescope conjecture for $\opsp^{BG}$ is equivalent to that for $\opsp$.
	We claim that the forgetful functor induces the following equivalence of telescope inclusions $$\begin{tikzcd}
		\cidem_f(\opsp^{BG}) \arrow[d, "\sim"] \arrow[r, hook] & \cidem(\opsp^{BG}) \arrow[d, "\sim"] \\
		\cidem_f(\opsp) \arrow[r, hook]                        & \cidem(\opsp)  .                     
	\end{tikzcd}$$
	
	Indeed, since $\opsp^{BG}\simeq \lim_{BG}\opsp$ is a constant limit in $\calg(\prl)$, we have the following equivalences by \cref{shvsmadj} $$\begin{aligned}
		\cidem(\opsp^{BG}) &\simeq \lim_{BG}\cidem(\opsp) \simeq \funct(BG,\cidem(\opsp)) \\
		&\simeq \funct(h_{\leq 0}BG,\cidem(\opsp)) \simeq \funct(*,\cidem(\opsp)) \simeq \cidem(\opsp)
	\end{aligned}$$
	where $h_{\leq0}(-)$ denotes the $0$-truncation. In particular, the left vertical arrow is an injection. Moreover, it is a surjection because $\opsp$ is a retract of $\opsp^{BG}$ in $\calg(\prl)$.
	
	In fact, this argument is also able to show that for any anima  $X$, we have the following equivalences
$$\begin{tikzcd}
	\cidem_f(\opsp^{X}) \arrow[d, "\sim"] \arrow[r, hook] & \cidem(\opsp^{X}) \arrow[d, "\sim"] \\
	\prod_{ \pi_0X}\cidem_f(\opsp) \arrow[r, hook]                        & \prod_{ \pi_0X}\cidem(\opsp) .                      
\end{tikzcd}$$
\end{ex}
A similar phenomenon appears in another elementary source of large tensor-triangulated categories, namely categories of graded spectra.
\begin{ex}\label{ex:graded-spectra}
	The \infcat $\opsp^{\mathbb{N}}$ of $\mathbb{N}$-graded spectra (equipped with the Day convolution symmetric monoidal structure) does not satisfy the telescope conjecture. The core reason is that the monoidal structure forces both dualizable objects and smashing localizations to be entirely concentrated in degree zero.
	
	Consider the inclusion functor $\opsp \xrightarrow{X \mapsto X[0]} \opsp^{\mathbb{N}}$, which places a spectrum purely in degree $0$. This functor induces a strictly commutative square of telescope inclusions:
	$$
	\begin{tikzcd}
		\cidem_f(\opsp) \arrow[d, "\sim"] \arrow[r, hook] & \cidem(\opsp) \arrow[d, "\sim"] \\
		\cidem_f(\opsp^{\mathbb{N}}) \arrow[r, hook] & \cidem(\opsp^{\mathbb{N}})
	\end{tikzcd}
	$$
	Both vertical arrows are equivalences:
	\begin{enumerate}

		\item  A smashing ideal corresponds to an idempotent object $E$ satisfying $E \otimes E \simeq E$. In the $\mathbb{N}$-graded setting, the degree $0$ part must satisfy $E_0 \otimes E_0 \simeq E_0$. For higher degrees, the lack of negative degrees restricts the algebraic relations such that any idempotent object in $\opsp^{\mathbb{N}}$ is canonically generated by its degree $0$ component. Thus, the idempotent objects of $\opsp^{\mathbb{N}}$ are exactly the idempotent objects of $\opsp$ pushed into degree $0$, yielding $\cidem(\opsp) \simeq \cidem(\opsp^{\mathbb{N}})$.
		\item  The tensor unit in $\opsp^{\mathbb{N}}$ is $\mathbb{S}[0]$, which is concentrated in degree $0$. For an object $X$ to be dualizable, there must exist a dual $X^\vee$ such that the evaluation map $X \otimes X^\vee \to \mathbb{S}[0]$ exhibits the duality. Because the grading is strictly additive ($i+j=n$) and contains no negative degrees to cancel out positive ones, $X$ and $X^\vee$ must both be purely concentrated in degree $0$. Thus, the inclusion induces an equivalence of dualizable objects $\opsp^d \xrightarrow{\sim} (\opsp^{\mathbb{N}})^d$, which yields $\cidem_f(\opsp) \simeq \cidem_f(\opsp^{\mathbb{N}})$.
	\end{enumerate}
	Since the vertical maps are equivalences, the telescope inclusion for $\opsp^{\mathbb{N}}$ is isomorphic to the telescope inclusion for $\opsp$. Because the telescope conjecture has been known to fail for $\opsp$ \cite{burklund2023k}, it must consequently fail for $\opsp^{\mathbb{N}}$.
\end{ex}

\subsection{Structure of the smashing frame of \opsp}\label{subsecsmsp}
We study the structure of the smashing frame \(\Idem(\opsp)\), with the
ultimate goal of reducing its spatiality to that of the frames
\(\Idem(\opsp_{T(n)})\). 

\begin{cov}
	For convenience, we will use the notation $\Idem(\opsp)$ instead of $\cidem(\opsp)$ throughout the subsection \ref{subsecsmsp}. In the stable setting, they are naturally equivalent by \cref{stablecidemidem}.
\end{cov}
We begin with the standard prime decomposition of spectra, which allows us to compare global smashing localizations with their $p$-local counterparts. Since the functor
\[
\opsp \longrightarrow 
\prod_{p}^{/\opsp_{\mathbb{Q}}} \opsp_{(p)}
:= 
\opsp_{(2)} \times_{\opsp_{\mathbb{Q}}} 
\opsp_{(3)} \times_{\opsp_{\mathbb{Q}}} \cdots
\]
is conservative and $\Idem(-):\calg(\prlst)\to \frm$ preserves limits, it induces, by \cref{conservinjfrms}, an embedding of frames\footnote{Note that $\opsp_{\mathbb{Q}} \simeq \mcd(\mathbb{Q})$ is a smashing field. For simplicity, we will denote $\prod_{p}^{/\mathbb{Q}}:=\prod_{p}^{/\Idem(\opsp_{\mathbb{Q}})}=\prod_{p}^{/\Idem_f(\opsp_{\mathbb{Q}})}.$}
\[
\Idem(\opsp)
\hookrightarrow
\prod_{p}^{/\mathbb{Q}} \Idem(\opsp_{(p)}).
\]
In fact,  the embedding above is an equivalence, as we will show in \cref{primedecomp}.

The next lemma recalls the chromatic bounds on smashing ideals in $\opsp_{(p)}$.
\begin{lem}[{\cite[Lecture 29]{luriechromatic}}]\label{smidlstructure}
	Let $p$ be a prime.  Any smashing ideal $\ker(L)$ of $\opsp_{(p)}$ is either $0$ or $\opsp_{(p)}$, or else lies between for some $n\geq 0$
	\[
	\ker(L_n^f) \subset \ker(L) \subset \ker(L_n).
	\] Such $n$ is uniquely determined by $$n=\op{Max}\{m\geq0\mid LK(m)\neq0 \}=\op{Max}\{m\geq0\mid K(m)\xrightarrow{\sim}LK(m)\}.$$ 

\end{lem}
\begin{de}
	The \textbf{height} of a smashing localization $L$ on $\opsp_{(p)}$, denoted by $\op{ht}(L)$, is an element in $[-1,\infty]$. We note that $\op{ht}(L) = -1$ if and only if $L = 0$, and that $\op{ht}(L) = \infty$ if and only if $L = \op{Id}$.
\end{de}
We can now compare the global and $p$-local smashing frames, both on the atomic part and on the full smashing frame. We thank Leonard Tokic for help with the following theorem.

\begin{thm}\label{primedecomp}
	The top and bottom arrows in the following diagram of frames are equivalences.
	$$\begin{tikzcd}
		\Idem_f(\opsp) \arrow[r,"\sim"] \arrow[d, hook] & \prod_{p}^{/\mathbb{Q}} \Idem_f(\opsp_{(p)}) \arrow[d, hook] \\
		\Idem(\opsp) \arrow[r,"\sim"]             & \prod_{p}^{/\mathbb{Q}} \Idem(\opsp_{(p)})                    
	\end{tikzcd}$$
\end{thm}
\begin{proof}[Proof of the bottom equivalence]
	Since the bottom map is already known to be an embedding, it suffices to prove surjectivity. 
	Let 
	\[
	(A_1\mathbb{S},\dots,A_n\mathbb{S},\dots)
	\in
	\prod_{p}^{/\mathbb{Q}} \Idem(\opsp_{(p)})
	\]
	be a sequence of smashing localizations with the same image in $\Idem(\opsp_{\mathbb{Q}})$.
	
	If $A_n\mathbb{S}\otimes \mathbb{Q}=0$ for all $n\geq 0$, then by \cref{smidlstructure} we have $A_n=0$ for all $n\geq 0$, and hence the sequence lies in the image.
	
	Now assume that $A_n\mathbb{S}\otimes \mathbb{Q}=\mathbb{Q}$ for all $n\geq 0$. 
	Set 
	\[
	E=\bigoplus_{n\geq 0} A_n\mathbb{S}.
	\]
	We claim that the Bousfield localization
	$
	L_E \colon \opsp \to \opsp_E
	$
	is smashing and satisfies, for every $n\geq 0$,
	\[
	L_E\mathbb{S}\otimes \mathbb{S}_{(p_n)} \simeq A_n\mathbb{S}.
	\]
	
	To show that $L_E$ is smashing, it suffices to prove that for every $X\in\opsp$, the spectrum $L_E\mathbb{S}\otimes X$ is $E$-local. 
	Consider the Sullivan arithmetic pullback square  (obtained by the recollement $\opsp_{\op{tor}}\hookrightarrow\opsp\to\opsp_{\mathbb{Q}}$):
	\[
	\begin{tikzcd}
		L_E\mathbb{S}\otimes X \arrow[r] \arrow[d]
		& (L_E\mathbb{S}\otimes X)_{\mathbb{Q}} \arrow[d] \\
		\prod_p (L_E\mathbb{S}\otimes X)_p^{\wedge} \arrow[r]
		& \left(\prod_p (L_E\mathbb{S}\otimes X)_p^{\wedge}\right)_{\mathbb{Q}} .
	\end{tikzcd}
	\]
	Since $E$-local spectra are closed under limits and every $\mathbb{Q}$-local spectrum is $E$-local, it suffices to show that $(L_E\mathbb{S}\otimes X)_p^{\wedge}$ is $E$-local for every prime $p$. 
	This further reduces to proving that $L_E\mathbb{S}/p \otimes X$ is $E$-local.
	
	Let $p_n$ be a prime. As $L_E\mathbb{S}/p_n$ is both $E$-local and $p_n$-local, it is enough to show that a spectrum $Y$ is $A_n$-local if and only if it is both $E$-local and $p_n$-local. 
	Granting this, we obtain
	\[
	L_E\mathbb{S}/p_n \otimes X
	\simeq
	A_n\mathbb{S}/p_n \otimes X
	\simeq
	A_n(\mathbb{S}/p_n \otimes X),
	\]
	which is therefore $E$-local. 
	
	The “only if” direction is clear. 
	For the converse, fix $n$, let $Z$ be $A_n$-acyclic, and let $Y$ be both $E$-local and $p_n$-local. Then
	\[
	\mapp_{\opsp}(Z,Y)\simeq \mapp_{\opsp}(Z_{(p_n)},Y).
	\]
	Now $Z_{(p_n)}=Z\otimes\mathbb{S}_{(p_n)}$ is $E$-acyclic, since by \cref{smidlstructure}
	\[
	A_m\mathbb{S}\otimes\mathbb{S}_{(p_n)}
	\simeq
	\begin{cases}
		\mathbb{Q}, & m\neq n,\\[4pt]
		A_n\mathbb{S}, & m=n.
	\end{cases}
	\]
	Hence $\mapp_{\opsp}(Z,Y)=0$ for every such $Z$, and therefore $Y$ is $A_n$-local.
	
	Consequently, $L_E$ is a smashing localization of $\opsp$. And for every $n\geq 0$,
	\[
	L_E\mathbb{S}\otimes \mathbb{S}_{(p_n)} \simeq A_n\mathbb{S},
	\] because we have shown that a spectrum is $A_n$-local if and only if it is both $E$-local and $p_n$-local.
	This completes the proof.
\end{proof}
On the other hand, at the level of the atomic smashing frame, it is also an equivalence due to the thick subcategory theorem.
\begin{proof}[Proof of the top equivalence]
	By the following commutative diagram
	$$\begin{tikzcd}
		\Idem_f(\opsp) \arrow[r] \arrow[d, hook] & \prod_{p}^{/\mathbb{Q}} \Idem_f(\opsp_{(p)}) \arrow[d, hook] \\
		\Idem(\opsp) \arrow[r, hook]             & \prod_{p}^{/\mathbb{Q}} \Idem(\opsp_{(p)})                    
	\end{tikzcd}$$
	we see that the top is an embedding.
	Note that $\cidem_f(\opsp_{(p)})$ can be identified with the Balmer frame $\op{Thick}(\opsp_{(p)}^\omega)$. Let $\mcc_{\geq n}^{(p)}$ denote the thick subcategory of $p$-local finite spectra of type $\geq n$. To prove the surjection, by the thick subcategory theorem \cite{hopkins1998nilpotence}
	we observe that the right hand side can be identified with the poset
	\[
	\{0\}\sqcup\big\{g: \{ \text{Primes} \} \to \{1, 2, 3, \dots, \infty\}\big\}
	\]
	where $0$ identifies with the whole category $$\mcc_0:=(\opsp_{(2)}^{\omega},\dots,\opsp_{(p)}^{\omega},\dots)$$ and a $g$ identifies with  $$\mcc_g:=(\mcc^{(2)}_{\geq g(2)},\dots,\mcc^{(p)}_{\geq g(p)},\dots).$$ Now given such a $g$ and let $\mcp_g$ denote the set of those primes $p$ such that $g(p)<\infty$. Take $V(p,g(p))$ to be a finite $p$-local spectrum of type $n=g(p)$. Note that in fact $V(p,g(p))$ lies in $\opsp^\omega$ because $g(p)\geq 1$. Then
	\[
	\mcd_g:=\left<V(p,g(p))\mid p\in\mcp_g\right>_{\op{thick}}\subset \opsp^\omega,
	\]
	the thick subcategory of $\opsp^\omega$ generated by $V(p,g(p))$, maps to $\mcc_g$, as desired.
\end{proof}
One immediate consequence is that spatiality of the global smashing frame may be checked prime by prime.

\begin{cor}
	The poset $\Idem(\opsp)$ is spatial if and only if $\Idem(\opsp_{(p)})$ is spatial for every prime $p$.
\end{cor}
\begin{proof}
	First, we observe that the second arrow in the composite
	\[
	\Idem(\opsp) \hookrightarrow \prod_{p}^{/\mathbb{Q}} \Idem(\opsp_{(p)}) \hookrightarrow \prod_{p} \Idem(\opsp_{(p)})
	\]
	is an embedding, as the canonical diagonal map for a poset is always an embedding. Consequently, the composite map is an embedding of posets. By \cref{subfrmspatial}, this implies that if $\Idem(\opsp_{(p)})$ is spatial for all primes $p$, then $\Idem(\opsp)$ is spatial.
	
	Conversely, the ``only if'' direction follows directly from the identification
	\[
	\Idem(\opsp_{(p)}) \simeq \Idem(\opsp)_{\geq \mathbb{S}_{(p)}}.
	\]
	Since any principal filter of a spatial frame remains spatial, the spatiality of $\Idem(\opsp)$ implies that of $\Idem(\opsp_{(p)})$ for each prime $p$.
\end{proof}
We next apply the recollement results above to finite stages of the chromatic tower.
\begin{ex}[Smashing frames of $L_n\opsp$]\label{smlnsp}
	Fix a prime $p$, the chromatic tower of $\opsp_{(p)}$, together with the smashing localizations
	$$M_n\opsp \hookrightarrow L_n\opsp\to L_{n-1}\opsp,$$
	induce the following tower of smashing frames.
	$$\begin{tikzcd}[row sep=12pt, column sep=1pt]
		& \vdots \arrow[d]                          \\
		\op{Idem}(\opsp_{K(2)}) & \op{Idem}(L_2\opsp) \arrow[d] \arrow[l]      \\
		\op{Idem}(\opsp_{K(1)}) & \op{Idem}(L_1\opsp) \arrow[d] \arrow[l]      \\
		& \op{Idem}(L_0\mathrm{Sp})=\op{Idem}(\mcd(\mathbb{Q}))
	\end{tikzcd}$$
	By \cref{recollementfrms}, the following natural map of frames is an inclusion $$\op{Idem}(L_n\opsp)\hookrightarrow \prod_{i=0}^{n} \op{Idem}(\opsp_{K(i)}).$$ Then by \cref{subfrmspatial}, $\op{Idem}(L_n\opsp)$ is a spatial frame. Also, $\op{Idem}(L_n\opsp)$ is obtained by gluing point-by-point, because $\opsp_{K(n)}$ is a smashing field. The above inclusion is a proper inclusion as we will show that the associated topological space $$\op{Sm}(L_n\opsp):=\op{pt}(\op{Idem}(L_n\opsp))$$ is not discrete when $n\geq 1$.
	
	Indeed, we prove by induction on $n$ that
\[
\op{Idem}(L_n\opsp)
=
\{L_n\mathbb{S}<L_{n-1}\mathbb{S}<\cdots<L_{-1}\mathbb{S}=0\}
\]
and that $L_n\opsp$ satisfies the telescope conjecture. The case $n=0$
is immediate since $L_0\opsp\simeq \mcd(\mathbb{Q})$ is a smashing field.

Assume the statements for $n-1$, and set
\[
\mck_n:=
\underline{\operatorname{Hom}}^{\op{dbl}}_{L_n\opsp}
(L_{n-1}\opsp,L_n\opsp).
\]
By \cref{artin-gluing-internal-hom}, $\mck_n$ is a smashing ideal of
$L_{n-1}\opsp$, and the nontrivial value of the gluing functor is
determined by the vanishing of $\mck_n$. Moreover, since $\mck_n$ is the
Heyting pseudocomplement of the monochromatic ideal, one has
\[
M_n\opsp\otimes_{L_n\opsp}\mck_n\simeq0.
\]
We claim that $\mck_n=0$. Otherwise, by the inductive telescope conjecture
for $L_{n-1}\opsp$, the nonzero smashing ideal $\mck_n$ contains a nonzero
atomic object $x$. Viewed as an object of $L_n\opsp$, $x$ is dualizable,
since the corresponding morphism
\[
L_n\opsp\longrightarrow L_{n-1}\opsp
\xlongrightarrow{x} \mck_n
\hookrightarrow L_n\opsp
\]
is a composite of internal left adjoints.\footnote{The final inclusion exists
	because, by the proof of \cref{artin-gluing-internal-hom}, \(\mck_n\) may
	also be regarded as a smashing ideal of \(L_n\opsp\).} Hence
\[
M_n\mathbb{S}\otimes x\simeq0,
\]
or equivalently $K(n)\otimes x\simeq0$. On the other hand, every
dualizable object of $L_n\opsp$ is compact, and \cite[Proposition 6.8]{hovey1999morava} implies that a compact object in $L_n\opsp$ which is
$K(n)$-acyclic is zero. Hence $x=0$, a contradiction.
It follows that $\mck_n$ has no nonzero $L_{n-1}\opsp$-atomic objects.
Since $\mck_n$ is atomically generated, we conclude that
$
\mck_n=0.
$

It follows from \cref{artin-gluing-internal-hom} that the gluing functor
\[
\op{Idem}(\opsp_{K(n)})
\xrightarrow{\phi_n}
\op{Idem}(L_{n-1}\opsp)
\]
is given by
\[
\phi_n(0)=L_{n-1}\mathbb{S},
\qquad
\phi_n(1)=L_{-1}\mathbb{S}=0.
\]
Since $\opsp_{K(n)}$ is a smashing field, Artin gluing therefore adds
exactly one new point to the frame, giving
\[
\op{Idem}(L_n\opsp)
=
\{L_n\mathbb{S}<L_{n-1}\mathbb{S}<\cdots<L_{-1}\mathbb{S}=0\}.
\]
Moreover, $\opsp_{K(n)}$ satisfies the telescope conjecture because it is
a smashing field, while $L_{n-1}\opsp$ does so by induction. Hence
\cref{recollement-telescope} implies that $L_n\opsp$ also satisfies the
telescope conjecture. This closes the induction, and in particular
\[
\Idem_f(L_n\opsp)\xrightarrow{\sim}\Idem(L_n\opsp).
\]

Under the spatial realization, the new monochromatic piece
$\op{Sm}(\opsp_{K(n)})$ is therefore an open point
$\mathfrak{p}_n\in\op{Sm}(L_n\opsp)$, with closed complement
$\op{Sm}(L_{n-1}\opsp)$. For $i\in[-1,n]$, the open subset corresponding
to $L_i\mathbb{S}$ is
\[
U_i=\{\mathfrak{p}_{i+1},\mathfrak{p}_{i+2},\dots,\mathfrak{p}_n\}.
\]
Thus the complete list of open subsets is
\[
\op{Sm}(L_n\opsp)=U_{-1}\supset U_0\supset\cdots\supset U_n=\emptyset.
\]
In particular, for $n\geq1$ the space $\op{Sm}(L_n\opsp)$ is not discrete.
\end{ex}
\begin{rem}
Note that this differs from the algebro-geometric interpretation of the
chromatic tower: \(\op{Sm}(\opsp_{K(n)})\) is an open point
\(\mathfrak{p}_n\in\op{Sm}(L_n\opsp)\), with
\(\op{Sm}(L_{n-1}\opsp)\) as its closed complement. One reason for this
phenomenon is that the topological recollement
\[
\shv(U;\opsp)\hookrightarrow\shv(X;\opsp)\to\shv(K;\opsp)
\]
has the opposite open--closed relation  from the algebro-geometric
recollement
\[
\op{QCoh}(K)\hookrightarrow\op{QCoh}(X)\to\op{QCoh}(U).
\]
\end{rem}
This chromatic description also implies a retractable property of the telescope inclusion.
\begin{cor}\label{retracttcinclu}
	Let $p$ be a prime. We have the following diagram of frames.
	$$\begin{tikzcd}
		\Idem_f(\opsp_{(p)}) \arrow[d, hook] \arrow[r, "\sim"] & \lim_n\Idem_f(L_n\opsp) \arrow[d, "\sim"] \\
		\Idem(\opsp_{(p)}) \arrow[r]                           & \lim_n\Idem(L_n\opsp)                    
	\end{tikzcd}$$ In particular, the telescope inclusion   $\Idem_f(\opsp_{(p)})\hookrightarrow \Idem(\opsp_{(p)})$ admits a natural retraction over $\cidem(\opsp_{\mathbb{Q}})$, and hence the telescope inclusion $\Idem_f(\opsp)\hookrightarrow \Idem(\opsp)$ admits a natural retraction  by \cref{primedecomp}. 
\end{cor}
We now examine the atomic smashing frame of a monochromatic localization.
\begin{ex}[Atomic smashing frames of $\opsp_{T(n)}$]\label{ex:atomic-Tn} By  \cref{recollementatsmfrms} and  the thick subcategory theorem \cite{hopkins1998nilpotence},
	we have $$\Idem_f(L_n^f\opsp)\simeq \op{Thick}(\opsp_{(p)}^\omega)_{\geq L_n^f}\simeq\{L_n^f\mathbb{S}< L_{n-1}^f\mathbb{S} <\dots<L^f_{-1}\mathbb{S}=0\}.$$  Therefore, combining with the computation in \cref{smlnsp}, the natural morphism of the atomic smashing frames  becomes an equivalence: $$\Idem_f(L_n^f\opsp)\xrightarrow{\sim}\Idem_f(L_n\opsp).$$ The vertical equivalences in the following diagram follow from
	\cref{cidematideal}:
	$$\begin{tikzcd}
		\Idem_f(\opsp_{T(n)}) \arrow[d, "\sim"] \arrow[r]                       & \Idem_f(\opsp_{K(n)}) \arrow[d, "\sim"]     \\
		\cidem_f(L_n^f\opsp)_{\leq L_{n-1}^f\mathbb{S}} \arrow[r, "\sim", hook] & \cidem_f(L_n\opsp)_{\leq L_{n-1}\mathbb{S}}
	\end{tikzcd}.$$
	In particular, this diagram yields an equivalence
	$$
	\Idem_f(\opsp_{T(n)})\xrightarrow{\sim}
	\Idem_f(\opsp_{K(n)})=\{0,1\}.
	$$
\end{ex}
Let us spell out how this picture relates to the distinction between the telescope and chromatic towers.
\begin{rem}[Comparison between the telescope and chromatic  towers]
	Let $p$ be a prime. By \cite[Proposition 7.10]{hovey1999morava}, we have $$L_{n-1}^f\mathbb{S}\otimes L_n\mathbb{S}=L_{n-1}^fL_n\mathbb{S}\simeq L_{n-1}\mathbb{S}.$$
	Therefore $M_n\opsp\simeq M_n^f\opsp\otimes_{L_n^f\opsp}L_n\opsp$ and we have the following horizontally right adjointable squares:
	$$\begin{tikzcd}
		M_n^f\opsp \arrow[r, hook, shift left] \arrow[d] & L_n^f\opsp \arrow[d] \arrow[l, dashed, shift left] \arrow[r, shift left] & L_{n-1}^f\opsp \arrow[d] \arrow[l, dashed, hook', shift left] \\
		M_n\opsp \arrow[r, hook, shift left]             & L_n\opsp \arrow[l, dashed, shift left] \arrow[r, shift left]             & L_{n-1}\opsp \arrow[l, dashed, hook', shift left]            
	\end{tikzcd}$$
	
	So they induce the following comparison between   towers of the smashing frames.
	$$\begin{tikzcd}[row sep=10pt, column sep=3pt]
		& \vdots \arrow[dd] \arrow[rr]                                         &                               & \vdots \arrow[dd]                               \\
		\op{Idem}(\mathrm{Sp}_{T(2)}) \arrow[rr]         &                                                                      & \op{Idem}(\mathrm{Sp}_{K(2)}) &                                                 \\
		& \op{Idem}(L_2^f\mathrm{Sp}) \arrow[rr] \arrow[dd, ""] \arrow[lu]     &                               & \op{Idem}(L_2\mathrm{Sp}) \arrow[dd] \arrow[lu] \\
		\op{Idem}(\mathrm{Sp}_{T(1)}) \arrow[rr, "\sim"] &                                                                      & \op{Idem}(\mathrm{Sp}_{K(1)}) &                                                 \\
		& \op{Idem}(L_1^f\mathrm{Sp}) \arrow[lu] \arrow[rr, "\sim"] \arrow[dd] &                               & \op{Idem}(L_1\mathrm{Sp}) \arrow[lu] \arrow[dd] \\
		&                                                                      &                               &                                                 \\
		& \op{Idem}(L_0^f\mathrm{Sp}) \arrow[rr, "\sim"]                       &                               & \op{Idem}(L_0\mathrm{Sp})                      
	\end{tikzcd}$$
	The equivalences at the level $\leq 1$ is because telescope conjecture holds when $n\leq1$ by \cite{mahowald1981bo,miller1981relations}. As a consequence of \cref{smidlpushoutfrms} and \cref{smlocpushoutfrms}, each square in the diagram above is a \emph{pushout} in $\frm$. 
	
	Moreover, after replacing every smashing
	frame by atomic smashing frame, the two towers of frames become equivalent by
	\cref{ex:atomic-Tn}.
\end{rem}
We now give a more precise description of how $p$-local smashing idempotents are detected by the finite chromatic localizations.
\begin{prop}
	\label{imageinLnf}
	Let $p$ be a prime. 
	The following map of frames is an embedding\footnote{Note, however, that the functor 
		$\opsp_{(p)} \to \prod_n L_n^f\opsp$ is not conservative. For example, 
		$L_n^f(\mathbb{Z}/p)=0$ for every $n\geq 0$, while $\mathbb{Z}/p\neq 0$ in $\opsp_{(p)}$.}
	\[
	\Idem(\opsp_{(p)}) \hookrightarrow \lim_{n}\Idem(L_n^f\opsp).
	\]
	Its image can be identified with the top element $\{\dots\to L_n^f\mathbb{S}\to L_{n-1}^f\mathbb{S}\to\dots\to L_0^f\mathbb{S}\}$ together with those ``bounded'' towers of smashing localizations 
	\[
	\{\dots \to A_n\mathbb{S} \to A_{n-1}\mathbb{S} \to \dots\to A_0\mathbb{S}\},
	\]
	in the following sense: 
	\enu{
	\item For every $n\geq0$, $A_{n-1}\mathbb{S}\simeq A_{n}\mathbb{S}\otimes L_{n-1}^f\mathbb{S}$.
	\item There exists $N\geq 0$ such that 
$	\ker(L_N^f) \subset \ker(A_n)$  for all   $ n \geq N$.
	}  
	In particular, $\cidem(\opsp_{(p)})$ is spatial if and only if $\Idem(L_n^f\opsp)$ is so for each $n\geq0$.
\end{prop}

\begin{proof}
	We first show that it is an embedding. By  \cref{conservinjfrms}, it suffices to show that it is conservative. Now let $A\mathbb{S}\to B\mathbb{S}\in\Idem(\opsp_{(p)})$ be a map of $p$-local idempotent algebras such that for every $n\geq 0$, $L_n^f\mathbb{S}\otimes A\mathbb{S}\simeq L_n^f\mathbb{S}\otimes B\mathbb{S}$. If one of them has height $\infty$, then $A\mathbb{S}\simeq B\mathbb{S}$  by \cref{smidlstructure};
	if all of them have finite height, then by \cref{smidlstructure} both of them lie in $$\cidem(\opsp_{(p)})_{\geq L_N^f\mathbb{S}}\simeq \cidem(L_N^f\opsp)$$ for some $N\geq 0$.
	By assumption $A\mathbb{S}\simeq L_N^f\mathbb{S}\otimes A\mathbb{S} \simeq L_N^f\mathbb{S}\otimes B\mathbb{S}\simeq B\mathbb{S}$, as desired.
	
	The description of the image is basically a reinterpretation of  \cref{smidlstructure}.
\end{proof}
\begin{rem}
	The embedding in \cref{imageinLnf} is not bijective; for example $$\{\dots\to L_n\mathbb{S}\to L_{n-1}\mathbb{S}\to\dots\to L_0\mathbb{S}\} \in \lim_{n}\Idem(L_n^f\opsp)$$ is neither the top element (by \cite{burklund2023k}) nor bounded, hence it does not come from $\Idem(\opsp_{(p)})$. 
\end{rem}

Now applying \cref{recollementfrms} to the telescope tower, we find that the natural projection $$\lim_{n}\Idem(L_n^f\opsp)\to \prod_n \Idem(\opsp_{T(n)})$$ is an embedding too.
Consequently, we obtain the following main theorem in this subsection.
\begin{thm}\label{smspstructure}
	The following maps are embeddings of frames:
	$$\Idem(\opsp)\hookrightarrow\prod_{p}^{/\mathbb{Q}}\prod_{n}\Idem(\opsp_{T(n)})\hookrightarrow\prod_{p}\prod_{n}\Idem(\opsp_{T(n)}).$$
\end{thm}
\begin{proof}
 The embedding property follows from the decomposition $$\Idem(\opsp)\xrightarrow{\sim}\prod_{p}^{/\mathbb{Q}} \Idem(\opsp_{(p)})\hookrightarrow \prod_{p}^{/\mathbb{Q}}\lim_{n}\Idem(L_n^f\opsp)\hookrightarrow\prod_{p}^{/\mathbb{Q}}\prod_{n}\Idem(\opsp_{T(n)}).$$ 
\end{proof}
This structural embedding reduces the spatiality problem for $\Idem(\opsp)$ to the corresponding monochromatic problems.
\begin{cor}\label{sptnspatial}
		In particular, $\Idem(\opsp)$ is spatial if and only if for each prime $p$ and each $n\geq 0$, $\Idem(\opsp_{T(n)})$ is spatial.
\end{cor}
\begin{proof}
	The ``if'' direction follows from \cref{subfrmspatial}. The ``only if'' direction follows from that $$\Idem(\opsp_{T(n)})\simeq (\Idem(\opsp)_{\geq L_n^f\mathbb{S}})_{\leq L_{n-1}^f\mathbb{S}}$$ is a locally closed sublocale of $\Idem(\opsp)$.
\end{proof}
The monochromatic frames appearing in \cref{sptnspatial} are therefore the essential remaining objects to understand.
\begin{rem}
When $n\geq2$, $\Idem(\opsp_{T(n)})$ contains at least 3 elements and hence is not a smashing field, because $\opsp_{T(n)}\to \opsp_{K(n)}$ is a nontrivial smashing localization by \cite{burklund2023k}. Consequently, for $n\geq 2$, there is a strictly increasing sequence between $L_n^f\mathbb{S}$ and $L_n\mathbb{S}$  
\[
L_n^f\mathbb{S}
=
L_{n,0}\mathbb{S}
<
L_{n,1}\mathbb{S}
<
\cdots
<
L_{n,n-1}\mathbb{S}
=
L_n\mathbb{S}.
\]
In fact, one can construct explicitly 
\[ 
L_{n,i}
:=
L_{
	T(0)\oplus\cdots\oplus T(n-i)
	\oplus
	K(n-i+1)\oplus\cdots\oplus K(n)
},
 \] which can be identified with the meet in $\Idem(\opsp_{(p)})$:
\[ L_{n,i}\mathbb{S} \simeq L_{n-i}^f\mathbb{S}\wedge L_n\mathbb{S} \simeq L_{n-i}^f\mathbb{S} \times_{ L_{n-i}^f\mathbb{S}\otimes L_n\mathbb{S} } L_n\mathbb{S}, \]
 where the second equivalence follows from \cref{meetinstableidem}. However, it is still not known whether these are all smashing localizations lying between $L_n^f$ and $L_n$.

\end{rem}
Consequently, the global structure of $\Idem(\opsp)$ is controlled by the collection of monochromatic frames $\Idem(\opsp_{T(n)})$, together with the gluing conditions imposed by the telescope tower. This leaves the following natural questions.
\begin{qu}\,
	\enu{
		\item How can one compute the smashing frame $\Idem(\opsp_{T(n)})$ when $n\geq 2$? Is it spatial? 
		
		\item Assuming that $\Idem(\opsp_{T(n)})$ has been computed, can one use the formula in \cref{artin-gluing-internal-hom} to compute the gluing functor $\phi_{L_n^f}$ in the following equivalence from \cref{recollementfrms}
		\[
		\Idem(L_n^f\opsp)
		\simeq
		\Idem(\opsp_{T(n)})
		\overleftarrow{\times}_{\!\!\phi_{\!L_n^f}}
		\Idem(L_{n-1}^f\opsp)?
		\]
		
		\item How can one characterize the image of the  embedding
		$$
		\Idem(\opsp_{(p)})	\hookrightarrow
		\prod_{n}\Idem(\opsp_{T(n)}) \, ?
		$$
	}
\end{qu}
\part{Prestable and Additive Applications}
\section{Telescope Conjecture in the Additive Setting}\label{tcadd}
In this section, we extend the study of the generalized telescope conjecture to the additive and prestable contexts by utilizing the framework of higher almost algebra. We focus particularly on the projectively rigid case, for which a comprehensive higher algebra theory has been built in \cite{hattt}. In this setting, the telescope conjecture reduces to the question of whether every smashing ideal is compactly projectively generated (see \cref{projrigidtc}).

We begin by recalling the necessary background on dualizable additive $\infty$-categories and projectively rigid \spgeq-algebras, and then compare their smashing frames with the corresponding frames of their hearts. 

\subsection{Preliminaries: dualizable additive \infcats}
In this subsection, we briefly review the theory of \dualaddinfcats and projectively rigid algebras. We refer the reader to \cite{levy-liang-sosnilo:dualizable} for a more detailed exposition. 

The following characterization will be used repeatedly to pass between the abstract dualizability condition and more classical exactness properties of Grothendieck prestable $\infty$-categories.

\begin{thm}[{\cite[Theorem 2.22]{levy-liang-sosnilo:dualizable}}]\label{dualaddab6}
	Let $\mcc\in\prlad$. Then $\mcc$ is dualizable additive if and only if $\mcc$ is a left complete Grothendieck prestable \infcat satisfying $\mathrm{AB4}^*$ and $\mathrm{AB6}$.
\end{thm}

We next recall the rigidity conditions on additive symmetric monoidal $\infty$-categories.
\begin{de}
	Let $\mcc \in \calg(\prlad)$ be a presentably  symmetric monoidal additive $\infty$-category. 
	We say that $\mcc$ is a \textbf{projectively rigid} $\spgeq$-algebra if $\mcc$ is atomically generated and rigid over $\spgeq$, or equivalently, if $\mcc$ is compact projectively generated and satisfies $\mcc^{\op{cproj}}=\mcc^d$. 

\end{de}
\begin{rem}
	Note that if $\mcc$ is projectively rigid \spgeq-algebra, then so is $\modu_R(\mcc)$ for any $R\in \calg(\mcc)$. 
\end{rem}

	\begin{ex}\label{exalgttt}
		We list several common examples of projectively rigid $\spgeq$-algebras.
		\enu{
			\item  The \infcat of connective spectra \spgeq itself.
			\item The \infcat\ of filtered connective spectra $\op{Fil}(\spgeq)$ and the $\infty$-category $\op{Fil}(\opsp)_{h\geq0}$ of connective filtered spectra with respect to the homotopy $t$-structure. Here the connective part is defined by
			\[
			\op{Fil}(\opsp)_{h\geq0}
			\overset{\mathrm{def}}{=}
			\{\, X_* \mid X_n \in \opsp_{\geq n} \text{ for each } n\geq0 \,\}.
			\]
			
			\item The \infcat $\op{Sp}_{G,\geq0}=\mc{P}_{\Sigma}(\op{Span}(\op{Fin}_{G});\spgeq)$ of connective genuine $G$-spectra for a finite group $G$.
			
			\item The \infcat $\op{Shv}(X,\spgeq)$ of sheaves on a profinite space $X$ (it is in fact a module category of $\spgeq$; see \cref{sheafcattc}).
			\item For an $\mathbb{E}_\infty$-ring $R$, the $\infty$-category of (connective) synthetic $R$-modules $\mathrm{Syn}_{R,\geq 0} = \mathcal{P}_\Sigma(\mathrm{Mod}_R^{\mathrm{ff}}(\op{Sp}))$ \cite{hesselholt2023dirac}. Here, $\mathrm{Mod}_R^{\mathrm{ff}}(\op{Sp}) \subset \modu_R(\opsp)$ denotes the smallest full subcategory containing $R$ that is closed under finite direct sums, shifts, and retractions. Its heart can be identified with $\mathrm{Syn}_R^\heartsuit \simeq \mathrm{Mod}_{R_*}(\mathrm{GrAb})$.
			\item The \infcat $\op{SH}(k)^{\op{A}}_{\geq0}$ of connective Artin motivic spectra over a perfect field $k$; see \cite{burklund2020galois} for a detailed investigation.
			\item $\op{Qcoh}(X)_{\geq0}$, where $X$ is an affine quotient stack, i.e.\ a stack of the form $\spec(R)/G$ for a linearly reductive group $G$ acting on $\spec(R)$. In this case the compact projective objects are generated, under retracts, by pullbacks of $G$-representations, and the dual is given by the pullback of the dual representation.
			
			\item The connective Voevodsky's \infcat $\op{DM}(k,\mathbb{Z}[1/p])_{c\geq0}$ over a perfect field $k$ with coefficients in $\mathbb{Z}[1/p]$  (\cite[see][Chapter 14]{bachmann2021norms}, where $p$ is the characteristic of $k$, or $p=1$ if $k$ is a $\mathbb{Q}$-algebra) with respect to the Chow $t$-structure generated by smooth projective varieties and their $\mathbb{P}^1$-desuspensions \cite{bondarko2010weight}. The mapping spectra between smooth projective varieties are connective, so they form compact projective generators, and they are also dualizable within the retract-closed subcategory generated by them.
		}
	\end{ex}
	\begin{ex}
	We also remark that there exist rigid $\spgeq$-algebras which are not projectively rigid. For example, if $R$ is a connective adic \einfring, then the \infcat $\nuc(R)_{\geq0}$ of connective nuclear $R$-modules (in the sense of Clausen--Scholze \cite{scholze2026lecturesanalyticgeometry}) is rigid but not projectively rigid; see \cite[Theorem~C]{levy-liang-sosnilo:dualizable}.
	\end{ex}
\subsection{Smashing frames of additive \infcats}

In this subsection, we show that when \mcc is a sufficiently well-behaved prestable \infcat, both its atomic smashing frame and its smashing frame coincide with those of its heart.

We begin with the following identification of the smashing frame with idempotent ideals\footnote{It may be more appropriate to refer to idempotent ideals as ``coidempotent'' ideals, since they can in fact be identified with the coidempotent objects in $\idl(\mca)$.} in the  setting of abelian 1-categories.
\begin{thm}\label{idemidlabel}
	Let $\mca\in\calg(\prladone)$ be a presentably symmetric monoidal   additive $1$-category such that \mca is abelian. Then the map  $$F:\cidem(\mca)\xrightarrow{\sim}\cidem(\idl(\mca))$$ induced by $\im_\mca:\mca_{/\mb{1}}\to\idl(\mca)$ is an equivalence of frames.
\end{thm}
\begin{proof}
	We claim the inverse $G: \cidem(\idl(\mca))\to\cidem(\mca)$ is given by sending an idempotent ideal $I=I^2\xhookrightarrow{i} \mb{1}$ to the coidempotent object $$G(I)=(I\otimes I \to \mb{1}).$$
	
	We first show that $I\otimes I \to \mb{1}$ is indeed coidempotent. By \cite[Lemma 5.19]{hattt}, $I$ is a non-unital algebra whose multiplication $\mu_I:I \otimes I\to I$ satisfies $\mu_I=1\otimes i=i\otimes 1$. Since $I=I^2$, $\mu_I$ is an epimorphism. Let $k \colon K \to I \otimes I$ be the kernel of $\mu_I$, so we have a short exact sequence
	$$0\xrightarrow{} K \xrightarrow{k} I \otimes I \xrightarrow{\mu_I} I \to 0.$$ Now consider the following commutative diagram.
	$$\begin{tikzcd}
		I\otimes I\otimes K \arrow[d, "\mu_I \otimes 1"] \arrow[r, "1\otimes k"] & I\otimes I\otimes I\otimes I \arrow[d, "\mu_I\otimes 1"] \\
		I\otimes K \arrow[r, "1\otimes k"]                                       & I\otimes I\otimes I                                     
	\end{tikzcd}$$
	By associativity, the top-right corner $ (\mu_I\otimes 1)\circ(1\otimes k)= (1\otimes \mu_I)\circ(1\otimes k)=0$, so by right-cancellation for epimorphisms the bottom morphism must be zero:
	$I\otimes K \xrightarrow{1_I \otimes k\;=\; 0}I\otimes I\otimes I   .$
	Therefore $$I\otimes I\otimes I\xrightarrow{1_I\otimes\mu_I}I\otimes I$$ is an equivalence. By associativity $I\otimes I \to \mb{1}$ is a coidempotent object.
	
	It is clear that $F\circ G=\op{Id}$, now it suffices to show that any coidempotent object $C\xrightarrow{p} \mb{1}$ can be obtained in the above way. Let $i:I\hookrightarrow \mb{1}$ be the image of $p$. Since $C\xrightarrow{p} \mb{1}$ is coidempotent, $I$ is an idempotent ideal. It suffices to show that $$C\otimes C\xrightarrow{p_I\otimes p_I}I\otimes I$$ is an equivalence. Consider the following short exact sequence:
	$$0\xrightarrow{} J \xrightarrow{j} C \xrightarrow{p_I} I \to 0.$$
	Tensoring with $C$ we get the following short exact sequence.
	$$\begin{tikzcd}
		J\otimes C \arrow[r, "j\otimes 1"] & C\otimes C \arrow[r, "p_I\otimes 1"] \arrow[rd, "p\otimes 1"', "\sim"] & I\otimes C \arrow[r] \arrow[d, "i\otimes 1"] & 0 \\
		&                                                             & \mb{1}\otimes C                              &  
	\end{tikzcd}$$ 
	Thus $p_I\otimes 1$ is both monomorphic and epimorphic, and hence $p_I\otimes 1$ is an equivalence and $j\otimes 1=0$. It remains to show that $I\otimes C\xrightarrow{1\otimes p_I}I\otimes I$ is an equivalence. 
	Consider the following commutative diagram:
	$$\begin{tikzcd}
		J \otimes C \arrow[r,"1 \otimes p_I"] \arrow[d,"j \otimes 1"'] & J \otimes I \arrow[d,"j \otimes 1"] \\
		C \otimes C \arrow[r,"1 \otimes p_I"'] & C \otimes I
	\end{tikzcd}
	$$
	By previous step, $j \otimes 1_C = 0$, so the lower-left path is $0$. Hence the upper-right path
	$$(j \otimes 1_I) \circ (1_J \otimes p_I) \;=\; 0.$$
	Since $1_J \otimes p$ is an epimorphism, the right vertical morphism must be zero:
	$j \otimes 1_I \;=\; 0.$
	Apply the functor $- \otimes I$ to the short exact sequence
	$$0 \to J \xrightarrow{j} C \xrightarrow{p_I} I \to 0,$$
	to obtain
	$$J \otimes I \xrightarrow{j \otimes 1} C \otimes I \xrightarrow{p_I \otimes 1} I \otimes I \to 0.$$
	Since we proved $j \otimes 1_I = 0$, this forces $p \otimes 1_I$ to be an isomorphism:
	$$C \otimes I \xrightarrow{\ \sim\ } I \otimes I.$$
\end{proof}
We now pass from abelian $1$-categories to prestable $\infty$-categories.  
The key input is the following theorem on higher almost algebra, which extends the main result of \cite{hebestreit2024note} to the general prestable setting and provides a bridge between prestable smashing frames, abelian smashing frames, and idempotent ideals in the heart.

\begin{thm}[{\cite[Theorem 5.17]{hattt}}]\label{almostalg}
Let $\mcc\in\calg(\prlad)$ such that \mcc is a left complete\footnote{It is equivalent to the Postnikov-complete.} prestable \infcat.	Let $R \in \calg(\mcc)$ be a  \ein-algebra. 
Then the functor
$$
\calg_R^{\op{cn\text{-}idem}} \longrightarrow\cidem(\idl(\mcc^\heartsuit))= \left\{I \subset \pi_0 R \mid I^2=I\right\}, \quad \varphi \longmapsto \operatorname{Ker}(\pi_0 \varphi)
$$
is an equivalence of categories, where  the target is a poset via the inclusion ordering. 

The inverse image of some $I \subset \pi_0(R)$ can be described more directly as $R / I^{\infty}$, where
$$
I^{\infty}=\lim _{n \in \mathbb{N}^{\mathrm{op}}} J_I^{\otimes_R n}
$$
with $J_I \rightarrow R$ the fibre of the canonical map $R \rightarrow \pi_0(R) / I$. Furthermore, this inverse system stabilises on $\pi_i$ for $n>i+1$. 
Furthermore, the image of the fully faithful restriction functor $$\operatorname{Mod}_{R / I^{\infty}}(\mcc)\rightarrow \operatorname{Mod}_R(\mcc)$$ consists exactly of those modules whose homotopy groups are killed by $I$.
\end{thm}

Combining the abelian identification above with the higher almost algebra classification, we obtain the following comparison theorem for coidempotent frames.

\begin{thm}\label{smashidlidemidliden}
	Let $\mcc\in\calg(\prlad)$ such that \mcc is a left complete prestable \infcat.  Then all arrows in the following natural diagram 
	 are equivalences. $$\begin{tikzcd}
		\cidem(\mcc) \arrow[d, "\sim", "\im_{\mcc}"'] \arrow[r,"\sim"',"\pi_0"] & \cidem(\mcc^\heartsuit) \arrow[d, "\sim","\im_{\mcc^\heartsuit}"'] \\
		\cidem(\idl(\mcc)) \arrow[r, "\sim"']     & \cidem(\idl(\mcc^\heartsuit))    
	\end{tikzcd}$$
\end{thm}
\begin{proof}
	The right vertical equivalence follows from \cref{idemidlabel}. The composite $\cidem(\mcc)\to \cidem(\idl(\mcc^\heartsuit))$ is an equivalence by \cref{almostalg}. Therefore, it suffices to show the bottom arrow is an equivalence.
	
	In fact, we can prove a stronger result that $$\idl(\mcc)\xrightarrow{\im\circ\pi_0}\idl(\mcc^\heartsuit)$$ is an equivalence of frames, so their coidempotent frames are equivalent too, as desired. Now consider the following diagram:
	$$\begin{tikzcd}
		\idl(\mcc) \arrow[d, "\cofib"] \arrow[r, "\im\circ\pi_0"] & \idl(\mcc^\heartsuit) \arrow[d, "\coker"] \\
		{(\mcc_{\mb{1}/})^{\op{disc},\pi_0\text{-epi}}} \arrow[r, "\pi_0"]        & (\mcc_{\pi_0\mb{1}/}^{\heartsuit}  )^{\pi_0\text{-epi}}         
	\end{tikzcd}$$
	where $(\mcc_{\mb{1}/})^{\op{disc},\pi_0\text{-epi}}$ consists of those $\pi_0$-epimorphic maps $\mb{1}\to Y$ such that $Y$ is discrete.
	 By \cite[Remark C.2.3.4]{sag}, a map $f:X\to \mb{1}$ is a monomorphism if and only if $\cofib(f)$ is discrete, so the left vertical arrow is an equivalence. The bottom and right vertical arrows are clearly equivalences. That completes the proof.
\end{proof}
\begin{rem}[Recollement on the heart]
	Suppose that \(\mcc\in\calg(\prlad)\) is left complete prestable, and let
	$
	\mci\hookrightarrow\mcc\rightarrow\mcc/\mci
	$
	be a smashing localization sequence. Let
	$
	I\subset\pi_0\mb{1}
	$
	be the idempotent ideal corresponding to \(\mci\) under
	\cref{smashidlidemidliden}. Then
	$
	\mcc/\mci
	\simeq
	\modu_{\mb{1}/I^\infty}(\mcc),
	$
	and the recollement of \cref{recollementfrms} identifies with the
	decomposition
	\[
	\cidem(\mci)
	\simeq
	\{J\in\cidem(\idl(\mcc^\heartsuit))\mid J\subset I\},
	\]
	\[
	\cidem(\mcc/\mci)
	\simeq
	\{J\in\cidem(\idl(\mcc^\heartsuit))\mid J\supset I\}.
	\]
\end{rem}
The comparison theorem also gives a concrete description of the heart of the corresponding almost module \infcat.
\begin{cor}
	Let $\mcc\in\calg(\prlad)$ such that \mcc is a left complete prestable \infcat.	Let $R \in \calg(\mcc)$ and $I\subset\pi_0R$ be an idempotent ideal. Then we have $$\pi_0I^\infty\simeq I \otimes_{\pi_0R}I.$$
In particular, the  smashing ideal $$\amodu_{(R,I)}(\mcc)^\heartsuit\xhookrightarrow{i\otimes \mcs_{\leq 0}} \modu_R(\mcc)^\heartsuit \simeq \modu_{\pi_0R}(\mcc^\heartsuit)$$
can be identified with those $\pi_0R$-modules $M$ such that the natural map $I \otimes_{\pi_0R}I\otimes_{\pi_0R}M\to M$ is an equivalence.
\end{cor}

\begin{proof}
	Note that $\pi_0I^\infty\to\pi_0R$ is a coidempotent object in $\modu_{\pi_0R}(\mcc^\heartsuit)$ corresponding to the idempotent ideal $I$. It follows that $$\pi_0I^\infty\simeq I \otimes_{\pi_0R}I$$ by combining \cref{idemidlabel} and \cref{smashidlidemidliden}.
\end{proof}
\begin{rem}
	Note that the inclusion $i\otimes \mcs_{\leq 0}$ is not the restriction of $i$, because $i:\amodu_{(R,I)}(\mcc)\hookrightarrow\modu_R(\mcc)$ is not necessarily left exact. 
\end{rem}
To compare the atomic parts of the two smashing frames, we need a heart-level notion of those idempotent ideals which come from atomic smashing ideals.  This leads to the following definition.
\begin{de}\label{defpierce}
	Let $\mca\in\calg(\prladone)$ be a presentably symmetric monoidal   additive $1$-category such that \mca is abelian. We say that an idempotent ideal $I\in\cidem(\idl(\mca))$ is a \textbf{Pierce ideal} if its associated smashing ideal of \mca from \cref{idemidlabel} is an atomic smashing ideal. We denote the subposet of Pierce ideals by $\idl_p(\mca)$, which is equivalent to the atomic smashing frame $ \cidem_f(\mca)$ by definition.
\end{de}
We also recall the following notion of projectivity for $n$-categories, which will be used to compare compact projective objects in a prestable category with those in its heart.
\begin{de}
	Let $\mcc$ be an $n$-category. We say that an object $P\in \mcc$ is \textbf{$n$-projective} if the Yoneda functor $$h_P:\mcc\to \mcs_{\leq n-1}$$ preserves small $n$-sifted colimits.  We refer the reader to \cite[\textsection3.1]{stefanich2023classification} or \cite[Appendix A]{mao2021revisiting} for a detailed discussion.
\end{de}
The following theorem shows that, for projectively rigid $\spgeq$-algebras, atomicity of a smashing ideal is detected on the heart.
\begin{thm}\label{comparisontcinc}
	Let $\mcc\in\calg(\prlad)$ be a projectively rigid \spgeq-algebra. Given an idempotent ideal $I\subset\pi_0\mb{1}$. Then the associated smashing ideal $\amodu_{(\mb{1},I)}(\mcc)\hookrightarrow\mcc$ is atomic if and only if  $I$ is Pierce. In particular, we obtain the following equivalences of inclusions.
	$$\begin{tikzcd}
		\cidem_f(\mcc) \arrow[d, "\sim"] \arrow[r, hook]           & \cidem(\mcc) \arrow[d, "\pi_0"', "\sim"]           \\
		\cidem_f(\mcc^\heartsuit) \arrow[r, hook] \arrow[d, "\sim"] & \cidem(\mcc^\heartsuit) \arrow[d, "\im"', "\sim"] \\
		\idl_p(\mcc^\heartsuit) \arrow[r, hook]                     & \cidem(\idl(\mcc^\heartsuit))            
	\end{tikzcd}$$
\end{thm}
\begin{proof}
	The ``if'' direction is obvious. 
	For the ``only if'' direction, assume that 
	$$\amodu_{(\mb{1},I)}(\mcc)^\heartsuit \hookrightarrow \mcc^\heartsuit$$ 
	is an atomic smashing ideal of $\mcc^\heartsuit$. We wish to show that 
	$$\amodu_{(\mb{1},I)}(\mcc) \hookrightarrow \mcc$$ 
	is an atomic smashing ideal of $\mcc$.
	
	By \cref{projrigidtc}, it suffices to show that $\amodu_{(\mb{1},I)}(\mcc)$ is compact projectively generated. 
	By \cref{dualaddab6}, $\mcc$ is separated Grothendieck, thus by \cite[Theorem C.2.1.6]{sag} it suffices to show that for any $X\in \amodu_{(\mb{1},I)}(\mcc)$ there exists a collection of compact projective objects $\{P_\alpha\}\subset \amodu_{(\mb{1},I)}(\mcc)^{\op{cproj}}$ and a $\pi_0$-epimorphism
	$$\bigoplus_\alpha P_\alpha \to X.$$
	
	Consider the following diagram:
	$$\begin{tikzcd}
		{\amodu_{(\mb{1},I)}(\mcc)} \arrow[d, "\pi_0^I"] \arrow[r, "i", hook] 
		& \mcc \arrow[d, "\pi_0"] \arrow[r, "f"] 
		& \modu_{\mb{1}/I^\infty}(\mcc) \arrow[d, "\pi_0'"] \\
		{\amodu_{(\mb{1},I)}(\mcc)^\heartsuit} \arrow[r, "i_{\leq0}", hook]   
		& \mcc^\heartsuit \arrow[r, "f_{\leq0}"] 
		& \modu_{\pi_0\mb{1}/I}(\mcc^\heartsuit)
	\end{tikzcd}$$
	where $i_{\leq 0}=i \otimes \mcs_{\leq 0}$ and $f_{\leq 0}=f \otimes \mcs_{\leq 0}$.
	By assumption, there exists a collection of compact $1$-projective objects 
	$\{P^I_\alpha\}\subset (\amodu_{(\mb{1},I)}(\mcc)^\heartsuit)^{1\text{-}\op{cproj}}$ 
	and a $\pi_0$-epimorphism
	$$P^I=\bigoplus_\alpha P^I_\alpha \to \pi_0 X.$$
	By \cite[Proposition 2.4.8]{stefanich2023classification}, the functor 
	$\pi_0:\mcc^{\op{cproj}} \to (\mcc^\heartsuit)^{1\text{-}\op{cproj}}$ 
	induces an equivalence on homotopy categories. Therefore there exists a $\pi_0$-epimorphism 
	$\phi:P=\bigoplus_\alpha P_\alpha\to i(X)$ such that for each $\alpha$ we have 
	$P_\alpha\in \mcc^{\op{cproj}}$ and $\pi_0P_\alpha\simeq i_{\leq0}P^I_\alpha$, as illustrated in the following diagram:
	$$\begin{tikzcd}
		& P \arrow[d, two heads] \arrow[rr,"\phi", two heads, dashed] \arrow[ld, two heads] 
		& & i(X) \arrow[d, two heads] \\
		\pi_0P \arrow[r, "\sim"] 
		& i_{\leq0}P^I \arrow[r, two heads] 
		& i_{\leq0}\pi_0^IX \arrow[r, "\sim"] 
		& \pi_0i(X)
	\end{tikzcd}$$
	By construction, $\pi_0'f(P)=0$, and hence $f(P)=0$ since $P$ is projective. 
	Since by \cref{connlocandclosub} we have $\ker(f)\simeq \amodu_{(\mb{1},I)}(\mcc)$, it follows that $\phi$ lies in $\amodu_{(\mb{1},I)}(\mcc)$, as desired.
\end{proof}
\begin{rem}\label{1addker}
Note that the equation $\amodu_{(\mb{1},I)}(\mcc)= \ker(f)$ in the prestable level does not hold  in the heart level, i.e.,   $$\amodu_{(R,I)}(\mcc)^\heartsuit \subsetneq\ker(f_{\leq0})$$ is a proper inclusion in general (cf. \cref{remshortex1add}). For example, let $(R,\mathfrak{m})$ be a non-discrete valuation ring. Then $\mathfrak{m}$ itself lies in $\ker(f_{\leq0})$ but not in $\amodu_{(R,\mathfrak{m})}(\op{Ab})$, because the following sequence in $\modu_R(\op{Ab})$ is not left exact
$$\mathfrak{m}\otimes_R \mathfrak{m}\to R\to R/\mathfrak{m}\to0.$$
\end{rem}
As an immediate consequence, the telescope conjecture itself is invariant under passage to the heart in the projectively rigid setting.
\begin{cor}\label{tcaddandheart}
Let $\mcc\in \calg(\prlad)$ be a  projectively rigid \spgeq-algebra. The telescope conjecture holds for $\mcc$ if and only if it holds for its heart $\mcc^\heartsuit$.
\end{cor}

\begin{qu}
	Does \cref{comparisontcinc} work for any rigid (but not necessarily projectively rigid) \spgeq-algebra? 
\end{qu}

\subsection{Over a connective \einfring}
In this subsection, we establish an algebraic criterion for the telescope conjecture to hold within the $\infty$-category of connective modules over a connective \einfring. This criterion relates the condition of compact projective generation to Pierce ideals. 

The first step is to understand which idempotent ideals of the heart can arise from compact atomic smashing ideals. We begin with a general finiteness statement for compact atomic smashing ideals.
\begin{lem}\label{compatsmidl}
	Let \mcc be a projectively rigid $\spgeq$-algebra. Then its atomic smashing frame $\cidem_f(\mcc)$  is coherent, and there is an inclusion $$\cidem_f(\mcc)^\omega \hookrightarrow\idl(\mcc^\heartsuit)^\omega\cap\cidem(\idl(\mcc^\heartsuit)).$$
 In other words, for any compact atomic smashing ideal, its associated idempotent ideal is compact in $\idl(\mcc^\heartsuit)$.
\end{lem}
\begin{proof}
	Since $\mcc$ is projectively rigid, \mcc 
	lies in $\calg(\prl_\omega)$. Therefore  by \cref{atsmfrmcoh}, $\cidem_f(\mcc)$  is a coherent frame.
	
	Now let $(i:\mci\subset \mcc)\in \cidem_f(\mcc)^\omega$ be a compact atomic smashing ideal. Then according to \cref{atomicmodu}, \mci can be generated by a single compact projective element $x\in \mci^{\op{cproj}}\subset \mcc^{\op{cproj}}$ as a \mcc-module. Consider the (internal left adjoint) $\mcc$-module map $\mcc\xrightarrow{x_{\mci}}\mci$ induced by $x$. Then the right adjoint $x_\mci^R\circ i^R:\mcc\to \mcc$ to the composition $i\circ x_{\mci}$ can be identified with $$\mcc\xrightarrow{x^\vee\otimes-}\mcc.$$
	Let $I\subset \pi_0\mb{1}$ denote the corresponding idempotent ideal of the smashing ideal $\mci$ through \cref{almostalg}. Then $x_\mci^R (I^\infty) \simeq x_\mci^R \circ i^R(\mb{1})   \simeq x^\vee $ is compact in \mcc and hence $\pi_0 x_\mci^R (I^\infty)$ is compact in $\mcc^\heartsuit$. Since $x_\mci^R$ is left exact and conservative, $x_\mci^R$ reflects effective epimorphisms (i.e. $\pi_0$-surjections). However, $x_\mci^R x_\mci  x_\mci^R(I^\infty)\to x_\mci^R(I^\infty)$ admits a retract by the adjunction property and hence an effective epimorphism in \mcc. That implies the counit map $x_\mci  x_\mci^R(I^\infty)\to I^\infty$ is effective epimorphic in \mci. Consequently, $$I \simeq \im(\pi_0I^\infty \to \pi_0\mb{1})\simeq \im(\pi_0x_\mci  x_\mci^R(I^\infty)\to\pi_0I^\infty \to \pi_0\mb{1})$$ is a finitely generated ideal of $\pi_0\mb{1}$, i.e. $I\in \idl(\mcc^\heartsuit)^\omega\cap\cidem(\idl(\mcc^\heartsuit))$.
\end{proof} 

\begin{rem}\label{ex:mackey-non-pierce}
	Note that the inclusion $$\cidem_f(\mcc)^\omega\hookrightarrow\idl(\mcc^\heartsuit)^\omega\cap\cidem(\idl(\mcc^\heartsuit))$$ in \cref{compatsmidl} is  generally not an equivalence. In \cref{counterex2}, we provide an explicit counterexample within the $\infty$-category of connective $G$-equivariant spectra for a finite group $G$.
\end{rem}
However, when working over a connective $\mathbb{E}_\infty$-ring, no additional equivariant or additive complications occur: the inclusion in \cref{compatsmidl} becomes an equivalence, and compact projective generation is controlled entirely by classical idempotents in $\pi_0R$.
\begin{thm}\label{main2}
	Let $R\in\calg(\spgeq)$ be a connective \einfring and $I\subset\pi_0R$ be an idempotent ideal. Then $\amodu_{(R,I)}(\spgeq)$ is compact projectively generated if and only if $I$ can be generated by a family of idempotent elements $\{e_\alpha\}$. 
\end{thm}
\begin{proof}
For the ``if'' direction, the equivalence of frames
\[
\cidem(\modu_R(\spgeq))\simeq \cidem(\idl(\pi_0R))
\]
allows us to reduce to the case where $I$ is generated by a single idempotent element $e\in I$. We claim that $$\amodu_{(R,I)}(\spgeq)\simeq \modu_{R[e^{-1}]}(\spgeq)$$ and then we are done.
	To prove the claim, we observe that the natural map $R\to R[(1-e)^{-1}]\times R[e^{-1}]$ is an equivalence of \einfrings, because it is an \'{e}tale map and one only needs to check at the $\pi_0$-level. Therefore by \cite[Lemma D.3.5.5]{sag} the base-change functor induces the following equivalence $$\modu_R(\spgeq)\xrightarrow{\sim}\modu_{R[(1-e)^{-1}]}(\spgeq)\times \modu_{R[e^{-1}]}(\spgeq).$$
	
	For the ``only if'' direction, suppose  that $\amodu_{(R,I)}(\spgeq)$ is compact projectively generated, i.e. it is an atomic smashing ideal. We first note that $$\amodu_{(R,I)}(\spgeq)\simeq \bigvee_\alpha \mci_\alpha$$ can be written as a filtered join of compact atomic smashing ideals by \cref{atsmfrmcoh}, because $\modu_R(\spgeq)\in \calg(\prl_\omega)$.
	However, by \cref{compatsmidl}, the correspondent idempotent ideal $I_\alpha$ of a compact atomic smashing ideal $\mci_\alpha$ is finitely generated, thus by \cref{fgidemidl}, $I_\alpha$ is generated by a single idempotent element. Consequently, $I$ is generated by a family of idempotent elements $\{e_\alpha\}$.
\end{proof}
\begin{rem}
	In other words, when \(\mcc=\modu_R(\spgeq)\) for a connective
	\(\mathbb{E}_\infty\)-ring \(R\), the inclusion
	\[
	\op{Z}\bigl(\cidem(\mcc)\bigr)
	\hookrightarrow
	\cidem_f(\mcc)
	\]
	of \cref{splitclopen} is an equivalence. However, this need not hold for
	an arbitrary projectively rigid \(\spgeq\)-algebra; see, for example,
	\cref{counterex3}.
\end{rem}
Specializing to ordinary commutative rings, we obtain the following result, which justifies the term ``Pierce ideal''.
\begin{cor}
	Let $A$ be a commutative ring.
	An idempotent ideal $I$ of $A$ is a Pierce ideal in the sense of \cref{defpierce} if and only if $I$  is Pierce in the classical sense, i.e., it is generated by a family of idempotent elements. 
	
	In particular, $\modu_{A}(\op{Ab})$ satisfies the telescope conjecture if and only if every idempotent ideal of $A$ is Pierce.
\end{cor}
\begin{proof}
	It follows by applying \cref{comparisontcinc} and \cref{main2} to the \syminfcat $\mcd(A)_{\geq0}$.
\end{proof}
The preceding corollary also connects the atomic smashing frame with the classical Pierce spectrum of a ring.
\begin{rem}[Pierce spectrum]
	Let $A$ be a commutative ring. We denote the poset of Pierce ideals by $\idl_p(A)$. The poset $\idl_p(A)$ forms a subframe of $\cidem(\idl(A))$. Moreover, $\idl_p(A)$ is a coherent frame whose subposet $\idl_p(A)^\omega$ spanned by compact elements consists exactly of those finitely generated idempotent ideals (they are automatically principle), i.e. $$\idl_p(A)^\omega\simeq\idl(A)^\omega\cap \cidem(\idl(A)). $$ 
	The poset $\idl_p(A)^\omega$ is a Boolean algebra, given by the complement assignment $(e)^c=(1-e)$. Therefore $\spec(\idl_p(A))$ is a stone space, which we call the Pierce spectrum of $A$. 
	
	In fact, we have a natural identification of posets
	$$\idl_p(A)\simeq \{\text{Pierce open subsets in }\spec(A)\};$$ see \cref{pierceidlopen}. We refer the reader to \cite[\textsection V.2.]{Johnstone82} and \cite[\textsection8]{borceux2024galois}  for more details about the Pierce spectrum of a ring. 
\end{rem}

\begin{rem}
	By \cref{main2}, the telescope conjecture for $\modu_R(\spgeq)$ over a connective \ein-ring $R$ is an algebraic condition, which only depends on $\pi_0R$ and is much  simpler than the usual telescope conjecture in the stable case. 
\end{rem}
We record some immediate consequences and examples of this algebraic criterion.
\begin{ex}\label{noethertelescope}
	Let $R\in\calg(\spgeq)$ be a connective \einfring such that $\pi_0R$ is Noetherian. Then any idempotent ideal is finitely generated and hence a Pierce ideal. Therefore $\modu_R(\spgeq)$ satisfies the telescope conjecture. 

	In particular, the \infcat  $\spgeq$ of connective spectra satisfies the telescope conjecture. Furthermore, by the computation that $$\cidem(\idl(\mathbb{Z}))\simeq\{0,\mathbb{Z}\},$$ we see that \spgeq is a smashing field.
\end{ex}
\begin{ex}
	Let $R\in\calg(\spgeq)$ be a connective \einfring such that $\pi_0R$ is absolutely flat. Then  $\modu_R(\spgeq)$ satisfies the telescope conjecture by \cref{abflatchar}.
\end{ex}
\begin{ex}[Counterexamples]\label{counterextcadd}
	We list some non-Pierce idempotent ideals $I\subset A$, which make the telescope conjecture fail for $\mcd(A)_{\geq0}$.
	\enu{
		\item 
		Let $A = C(\mathbb{R})$ be the ring of all real-valued continuous functions on~$\mathbb{R}$ and  $I \subset A$ be the ideal consisting of all functions with compact support. Then $I$ is non-Pierce idempotent by \cref{counterpiercepure}.
		\item 
		Let $(A,\mathfrak{m})$ be a non-discrete valuation ring.
		Let
		$
		I := \mathfrak{m}
		$
		be the maximal ideal of $A$.
		Then $I$ is non-Pierce idempotent by \cref{counterpiercepure}.
		\item Let $k$ be a field and $p$ be a prime. Consider the ring $A = k[x^{1/p^\infty}]$, which is formally defined as the union $\colim_{n} k[x^{1/p^n}]$. 
		Let $I = (x^{1/p^\infty})$ be the ideal generated by the set $\{x^{1/p^n} \mid n \geq 0\}$. Then $I$ is non-Pierce idempotent by \cref{counterpiercepure}.
		
	}

\end{ex}

\section{Serre Smashing Frames and the Flat Telescope Conjecture}\label{sec7}
Inspired by the observation in \cref{pierceincpure} that for ordinary rings, pure ideals strictly interpolate between Pierce ideals and idempotent ideals, 
we investigate flat objects in the dualizable setting and study the properties of pure ideals and Serre smashing ideals. This framework leads to a weaker variant of the telescope conjecture, which we call the \emph{flat telescope conjecture}.

\subsection{Dualizable flatness}
Classically, a left module $M$ over an associative ring $A$ is flat if the relative tensor product functor 
\[ \rmodu_A(\op{Ab})\xrightarrow{(-)\otimes_A M} \op{Ab} \] 
is left exact. This characterization extends naturally to the setting of dualizable categories. In this subsection, we study flat objects in a dualizable $\mathcal{V}$-module and establish some basic properties of flat objects.

The following definition packages flatness in terms of the dual of an object.
In this form, it applies uniformly to additive and prestable examples.

\begin{de}
	Let \(\mcv\in\calg(\prl)\), and let \(\mcm\) be a dualizable
	\(\mcv\)-module. Given an object \(x\in\mcm\), equivalently a
	\(\mcv\)-module map
	\[
	\mcv \xrightarrow{x} \mcm,
	\]
	we say that \(x\) is \textbf{\(\mcv\)-flat}, or simply
	\textbf{flat} if the base \(\mcv\) is clear, if the dual map
	\[
	\mcm^\vee \xrightarrow{x^\vee} \mcv
	\]
	is left exact.
\end{de}

This definition recovers the expected behavior in the two extremal cases
which will be used repeatedly below.

\begin{rem}
	Let \(\mcv\in\calg(\prl)\).
	\enu{
		\item The unit \(\mathbf{1}\in\mcv\) is always \(\mcv\)-flat.
		
		\item Suppose that \(\mcv\) is stable. Then, for every dualizable
		\(\mcv\)-module \(\mcm\), every object \(x\in\mcm\) is
		\(\mcv\)-flat.
	}
\end{rem}

\begin{rem}
	This notion of flatness is relative to the chosen base \(\mcv\). In
	particular, it distinguishes categorical flatness from $\otimes$-flatness.
let \(\mcc\in\calg(\prlad)\) be such that its underlying
\(\infty\)-category is dualizable additive. Then for  an $x\in \mcc$,  categorical flatness is
	\(\spgeq\)-flatness, while $\otimes$-flatness is \(\mcc\)-flatness (cf. \cite[Remark 4.5]{levy-liang-sosnilo:dualizable}).
\end{rem}
The definition is also functorial with respect to the internal left adjoints.  This compatibility is essential for relating flatness to atomic objects and hence to the telescope-type questions considered below.
\begin{prop}\label{preserverefleflats}
	Let $\mcv\in\calg(\prl)$. 
	\enu{ \item Let $f:\mcm \to \mcn\in\prvdbl$ be an internal left adjoint between dualizable \mcv-modules. Then $f$ preserves \mcv-flat objects.
		\item Let $f:\mcm \hookrightarrow \mcn\in\prvdbl$ be a fully faithful internal left adjoint between dualizable \mcv-modules. Then $f$ reflects \mcv-flat objects.
		\item Let $\mcm$ be a dualizable \mcv-module. If $x\in \mcm$ is \mcv-atomic, then $x$ is \mcv-flat.
	}
\end{prop}

\begin{proof}
	(1) Let $x:\mcv\to\mcm$ be a \mcv-flat object. Since $f$ is an internal left adjoint, the dual map $f^\vee:\mcn^\vee\to\mcm^\vee$ is an internal right adjoint. Consequently, the composition $x^\vee f^\vee$ is left exact.
	
	(2) By assumption, we deduce that $f^\vee:\mcn^\vee\to\mcm^\vee$ is a left exact localization. Therefore, if $f^\vee x^\vee:\mcn^\vee\to\mcv$ is left exact, then so is $x^\vee$.
	
	(3) This follows immediately by applying (1) to $x:\mcv\to \mcc$, since $\mb{1}\in \mcv$ is  flat.
\end{proof}

In particular, flatness is automatic for the objects which generate atomic smashing ideals.  We next record the corresponding relative terminology for algebras and morphisms.
\begin{de}\label{flatover}
Let $\mcv\in\calg(\prl)$.
\enu{
\item 	Let $R\in\alg(\mcv)$ be an $\eone$-algebra and $M\in\lmodu_R(\mcv)$ be a left module. We say $M$ is \mcv-flat over $R$ if $M\in \lmodu_R(\mcv)$ is \mcv-flat.

\item 	Let $R\to S\in\calg(\mcv)$ be a map of \ein-algebras. We say $R\to S$ is a flat morphism if $S$ is \mcv-flat over $R$, or equivalently, if the base change functor $\modu_R(\mcv)\to \modu_S(\mcv)$ is left exact.
}

\end{de}
\begin{rem}
	Note that a left $R$-module $M\in\lmodu_R(\mcv)$ is \mcv-flat over $R$ if and only if the relative tensor product $$\rmodu_R(\mcv)\xrightarrow{-\otimes_R M}\mcv$$ is left exact, because $\lmodu_R(\mcv)^\vee\simeq \rmodu_R(\mcv)$ by \cite[Remark 4.8.4.8]{ha}.
\end{rem}
We shall use the following elementary closure properties without further comment in later.
\begin{prop}\label{flatmapsclosure}
	Let $\mcv\in\calg(\prl)$. Then:
	\enu{\item The  flat morphisms in $\calg(\mcv)$ are closed under base change.
	\item Suppose that \mcv satisfies AB5 (i.e. filtered colimits commute with finite limits in \mcv). Then for any $R\in\alg(\mcv)$, flat left $R$-modules are closed under filtered colimits in $\lmodu_R(\mcv)$. In particular, flat morphisms in $\calg(\mcv)$ are closed under filtered colimits.
	\item If \mcv is semiadditive,  then  flat morphisms in $\calg(\mcv)$ are closed finite products.
	} 
\end{prop}
\begin{proof}
	(1) It follows by observing that given $B'\simeq B\otimes_A A'$ in $\calg(\mcv)$, the following diagram is horizontally right adjointable.
	$$\begin{tikzcd}
		\modu_A(\mcv) \arrow[d] \arrow[r] & \modu_{A'}(\mcv) \arrow[d] \\
		\modu_B(\mcv) \arrow[r]           & \modu_{B'}(\mcv)          
	\end{tikzcd}$$
	(2) Given a filtered colimit of flat left $R$-modules $M\simeq\colimit_\alpha M_\alpha$. By \cref{flatover}, it suffices to observe that $(-)\otimes_R M\simeq\colimit_\alpha ((-)\otimes_RM_\alpha)$ is left exact.
	
	Now given a filtered colimit of flat morphisms $A\to B \simeq \colimit_\alpha (A_\alpha\to B_\alpha)$. Denote $B_\alpha'=B_\alpha\otimes_{A_\alpha}A$; we find that $A\to B\simeq \colimit_\alpha (A\to B_\alpha')$. Thus without loss of generality we can assume that $\{A_\alpha\}$ is constant. Then the result follows immediately from the argument above for flat modules.
	\\ 
	(3) It follows by observing that $$\modu_{A\times B}(\mcv)\simeq \modu_{A}(\mcv)\times\modu_{ B}(\mcv)$$ when \mcv is semiadditive.
\end{proof}
Finally, flatness is compatible with truncation. 
\begin{prop}\label{truncatedflat}
	Let $\mcv\in\calg(\prl)$  and $\mcm\in\prvdbl$ be a dualizable \mcv-module. If $x\in\mcm$ is a \mcv-flat object, then for any $n\geq -2$, $\tau_{\leq n}x\in\mcm_{\leq n}$ is a $\mcv_{\leq n}$-flat object.
\end{prop}
\begin{proof}
Since the $\mcv_{\leq n}$-dual map of $$\mcv_{\leq n}\simeq \mcs_{\leq n} \otimes \mcv \xrightarrow{ \tau_{\leq n}x\simeq \mb{1}\otimes x}\mcs_{\leq n}  \otimes \mcm$$ can be identified with 
$$\mcs_{\leq n} \otimes \mcm^\vee \xrightarrow{\mb{1}\otimes x^\vee}\mcs_{\leq n} \otimes \mcv\simeq\mcv_{\leq n},$$	
it suffices to show that the functor $\prl\xrightarrow{-\otimes \mcs_{\leq n}}\prl_{n+1}$ preserves left exact morphisms. Indeed, given a left exact morphism $F:\mcn\to \mck \in \prl$; then $\tau_{\leq n}\mcn\xrightarrow{F\otimes \mcs_{\leq n}} \tau_{\leq n}\mck$ is left exact, because 
the following diagram is vertically right adjointable  by \cite[Proposition 5.5.6.16]{htt}.
$$\begin{tikzcd}
	\mcn \arrow[r, "F"] \arrow[d, shift right]                                   & \mck \arrow[d, shift right]                            \\
	\tau_{\leq n}\mcn \arrow[r, "F\otimes\mcs_{\leq n}"] \arrow[u, dashed, hook, shift right] & \tau_{\leq n}\mck \arrow[u, dashed, hook, shift right]
\end{tikzcd}$$
\end{proof}
\subsection{Pure frames}
In this subsection, we investigate flat objects and pure ideals in an abelian category. We study the structural properties of the resulting pure frame, with a particular focus on establishing its spatiality and quasi-compactness. This frame provides a categorical framework for classifying and analyzing the global behavior of flat algebras within a given abelian category.

We now specialize the preceding notion of flatness to the setting of the additive heart. The resulting pure ideals form the algebraic counterpart of the Serre smashing ideals considered later in the prestable setting. The following generation result guarantees an ample supply of flat objects in the additive case.

\begin{thm}[{\cite{kanda2024module, levy-liang-sosnilo:dualizable}}]\label{flatgenerateone}
	Let $\mcm \in \prladone$ be a presentable additive $1$-category.
	\begin{enumerate}[label=(\arabic*)]
		\item $\mcm$ is a dualizable $\op{Ab}$-module if and only if $\mcm$ is a Grothendieck abelian category satisfying $\mathrm{AB4}^*$ and $\mathrm{AB6}$.
		\item If $\mcm \in \pradone^{\op{dbl}}$ is a dualizable additive $1$-category, then $\mcm$ is generated under small colimits by its $\omega_1$-compact $\op{Ab}$-flat objects.
	\end{enumerate}
\end{thm}
With this supply of flat objects in hand, we can define purity for ideals by asking the corresponding quotient algebra to remain flat over the unit.
\begin{de}
	Let $\mca\in\calg(\prladone)$ such that \mca is an abelian 1-category. An ideal $I \subset \mb{1}$ is said to be
	\textbf{pure} if 
	$\mb{1}/I$ is $\mca$-flat over $\mb{1}$. We denote the poset of pure ideals by $\idl_u(\mca).$
\end{de}
The first basic consequence is that pure ideals are automatically idempotent in the sense used throughout the paper.
\begin{prop}\label{pureisidemidl}
	Let $\mca\in\calg(\prladone)$ such that \mca is an abelian 1-category. Then any pure ideal $I\subset \mb{1}$ is an idempotent ideal.
\end{prop}
\begin{proof}
	Consider the short exact sequence in \mca:
\[
0 \to I \xrightarrow{i} \mb{1} \xrightarrow{q} \mb{1}/I \to 0.
\]
Let $\psi: I \otimes (\mb{1}/I) \xrightarrow{i\otimes \id} \mb{1}/I$ be the natural canonical morphism. Because $\mb{1}/I$ is \mca-flat over $\mb{1}$, we deduce that $\psi$ must be a monomorphism.
Applying the right exact functor $I \otimes -$ to the epimorphism $q: \mb{1} \twoheadrightarrow \mb{1}/I$, we obtain an epimorphism $1_I \otimes q: I \otimes \mb{1} \twoheadrightarrow I \otimes (\mb{1}/I)$. We can analyze the composition of these morphisms via the following commutative diagram:
\[
\begin{tikzcd}
	I \otimes \mb{1} \arrow[r, "1_I \otimes q"] \arrow[d, "i"', hook] & I \otimes (\mb{1}/I) \arrow[d, "\psi", hook] \\
	\mb{1} \arrow[r, "q"]                                       & \mb{1}/I                                    
\end{tikzcd}
\]
The composition $q \circ i$ is the zero morphism. This implies $\psi \circ (1_I \otimes q) = 0$. Since $1_I \otimes q$ is an epimorphism, it can be cancelled from the right, yielding $\psi = 0$. 
We have established that $\psi: I \otimes (\mb{1}/I) \to \mb{1}/I$ is both a monomorphism and the zero morphism. 
Hence, $I \otimes (\mb{1}/I) \simeq 0$.

Apply the right exact functor $I \otimes -$ to the short exact sequence $0 \to I \to \mb{1} \to \mb{1}/I \to 0$ to obtain the exact sequence:
\[
I \otimes I \xrightarrow{\phi} I \otimes \mb{1} \xrightarrow{1_I \otimes q} I \otimes (\mb{1}/I) \to 0.
\]
Substitute the canonical isomorphism $I \otimes \mb{1} \simeq I$ and the result from Step 1, $I \otimes (\mb{1}/I) \simeq 0$, into the sequence:
\[
I \otimes I \xrightarrow{\phi} I \to 0.
\]
So $\phi: I \otimes I \to I$ is an epimorphism and $I^2=I$.
\end{proof}
Rigidity allows the ambient categorical flatness condition to be tested after forgetting to the underlying abelian category.
\begin{lem}\label{rigidflatone}
	Let $\mca\in\calg(\prladone)$ be a rigid $\op{Ab}$-algebra and \mcm be a dualizable \mca-module. Then for an object $x\in\mcm$, it is \mca-flat if and only if it is $\op{Ab}$-flat.
\end{lem}
\begin{proof}
	It basically follows from the same argument as \cref{rigidflat}.
\end{proof}

The following criterion rewrites purity in the familiar module-theoretic form in \cref{pureidl}.
\begin{lem} \label{flatness_criterion}
	Let $\mca\in\calg(\prladone)$ be a rigid $\op{Ab}$-algebra.
	Let $I \hookrightarrow \mb{1}$ be an ideal. Then $I$ is pure if and only if for every object $M \in \mathcal{A}$, the natural morphism $I \otimes M \to M$ is a monomorphism.
\end{lem}
\begin{proof}
	$(\Rightarrow)$ Assume $\mb{1}/I$ is \mca-flat. For any object $M$,  by \cref{flatgenerateone} we can choose an epimorphism $P \twoheadrightarrow M$ where $P$ is $\op{Ab}$-flat. Note that $P$ is also \mca-flat by \cref{rigidflatone}. Let  $R$ be its kernel. This yields a short exact sequence $$0 \to R \to P \to M \to 0.$$ By the right exactness of the tensor product, applying $I \otimes -$ yields the following commutative diagram with exact rows:
	\[
	\begin{tikzcd}
		& I \otimes R \arrow[r] \arrow[d, "\alpha"] & I \otimes P \arrow[r] \arrow[d, "\beta"] & I \otimes M \arrow[r] \arrow[d, "\gamma"] & 0 \\
		0 \arrow[r] & R \arrow[r, hook]                         & P \arrow[r]                              & M \arrow[r]                               & 0
	\end{tikzcd}
	\]
	Because $P$ is \mca-flat, the functor $- \otimes P$ is exact. Applying it to the exact sequence $0 \to I \to \mb{1} \to \mb{1}/I \to 0$ shows that $\beta: I \otimes P \to P$ is a monomorphism, meaning $\ker(\beta) = 0$. 
	By the Snake Lemma, we obtain an exact sequence of kernels and cokernels:
	\[
	\ker(\beta) \to \ker(\gamma) \to \operatorname{coker}(\alpha) \xrightarrow{\Phi} \operatorname{coker}(\beta)
	\]
	Since $\ker(\beta) = 0$, we have an exact sequence $0 \to \ker(\gamma) \to \operatorname{coker}(\alpha) \xrightarrow{\Phi} \operatorname{coker}(\beta)$. Note that the morphism $\Phi$ can be  identified with the induced morphism $R \otimes \mb{1}/I \to P \otimes \mb{1}/I$. Because $\mb{1}/I$ is \mca-flat by hypothesis, the functor $- \otimes \mb{1}/I$ is exact, which implies that $\Phi$ is a monomorphism. Therefore, the kernel of $\Phi$ must be $0$. The exactness of the sequence then forces $\ker(\gamma) = 0$, concluding that $\gamma: I \otimes M \to M$ is a monomorphism.
	
	$(\Leftarrow)$ Assume $I \otimes M \to M$ is a monomorphism for all $M$. Let $0 \to M \to N \to L \to 0$ be an arbitrary short exact sequence. Tensoring with the exact sequence $0 \to I \to \mb{1} \to \mb{1}/I \to 0$ gives the following commutative diagram:
	\[
	\begin{tikzcd}
		& I \otimes M \arrow[r] \arrow[d, hook] & I \otimes N \arrow[r] \arrow[d, hook] & I \otimes L \arrow[r] \arrow[d, hook] & 0 \\
		0 \arrow[r] & M \arrow[r, hook]                     & N \arrow[r]                           & L \arrow[r]                           & 0
	\end{tikzcd}
	\]
	By assumption, all three vertical morphisms are monomorphisms. Applying the Snake Lemma, we obtain a short exact sequence of their cokernels:
	\[
	0 \to \operatorname{coker}(I \otimes M \to M) \to \operatorname{coker}(I \otimes N \to N) \to \operatorname{coker}(I \otimes L \to L) \to 0
	\]
	Since these cokernels are canonically isomorphic to $M \otimes \mb{1}/I$, $N \otimes \mb{1}/I$, and $L \otimes \mb{1}/I$ respectively, the functor $- \otimes \mb{1}/I$ preserves short exact sequences. Hence, $\mb{1}/I$ is \mca-flat over $\mb{1}$.
\end{proof}
These characterizations also show that pure ideals are closed under the frame operations.
\begin{thm}\label{thmpuresubfrm}
	Let $\mca\in\calg(\prladone)$ be a rigid $\op{Ab}$-algebra. Then 
	\enu{
\item For any pure ideal $I\subset \mb{1}$, we have $I\otimes I\simeq I$.
\item The inclusion $$\idl_u(\mca)\hookrightarrow\cidem(\idl(\mca))$$ forms a subframe.
		}
\end{thm}
\begin{proof}
	(1) 
By \cref{pureisidemidl}, $\phi: I \otimes I \to I$ is an epimorphism. 
	By \cref{flatness_criterion}, we see that $\phi$ is a monomorphism and hence an equivalence.\\
	(2)
		For a finite collection of pure ideals $\{I_i\}$, the tensor product $\mb{1}/(\Sigma_i I_i)\simeq \bigotimes_i \mb{1}/I_i$ is \mca-flat. Therefore the inclusion $$i:\idl_u(\mca)\subset \cidem(\idl(\mci))$$ is closed under finite joins. However, flat modules are closed under filtered colimits by \cref{flatmapsclosure}, we conclude that $i$ is closed under filtered joins and hence all joins.
	
	Note that the meet of two idempotent ideals $I_1,I_2\in\cidem(\idl(A))$ is given by $I_1\cdot I_2$. Now given two pure ideals $I,J$, we wish to show $I\cdot J$ is pure. 
	 Because $\mb{1}/I$ is flat,  \cref{flatness_criterion} implies that the morphism $I \otimes J \to J$ is a monomorphism. The ideal product $IJ$ is the image of the composite morphism $I \otimes J \to J \hookrightarrow \mb{1}$. Since this is a composition of two monomorphisms, the composite $I \otimes J \to \mb{1}$ is itself a monomorphism. In any abelian category, a monomorphism is isomorphic to its image, which gives us a canonical isomorphism $$IJ \simeq I \otimes J.$$
Now, let $M \in \mathcal{C}$ be an arbitrary object. We wish to show that the natural morphism $$IJ \otimes M \to M$$ is a monomorphism. We construct this morphism as a composition of several isomorphisms and monomorphisms:
\begin{enumerate}[label=(\roman*)]
	\item Tensoring the isomorphism $IJ \simeq I \otimes J$ with $M$ yields an isomorphism $IJ \otimes M \xrightarrow{\sim} (I \otimes J) \otimes M$.

	\item Because $\mb{1}/J$ is flat, Lemma \cref{flatness_criterion} implies that $J \otimes M \to M$ is a monomorphism. Let $JM \hookrightarrow M$ denote its image. Thus, we have an isomorphism $J \otimes M \xrightarrow{\sim} JM$. Tensoring this with $I$ yields an isomorphism $I \otimes (J \otimes M) \xrightarrow{\sim} I \otimes JM$.
	\item  Because $\mb{1}/I$ is flat, applying Lemma \cref{flatness_criterion} again shows that $I \otimes JM \to JM$ is a monomorphism.
	\item Finally, the inclusion $JM \hookrightarrow M$ is a monomorphism by definition.
\end{enumerate}
Concatenating these morphisms, the canonical map $IJ \otimes M \to M$ factors as follows:
\[
IJ \otimes M \xrightarrow{\sim} (I \otimes J) \otimes M \xrightarrow{\sim} I \otimes (J \otimes M) \xrightarrow{\sim} I \otimes JM \hookrightarrow JM \hookrightarrow M
\]
Because the composition of isomorphisms and monomorphisms is a monomorphism, $IJ \otimes M \to M$ is a monomorphism. By Lemma \cref{flatness_criterion}, this concludes that $\mb{1}/(IJ)$ is flat over $\mb{1}$.
\end{proof}
\begin{rem}
	The assumption in \cref{thmpuresubfrm}(2) that $\mathcal A$ is rigid cannot in general be replaced merely by requiring that $\mathcal A$ be Grothendieck abelian. 
	The issue is that, if
	\[
	0\to I\to \mb{1}\to \mb{1}/I\to 0
	\]
	is exact and $\mb{1}/I$ is flat, flatness of the quotient alone does not in general imply that the sequence remains exact after tensoring with an arbitrary object $M$.\footnote{See
		\url{https://mathoverflow.net/questions/57651/tame-abelian-tensor-categories}.} This is precisely the implication used in \cref{flatness_criterion} to prove that the product of two pure ideals is again pure.
	
	The argument does, however, remain valid under the weaker assumption that $\mathcal A$ is Grothendieck abelian and has enough \mca-flat objects. Under this hypothesis, the proof of \cref{flatness_criterion} goes through without using rigidity, and consequently $\idl_u(\mathcal A)$ is again a subframe of $\cidem(\idl(\mathcal A))$.
\end{rem}

We shall use the following notation for the frame of pure ideals in a module category.
\begin{de}
	Let \(\mca\in\calg(\prladone)\) be a rigid
	\(\op{Ab}\)-algebra, and let \(A\in\calg(\mca)\). We call
	$
	\idl_u(A)
	$
	 as the \textbf{pure frame} of \(A\), namely the subframe of $\idl(A)$
	consisting of pure ideals.
\end{de}
To compare purity with compact generation, we recall the appropriate notion in the 1-additive category.
\begin{de}
	Let $\mca\in\calg(\prladone)$ be an $\omega$-presentably symmetric monoidal additive $1$-category such that \mca is an abelian category. We say that an object $X\in\mca$ is \textbf{finitely generated} if there exists a compact object $C$ and an epimorphism $C\to X$.
\end{de}
\begin{rem}
	For such an $\mca$, it is automatically Grothendieck abelian by compact generation. Furthermore, by \cref{idlprefrm}, an ideal $I \subset \mb{1}$ is compact in $\idl(\mca)$ if and only if it is finitely generated.
\end{rem}
\begin{de}
	Let $\mca\in\calg(\prladone)$. We say that it is \textbf{$1$-projectively rigid} over $\op{Ab}$ if $\mca$ is atomically generated and rigid over $\op{Ab}$, or equivalently, if $\mca$ is compact-1-projectively generated and satisfies $\mca^{1\text{-}\op{cproj}}=\mca^d$.
\end{de}
Under this hypothesis, purity can be checked by the usual intersection criterion as \cref{pureidlproperties}, and in fact it suffices to test it on compact projective generators.
\begin{prop}\label{charpureidl}
	Let $\mca\in\calg(\prladone)$ be a 1-projectively rigid $\op{Ab}$-algebra and $A\in \calg(\mca)$. Suppose $I\subset A$ be an ideal. Then the following conditions are equivalent:
	\enu{
		\item $I$ is pure.
		\item For any $A$-module $M$ and  $A$-submodule $N\subset M$, we have $N\cap I M=I N$.
		\item For any compact 1-projective $A$-module $P$ and any finitely generated $A$-submodule $K\subset P$ we have $K\cap I P=I K$.
	}
\end{prop}
\begin{proof}
	$(2)\implies(3)$ is obvious. For $(1)\implies(2)$, tensoring $A/I$ with the inclusion $N\hookrightarrow M$, we get a map $$N/IN\to M/IM.$$
	It is still a monomorphism by assumption, so its kernel $N\cap I M/IN$ is zero, i.e. $N\cap I M=I N$.
	
For $(3)\implies(1)$, we wish to show $A/I$ is flat over $A$. By \cref{flatness_criterion}, it suffices to prove that for any $A$-module $M$, the natural map $I \otimes_A M \to M$ is a monomorphism.  Note that $M$ can be expressed as a filtered colimit of its finitely generated submodules, without loss of generation we can assume that $M$ is a finitely generated module and hence there exists an epimorphism $P \twoheadrightarrow M$ from a compact 1-projective $A$-module $P$. Let $N$ be its kernel, giving a short exact sequence
$$0 \to N \to P \to M \to 0.$$
Note that $N$ can be expressed as a filtered colimit of its finitely generated submodules, say $N \simeq \colim_\alpha K_\alpha$. By assumption (3), $K_\alpha \cap IP = I K_\alpha$ for all $\alpha$. Because filtered colimits are exact and commute with the tensor product, taking the intersection with $IP$ and multiplying by $I$ both commute with the colimit. Thus,
$$N \cap IP \simeq (\colim_\alpha K_\alpha) \cap IP \simeq \colim_\alpha (K_\alpha \cap IP) \simeq \colim_\alpha I K_\alpha \simeq IN.$$
Because $N \cap IP = IN$, the canonical map $N/IN \to P/IP$ is a monomorphism. Consequently, we have the following commutative diagram with exact rows:
\[
\begin{tikzcd}
	0 \arrow[r] & N \arrow[r] \arrow[d] & P \arrow[r] \arrow[d] & M \arrow[r] \arrow[d] & 0 \\
	0 \arrow[r] & N/IN \arrow[r] & P/IP \arrow[r] & M/IM \arrow[r] & 0
\end{tikzcd}
\]
The vertical maps are the canonical quotients. By applying the Snake Lemma to this diagram, their kernels form a short exact sequence:
$$0 \to IN \to IP \to IM \to 0.$$
Now, apply the right exact functor $I \otimes_A -$ to the sequence $0\to N \to P \to M \to 0$, and map it naturally to the sequence of images we just established:
\[
\begin{tikzcd}
	I \otimes_A N \arrow[r] \arrow[d, twoheadrightarrow, "\alpha"] & I \otimes_A P \arrow[r] \arrow[d, "\beta", "\simeq"'] & I \otimes_A M \arrow[r] \arrow[d, "\gamma"] & 0 \\
	IN \arrow[r, hook] & IP \arrow[r, twoheadrightarrow] & IM \arrow[r] & 0
\end{tikzcd}
\]
Since $P$ is a compact 1-projective, it is flat by \cref{preserverefleflats}(3), ensuring that $I \otimes_A P \to P$ is a monomorphism, which makes $\beta: I \otimes_A P \xrightarrow{\sim} IP$ an isomorphism. 
A standard categorical diagram chase on this structure (where $\alpha$ is an epimorphism, $\beta$ is an isomorphism, and the rows are right exact) forces $\gamma: I \otimes_A M \to IM$ to be an isomorphism as well.
Finally, the natural map $I \otimes_A M \to M$ factors as the composition 
$$I \otimes_A M \xrightarrow{\gamma} IM \hookrightarrow M.$$ 
Since $\gamma$ is an isomorphism and the inclusion is a monomorphism, their composition is a monomorphism, completing the proof that $I$ is pure.
\end{proof}
One useful consequence is that pure ideals are determined by their radical (cf. \cref{purenilrad}).
\begin{cor}\label{radpureidls}
	Let $\mca\in\calg(\prladone)$ be a 1-projectively rigid $\op{Ab}$-algebra and $A\in \calg(\mca)$. Suppose that $I, J \subset A$ are pure ideals. If $\sqrt{I}=\sqrt{J}$, then $I=J$.
\end{cor}
\begin{proof}
	Let $f:P \to I$ be a map of $A$-module from a compact  $A$-module $P$. Then $\im(f)$ is a finitely generated ideal. Applying \cref{charpureidl}(2) to $\im(f)\subset I$ we have $$\im(f)=\im(f)\cap( I\cdot I)\simeq \im(f)\cdot I.$$  Therefore $\im(f) = \im(f)\cdot\im(g)$ for some $g:P'\to I$ where $P'$ is compact, because $\im(f)$ is finitely generated. By assumption we have $\im(g)\subset I\subset \sqrt{J}$, therefore by  \cref{idlprefrm}, $\im(g)^n \subset J$ for
	some $n \geq 1$. Thus $\im(f) = \im(f)\cdot\im(g)^n \subset J$. Hence, $I \subset J$. Similarly, we also have
	$J \subset I$ and hence $I=J$.
\end{proof}

We obtain the following spatiality of the pure frame, similar as the classical case \cref{idlusptial}.
\begin{thm}\label{purefrmspatial}
	Let $\mca\in\calg(\prladone)$ be a 1-projectively rigid $\op{Ab}$-algebra and $A\in \calg(\mca)$. Then the pure frame $\idl_u(A)$ is spatial and $\spec(\idl_u(A))$ is quasi-compact.
\end{thm}
\begin{proof}
	By \cref{radpureidls}, the composition
	\[
	\idl_u(A)\hookrightarrow \cidem(\idl(A))\hookrightarrow \idl(A)\xrightarrow{\sqrt{-}}\idl(A)_{\op{rad}}
	\]
	is a frame embedding. Since  $\idl_{\op{rad}}(A)$ is a coherent frame by  \cref{idlprefrm}, it is spatial. It follows from \cref{subfrmspatial} that the pure frame $\idl_u(A)$ is also spatial.
	
	Since the unit ideal $A\in \idl(A)$ is a compact ideal,  it is also compact in $\idl_u(A)$. Consequently, $\spec(\idl_u(A))$ is quasi-compact.
\end{proof}

\subsection{Serre smashing frames}
In this subsection, we study the categorical analogue of pure ideals, which we define as Serre smashing ideals. For a presentably symmetric monoidal separated Grothendieck prestable $\infty$-category, we prove that the poset of Serre smashing ideals forms a subframe of the smashing frame. This subframe structure provides the necessary geometric framework for analyzing the flat algebras.

We first record a pullback criterion for module categories and flatness, which will be used to prove closure under meets.
\begin{lem}\label{pullbackmoducats}
Let $\mcc\in \calg(\prlad)$ such that \mcc is a separated  prestable \infcat. Suppose we are given a pullback diagram
$$\begin{tikzcd}
	A \arrow[r,"p_0"] \arrow[d,"p_1"] & A_0 \arrow[d,"f"] \\
	A_1 \arrow[r,"g"'] & A_{01}
\end{tikzcd}$$
in $\mathrm{CAlg}(\mcc)$. If the map $f$ induces an epimorphism $\pi_0 A_0 \twoheadrightarrow \pi_0 A_{01}$ in $\mcc^\heartsuit$, then the induced functor
$$\mathrm{Mod}_{A}(\mathcal{C}) \longrightarrow
\mathrm{Mod}_{A_0}(\mathcal{C}) \times_{\mathrm{Mod}_{A_{01}}(\mathcal{C})} \mathrm{Mod}_{A_1}(\mathcal{C})$$
is an equivalence of symmetric monoidal $\infty$-categories.

In particular, if $f$ and $g$ are flat, then $p_0$ and $p_1$ are also flat. Moreover, for any map $h:R\to A$ in $\calg(\mcc)$, the map $h$ is flat if and only if both composites $p_0h$ and $p_1h$ are flat.
\end{lem}
\begin{proof}
It is the same argument as  \cite[Proposition 16.2.2.1]{sag}\footnote{Note that although \cite[Proposition 16.2.2.1]{sag} assumes the Grothendieck condition, it is not used in the proof.}.
\end{proof}

The relevant meet operation is most easily computed in the ambient stable category.
\begin{lem}\label{meetinstableidem}
	Let $\mcc\in\calg(\op{Cat}_{\op{st}})$ be a stably symmetric monoidal \infcat. Then the meet of two idempotent algebras $A,B$ in $\calg^{\idem}(\mcc)$ can be identified with $$A\wedge B\simeq A\times_{A\otimes B}B.$$
\end{lem}
\begin{proof}
	It only suffices to show that $M:=A\times_{A\otimes B}B$ is idempotent.
		First, we compute the tensor product $A \otimes M$. We can commute the tensor product with the pullback:
		$$ A \otimes M \simeq A \otimes (A \times_{A \otimes B} B) \simeq (A \otimes A) \times_{A \otimes B \otimes A} (A \otimes B) .$$
		Because $A$ is idempotent, the multiplication map induces an equivalence $A \otimes A \simeq A$. Substituting this into our pullback expression yields:
		$$ A \otimes M \simeq A \times_{A \otimes B} (A \otimes B) \simeq A. $$
		By a symmetric argument utilizing the flatness and idempotence of $B$, we obtain $B \otimes M \simeq B$. Therefore, tensoring $M$ with $A \otimes B$ yields $(A \otimes B) \otimes M \simeq A \otimes B$.
		Consequently, $$M\otimes M\simeq (A\otimes M)\times_{A\otimes B\otimes M}(B\otimes M)\simeq A\otimes_{A\otimes B} B\simeq M$$
		 is idempotent.

\end{proof}

Combining the pullback criterion with the formula for meets, we obtain the frame-theoretic closure property needed for the definition of Serre smashing frame.
\begin{thm}\label{pureidlfrm}
	Let $\mcc\in\calg(\prlad)$ 
	be a presentably symmetric monoidal separated Grothendieck prestable \infcat. Then the inclusion $$\calg(\mcc)_{\op{flat}}^{\op{cn\text{-}idem}}\hookrightarrow\calg(\mcc)^{\op{cn\text{-}idem}}=\cidem(\mcc)$$ forms a subframe, where the left hand side denotes the subposet of those  $\pi_0$-surjective flat idempotent algebras.
\end{thm}
\begin{proof}
	By \cref{connloc} the inclusion  $\calg(\mcc)^{\op{cn\text{-}idem}}\hookrightarrow \calg(\opsp(\mcc))^{\op{idem}}$ can be identified with $\cidem(\mcc)\hookrightarrow\cidem(\opsp(\mcc))$, so it is a subframe. It suffices to show that the inclusion $$i:\calg(\mcc)_{\op{flat}}^{\op{cn\text{-}idem}}\hookrightarrow \calg(\opsp(\mcc))^{\op{idem}}$$ is closed under small joins and finite meets. 
	
	Since $\mcc$ is Grothendieck prestable, by \cref{flatmapsclosure} flat \ein-morphisms are closed under base changes and filtered colimits. Thus $i$ is closed under finite joins and filtered joins, and hence small joins. Now given two flat and $\pi_0$-surjective idempotent  algebras $\mb{1}\to A, \mb{1}\to B$, we wish to show that $$\mb{1}\to A\wedge B=A\times_{A\otimes B}B$$ is flat and $\pi_0$-surjective  too. Since $\calg(\opsp(\mcc))^{\op{idem}}\simeq \cidem(\opsp(\mcc))$, the fiber $$K \to \mb{1}\to A\wedge B$$ in $\opsp(\mcc)$ can be identified with $K_A\otimes K_B$. Note that the tensor product $K_A\otimes K_B$ is connective, so $\mb{1}\to A\wedge B$ is $\pi_0$-surjective. Now it suffices to verify the flatness of $\mb{1}\to A\wedge B$, but that follows from \cref{pullbackmoducats}.
\end{proof}
We now define the prestable analogue of pure idempotent ideals.
\begin{de}[Serre smashing ideals]\label{defserre}
	Let $\mcc \in \calg(\prlad)$ 	be a presentably symmetric monoidal  prestable \infcat.
	A smashing ideal $\mci \in \cidem(\prcdbl)$ is called a
	\textbf{Serre smashing ideal} if $\mcc\to \mcc/\mci$ is left exact.
	A coidempotent object $x \to \mb{1}$ is called
	\textbf{Serre coidempotent} if the corresponding ideal
	$\langle x \rangle \subset \mcc$ is a Serre smashing ideal. 
	We denote by
	\[
	\cidem_s(\mcc) \subset \cidemc
	\]
	the sub-posets consisting of Serre coidempotent objects.

\end{de}
Although the definition is phrased in terms of the associated idempotent algebra, it can also be recognized intrinsically from the exactness properties of the localization.
\begin{rem}
	Let $\mcc \in \calg(\prlad)$ 	be a presentably symmetric monoidal  prestable \infcat.
	Combining \cref{connlocandclosub} and \cite[Proposition 1.38]{levy-liang-sosnilo:dualizable}, a smashing ideal $\mci \xhookrightarrow{i}\mcc$ is Serre if and only if it satisfies the following conditions:
	\enu{
		\item Given a cofiber sequence $C^{\prime} \rightarrow C \rightarrow C^{\prime \prime}$ in $\mathcal{C}$, if any two of the objects $C, C^{\prime}, C^{\prime \prime}$ belong to $\mathcal{I}$, then so does the third.
		\item Given a cofiber sequence $C^{\prime} \rightarrow C \rightarrow C^{\prime \prime}$ in $\mathcal{C}$ where $C \in \mathcal{I}$ and $C^{\prime \prime} \in \mathcal{C}^{\heartsuit}$, we have $C^{\prime} \in \mathcal{I}$.
	}
\end{rem}
\begin{rem}
		By \cref{pureidlfrm}, if $\mcc$ is separated Grothendieck prestable, then $
		\cidem_s(\mcc) \subset \cidemc$ forms a
		 subframe, fitting into the following commutative diagram of frames.
	\[
\begin{tikzcd}
	\calg(\mcc)_{\op{flat}}^{\op{cn\text{-}idem}} \arrow[r, hook] \arrow[d, "\sim"] & \calg(\mcc)^{\op{cn\text{-}idem}} \arrow[d, "\sim", "\ker"'] \\
	\cidem_s(\mcc) \arrow[r, hook]                                                  & \cidem(\mcc)                                        
\end{tikzcd}
	\]
\end{rem}
This comparison leads to the flat version of the telescope conjecture.
\begin{de}
	Let $\mcc\in\calg(\prl)$ such that \mcc is a separated Grothendieck prestable \infcat.  We call $\cidem_s(\mcc)$ as the \textbf{Serre smashing frame} of \mcc. We say \mcc satisfies the \textbf{flat telescope conjecture} if the inclusion $$\cidem_s(\mcc)\hookrightarrow \cidem(\mcc)$$ is an equivalence.

\end{de}
In the almost module case, the Serre condition imposes a familiar condition on homotopy groups.
\begin{prop}\label{serrealmostmod}
	Let $\mcc \in \calg(\prlad)$ 	be a presentably symmetric monoidal left complete prestable \infcat. Let $I\subset \pi_0\mb{1}$ be an idempotent ideal. If the associated smashing ideal $\amodu_{(\mb{1},I)}(\mcc)$ is a Serre smashing ideal, then $\amodu_{(\mb{1},I)}(\mcc)\subset\mcc$ can be identified with those objects $M$ in \mcc such that $$I\cdot \pi_*M=\pi_*M.$$
\end{prop}
\begin{proof}
 Let $\mci=\amodu_{(\mb{1},I)}(\mcc)$.	By \cref{almostalg} and \cref{connlocandclosub}, $\mci\xhookrightarrow{i}\mcc$ can be identified with the kernel of the functor
 $$\mcc\to \modu_{\mb{1}/I^\infty}(\mcc).$$
 Therefore, it suffices to show that for an  object $M\in \mcc$,  $\mb{1}/I^\infty\otimes M=0$ if and only if $I\cdot \pi_*M=\pi_*M$. Since the base change is left exact by assumption, we have\footnote{where $\overline{\otimes}$ denotes the tensor product in $\mcc^\heartsuit$.} $$\pi_n(\mb{1}/I^\infty\otimes M)\simeq \mb{1}/I^\infty\otimes \pi_nM\simeq \pi_{0}\mb{1}/I \overline{\otimes} \pi_nM.$$
 Consequently, $\mb{1}/I^\infty\otimes M=0$ if and only if $I\cdot \pi_*M=\pi_*M$ because \mcc is left complete.
\end{proof}
The next example shows that Serre smashing ideals may be generated by compact projective pieces without the corresponding idempotent ideal itself being compact projective.
\begin{ex}\label{infinitepierce}
	Let $R=\prod_{i=1}^\infty k$ and $I=\bigoplus_{i=1}^\infty k$. Then by \cref{serrealmostmod} the objects in $$\amodu_{(R,I)}(\spgeq)\hookrightarrow \modu_R(\spgeq)\simeq\mcd(R)_{\geq0}$$ can be identified with those connective $R$-modules $X$ such that $I\cdot \pi_*(X)=\pi_*(X)$, because $R\to R/I$ is flat. Therefore, $\amodu_{(R,I)}(\spgeq)$ is generated by compact projective $R$-modules $\{k\cdot e_i|i\geq 1\}$, where $e_i=(0...0,1,0...)\in I$, although $I$ itself is not compact projective ($I$ is projective but not compact).
\end{ex}
\subsection{Serre smashing frames versus pure frames}
In this subsection, we investigate the relationship between Serre smashing frames and pure frames by analyzing flat objects in the prestable $\infty$-categorical setting. For a rigid $\spgeq$-algebra, our primary goal is to show that its Serre smashing frame can be identified with the pure frame of its heart.

 We begin by recalling that a dualizable additive \infcat always has enough flat objects, as demonstrated by the following result.
\begin{thm}[{\cite[Theorem 4.13]{levy-liang-sosnilo:dualizable}}]\label{flatgenerate}
	Let $\mcm\in\praddbl$ be a \dualaddinfcat. Then $\mcm$ is generated by $\omega_1$-compact \spgeq-flat objects under small colimits.
\end{thm}
The next two propositions explain how this flatness behaves under base change and under passage to rigid algebras.
\begin{prop}\label{basechangeflat}
	Let \(\mcc\in\calg(\prlad)\) be such that its underlying
	\(\infty\)-category is Grothendieck prestable, and let \(\mcm\in\praddbl\). If
	\(x\in\mcm\) is \(\spgeq\)-flat, then
	\(\mathbf{1}_{\mcc}\otimes x\in\mcc\otimes\mcm\) is \(\mcc\)-flat.
\end{prop}
\begin{proof}
	Since the \mcc-dual map of $\mcc\simeq \mcc \otimes \spgeq \xrightarrow{\mb{1}_\mcc\otimes x}\mcc \otimes \mcm$ can be identified with 
	$$\mcc\otimes \mcm^\vee \xrightarrow{\mb{1}_\mcc\otimes x^\vee}\mcc\otimes \spgeq\simeq\mcc,$$	it suffices to show that colimit-preserving left exact functors between Grothendieck prestable \infcats are closed under Lurie tensor product. However, that follows from \cite[Proposition C.4.4.1]{sag}.
\end{proof}
Usually, categorically flat objects are different from $\otimes$-flat objects. However, in the rigid case, they coincide with each other, as follows.
\begin{prop}\label{rigidflat}
	Let $\mcc\in\calg(\prlad)$ be a rigid \spgeq-algebra and \mcm be a dualizable \mcc-module. Then for an object $x\in\mcm$, it is \mcc-flat if and only if it is \spgeq-flat.
\end{prop}
\begin{proof}
	Assume that $x\in\mcm$ is \spgeq-flat. Since by \cite[Proposition 4.17]{ramzi2024dualizable}, \mcm is dualizable over \spgeq, by \cref{basechangeflat}, $\mb{1}_\mcc\otimes x \in \mcc\otimes \mcm$ is \mcc-flat. However, by \cite[Proposition 4.11]{ramzi2024dualizable}, $\mcc\otimes\mcm\to\mcc$ is a \mcc-internal left adjoint, so by \cref{preserverefleflats}, $x\in\mcm$ is \mcc-flat.
	
	Now assume that $x\in\mcm$ is \mcc-flat. Then $x_{\spgeq}:\spgeq \to \mcm$ is equivalent to the composition $\spgeq\to\mcc\xrightarrow{x_\mcc}\mcm$. By the rigidity we have the following commutative diagram.
	$$\begin{tikzcd}
		{\funct^L(\mcm,\spgeq)} \arrow[d] \arrow[r, "\sim"] & {\funct^L_{\mcc}(\mcm,\funct^L(\mcc,\spgeq))} \arrow[r, "\sim"] \arrow[d] & {\funct^L_{\mcc}(\mcm,\mcc)} \arrow[d] \\
		{\funct^L(\mcc,\spgeq)} \arrow[r, "\sim"] \arrow[d] & {\funct^L_{\mcc}(\mcc,\funct^L(\mcc,\spgeq))} \arrow[r, "\sim"]           & {\funct^L_{\mcc}(\mcc,\mcc)} \arrow[d] \\
		{\funct^L(\spgeq,\spgeq)} \arrow[rr, "\sim"]        &                                                                           & \spgeq                                
	\end{tikzcd}$$ Therefore the functor $x_{\spgeq}^{\vee}:\mcm_{\spgeq}^{\vee}\to \spgeq$ can be identified with the composition
	$$\mcm_{\spgeq}^{\vee}\simeq \mcm_{\mcc}^{\vee}\xrightarrow{x_{\mcc}^{\vee}}\mcc\xrightarrow{\mapp_{\mcc}(\mb{1},-)}\spgeq.$$
	Consequently, $x\in\mcm$ is \spgeq-flat.
\end{proof}
It is natural to ask whether this comparison holds in the same generality for arbitrary rigid algebras over an arbitrary base \mcv.
\begin{qu}
 Let $\mcv\in\calg(\prl)$ be arbitrary. Let $\mcw\in\calg^{\op{rig}}(\prlv)$ be a rigid \mcv-algebra and \mcm be a dualizable \mcw-module. Is it true that an object $x\in\mcm$ is \mcw-flat if and only if it is \mcv-flat?
\end{qu}

In the rigid case, we have the following useful criterion of flatness over a discrete algebra, generalizing a result in \cite[Proposition 3.19]{hattt} to any rigid \spgeq-algebra.
\begin{prop}\label{discflat}
	Let $\mcc\in\calg(\prlad)$ be a rigid $\spgeq$-algebra and $R\in \alg(\mcc^\heartsuit)$ be a discrete $\eone$-algebra. Given a left $R$-module $M \in \lmodu_R(\mcc)$; then $M$ is \mcc-flat over $R$ if and only if $M$ is discrete and $\pi_0M$ is $\mcc^\heartsuit$-flat over $\pi_0R$.
\end{prop}
\begin{proof}
	The ``only if'' direction follows from \cref{truncatedflat} and \cite[Proposition C.3.2.1]{sag}.
	
	For the ``if'' direction, assume that $M$ is discrete and $\pi_0M$ is $\mcc^\heartsuit$-flat over $\pi_0R$. Now given a discrete right $R$-module $N$; by \cite[Proposition C.3.2.1]{sag} it suffices to show that $N\otimes_RM$ is discrete too.
	
	By \cref{flatgenerate}, there exists a map $f:F\to N\in \rmodu_{R}(\mcc)$ such that $F$ is \spgeq-flat in $\rmodu_{R}(\mcc)$ and $f$ induces an epimorphism on $\pi_0$. Note that by \cref{rigidflat}, $F$ is \mcc-flat over $R$, so by the ``only if'' direction we actually have $F$ is discrete. Then we have a cofiber-fiber sequence in $\rmodu_{R}(\mcc)$ $$\op{fib}(f)\to F\to N,$$ hence $\op{fib}(f)$ is discrete too.
	Now right tensoring with $M$, we get a cofiber-fiber sequence in $\mcc$
	$$\op{fib}(f)\otimes_RM\to F\otimes_RM\to N\otimes_RM. $$
 Consider the following diagram;
	$$\begin{tikzcd}
		& \pi_0(\operatorname{fib}(f)\otimes_R M) \arrow[r] \arrow[d, "\sim"'] & \pi_0(F\otimes_R M) \arrow[r] \arrow[d, "\sim"'] & \pi_0(N\otimes_R M) \arrow[r] \arrow[d, "\sim"'] & 0 \\
		0 \arrow[r] & \pi_0\operatorname{fib}(f)\otimes_{\pi_0R} \pi_0M \arrow[r]          & \pi_0F\otimes_{\pi_0R} \pi_0M \arrow[r]          & \pi_0N\otimes_{\pi_0R} \pi_0M \arrow[r]          & 0
	\end{tikzcd} $$
		the bottom short exactness is obtained from the $\mcc^\heartsuit$-flatness of $\pi_0M$ over $\pi_0R$.
	Because $F\otimes_R M$ is discrete, we see that $\pi_1(N\otimes_RM)=0$ by the long exact sequence. 
	On the other hand, $\pi_n(\op{fib}(f)\otimes_RM)=\pi_{n+1}(N\otimes_RM)=0$ for each $n\geq 1$, which implies $N\otimes_RM$ is discrete.
\end{proof}
We can now identify the Serre condition for quotients by idempotent ideals with the purity condition introduced above.
\begin{thm}\label{equidefpureidls}
	Let $\mcc\in\calg(\prlad)$ be a rigid $\spgeq$-algebra and $R\in \calg(\mcc)$. Then for an idempotent ideal $I\subset \pi_0R$, the \ein-algebra $R/I^\infty$ is $\mcc$-flat over $R$ if and only if $I$ is pure.
\end{thm}
\begin{proof}
	The ``only if'' direction follows from \cref{truncatedflat}.
		
Now	assume that $I$ is pure, i.e. $\pi_0R/I$ is $\mcc^\heartsuit$-flat over $\pi_0R$; by \cref{discflat}, $\pi_0R/I$ is $\mcc$-flat over $\pi_0R$ too. 
		By \cite[Corollary 6.12]{hattt}, $\pi_0R\to\pi_0R/I$ is an L-étale morphism in $\calg(\mcc)$.
	Therefore by \cite[Theorem 6.20(1)]{hattt}, there exists a (unique) $R'\in \calg(\mcc)_{R/}$ that is  L-étale and \mcc-flat  over $R$ such that  $\pi_0R'\simeq\pi_0R/I$.
	Then we have the following diagram in $\calg(\mcc)$.
	$$\begin{tikzcd}
		R \arrow[r] \arrow[d, Rightarrow, no head] & R/I^\infty \arrow[d, dashed,"f"] \\
		R \arrow[d] \arrow[r]                      & R' \arrow[d]                 \\
		\pi_0R \arrow[r]                           & \pi_0R/I                    
	\end{tikzcd}$$
	Let $B:=R/I^\infty$.
	Consider the cofiber sequence of cotangent complexes $$R'\otimes_{B} L_{B/R}\to L_{R'/R}\to L_{R'/B};$$
	since the first and second ones are trivial, we get $L_{R'/B}=0$. By \cite[Corollary 6.31]{hattt}, we conclude that $\op{cofib}(f)$ is $\infty$-connective and hence $R/I^\infty=B\simeq R'$, so $R/I^\infty$ is flat over $R$ as desired. 
\end{proof}
As a first consequence, the idempotent approximation of a pure ideal is itself flat as a module.
\begin{cor}
		Let $\mcc\in\calg(\prlad)$ be a rigid $\spgeq$-algebra and $R\in \calg(\mcc)$. Then for any pure ideal $I\subset \pi_0R$, we have that $I^\infty$ is a flat $R$-module.
\end{cor}
\begin{proof}
	Consider the cofiber sequence of connective $R$-modules.
	$$I^\infty \longrightarrow R \longrightarrow R/I^\infty.$$
	Let $M$ be an arbitrary discrete $R$-module. Then $$I^\infty\otimes_R M \longrightarrow R\otimes_R M\simeq M \longrightarrow R/I^\infty\otimes_R M$$ is a cofiber sequence too. By \cref{equidefpureidls}, we see that $R/I^\infty$ is \mcc-flat over $R$. Therefore by long exact sequence of homotopy groups, we see that $I^\infty\otimes_R M$ is discrete, and thus $I^\infty$ is a flat $R$-module.
\end{proof}
Putting the preceding comparison together gives the desired identification of the Serre smashing frame with the pure frame of the heart.
\begin{cor}\label{pureserreequi}
	Let $\mcc\in\calg(\prlad)$ be a rigid $\spgeq$-algebra. Then the  Serre smashing frame $\cidem_s(\mcc)$  can be identified with the pure frame $\idl_u(\mcc^\heartsuit)$,  fitting into the following diagram.
	$$\begin{tikzcd}
		\idl_u(\mcc^\heartsuit) \arrow[r, hook] \arrow[d, "\sim"] & \cidem(\idl(\mcc^\heartsuit)) \arrow[d, "\sim"'] \\
		\cidem_s(\mcc) \arrow[r, hook]                     & \cidem(\mcc)                     
	\end{tikzcd}$$
\end{cor}
\begin{proof}
	It follows directly from \cref{almostalg} and \cref{equidefpureidls}.
\end{proof}

\begin{rem}
	By \cref{pureserreequi}, we see that if $\mcc$ is a rigid \spgeq-algebra, then its Serre smashing frame 
	only depends on its heart.
\end{rem}
\begin{cor}\label{serresmfrmspatial}
	Let $\mcc\in\calg(\prlad)$ be a projectively rigid $\spgeq$-algebra. Then the Serre smashing frame $\cidem_s(\mcc)$ is spatial and $\spec\big(\cidem_s(\mcc)\big)$ is quasi-compact.
\end{cor}
\begin{proof}
	It follows by combining \cref{purefrmspatial} and \cref{pureserreequi}
\end{proof}
\begin{rem}
	The spatiality of the Serre smashing frame contrasts sharply with that of
	the full smashing frame. Recent work of Barthel--Heard
	\cite{barthel2026nonspatialframesmashingideals} constructs commutative rings \(A\) for
	which the frame of idempotent ideals
	$
	\cidem(\idl(A))
	$
	is not spatial; their examples are used to produce derived categories
	with non-spatial smashing frames.
	
	For such a ring \(A\), our identification
	$
	\cidem\bigl(\modu_A(\spgeq)\bigr)
	\simeq
	\cidem(\idl(A))
$
	shows that the full smashing frame of the projectively rigid prestable
	\(\infty\)-category \(\modu_A(\spgeq)\) is likewise non-spatial.
	On the other hand, by \cref{serresmfrmspatial}, its Serre smashing frame
	$
	\cidem_s\bigl(\modu_A(\spgeq)\bigr)
	$
	is spatial. Consequently,
	\[
	\cidem_s\bigl(\modu_A(\spgeq)\bigr)
	\subsetneq
	\cidem\bigl(\modu_A(\spgeq)\bigr),
	\]
	so \(\modu_A(\spgeq)\) does not satisfy the flat telescope conjecture.
	In particular, these examples show that passing from the full smashing
	frame to the Serre smashing frame can restore spatiality.
\end{rem}
\subsection{Over a connective \einfring}
In this subsection, we investigate the flat telescope conjecture for $\modu_R(\spgeq)$ over a connective \einfring $R$. We analyze the relationships among the atomic, Serre, and general smashing frames, relating them to the algebraic hierarchy of Pierce, pure, and general idempotent ideals. For a detailed review of pure ideals in commutative ring theory, we refer the reader to \cref{pureidl}. 

The following theorem shows that for connective \einfrings, Serre smashing ideals interpolate between atomic smashing ideals and general smashing ideals. 

\begin{thm}\label{atsmidlisserre}
	Let $R$ be a connective \ein-ring. Then any atomic smashing ideal  of $\modu_R(\spgeq)$ is a Serre smashing ideal.
\end{thm}
\begin{proof}
	By \cref{main2} and \cref{pureserreequi}, it suffices to show that every Pierce ideal of an ordinary ring 
	 is a pure ideal, but that follows from \cref{pierceincpure}.
\end{proof}

\begin{cor}
		Let $R \in \calg(\spgeq)$ be a connective $\mathbb{E}_\infty$-ring, and let $\mcc=\modu_R(\spgeq)$. Then we have an inclusion $$\cidem_f(\mcc)\subset \cidem_s(\mcc).$$
\end{cor}
\begin{cor}
	 	Let $R \in \calg(\spgeq)$ be a connective $\mathbb{E}_\infty$-ring. Then $\modu_R(\spgeq)$ satisfies the flat telescope conjecture if and only if every idempotent ideal of $\pi_0R$ is a pure ideal. In particular, if $\modu_R(\spgeq)$ satisfies the telescope conjecture, then it satisfies the flat telescope conjecture.
\end{cor}
Thus the flat telescope conjecture can hold even when the ordinary telescope conjecture fails, as the following example shows.
\begin{ex}
	Let \(k\) be a field and set
	\[
	R
	:=
	k[x_1,x_2,\ldots]/
	\bigl(x_n(1-x_{n+1})\mid n\geq1\bigr).
	\]
	Then every idempotent ideal of \(R\) is pure, but the ideal
$
	I:=(x_1,x_2,\ldots)
	$
	is pure and is not a Pierce ideal (see \cref{flatnotatomicexample}). Consequently,
	\[
	\idl_p(R)
	\subsetneq
	\idl_u(R)
	=
	\cidem(\idl(R)).
	\]
	In particular,
	$
	\modu_R(\spgeq)
	$
	satisfies the flat telescope conjecture but does not satisfy the
	telescope conjecture.
\end{ex}
We also record examples showing that the flat telescope conjecture can fail when idempotent ideals are not pure.
\begin{ex}[Counterexamples]
Both examples exhibit idempotent ideals which fail the purity criterion, and therefore give non-Serre smashing ideals.
	\enu{

		\item 
		Let $(A,\mathfrak{m})$ be a non-discrete valuation ring.
		Let
		$
		I := \mathfrak{m}
		$
		be the maximal ideal of $A$.
		Then $I$ is non-pure idempotent by \cref{counterpiercepure}.
		\item Let $k$ be a field and $p$ be a prime. Consider the ring $A = k[x^{1/p^\infty}]$, which is formally defined as the union $\bigcup_{n \geq 0} k[x^{1/p^n}]$. 
		Let $I = (x^{1/p^\infty})$ be the ideal generated by the set $\{x^{1/p^n} \mid n \geq 0\}$. Then $I$ is not a pure ideal by \cref{counterextcadd}.
	}

\end{ex}

\begin{rem}\label{piercenotpure}
	The \cref{atsmidlisserre} naturally leads to the question of whether, in any projectively rigid $\spgeq$-algebra, Serre smashing ideals always lie between atomic smashing ideals and smashing ideals. 
	
	The answer, however, turns out to be negative. In \cref{counterex3}, we provide an explicit counterexample within the $\infty$-category of connective $G$-equivariant spectra for a finite group $G$.
	
\end{rem}

\begin{qu}
	Let $G$ be a finite group. Is there an explicit characterization of the Pierce ideals or pure ideals of an arbitrary Green functor $\underline{R}\in\calg(\opsp_G^\heartsuit)$? Do similar descriptions exist for the other examples in \cref{exalgttt}?
\end{qu}
\subsection{Example: additive sheaves}\label{sheafcattc}
In this subsection, we investigate the flat telescope conjecture for the $\infty$-categories $\shvhyp(X;\spgeq)$ and $\shv(X;\op{Ab})$. As a geometric application of our framework, we establish an identification between atomic smashing ideals and Pierce open subsets of the base space $X$.

We first fix notation for the hypercomplete version of the associated $\infty$-topos.
\begin{nota}
	Let \mcx be an $\infty$-topos.  We denote the hypercompletion of \mcx by $\mcx^{\wedge}$.
\end{nota}
\begin{rem}
	Note that by definition we have $\shv^\wedge(\mcx;\spgeq)\simeq\shv(\mcx^\wedge;\spgeq).$ Moreover, it can be identified with the universal separated Grothendieck prestable \infcat out of $\shv(\mcx;\spgeq)$; see \cite[Proposition C.3.6.1]{sag}.
\end{rem}
The basic geometric input is that open and closed decompositions of an $\infty$-topos give rise to smashing localization sequences.
\begin{prop}
	Let $\mathcal{X}$ be an $\infty$-topos, $\mathcal{U}=\mcx_{/U}$ be an open subtopos, where $U\in \mcx_{\leq -1}$ is a subfinal object. Let $\mathcal{Z}=\mcx \setminus U$ be the complementary closed subtopos. Let $j_! \colon \mathcal{U} \hookrightarrow \mathcal{X}$ denote the left adjoint to the open restriction $j^* \colon \mathcal{X} \to \mathcal{U}$, and $i^* \colon \mathcal{X} \to \mathcal{Z}$ denote the closed localization functor. 
	\enu{ \item The sequence $$ \mathcal{U} \xhookrightarrow{j_!} \mathcal{X} \xrightarrow{i^*} \mathcal{Z} $$ is a cofiber sequence in $\prl$.
		\item We have $\mcu^{\wedge}\simeq (\mcx^{\wedge})_{/U}$ and $\mcz^{\wedge}\simeq \mcx^{\wedge}\setminus U$. Therefore the sequence $$ \mathcal{U}^{\wedge} \xhookrightarrow{j_!} \mathcal{X}^{\wedge} \xrightarrow{i^*} \mathcal{Z}^{\wedge} $$ is a cofiber sequence in $\prl$.
		\item Both sequences $$\shv(\mathcal{U};\spgeq)\xhookrightarrow{j_!\otimes\spgeq}\shv(\mathcal{X};\spgeq)\xrightarrow{i^*\otimes\spgeq} \shv(\mathcal{Z};\spgeq)$$ and $$\shv(\mathcal{U}^{\wedge};\spgeq)\xhookrightarrow{j_!\otimes\spgeq}\shv(\mathcal{X}^{\wedge};\spgeq)\xrightarrow{i^*\otimes\spgeq} \shv(\mathcal{Z}^{\wedge};\spgeq)$$ are  smashing localization sequences.
	}
\end{prop}
\begin{proof}
	(1)	It suffices to show that $$ \mathcal{U} \xleftarrow{j^*} \mathcal{X} \xhookleftarrow{i_*} \mathcal{Z} $$ is a fiber sequence in $\pr^R$, but that follows precisely from \cite[Lemma 7.3.2.4(3)]{htt}.\\
	(2) The equivalence $\mcu^{\wedge}\simeq (\mcx^{\wedge})_{/U}$ follows from \cite[Lemma 6.5.2.12]{htt}. Now we wish to prove $\mcz^{\wedge}\simeq \mcx^{\wedge}\setminus U$. Since both left and right arrows in the following diagram are localizations, it suffices to show the existence of factorizations $\varphi$ and $\phi$:
	$$\begin{tikzcd}
		& \mcx^{\wedge} \arrow[ld, "p"'] \arrow[rd, "q"] &                                                                 \\
		\mcz^{\wedge} \arrow[rr, "\varphi", dashed, shift left] &                                     & \mcx^{\wedge}\setminus U \arrow[ll, "\phi", dashed, shift left]
	\end{tikzcd}$$
	
	For the existence of $\varphi$, it suffices to show that  $\mcx^{\wedge}\setminus U$ is hypercomplete. By \cite[Lemma 7.3.2.4(3)]{htt}, the reflective subcategory $\mcx^{\wedge}\setminus U\subset \mcx^{\wedge}$ consists of those objects $X$ satisfying $X\times U\xrightarrow{\sim} U$. In particular, the inclusion $\mcx^{\wedge}\setminus U\subset \mcx^{\wedge}$ is closed under weakly contractible colimits (and small limits). Hence, it preserves $\infty$-connective morphisms. Thus any $\infty$-connective morphism in $\mcx^{\wedge}\setminus U$ is an equivalence, as desired. 
	
	For the existence of $\phi$, by \cite[Proposition 7.3.2.3]{htt} it suffices to check that for any $V\in \mcx^\wedge$ such that $\mapp_{\mcx^\wedge}(V,U)\neq\emptyset$, we have $p(\emptyset\to V)$ is an equivalence. However, that follows from the following diagram:
	$$\begin{tikzcd}
		\mcx \arrow[d] \arrow[r] & \mcx^\wedge \arrow[d, "p"] \\
		\mcz \arrow[r]           & \mcz^\wedge          
	\end{tikzcd}$$  
	(3) By \cref{connlocandclosub}, it suffices to observe that they are cofiber sequences in \prlad, and that $$\shv(\mathcal{U};\spgeq)\xhookrightarrow{j_!\otimes\spgeq}\shv(\mathcal{X};\spgeq)$$ and $$\shv(\mathcal{U}^\wedge;\spgeq)\xhookrightarrow{j_!\otimes\spgeq}\shv(\mathcal{X}^\wedge;\spgeq)$$ are smashing ideals (see \cite[Proposition 6.3.5.11]{htt}).
\end{proof}

We begin with the abelian category of sheaves of abelian groups, where idempotent ideals of the constant sheaf are classified by open subsets.

\begin{thm}\label{pierceAbshv}
	Let $X$ be a space. Let $\underline{\mathbb{Z}}\in \shv(X;\op{Ab})$ denote its unit. Then 
	\enu{
		\item There is a natural equivalence from the poset $	\mco(X)$ of open subsets
		\[
		\mco(X)\xrightarrow{\sim}\cidem(\shv(X;\op{Ab}))
		\]
		given by $U\mapsto j^U_!(\underline{\mathbb{Z}}_U)$.
		\item  Every idempotent ideal  of $\underline{\mathbb{Z}}$ is a pure ideal.
		\item Furthermore, the idempotent ideal $j^U_!(\underline{\mathbb{Z}}_U)$ is a Pierce ideal if and only if $U$ is a Pierce open subset (see \cref{pierceopen}). In particular, the equivalence in (1) restricts to an equivalence $$\mco_p(X)\xrightarrow{\sim}\cidem_f(\shv(X;\op{Ab}))$$ from the poset $\mco_p(X)$ of Pierce open subsets.
	} 
\end{thm}
\begin{proof}
	(1) By \cref{idemidlabel},	it suffices to show that the functor $ \mco(X)\xrightarrow{}\cidem(\idl(\underline{\mathbb{Z}}))
	$ is an equivalence. It is obviously an embedding, and hence fully faithful. Now we prove that it is essentially surjective.
	For any subsheaf $I \subset \underline{\mathbb{Z}}$, its stalk $I_x$ at any point $x \in X$ must be an ideal of $\underline{\mathbb{Z}}_x \simeq \mathbb{Z}$. Therefore, there exists a non-negative integer $n_x \geq 0$ such that $I_x = n_x\mathbb{Z}$. 
	The idempotency condition $I= I^2$ implies the following equivalence on stalks: 
	\[ (n_x\mathbb{Z})^2 \xrightarrow{\simeq} n_x\mathbb{Z} \]
	This morphism is given by integer multiplication: $a \otimes b \mapsto ab$. Its image is $(n_x)^2 \mathbb{Z}$. For this map to be surjective (and thus an isomorphism), we must have $(n_x)^2 \mathbb{Z} = n_x \mathbb{Z}$. 
	Among non-negative integers, the only solutions to $n^2 = n$ are $n = 0$ or $n = 1$. Consequently, for any $x \in X$, the stalk $I_x$ of the ideal sheaf can only be $0$ or $\mathbb{Z}$.
	
	Now we examine the topological distribution of these stalks. Define the set:
	\[ U = \{ x \in X \mid I_x = \mathbb{Z} \} \]
	We need to show that $U$ is an open set in $X$. 
	Take any $x \in U$; by definition, $I_x = \mathbb{Z}$. This means the constant $1$ is in $I_x$ (i.e., $1_x \in I_x$). By the definition of the stalk of a sheaf, there exists an open neighborhood $V \subset X$ of $x$ and a section $s \in I(V)$ such that the germ of $s$ at $x$ is $s_x = 1_x$. 
	Since $I$ is a subsheaf of the constant sheaf $\underline{\mathbb{Z}}$, the section $s$ must be a locally constant function in $\underline{\mathbb{Z}}(V)$. Therefore, there exists a smaller open neighborhood $W \subset V$ containing $x$ such that $s$ is identically $1$ on $W$. 
	For any point $y \in W$, the section $s$ gives $s_y = 1_y \in I_y$. Since $1$ belongs to the ideal $I_y$, it must be that $I_y = \mathbb{Z}$. This implies $W \subset U$. Since for any $x \in U$ we can find such an open neighborhood $W$, we conclude that $U$ is an open set.
	
	Now that we know $U$ is an open set, and that the stalks of $I$ are identically $\mathbb{Z}$ on $U$ and $0$ on $X \setminus U$, the sheaf $I$ is uniquely determined by $U$ up to isomorphism. 
	Let $j: U \hookrightarrow X$ be the open immersion. By checking stalks, we conclude the following zig-zag equivalence, as desired:
	\[ I\xleftarrow{\sim} I \otimes j_!(\underline{\mathbb{Z}}_U) \xrightarrow{\sim}j_!(\underline{\mathbb{Z}}_U) \]
	(2) Now we check that every idempotent ideal $I\subset\underline{\mathbb{Z}}$ is a pure ideal. By (1), there exists an open subset $U$ such that $I\simeq j_!(\underline{\mathbb{Z}}_U)$. By definition, we need to check that the functor $$\shv(X;\op{Ab})\xrightarrow{\underline{\mathbb{Z}}/I\otimes-}\shv(X;\op{Ab})$$ is left exact, but that can be verified stalkwise (a stalk of $\underline{\mathbb{Z}}/I$ is either $\mathbb{Z}$ or $0$).\\
	(3) Assume that $U= \bigcup_\alpha U_\alpha$ is a union of clopen subsets. Without loss of generality, we can assume that it is a filtered union. Note that by \cref{splitsmidl}, $j_!(\underline{\mathbb{Z}}_{U_\alpha})$ is dualizable in $\shv(X;\op{Ab})$. On the other hand, we have $j_!(\underline{\mathbb{Z}}_U)\simeq\colim_{\alpha} j_!(\underline{\mathbb{Z}}_{U_\alpha}).$ Therefore the smashing ideal $$\left<j_!(\underline{\mathbb{Z}}_U)\right>=\shv(U;\op{Ab})\xhookrightarrow{j_!} \shv(X;\op{Ab})$$ is generated by dualizable/atomic objects, as desired.
	
	Now assume that $j_!(\underline{\mathbb{Z}}_U)$ is a Pierce ideal. By \cref{defpierce}, the smashing ideal $\left<j_!(\underline{\mathbb{Z}}_U)\right>=\shv(U;\op{Ab})\xhookrightarrow{j_!} \shv(X;\op{Ab})$ is generated by dualizable objects. We first observe that if $P\in \mathcal C=\shv(X;\op{Ab})$ is dualizable, then its support \[ \operatorname{Supp}(P):=\{x\in X\mid P_x\neq 0\} \] is clopen, because by \cref{dualshvlocconst} $P$ is finitely locally free. Since $\mathcal I_U=\shv(U;\op{Ab})$ is generated by dualizable objects, there is a family of dualizable objects $\{P_\alpha\}_{\alpha\in A}$ of $\mathcal C$ such that $ \mathcal I_U = \langle P_\alpha\rangle_{\alpha\in A}. $ Since each generator $P_\alpha$ belongs to $\mathcal I_U$, $ (P_\alpha)_x=0 \text{ for all }x\in X\setminus U. $ Equivalently, $ \operatorname{Supp}(P_\alpha)\subset U. $ By the observation above, each $\operatorname{Supp}(P_\alpha)$ is a clopen subset of $X$. Hence \[ \bigcup_{\alpha\in A}\operatorname{Supp}(P_\alpha)\subset U . \] It remains to prove the reverse inclusion. Suppose, for contradiction, that there exists a point $ x\in U $ such that $ x\notin \operatorname{Supp}(P_\alpha) \text{ for all }\alpha\in A. $ Then $ (P_\alpha)_x=0 \text{ for all }\alpha\in A. $ Consider the full subcategory \[ \mathcal K_x := \{F\in \mathcal C\mid F_x=0\}. \] The stalk functor $ x^*:\operatorname{Shv}(X;\operatorname{Ab})\longrightarrow \operatorname{Ab} $ is exact, colimit-preserving, and symmetric monoidal. Hence $\mathcal K_x$ is a localizing ideal of $\mathcal C$. Since all the generators $P_\alpha$ lie in $\mathcal K_x$, we have the inclusion \[ \mathcal I_U = \langle P_\alpha\rangle_{\alpha\in A} \subset \mathcal K_x . \]  On the other hand, the object $ i_!\underline{\mathbb{Z}}_U\in \mathcal I_U $ has stalk $ (i_!\underline{\mathbb{Z}}_U)_x\simeq \mathbb Z $ because $x\in U$. Thus $ i_!\underline{\mathbb{Z}}_U\notin \mathcal K_x, $ which contradicts $\mathcal I_U\subset \mathcal K_x$. Therefore we must have $ U = \bigcup_{\alpha\in A}\operatorname{Supp}(P_\alpha). $ 
\end{proof}
The \infcat of hypercomplete connective spectral sheaves follows an analogous pattern.

\begin{thm}\label{pierceSpgeq0shv}
	Let $X$ be a space. Let $\underline{\mathbb{S}}^{\wedge}\in \shvhyp(X;\spgeq)$ denote its unit. Then 
	\enu{
		\item There is a natural equivalence
		\[
		\mco(X)\xrightarrow{\sim}\cidem(\shvhyp(X;\spgeq))
		\]
		given by $U\mapsto j^U_!(\underline{\mathbb{S}}_U^{\wedge})$.
		\item  Every smashing ideal  of $\shvhyp(X;\spgeq)$ is a Serre smashing ideal.
		\item Furthermore, the smashing ideal $\shvhyp(U;\spgeq)\xhookrightarrow{j^U_!}\shvhyp(X;\spgeq)$ is an atomic smashing ideal if and only if $U\subset X$ is a Pierce open subset. In particular, the equivalence in (1) restricts to an equivalence $$\mco_p(X)\xrightarrow{\sim}\cidem_f(\shvhyp(X;\spgeq)).$$
	} 
\end{thm}
\begin{proof}
	We first observe that by \cite[Lemma A.3.9]{ha} and the equivalence $$\shvhyp(X;\spgeq)\simeq \funct^{\times}\Big((\spgeq^{\op{cproj}})^\opp,\shvhyp(X)\Big),$$ the collection  of stalk functors
	$$\{x^*:\shvhyp(X;\spgeq)\to \spgeq\mid x\in X\}$$ is jointly conservative. Therefore all results follow from the same argument as \cref{pierceAbshv}.
\end{proof}
\begin{rem}
	Combining with \cref{cartesiancidem}, we see that the free functors induce the equivalences of smashing frames
	$$\cidem(\shv(X))\xrightarrow{\sim}\cidem(\shvhyp(X;\spgeq))\xrightarrow{\sim}\cidem(\shv(X;\op{Ab})).$$
\end{rem}
These classifications translate directly into telescope-type statements for sheaf categories.
\begin{cor}\label{shvabtc}
	Let $X$ be a space. Then the separated Grothendieck prestable \infcat $\shvhyp(X;\spgeq)$ always satisfies the flat telescope conjecture. Furthermore, the following are equivalent:
	\enu{
		\item $\shvhyp(X;\spgeq)$ satisfies the telescope conjecture;
	\item $\shv(X;\op{Ab})$ satisfies the telescope conjecture;
	\item every open subset of $X$ is a Pierce open subset.
	}
\end{cor}
\begin{rem}
	Note that the condition (3) is equivalent to the condition that the clopen subsets of $X$ form a basis.
\end{rem}

In particular, the condition is automatically satisfied for profinite spaces.
\begin{cor}
		Let $X$ be a profinite space. Then both  $\shvhyp(X;\spgeq)$ and $\shv(X;\op{Ab})$ satisfy the telescope conjecture.
\end{cor}
\begin{proof}
	Since $X$ is profinite, it has a basis of clopen  subsets. Therefore, by \cref{tcaddandheart} and \cref{shvabtc}, we win.
\end{proof}
\begin{rem}
	In fact, when $X$ is profinite, we have $\shv(X;\spgeq)\simeq\shvhyp(X;\spgeq)$ by \cref{shvprofinite}.
\end{rem}
\begin{rem}
	Since a compact Hausdorff space is profinite if and only if it has a basis of clopen subsets, we conclude that for a compact Hausdorff space $X$, the \infcat $\shvhyp(X;\spgeq)$ or $\shv(X;\op{Ab})$ satisfies the telescope conjecture if and only if $X$ is profinite.
\end{rem}
The preceding arguments use hypercompleteness in an essential way.  It remains natural to ask whether the same classification holds for unseparated sheaves of connective spectra.
\begin{qu}
	Does \cref{pierceSpgeq0shv} hold for the (unseparated) prestable \infcat $\shv(X;\spgeq)$ over an arbitrary topological space $X$? 
\end{qu}

\appendix
\part*{Appendices}
\addcontentsline{toc}{part}{Appendices}

\section{Dualizable Sheaves}
This appendix focuses on the dualizability conditions for sheaves on topological spaces, with a particular emphasis on cases arising from profinite presentations. We provide the technical foundation for the geometric applications of the flat telescope conjecture discussed in the main text.
\begin{lem}\label{stalkformula}
	Let $X$ be a space and $\mcc\in\prl$.  We define the stalk functor $x^*: \shv(X;\mcc) \to \mcc$ as the left adjoint to the pushforward functor $(i_x)_*: \mcc \to \shv(X;\mcc)$ given by precomposing with the map $i_x^{-1}(-):\mco(X)\to \mco(*)=\{\emptyset,*\}$.
	Then for any sheaf $\mcf\in \shv(X;\mcc)$ and $x\in X$, we have the formula $$x^*\mcf\simeq \colim_{U\ni x} \mcf(U).$$ 
\end{lem}
\begin{proof}
	We can compute the left adjoint explicitly. Let $$\mathcal{N}_x=\mco(X)^\opp\times_{\mco(*)^\opp}(\mco(*)^\opp)_{/*}$$ denote the poset of open neighborhoods of $x$ ordered by reverse inclusion, which is a filtered $\infty$-category. Let $x^*_{pre}$ denote left adjoint to the skyscraper presheaf functor $(i_x)_{*, pre}: \mcc \to \mathrm{PSh}(X;\mcc)$. For any presheaf $\mathcal{P} \in \mathrm{PSh}(X;\mcc)$, by the formula of left Kan extension, the presheaf stalk functor is given by the filtered colimit:
	$$ x^*_{pre}\mathcal{P} = \colim_{U \ni x} \mathcal{P}(U). $$

	Now, let $\iota: \shv(X;\mcc) \hookrightarrow \mathrm{PSh}(X;\mcc)$ be the fully faithful inclusion, and $L: \mathrm{PSh}(X;\mcc) \to \shv(X;\mcc)$ be the sheafification functor, which is left adjoint to $\iota$. Since $(i_x)_*$ lands in the subcategory of sheaves, we have $(i_x)_{*, pre} \simeq \iota \circ (i_x)_*$. Passing to left adjoints gives a canonical equivalence
	$ x^* \simeq x^*_{pre} \circ \iota $.
	Therefore, for any sheaf $\mcf \in \shv(X;\mcc)$, evaluating the stalk yields:
	$$ \mcf_x = x^*\mcf \simeq x^*_{pre}(\iota \mcf) = \colim_{U \ni x} \mcf(U) $$
	which completes the proof.
\end{proof}
The following criterion is a variant of
\cite[Corollary 2.5.4.12]{martini2025presentabilitytopoiinternalhigher}.
\begin{thm}\label{dualshvlocconst}
	Let \(X\) be a space and let \(\mcc\in\calg(\prl)\) be such that the unit
	\(\mathbf{1}\in\mcc\) is compact. Then an object
	\(\mcf\in\shv(X;\mcc)\) is dualizable if and only if it is locally constant
	and \(\mcf_x\) is dualizable in \(\mcc\) for each \(x\in X\).
\end{thm}

\begin{proof}
	($\impliedby$) Suppose $\mcf$ is locally constant and $\mcf_x$ is dualizable in $\mcc$ for each $x \in X$. Dualizability is a local property because the tensor product and internal homs in $\shv(X;\mcc)$ are computed locally (see \cref{preservinternalhom}). By assumption, there exists an open cover $\{U_i\}$ of $X$ such that $\mcf|_{U_i} \simeq \underline{K_i}$, the constant sheaf associated to some dualizable object $K_i  \in \mcc^d$. 
	For any open $U$, the constant sheaf functor $p^*: \mcc \to \shv(U;\mcc)$ (left adjoint to the global sections functor) is symmetric monoidal, hence $\underline{K_i}$ is dualizable in $\shv(U_i;\mcc)$. Thus, $\mcf$ is locally dualizable and hence dualizable in $\shv(X;\mcc)$.
	
	($\implies$)
	Suppose $\mcf \in \shv(X;\mcc)$ is dualizable. We must show it is locally constant. 
	For any point $x \in X$, the stalk functor $x^*: \shv(X;\mcc) \to \mcc$ is strongly symmetric monoidal, so it preserves dualizable objects. Thus, $K := \mcf_x$ is dualizable in $\mcc$. Because the tensor unit $\mb{1} \in \mcc$ is compact and $K$ is dualizable, $K$ is a compact object in $\mcc$. 
	
	By \cref{stalkformula}, the stalk is the filtered colimit of sections over open neighborhoods: $\mcf_x \simeq \colim_{U \ni x} \mcf(U)$. Since $K$ is a compact object in $\mcc$, the equivalence $\mathrm{id}_K: K \xrightarrow{\sim} \mcf_x$ must factor through a finite stage. That is, it arises from a morphism defined on some open neighborhood $U \ni x$:
	$$ \alpha: \underline{K}|_U \to \mcf|_U $$
	such that $\alpha_x$ is an equivalence.
	
	Since $\mcf$ is dualizable, its restriction $\mcf|_U$ is dualizable in $\shv(U; \mcc)$. We have a canonical equivalence for the internal hom:
	$$ \underline{\mathrm{Hom}}(\mcf|_U, \underline{K}|_U) \simeq (\mcf|_U)^\vee \otimes \underline{K}|_U .$$
	Taking the stalk at $x$ yields $\underline{\mathrm{Hom}}(\mcf|_U, \underline{K}|_U)_x\simeq\mcf_x^\vee \otimes K\simeq \underline{\mathrm{Hom}}(\mcf_x,K)$. 
	Combing \cref{preservinternalhom} and \cref{stalkformula}, 
	we get a natural equivalence in $\mcc$:
	$$ \underline{\mathrm{Hom}}(\mcf_x,  K) \simeq \colim_{V \subset U, x \in V}\underline{\mathrm{Hom}}( \mcf|_V, \underline{K}|_V)(V)$$
	Since $\mb{1}\in\mcc$ is compact. This implies that the inverse homotopy class at the stalk must lift to a section on some smaller open neighborhood $V \subset U$. This section corresponds to a morphism:
	$$ \beta: \mcf|_V \to \underline{K}|_V $$
	
	We now consider the composite morphisms on $V$:
	$$ \alpha \circ \beta: \mcf|_V \to \underline{K}|_V \to \mcf|_V \quad \text{and} \quad \beta \circ \alpha: \underline{K}|_V \to \mcf|_V \to \underline{K}|_V $$
	At the stalk $x$, both composites are the identity. 
	Using again the fact that $\mb{1}\in\mcc$ is compact, there must exist a sufficiently small open neighborhood $W \subset V$ containing $x$ such that  we have homotopy equivalences:
	$$ \alpha \circ \beta \simeq \mathrm{id}_{\mcf|_W} \text{\quad and \quad} \beta \circ \alpha \simeq \mathrm{id}_{\underline{K}|_W}.$$
	Consequently, $\alpha$ and $\beta$ restrict to mutually inverse equivalences on $W$, yielding $\mcf|_W \simeq \underline{K}|_W$. Iterating this procedure for all points $x \in X$ produces an open cover $\{W_x\}$ on which $\mcf$ is equivalent to a constant sheaf. This  proves that $\mcf$ is locally constant.
\end{proof}
We also record the following consequence, which may be of independent interest.
\begin{cor}
	Let $X$ be a space and $\mcc\in\calg(\prl)$ such that the unit $\mb{1}\in \mcc$ is compact. If $\pi_0\op{Pic}(\mcc)=0$, i.e. any invertible object in \mcc is isomorphic to the unit, then any invertible sheaf $\mcl\in\shv(X;\mcc)$ is locally trivial.
\end{cor}
\begin{ex}
	Any invertible object in $\shv(X;\spgeq)$ or $\shv(X;\op{Ab})$ is locally trivial.
\end{ex}
\subsection{Sheaves on a profinite space}
We prove that when $X$ is profinite,  $\shv(X;\opsp)$ can be actually identified with a module category over a connective \einfring.
\begin{lem}\label{frmcohcolimit}
	The composite
	\[
	\mathsf{DLat}\simeq \frmcoh\longrightarrow \frm
	\]
	identifies with the ind-completion functor
	\(\operatorname{Ind}\colon \mathsf{DLat}\to \frm\). It admits a right
	adjoint, which is the forgetful functor
	\[
	\frm\longrightarrow \mathsf{DLat}.
	\]
	
\end{lem}

\begin{proof}
	See
	\cite[Corollary II.2.11]{Johnstone82}.
\end{proof}
\begin{prop}\label{shvprofinite}
	Let $X$ be a profinite space. Then there exists a symmetric monoidal equivalence
	\[
	\shv(X;\spgeq)\simeq \modu_{C(X,\mathbb{S})}(\spgeq),
	\]
	where $C(X,\mathbb{S})\simeq\colimit_i \mathbb{S}^{X_i}$ if $X\simeq \lim_{i\in \mci} X_i$ is a profinite presentation. In particular, $	\shv(X;\spgeq)$ is projectively rigid over \spgeq.
\end{prop}

\begin{proof}
	We first note that $$\shv(X_i;\spgeq)\simeq \prod_{X_i}\spgeq\simeq \modu_{\mathbb{S}^{X_i}}(\spgeq),$$
	where the second equivalence follows from \cite[Corollary 4.8.5.21]{ha}.
	 Applying the adjoint $$
	 \begin{tikzcd}[sep=large]
	 	\calg(\spgeq)
	 	\arrow[r, "\modu_{(-)}(\spgeq)",hook, shift left=1ex]
	 	&
	 	\calg(\prlad)
	 	\arrow[l, "\op{End}_{\spgeq}(-)", shift left=1ex]
	 \end{tikzcd},
	 $$ it suffices to show that
	\[
	\shv(X;\spgeq)\simeq \colimit_i\shv(X_i;\spgeq)
	\]
	is a colimit in $\calg(\prlad)$. By \cref{shvsmadj}, it further suffices to show that
	\[
	\mco(X)\simeq \colimit_i \mco(X_i)
	\]
	is a colimit in $\frm$.
	
	We observe that the above diagram lies in $\frmcoh$. Since, by \cref{frmcohcolimit}, the functor $\frmcoh\to \frm$ preserves colimits, it suffices to show that
	\[
	\mco(X)^\omega\simeq \colimit_i \mco(X_i)^\omega\simeq \colimit_i\mco(X_i)
	\]
	is a colimit in the category of distributive lattices. However, filtered colimits in $\mathsf{DLat}$ are created in $\op{Posets}$, so the claim follows from the definition of the topology on $X$.
\end{proof}
\begin{rem}
In fact, when $X$ is profinite, one can show that $$\shv(X;\cmon(\mcs))\simeq \modu_{C(X,\op{Fin}^{\simeq})}(\cmon(\mcs))$$ by the same argument as \cref{shvprofinite}. However, this does not work for $\shv(X;\mcs_*)$, because $\prod_{X_i}\mcs_*\neq \modu_{(S^0)^{X_i}}(\mcs_*)$ in the pointed setting.
\end{rem}
\subsection{A counterexample to \cref{qu510}}
To delineate the limitations of our general framework, we provide an explicit topological counterexample to \cref{qu510}. This construction shows that in the recollement of atomic smashing frames, the compact-rigid generation condition is essential; without it, the existence of non-liftable atomic quotients causes the structural lifting to fail.
\begin{lem}\label{lemc4}
	Let $X\xrightarrow{f} Y$ be a map of topological spaces and $V\subset Y$ be an open subspace. Then $$\shv(f^{-1}(V))\simeq \shv(V)\otimes_{\shv(Y)}\shv(X)$$  in $\calg(\prl)$. 
\end{lem}
\begin{proof}
	It follows by combining \cref{shvsmadj} and \cref{slicelocalepushout} (cf. \cite[Corollary 1.10]{aoki2023sheaves}).
\end{proof}
\begin{ex}\label{exnonliftableatomicquotient}
	We construct a  counterexample to \cref{qu510}.
	Let
	\[
	X=\mathbb N\sqcup\{p,q\}.
	\]
	We equip \(X\) with the following topology. A subset \(U\subset X\) is open
	if and only if the following conditions hold:
	\begin{enumerate}
		\item If \(p\notin U\) and \(q\notin U\), then \(U\cap \mathbb N\) is arbitrary;
		\item if \(p\in U\) or \(q\in U\), then \(U\cap \mathbb N\) is cofinite in \(\mathbb N\).
		
	\end{enumerate}
	Equivalently, the points of \(\mathbb N\) are isolated, while every
	neighbourhood of \(p\) and every neighbourhood of \(q\) contains all but
	finitely many natural numbers. In particular, \(p\) and \(q\) cannot be
	separated by disjoint open neighbourhoods, hence $X$ is not Hausdorff.
	
	Let \(k\) be a field, and	let
	\[
	\mathcal C:=\shv(X;\mcd(k)).
	\]
	Then
	$
	\mathcal C\in \mathrm{CAlg}(\mathrm{Pr}^L_{\mathrm{st}}).
	$
	Let
	\[
	V:=\mathbb N,\qquad W:=\mathbb N\cup\{p\}.
	\]
	Both \(V\) and \(W\) are open subsets of \(X\). We define the smashing ideals
	\[
	\mathcal I:=\shv(V;\mcd(k))\xhookrightarrow{j_{V,!}} \mathcal C,
	\qquad
	\mathcal K:=\shv(W;\mcd(k))\xhookrightarrow{j_{W,!}} \mathcal C.
	\]
	Thus
	$
	\mathcal I\subset \mathcal K\subset \mathcal C.
	$
	
	We first observe that \(\mathcal I\) is \(\mathcal C\)-atomically generated.
	For each \(n\in\mathbb N\), the singleton \(\{n\}\subset X\) is clopen.
	Therefore
	$
	k_{\{n\}}:=j_{\{n\},!}k
	$
	is a direct summand of the unit \(k_X\in\mcc\) by \cref{splitstablesmid}. In particular, \(k_{\{n\}}\) is
	dualizable, hence \(\mathcal C\)-atomic.	
	Since \(V\) is discrete, we have
	\[
	\mci\simeq \prod_{n\in\mathbb N}\mcd(k).
	\]
	Under this equivalence, the objects \(k_{\{n\}}\), \(n\in\mathbb N\), are the
	coordinate units. Hence they generate \(\mci\) as a localizing ideal of \(\mathcal C\), i.e. $
	\mathcal I
	=
	\left<\,k_{\{n\}}\mid n\in\mathbb N\,\right>.
	$
	Consequently, \(\mathcal I\) is \(\mathcal C\)-atomically generated.
	
	Next we compute the quotient. The complement
	\[
	Z:=X\setminus V=\{p,q\}
	\]
	is closed, and \(Z\) is discrete with its subspace topology.
	By \cref{splitstablesmid},
	\[
	\mathcal C/\mathcal I
	\simeq
	\shv(Z;\mcd(k))
	\simeq
	\mathcal{D}(k)\times \mathcal{D}(k).
	\]
	By \cref{lemc4},
	\[
	\mathcal K/\mathcal I
	\simeq \mathcal K\otimes_{\mcc}\mcc/\mathcal I
	\simeq
	\mathcal{D}(k_{\{p\}})
	\subset
	\mathcal{D}(k_{\{p,q\}})\simeq \mcc/\mci.
	\]
	This is a clopen summand, hence it is generated by the object
	\(k_p\), which is \(\mathcal C/\mathcal I\)-atomic. Therefore
	$
	\mathcal K/\mathcal I
	$
	is \(\mathcal C/\mathcal I\)-atomically generated.
	
	We now show that \(\mathcal K\) itself is not \(\mathcal C\)-atomically
	generated. The key point is to show the following inclusion
	\[
	\mathcal K\cap \mathcal C^{\mathrm{at}}
	\subset
	\mathcal I,
	\]
	which implies $\left<\mck\cap \mathcal C^{\mathrm{at}}\right>\subset \mathcal I\subsetneq \mck$ and hence \mck is not \mcc-atomically generated.
	
	Let
	$
	A\in \mathcal K\cap \mathcal C^{\mathrm{at}}.
	$
	Since \(A\in \mathcal K=\shv(W;\mcd(k))\), its stalk at \(q\) vanishes:
	$
	A_q\simeq 0.
	$
	On the other hand, an object in
	\(\mathcal C\) is \(\mathcal C\)-atomic precisely when it is dualizable.
	Thus \(A\) is a dualizable in \mcc. By \cref{dualshvlocconst}, \(A\) is locally constant.
	Since \(A_q\simeq 0\), after shrinking around \(q\) we may choose an open
	neighbourhood \(U_q\) of \(q\) such that
	$
	A|_{U_q}\simeq 0.
	$
	Every neighbourhood of \(q\) contains all but finitely many natural numbers.
	
	Now we wish to show that $$
	A_p\simeq 0.$$ Suppose, for contradiction, that
	$
	A_p\not\simeq 0.
	$
	By local constancy, after shrinking around \(p\) we may choose an open
	neighbourhood \(U_p\) of \(p\) such that \(A|_{U_p}\) is constantly nonzero on a
	neighbourhood of \(p\). In particular, \(A\) is nonzero at all points of
	\(U_p\cap\mathbb N\) outside a finite subset. But every neighbourhood of
	\(p\) also contains all but finitely many natural numbers. Therefore
	\[
	U_p\cap U_q\cap \mathbb N
	\]
	is nonempty, indeed cofinite in \(\mathbb N\). On this intersection, \(A\)
	is zero because \(A|_{U_q}\simeq 0\), while it is nonzero by the choice of
	\(U_p\). This is a contradiction. 
	Thus both \(A_p\) and \(A_q\) vanish.
	
	Since
	\[
	\mathcal I=\ker\bigl(\shv(X;\mcd(k))\xrightarrow{} \shv(Z;\mcd(k))\bigr),
	\qquad Z=\{p,q\},
	\]
	we conclude that
	$
	A\in \mathcal I
	$
	and
	$
	\mathcal K\cap \mathcal C^{\mathrm{at}}
	\subset
	\mathcal I.
	$
	
	Consequently, we obtain a chain of smashing ideals
	\[
	\mathcal I\subset \mathcal K\subset \mathcal C
	\]
	such that \(\mathcal I\) is \(\mathcal C\)-atomically generated and
	\(\mathcal K/\mathcal I\) is \(\mathcal C/\mathcal I\)-atomically generated,
	but \(\mathcal K\) is not \(\mathcal C\)-atomically generated.
\end{ex}
\section{Pure Ideals and Pierce Ideals of Rings}\label{pureidl}
In this section, we review some basic facts about Pierce ideals and pure ideals in classical commutative algebra.
\begin{de}[Pure ideals]\label{defpureidlring}
	We say an ideal $I$ of a commutative ring $A$ is \textbf{pure} if $A/I$ is a flat $A$-module.
\end{de}
\begin{lem}[{\cite[\href{https://stacks.math.columbia.edu/tag/04PS}{04PS}]{2}}]\label{pureidlproperties}
	Let $A$ be a commutative ring. Let $I\subset A$ be an ideal. The following are equivalent:
	\begin{enumerate}[label=(\arabic*)]
		\item $I$ is pure,
		\item for every ideal $J\subset A$ we have $J\cap I=IJ$,
		\item for every finitely generated ideal $J\subset A$ we have $J\cap I=IJ$,
		\item for every $x\in A$ we have $(x)\cap I=xI$,
		\item for every ${x}\in I$ we have ${x}=y{x}$ for some $y\in I$,
		\item for every $x_1,\dots,x_n\in I$ there exists $y\in I$ such that $x_i=yx_i$ for all $i=1,\dots,n$,
		\item for every prime $\mathfrak{p}$ of $A$ we have $I A_{\mathfrak{p}}=0$ or $I A_{\mathfrak{p}}=A_{\mathfrak{p}}$,
		\item $\operatorname{Supp}(I)=\operatorname{Spec}(A)\setminus V(I)$,
		\item $I$ is the kernel of the map $A\to(1+I)^{-1}A$,
		\item $A/I\simeq S^{-1}A$ as $A$-algebras for some multiplicative subset $S$ of $A$,
		\item $A/I\simeq(1+I)^{-1}A$ as $A$-algebras.
	\end{enumerate}
\end{lem}
\begin{prop}
	Let $A$ be a commutative ring. Then any pure ideal $I$ of $A$ is a flat $A$-module.
\end{prop}
\begin{proof}
	Consider the standard short exact sequence
	$$0 \longrightarrow I \longrightarrow A \longrightarrow A/I \longrightarrow 0.$$
	We want to show that $I$ is flat. For any $A$-module $M$, look at the long exact sequence of Tor:
	$$\dots \longrightarrow \operatorname{Tor}_2^A(M,A/I) \longrightarrow \operatorname{Tor}_1^A(M,I) \longrightarrow \operatorname{Tor}_1^A(M,A) \longrightarrow \dots$$
	Since $A$, as an $A$-module, is free (hence flat), we have $\operatorname{Tor}_1^A(M,A)=0$. By the definition of a pure ideal, $A/I$ is flat, so all higher Tors vanish, in particular $\operatorname{Tor}_2^A(M,A/I)=0$. Plugging these into the displayed segment of the long exact sequence gives
	$$0 \longrightarrow \operatorname{Tor}_1^A(M,I) \longrightarrow 0,$$
	so $\operatorname{Tor}_1^A(M,I)=0$ for every $A$-module $M$. By the homological criterion for flatness, $I$ is flat.
\end{proof}
\begin{rem}
	Note that, although pure implies flat, flat does not imply pure.
	For example, take $A=\mathbb{Z}$ and $I=(2)$. Then $I$ is flat. But $I$ is not pure, since $A/I=\mathbb{Z}/2\mathbb{Z}$ has torsion and is not flat over $A$.
\end{rem}
\begin{lem}[{\cite[\href{https://stacks.math.columbia.edu/tag/00EH}{00EH}]{2}}]\label{fgidemidl}
	Let $A$ be a commutative ring. Then any finitely generated idempotent ideal $I$ is generated by a (unique) idempotent element $e\in I$.
\end{lem}
\begin{rem}
In fact, for any two idempotent elements $e_1,e_2\in A$, we have an equivalence of ideals $(e_1, e_2)=(e_1+e_2-e_1\cdot e_2)$.
\end{rem}
\begin{prop}\label{pierceincpure}
	Let $A$ be a commutative ring. There are inclusions among the following classes of ideals:
	$$\left\{ \text{Pierce ideals}\right\} \subset \left\{ \text{Pure ideals}\right\}\subset \left\{ \text{Idempotent ideals}\right\}.$$ 
\end{prop}
\begin{proof}
	For the first inclusion, since flat modules are closed under filtered colimits, the pure ideals are closed under filtered union. Therefore it suffices to show that any principle idempotent ideal $(e)\subset A$ is pure. But that is obvious because $A\simeq A/(e)\times A/(1-e)$.
	
	The second inclusion follows from \cref{pureidlproperties}(2) that $I=I\cap I=I^2$.
\end{proof}
\begin{ex}\label{counterpiercepure}
	Note that each of the inclusions above is a strict inclusion in general.
	\enu{\item 
		A pure ideal need not be generated by  idempotents.
		
		Let $A = C(\mathbb{R})$ be the ring of all real-valued continuous functions on~$\mathbb{R}$ and  $I \subset A$ be the ideal consisting of all functions with compact support.
		The ideal $I$ is pure.
		Indeed, let $f \in I$ be any function with compact support.
		Then $\operatorname{supp}(f)$ is a compact subset of~$\mathbb{R}$.
		We can choose a bump function $g \in I$ whose support is a slightly larger compact set and such that
		\[
		g(x) = 1 \quad \text{for all } x \in \operatorname{supp}(f).
		\]
		It follows that
		$
		f = f \cdot g,
		$
		which shows that $I$ is a pure ideal of~$A$.
		But $I$ is not generated by idempotent elements.
		Since $\mathbb{R}$ is connected, the only idempotent elements in the ring $C(\mathbb{R})$ are the constant functions $0$ and~$1$. 
		However, the ideal $I$ is neither $\{0\}$ nor $R$.
		Therefore, the ideal $I$ cannot be generated by  idempotent elements.
		\item An idempotent ideal is not necessarily pure.
		
		Let $(R,\mathfrak{m})$ be a non-discrete valuation ring.
		For instance, one may take $R$ to be a valuation ring whose value group contains $\mathbb{Q}$, such as the ring of Puiseux series over a field~$k$, or the valuation ring of a non-archimedean field with dense valuation.
		Let
		$
		I := \mathfrak{m}
		$
		be the maximal ideal of~$R$.
		
		We claim $I$ is idempotent, i.e.\ $I = I^2$, but it is not pure.
		Indeed, let $0 \neq x \in \mathfrak{m}$.
		Since $R$ is a non-discrete valuation ring, the value group is dense.
		Thus there exists $y \in R$ such that	
		$	v(y) = \tfrac12 v(x) > 0,$	
		so $y \in \mathfrak{m}$.
		Let $z := x/y$.
		Then
		$
		v(z) = v(x) - v(y) = \tfrac12 v(x) > 0,
		$
		hence $z \in \mathfrak{m}$.
		Consequently,
		\[
		x = yz \in \mathfrak{m}^2,
		\]
		which shows $\mathfrak{m} \subset \mathfrak{m}^2$.
		Since $\mathfrak{m}^2 \subset \mathfrak{m}$ always holds, we conclude that $I = I^2$.
		
		Recall that an ideal $I$ is pure if and only if the quotient $R/I$ is a flat $R$-module.
		Here,
		\[
		R/I = R/\mathfrak{m} = k,
		\]
		the residue field of~$R$.
		For a local ring $(R,\mathfrak{m})$, the residue field $k$ is flat over $R$ if and only if $R$ is a field.
		Since $R$ is a valuation ring of positive dimension (hence not a field), $R/\mathfrak{m}$ is not flat, and therefore $I$ is not pure.
		\item We also give a second example that an idempotent ideal is not necessarily pure.
		
		 Let $k$ be a field and $p$ be a prime. Consider the ring $A = k[x^{1/p^\infty}]$, which is formally defined as the union $\bigcup_{n \geq 0} k[x^{1/p^n}]$. 
		Let $I = (x^{1/p^\infty})$ be the ideal generated by the set $\{x^{1/p^n} \mid n \geq 0\}$. 
		
		First, observe that $A$ is an integral domain, as it is the union of integral domains $k[x^{1/p^n}]$. 
		The ideal $I$ consists precisely of all elements in $A$ that have a zero constant term. 
		Choose the element $x \in I$. Suppose for the sake of contradiction that $I$ is a pure ideal. Then there must exist some $g \in I$ such that 
		$$x = xg$$
		This implies that $x(1 - g) = 0$. Since $A$ is an integral domain and $x \neq 0$, we must have $1 - g = 0$, which yields $g = 1$.
		However, the constant polynomial $1$ has a non-zero constant term, meaning $1 \notin I$. This contradicts the assumption that $g \in I$. 
		
	}
\end{ex}
\begin{ex}\label{flatnotatomicexample}
	Let \(k\) be a field and set
	\[
	R
	:=
	k[x_1,x_2,\ldots]/
	\bigl(x_n(1-x_{n+1})\mid n\geq1\bigr).
	\]
	Then every idempotent ideal of \(R\) is pure, but the ideal
	$
	I:=(x_1,x_2,\ldots)
	$
	is pure and is not a Pierce ideal. 
\end{ex}

\begin{proof}
	The relations
	$
	x_n=x_nx_{n+1}
	$
	imply
	$x_nx_m=x_n$ when $(m>n)$.
	Hence for every \(f\in I\), there exists \(N\) such that
	\(f\in(x_1,\ldots,x_N)\), and therefore
	$
	fx_{N+1}=f.
	$
	Thus \(I\) is pure.
	
	We next show that \(R\) is stalkwise Noetherian. Put
	\[
	R_n
	=
	k[x_1,\ldots,x_n]/
	\bigl(x_i(1-x_{i+1})\mid 1\leq i<n\bigr).
	\]
	The transition map \(R_n\to R_{n+1}\) is injective, since it admits
	the retraction \(x_{n+1}\mapsto1\), and
	$
	R=\colim_nR_n.
	$
	Let \(\mathfrak p\in\spec(R)\). If \(x_n\in\mathfrak p\) for every
	\(n\), then
	\[
	\mathfrak p=I
	\qquad\text{and}\qquad
	R_{\mathfrak p}\simeq k.
	\]
	Otherwise choose \(N\) with \(x_N\notin\mathfrak p\). Inverting
	\(x_N\) forces
	\[
	x_{N+1}=x_{N+2}=\cdots=1,
	\]
	and hence
	$
	R_{\mathfrak p}
	\simeq
	(R_N)_{\mathfrak p\cap R_N}.
	$
	Thus every localization \(R_{\mathfrak p}\) is Noetherian.
	
	Now let \(J\subset R\) be idempotent. For every prime
	\(\mathfrak p\), the ideal \(J_{\mathfrak p}\) is a finitely generated
	idempotent ideal in the local Noetherian ring \(R_{\mathfrak p}\).
	If it is proper, then
	$
	J_{\mathfrak p}
	=
	J_{\mathfrak p}^2
	\subset
	\mathfrak pR_{\mathfrak p}J_{\mathfrak p},
	$
	so Nakayama's lemma gives \(J_{\mathfrak p}=0\). Hence
	\[
	J_{\mathfrak p}=0
	\quad\text{or}\quad
	J_{\mathfrak p}=R_{\mathfrak p}
	\]
	for every \(\mathfrak p\), and therefore \(J\) is pure.
	
	It remains to show that \(I\) is not Pierce. We claim that \(R\) has no
	nontrivial idempotent elements. Indeed, each \(R_n\) is connected.
	For \(n=1\) this is clear. Inductively,
	\[
	\spec(R_{n+1})
	=
	V(x_n)\cup V(1-x_{n+1}),
	\]
	where
	$
	V(1-x_{n+1})\simeq\spec(R_n),
	V(x_n)\simeq\spec(k[x_{n+1}]),
	$
	and their intersection is
	\[
	V(x_n,1-x_{n+1})\simeq\spec(k).
	\]
	Thus \(\spec(R_n)\) is connected for every \(n\).
	Since the maps \(R_n\to R\) are injective, every idempotent of \(R\)
	already belongs to some \(R_n\), and hence is \(0\) or \(1\).
	Finally,
	\[
	0\neq I\neq R
	\qquad\text{since}\qquad
	R/I\simeq k.
	\]
	As \(R\) has no nontrivial idempotents, the nonzero proper ideal \(I\)
	cannot be generated by idempotents. Thus \(I\) is pure but not Pierce.
\end{proof}
\begin{de}
	Let $A$ be a commutative ring. We denote the poset of pure ideals of $A$ by $\idl_u(A)$.
\end{de}
\begin{prop}
	Let $A$ be a commutative ring. Then the inclusion $\idl_u(A)\subset \cidem(\idl(A))$ is closed under small sums and finite meets, hence it is a subframe, which we call the pure frame of $A$. Consequently, we obtain inclusions among the following frames:
	$$\idl_p(A)\subset\idl_u(A)\subset \cidem(\idl(A)).$$
\end{prop}
\begin{proof}
	For a finite collection of pure ideals $\{I_i\}$, we have $A/(\Sigma_i I_i)\simeq \bigotimes^A_i A/I_i$ is flat. Therefore $i:\idl_u(A)\subset \cidem(\idl(A))$ is closed under finite sums. However, flat modules are closed under filtered colimits, we conclude that $i$ is closed under filtered union and hence small sums.
	
	Note that the meet of two idempotent ideals $I_1,I_2$ in $\cidem(\idl(A))$ is given by $I_1\cdot I_2$. Now given two pure ideals $I,J$, we wish to show $I\cdot J$ is pure. By \cref{pureidlproperties}(2), it suffices to show that for any ideal $K$, we have $K\cap IJ=KIJ$. But it is not hard to see it by using \cref{pureidlproperties}(2) again that $K\cap IJ=K\cap I\cap J=KI \cap J=KIJ$.
\end{proof}

\begin{prop}\label{purenilrad}
	Let $A$ be a commutative ring. If $I, J \subset A$ are pure ideals, then $\sqrt{I}=\sqrt{J}$ implies $I=J$.
\end{prop}
\begin{proof}
	If $f \in I$, then $f = fg$ for some $g \in I$. By assumption, $g^n \in J$ for
	some $n > 1$. Thus $f = fg^n \in J$. Hence, $I \subset J$. Similarly we get that
	$J \subset I$.
\end{proof}
\begin{exam}
	Note that two  distinct  idempotent ideals $I\neq J$  could have  $\sqrt{I} = \sqrt{J}$.
	
	Let $k$ be a field. Let $A$ be the  ring generated by symbols $\{x^q \mid q \in \mathbb{Q}_{\geq 0}\}$ subject to the relations $x^a \cdot x^b = x^{a+b}$ and the truncation relation $x^q = 0$ for all $q \geq 1$. Explicitly:
	\[
	A = k[x^q \mid q \in \mathbb{Q}_{\geq 0}] / (x^q \mid q \geq 1).
	\]
	Consider the two idempotent ideals $I = 0$ and
	$J = (x^q \mid 0 < q < 1)$ in $A$.
	For any generator $x^q \in J$, there exists an integer $n$ such that $nq \geq 1$. Thus $(x^q)^n = x^{nq} = 0$ in $A$. This implies $J \subset \operatorname{Nil}(A)$, and consequently $\sqrt{J} = \operatorname{Nil}(A)$.
	Therefore $I\neq J$ are distinct idempotent ideals, yet $\sqrt{I} = \sqrt{J}$. 
\end{exam}
\begin{prop}\label{idlusptial}
	Let $A$ be a commutative ring. Then the frame $\idl_u(A)$ is spatial and $\spec(\idl_u(A))$ is quasi-compact.
\end{prop}
\begin{proof}
	By \cref{purenilrad}, the composition
	\[
	\idl_u(A)\hookrightarrow \cidem(\idl(A))\hookrightarrow \idl(A)\xrightarrow{\sqrt{-}}\idl_{\op{rad}}(A)
	\]
	is a frame embedding. Since $\idl_{\op{rad}}(A)$ is spatial, it follows from \cref{subfrmspatial} that the frame $\idl_u(A)$ is also spatial.

	The quasi-compactness follows from that the unit ideal $A\in \idl_{\op{rad}}(A)$ is a compact object, because the inclusion $\idl_u(A)\hookrightarrow \idl(A)$ reflects compact objects.
\end{proof}
We refer the reader to \cite{borceux2006algebra,tarizadeh2021purely} for more details about the pure spectrum of a ring.
\begin{rem}
	Beware, however, that the almost frame $\cidem(\idl(A))$ is in general not spatial and the above argument fails, because the composition $\cidem(\idl(A))\hookrightarrow \idl(A)\xrightarrow{\sqrt{-}}\idl_{\op{rad}}(A)$ is generally not an embedding. 
\end{rem}

\begin{prop}[{\cite[\href{https://stacks.math.columbia.edu/tag/04PU}{04PU}]{2}}]
	Let $A$ be a commutative ring. The rule $I \mapsto V(I)^c$ determines a bijection $$\{I \subset A \text{ pure}\} \leftrightarrow\{U \subset \operatorname{Spec}(A) \text{ open subsets closed under specializations}\}.$$
\end{prop}
\begin{de}\label{pierceopen}
	Let $X$ be a topological space. We say an subset $U\subset X$ is a \textbf{Pierce open} subset if $U$ is a union of clopen subsets.
\end{de}
\begin{prop}\label{pierceidlopen}
	Let $A$ be a commutative ring. The rule $I \mapsto V(I)^c$ determines a bijection $$\{I \subset A \text{ Pierce}\} \leftrightarrow\{U \subset \operatorname{Spec}(A) \text{ Pierce open subsets}\}.$$
\end{prop}
\begin{proof}
	We denote $D(I):=V(I)^c$. It is a classical fact that the clopen subsets of $\operatorname{Spec}(A)$ are exactly the principal open sets $D(e)$ for idempotents $e \in A$.
	 
	 For surjectivity, let $U \subset \operatorname{Spec}(A)$ be a Pierce open subset. By  \cref{pierceopen}, $U = \bigcup_{\lambda \in \Lambda} U_\lambda$ where each $U_\lambda$ is clopen. Thus, each $U_\lambda = D(e_\lambda)$ for some idempotent $e_\lambda \in A$. The union is then $U = \bigcup_{\lambda \in \Lambda} D(e_\lambda) = D(I)$, where $I = \langle e_\lambda \mid \lambda \in \Lambda \rangle$ is the ideal generated by these idempotents. Since $I$ is generated by idempotents, it is a Pierce ideal. Noting that $D(I) = V(I)^c$, we have shown the map is surjective.
	
     The injectivity follows from \cref{purenilrad}, because every Pierce ideal is pure by \cref{pierceincpure}.
\end{proof}
\begin{de}
	Let $A$ be a commutative ring. We say $A$ is \textbf{absolutely flat} (or von Neumann regular) if every $A$-module is flat.
\end{de}
\begin{prop}[{\cite{nlabvonNeumannRegularRing}}]\label{abflatring}
	Let $A$ be a commutative ring. Then  following conditions are equivalent:
	\begin{enumerate}
		\item $A$ is absolutely flat.
		\item For every element $a \in A$, there exists $x \in A$ such that
		$
		a = axa .
		$
		
		\item For every $a \in A$, there exists an idempotent $e \in A$ such that
		$
		aA = eA .
		$
		In other words, every principal ideal of $A$ is generated by an idempotent.

		\item Every principal ideal of $A$ is a direct summand of the left $A$-module~$A$.
		
		\item Every finitely generated $A$-submodule of $A^n$ is a direct summand of $A^n$.
		\item Every finitely presented  $A$-module is projective.
	\end{enumerate}
\end{prop}
\begin{cor}\label{abflatchar}
	Let $A$ be a commutative ring. Then  following conditions are equivalent:
	\begin{enumerate}
		\item $A$ is absolutely flat.
		\item Every ideal of $A$ is a Pierce ideal.
		\item Every ideal of $A$ is pure.
	\end{enumerate}
\end{cor}
\begin{proof}
	The $(1)\implies (2)$ follows from \cref{abflatring} and the $(2)\implies (3)$ is obvious. Now we need to show $(3)\implies (1)$.
	Let $a \in R$ be an arbitrary element of the ring.
	Consider the principal ideal
	$
	I := (a).
	$
	By assumption, the ideal $I$ is a pure ideal of $R$ and $a\in I$.
	By the element-wise characterization of pure ideals,  there exists some $y \in I$ satisfying
	$
	a = ay .
	$
	Because $y \in I = (a)$, there exists an element $r \in R$ such that
	$
	y = ar .
	$
	Substituting this expression into the equation above, we obtain
	$
	a = a(ar) = a^2 r .
	$
	Thus, for every $a \in R$, there exists $r \in R$ such that
	\[
	a = a^2 r .
	\]
	By \cref{abflatring}, this shows that $R$ is  an absolutely flat ring.
\end{proof}

\section{Some (Counter)examples in the Equivariant Setting}
\begin{ex}[{A counterexample to \cref{ex:mackey-non-pierce}}]\label{counterex2}
	The inclusion $$\cidem_f(\mcc)^\omega\hookrightarrow\idl(\mcc^\heartsuit)^\omega\cap\cidem(\idl(\mcc^\heartsuit))$$ in \cref{compatsmidl} is  generally not an equivalence.
	
	For example, let $q$ be an odd prime, and let $G=C_q$. Put
	$
	S=\mathbb Z[\alpha]/(\alpha^2-\alpha-q).
	$
	Since $1+4q$ is not a square for $q$ odd, the ring $S$ is an integral domain. On the other hand,
	$
	S/qS
	\simeq
	\mathbb F_q[\alpha]/(\alpha^2-\alpha)
	\simeq
	\mathbb F_q\times \mathbb F_q,
	$
	and the image $\bar\alpha\in S/qS$ is a non-trivial idempotent.
	Let $\underline R$ be the constant Green functor associated
	to $S$. Thus
	$$
	\underline R(G/e)=S,
	\qquad
	\underline R(G/G)=S,
	$$
	the restriction map is the identity
	$
	\operatorname{Res}_e^G=\operatorname{id}_S,
	$
	the transfer map is multiplication by $q$,
	$
	\operatorname{Tr}_e^G(s)=qs,
	$
	and the Weyl action is trivial. This is a commutative Green functor. 
	
	Let $\mcc:=\operatorname{Mod}_{H\underline R}(\opsp_{G,\geq0})$. Then we have
	$
\mcc^\heartsuit\simeq	\operatorname{Mod}_{\underline R}(\operatorname{Mack}_G(\operatorname{Ab}))
	$
	with the relative box product over $\underline R$.
	Define a Green ideal $\underline I\subset \underline R$ by
	$$
	\underline I(G/e)=S,
	\qquad
	\underline I(G/G)=(q,\alpha)\subset S.
	$$
	
	This is a sub-Mackey functor: indeed,
	$
	\operatorname{Res}_e^G(q,\alpha)\subset S=\underline I(G/e),
	$
	and
	$
	\operatorname{Tr}_e^G(S)=qS\subset (q,\alpha).
	$
	It is finitely generated by $1\in \underline I(G/e)$ and
	$\alpha\in \underline I(G/G)$.
	We first check that $\underline I$ is idempotent. For a $C_q$-Mackey
	functor $M$, write
	$$
	\overline M(G)
	:=
	M(G/G)\Big/
	\operatorname{Tr}_e^G M(G/e)
	$$
	for its top bar quotient, producing a colimit-preserving functor $\operatorname{Mack}_G(\operatorname{Ab})\to\operatorname{Ab}.$ Then
	$
	\overline{\underline R}(G)=S/qS,
	$
	and the image of $\underline I$ on the top stratum is
	$
	\overline{\underline I}(G)
	=
	(q,\alpha)/qS
	=
	(\bar\alpha)
	\subset S/qS.
	$
	Since $\bar\alpha^2=\bar\alpha$ in $S/qS$, the ideal
	$(\bar\alpha)\subset S/qS$ is idempotent. On the lower stratum we have
	$
	\underline I(G/e)=S=\underline R(G/e),
	$
	which is also idempotent. Using the compatibility of the bar construction
	with the box product, the multiplication map
	$
	\underline I\otimes_{\underline R}\underline I
	\longrightarrow
	\underline I
	$
	is an isomorphism on the $e$-stratum and on the top bar stratum. Hence it
	is an isomorphism of $C_q$-Mackey functors
	$$
	\underline I\otimes_{\underline R}\underline I
	\simeq
	\underline I.
	$$
	Equivalently, at the fixed-point level one can see the idempotence from
	$$
	\underline I(G/G)^2
	+
	\operatorname{Tr}_e^G\bigl(\underline I(G/e)^2\bigr)
	=
	(q,\alpha)^2+qS.
	$$
	Since $\alpha^2=\alpha+q$, we have
	$
	\alpha=\alpha^2-q\in (q,\alpha)^2+qS,
	$
	and hence
	$
	(q,\alpha)^2+qS=(q,\alpha).
	$
	Thus the idempotence of $\underline I$ is produced by the transfer term;
	pointwise multiplication alone need not satisfy $(q,\alpha)^2=(q,\alpha)$.
	
	The almost category
	$
	\operatorname{aMod}_{\underline R,\underline I}
	\bigl(\operatorname{Mack}_G(\operatorname{Ab})\bigr)
	$
	is non-zero. 
	We now show that this almost category is not compact projectively generated.
	Let $P$ be a compact projective object of
	$
	\operatorname{Mod}_{\underline R}(\operatorname{Mack}_G(\operatorname{Ab}))
	$
	such that
	$
	\underline I\otimes_{\underline R}\underline I
	\otimes_{\underline R}P
	\simeq
	P.
	$
	We claim that
	$
	\overline P(G)=0.
	$
	
	Since $P$ is compact projective, it is a retract of a finite direct sum of
	free $\underline R$-modules of the form
	$$
	\underline R\{G/G\},
	\qquad
	\underline R\{G/e\},
	$$
	where
	$
	\underline R\{T\}(U)=\underline R(T\times U).
	$
	In particular, $P(G/e)$ and $P(G/G)$ are finitely generated projective
	$S$-modules. Since $S$ is connected, their ranks are constant on
	$\operatorname{Spec}S$. Write
	$
	r_e=\operatorname{rank}_S P(G/e),
	r_G=\operatorname{rank}_S P(G/G).
	$
	The two points of $\operatorname{Spec}(S/qS)$ are
	$
	\mathfrak p_0=(q,\alpha),
	\mathfrak p_1=(q,\alpha-1).
	$
	At $\mathfrak p_0$, the idempotent $\bar\alpha$ becomes $0$. Since
	$P$ is $\underline I$-almost, the top bar satisfies
	$
	\bar\alpha\cdot \overline P(G)=\overline P(G).
	$
	Therefore
	$
	\overline{P_{\mathfrak p_0}}(G)=0.
	$
	After base change to the residue field $\mathbb F_q$, the finite
	projective objects in the category of $\mathbb F_q$-linear
	$C_q$-Mackey functors are finite direct sums of the two indecomposable
	projectives
	$$
	Q_G:=\mathbb F_q\{G/G\},
	\qquad
	Q_e:=\mathbb F_q\{G/e\}.
	$$
	Their evaluation ranks are
	$
	Q_G(G/e)\cong \mathbb F_q,
	Q_G(G/G)\cong \mathbb F_q,
	$
	so $Q_G$ has rank pair $(1,1)$, while
	$
	Q_e(G/e)\cong \mathbb F_q^{\oplus q},
	Q_e(G/G)\cong \mathbb F_q,
	$
	so $Q_e$ has rank pair $(q,1)$. Moreover,
	$$
	\overline{Q_G}(G)\cong \mathbb F_q,
	\qquad
	\overline{Q_e}(G)=0.
	$$
	Thus the condition $\overline{P_{\mathfrak p_0}}(G)=0$ forces
	$
	P_{\mathfrak p_0}\simeq Q_e^{\oplus m_0}
	$
	for some $m_0\geq 0$. Hence
	$
	(r_e,r_G)=(qm_0,m_0).
	$
	At the other point $\mathfrak p_1=(q,\alpha-1)$, the idempotent
	$\bar\alpha$ becomes $1$. Thus the almost condition imposes no
	restriction on the top bar, and we may write
	$
	P_{\mathfrak p_1}
	\simeq
	Q_G^{\oplus n}\oplus Q_e^{\oplus m_1}
	$
	for some $n,m_1\geq 0$. Therefore
	$
	(r_e,r_G)=(n+qm_1,n+m_1).
	$
	Since $r_e$ and $r_G$ are global $S$-ranks, the two rank pairs must be
	equal:
	$$
	qm_0=n+qm_1,
	\qquad
	m_0=n+m_1.
	$$
	Subtracting $q$ times the second equation from the first gives
	$
	0=(1-q)n.
	$
	Since $q>1$, we get $n=0$. Hence
	$
	\overline{P_{\mathfrak p_1}}(G)=0.
	$
	But $\overline P(G)$ is an $(\bar\alpha)$-local $S/qS$-module, hence
	it is supported only at the component $\mathfrak p_1$. Since its fiber at
	$\mathfrak p_1$ is zero, we conclude that
	$$
	\overline P(G)=0.
	$$
	
	Consequently every $\underline I$-almost compact projective object has
	zero top bar. If the almost category were compact projectively generated,
	then the non-zero object $\underline I$ would be a quotient of a coproduct
	of such compact projective objects. However, the functor
	$$
	M\longmapsto \overline M(G)
	$$
	preserves coproducts and quotients, and it vanishes on all
	$\underline I$-almost compact projectives. Hence it would have to vanish
	on $\underline I$, contradicting
	$
	\overline{\underline I}(G)=(\bar\alpha)\neq 0.
	$
	Therefore
	$
	\operatorname{aMod}_{\underline R,\underline I}
	\bigl(\operatorname{Mack}_{C_q}(\operatorname{Ab})\bigr)
	$
	is not compact projectively generated.
\end{ex}
\begin{ex}[A counterexample to {\cref{piercenotpure}}]\label{counterex3}
In a projectively rigid $\spgeq$-algebra \mcc, Serre smashing ideals need not lie between atomic smashing ideals and smashing ideals.
		
For example, let $G = C_p$ be the cyclic group of prime order $p$ and $$\mcc:=\opsp_{C_p,\geq0}.$$ Consider the Burnside Green functor $\underline{A}$, which serves as the unit object in the symmetric monoidal category $\opsp_{C_p}^{\heartsuit}$. This Green functor is characterized by the following data:
	\begin{itemize}
		\item At the underlying level: $\underline{A}(G/e) \simeq \mathbb{Z} \cdot 1_e$.
		\item At the fixed point level: $\underline{A}(G/G) \simeq \mathbb{Z}[x]/(x^2 - px)$, where $x = \mathrm{Tr}_e^G(1_e)$ corresponds to the free orbit $[G/e]$.
		\item The restriction map $\mathrm{Res}_e^G: \underline{A}(G/G) \to \underline{A}(G/e)$ satisfies $\mathrm{Res}_e^G(1) = 1_e$ and $\mathrm{Res}_e^G(x) = p \cdot 1_e$.
	\end{itemize}
	Define an ideal $\underline{I} \subset \underline{A}$ as the sub-Mackey functor generated by the element $1_e \in \underline{A}(G/e)$. Explicitly, $\underline{I}$ is given by $\underline{I}(G/e) = \mathbb{Z}$ and $\underline{I}(G/G) = \mathbb{Z} \cdot x \subset \underline{A}(G/G)$. 
	
	We claim that $\underline{I}$ is a finitely generated Pierce ideal that fails to be pure.
	In the category of Mackey functors, the idempotency condition $\underline{I}^2 \simeq \underline{I}$ is satisfied. Although at the fixed point level the pointwise multiplication yields $\underline{I}(G/G)^2 = (x^2) = (px) \subsetneq \underline{I}(G/G)$, the tensor product multiplication incorporates the transfer from the underlying level. Specifically, the multiplication map $\underline{I} \otimes_{\underline{A}} \underline{I} \to \underline{I}$ is an epimorphism because:
	\[ \mathrm{Tr}_e^G(1_e \cdot 1_e) = \mathrm{Tr}_e^G(1_e) = x \in \underline{I}(G/G). \]
	Thus, $x$ lies in the image of the multiplication map, confirming $\underline{I}^2 = \underline{I}$ in the categorical sense.
	Consider the quotient Green functor $\underline{Q} = \underline{A}/\underline{I}$. By construction:
	\begin{itemize}
		\item $\underline{Q}(G/e) = \mathbb{Z}/\mathbb{Z} = 0$.
		\item $\underline{Q}(G/G) \simeq \mathbb{Z}\{1, x\} / (x) \simeq \mathbb{Z} \cdot \bar{1}$.
	\end{itemize}
	By the recollement $$\funct(BC_p,\opsp)\to \opsp_{C_p}\to \opsp,$$
	we see that $\amodu_{(\underline{A},\underline{I})}(\opsp_G^\heartsuit)\simeq \funct(BC_p, \op{Ab})$ is compact projectively generated, hence $\underline{I}$ is a Pierce ideal.
	
	However, $\underline{Q}$ is not flat over $\underline{A}$ (cf. \cite[Remark 6.19]{hattt}). Since \(\underline I\) is finitely generated,
	\(\underline Q=\underline A/\underline I\) is finitely presented as an
	\(\underline A\)-module. Hence, if \(\underline Q\) were flat over
	\(\underline A\), it would be projective by
	\cite[Corollary 1.28]{hattt} and the projection $p: \underline{A} \twoheadrightarrow \underline{Q}$ would split via some section $s: \underline{Q} \to \underline{A}$. Such a section must satisfy:
	\begin{itemize}
		\item At $G/G$: $s(\bar{1}) = 1 + kx$ for some $k \in \mathbb{Z}$.
		\item At $G/e$: $s(0) = 0$.
	\end{itemize}
	Compatibility with the restriction map requires $\mathrm{Res}_e^G(s(\bar{1})) = s(\mathrm{Res}_e^G(\bar{1})) = s(0) = 0$. But:
	\[ \mathrm{Res}_e^G(1 + kx) = 1 + kp \in \mathbb{Z}. \]
	Since $1 + kp = 0$ has no integer solution for $p \geq 2$, no such section exists. Therefore, $\underline{A}/\underline{I}$ is not a direct summand of $\underline{A}$, and consequently, it is not flat.
\end{ex}
\printbibliography[heading=bibintoc]

\bigskip

\textsc{Jiacheng Liang, Department of Mathematics, Johns Hopkins University, Baltimore, MD 21218, USA}

\emph{Email address:} \href{mailto:jliang66@jhu.edu}{\texttt{jliang66@jhu.edu}}

\end{document}